%% file: Involutions_of_Azumaya_Algebras_v2.tex
\documentclass[11pt, oneside, final]{amsart}
\usepackage{palatino, mathpazo}

\input{header}

\input{UriyaHeader}

\newcommand{\benw}[2][]{\ifdraft{\todo[linecolor=Green,backgroundcolor=Green!25,bordercolor=Green,#1]{#2---Ben W}\ }{}}
\newcommand{\uriyaf}[2][]{\ifdraft{\todo[linecolor=Red,backgroundcolor=Red!25,bordercolor=Red,#1]{#2---Uriya}\ }{}}
\numberwithin{thm}{subsection}

\begin{document}

\title{Involutions of Azumaya algebras}

\author{Uriya A.\ First} 
\address{Uriya A.\ First \\ Department of Mathematics \\ University of Haifa \\Haifa 31905\\ Israel}
\email{uriya.first@gmail.com}
\thanks{First was supported, in part, by a UBC postdoctoral fellowship.}
% Change address to Haifa University if the paper is submitted after October 1.

\author{Ben Williams}
\address{Ben Williams \\ Department of Mathematics \\ University of British Columbia \\Vancouver~BC V6T~1Z2\\ Canada }
\email{tbjw@math.ubc.ca}

%\date{October 22, 2017.}

\begin{abstract}
 We consider the general circumstance of an Azumaya algebra $A$ of degree $n$ over a locally ringed topos $(\cat X, \sh
 O_\cat X)$ where the latter carries a (possibly trivial) involution, denoted $\lambda$. This generalizes the usual notion of
 involutions of Azumaya algebras over schemes with involution, which in turn generalizes the notion of involutions of
 central simple algebras. 
 We provide a criterion to determine whether two Azumaya algebras with involutions extending $\lambda$
 are locally isomorphic, describe the equivalence classes obtained by this relation,
 and settle the question of when an Azumaya algebra $A$ is Brauer equivalent to an algebra carrying an involution extending
 $\lambda$, by giving a cohomological condition. We remark that these results are novel even in the case of schemes,
 since we allow ramified, non-trivial involutions of the base object. We observe that, if the cohomological condition is
 satisfied, then $A$ is Brauer equivalent to an Azumaya algebra of degree $2n$ carrying an involution. By comparison
 with the case of topological spaces, we show that the integer $2n$ is minimal, even in the case of a nonsingular
 affine variety $X$ with a fixed-point free involution. As an incidental step, we show that if $R$ is a commutative ring with
 involution for which the fixed ring $S$ is local, then either $R$ is local or $R/S$ is a quadratic \'etale extension
 of rings.
\end{abstract}
\maketitle

%%%%%%%%%%%%%%%%%%%%%%%%%%%%%%%%%%%%%%%%%%%%%%%%%%%%%%%%%%%%%%%%%

\setcounter{tocdepth}{1} % Default: 2

\tableofcontents

\setcounter{tocdepth}{2} % to make hyperref include subsections in the "side ToC" of the generated pdf.

%%%%%%%%%%%%%%%%%%%%%%%%%%%%%%%%%%%%%%%%%%%%%%%%%%%%%%%%%%%%%%%%%

\section{Introduction} \label{sec:Introduction}

\subsection{Motivation}
\label{subsec:motivation}

Let $A$ be a central simple algebra over a field $K$ and let $\tau:A\to A$ be an involution, i.e., an anti-automorphism
satisfying $a^{\tau\tau}=a$ for all $a\in A$. Recall that $\tau$ can be of the \textit{first kind} or of the \textit{second kind},
depending on whether $\tau$ restricts to the identity on the centre $K$ or not. We further say that $\tau$ is a
$\lambda$-involution where $\lambda=\tau|_K$.

Central simple algebras and their involutions play a major role in the theory of classical algebraic groups, and also in
Galois cohomology. For example, letting $F$ denote the fixed field of $\lambda:K\to K$, it is well-known that the
absolutely simple adjoint classical algebraic groups over $F$ are all given as the neutral connected component of projective
unitary groups of algebras with involution $(A,\tau)$ as above, where $K$ varies (here we also allow $K=F\times F$ with the switch
involution), see \cite[\S26]{knus_book_1998-1}. In fact, all simple algebraic groups of types $A$, $B$, $C$, $D$,
excluding $D_4$, can be described by means of central simple algebras with involution. Involutions of central simple
algebras also arise naturally in representation theory, either since group algebras admit a canonical involution, or in
the context of orthogonal, unitary, or symplectic representations, see, for instance, \cite{riehm_orthogonal_rep_2001}.

Azumaya algebras are generalizations of central simple algebras in which the base field is replaced with a ring, or more
generally, a scheme. As with central simple algebras, Azumaya algebras and their involutions are important in the study
of classical reductive group schemes, as well as in \'etale cohomology and in the representation theory of finite groups
over rings; see \cite{knus_quadratic_1991}.

Suppose that $K/F$ is a quadratic Galois extension of fields and let $\lambda$ denote the non-trivial $F$-automorphism of
$K$. A theorem of Albert, Riehm and Scharlau, \cite[Thm.~3.1(2)]{knus_book_1998-1}, asserts that a central
simple $K$-algebra $A$ admits a $\lambda$-involution if and only if $[A]$, the Brauer class of $A$, lies in the kernel
of the corestriction map $\cores_{K/F}:\Br(K)\to \Br(F)$. Saltman \cite[Thm.~3.1b]{saltman_azumaya_1978} later showed
that if $K/F$ is replaced with a quadratic Galois extension of rings $R/S$, then the class $[A]$ lies in the kernel of $
\cores_{R/S}:\Br(R)\to\Br(S) $ if and only if some representative $A'\in [A]$ admits a $\lambda$-involution. Here, in
distinction to the case of fields, an arbitrary representative may not posses an involution. However, a later proof by Knus,
Parimala and Srinivas \cite[Th.~4.2]{knus_azumaya_1990}, which applies to Azumaya algebras over schemes, implies that
one can take $A'\in [A]$ such that $\deg A'=2\deg A$.

The aforementioned results all have counterparts for involutions of the first kind in which the condition
$\cores_{R/S}[A]=0$ is replaced by $2[A]=0$.

In this article, we generalize this theory to more general sites and more general involutions. We have two purposes in
doing so. The first, our initial motivation, is to demonstrate that the upper bound in Saltman's theorem,
$\deg A'\leq 2\deg A$ guaranteed by \cite[Th.~4.2]{knus_azumaya_1990}, cannot be improved in general for
involutions of the second kind. The statement in the case of involutions of the first kind was established in
\cite{asher_auel_azumaya_2017}. Our general approach here is similar to \cite{asher_auel_azumaya_2017},
\cite{antieau_unramified_2014} and related works. That is, the desired example is constructed by approximating a
suitable classifying space, and topological obstruction theory is used to show that it has the required properties. In contrast with \cite{asher_auel_azumaya_2017} and
\cite{antieau_unramified_2014}, the obstruction is obtained by means 
of equivariant homotopy theory.\uriyaf{Remove
or edit this sentence if you wish.} 

We therefore introduce and study involutions of the second kind of Azumaya algebras on topological spaces. In
fact, we develop the necessary foundations in the generality of connected locally ringed topoi with involution, and show that
Saltman's theorem holds in this setting. In doing so, we stumbled into our second purpose, which we now explain.

Any involution of a field is either trivial or comes from a quadratic Galois extension, which is why the classical theory sees a dichotomy into
involutions of the first or second kind. For a ring, the analogous involutions are the trivial involutions or those
arising as the non-trivial automorphism of a quadratic \'etale extension $R/S$. Geometrically, these correspond to extreme cases where one has
either a trivial action of the cyclic group $C_2 = \{1, \lambda\}$ on a scheme, or where the action is
scheme-theoretically free. One may also view the free case as corresponding to an \emph{unramified} map
$\pi: X \to X/C_2$. This dichotomy has been preserved in the literature on involutions of Azumaya algebras over schemes,
say for instance \cite{knus_azumaya_1990} and \cite{knus_quadratic_1991}, by considering only trivial or unramified
involutions of the base ring.

There are, of course, involutions $\lambda:R\to R$ which are neither trivial nor wholly unramified. For instance, one
may encounter involutions of varieties that are generically free but fix a nonempty closed subscheme. Alternatively,
there are involutions of nonreduced rings that restrict to trivial involutions of the reduction---these are
geometrically ramified everywhere, but nonetheless non-trivial.

Our second purpose therefore became developing the theory of $\lambda$-involutions on Azumaya algebras with minimal
assumptions on $\lambda$. 
We establish a generalization of Saltman's theorem, and present a
classification of $\lambda$-involutions into \emph{types}, generalizing the classification of involutions 
of central simple algebras as
\emph{orthogonal}, \emph{symplectic}---both of the first kind---or \emph{unitary}---of the second.

In more detail, given a field $K$ of characteristic not $2$ and an involution $\lambda:K\to K$, recall that two
degree-$n$ central simple $K$-algebras with involutions extending $\lambda $ are of the same type if they become
isomorphic after base change to a separable closure of the fixed field of $\lambda$. This definition extends
naturally to the case of a general connected ring $R$ in which $2$ is a unit by replacing ``a separable closure'' by
an \'etale extension of $S$, the fixed ring of $\lambda:R\to R$. It is natural to ask how many types are obtained in
this manner, and how to distinguish them effectively. In the classically-considered cases of trivial
or %quadratic-\'etale
unramified involutions on $R$, the situation is known to be similar to case of fields: When $R=S$, there are at most two types
--- the orthogonal, which occurs for all $n$, and the symplectic, which occurs only for even $n$. When $R/S$ is
quadratic \'etale, only one type, called the unitary type, occurs for all $n$.

We describe the types for arbitrary $\lambda:R\to R$ and 
give a cohomological criterion to determine when two involutions
are of the same type. 
This criterion implies in particular that the type of an Azumaya algebra with involution $(A, \tau)$ 
is determined
entirely by the restriction of $(A, \tau)$ to the ramification locus of $ \Spec R \to \Spec S$.
More than two types may occur. With this new subtlety, one can further ask, in the context
of Saltman's theorem, what are the types of $\lambda$-involutions which can be exhibited
on representatives of a given Brauer class in $ \Br R$. Our generalization of Saltman's theorem
answers this question.

To demonstrate some of the ideas above, let us consider a field $k$ of characteristic different from $2$ and the ring
$R:=k[x, x^{-1}]$ of Laurent polynomials with the involution $\lambda: x \mapsto x^{-1}$. The fixed ring is
$S:=k[x+x^{-1}]$. The map $\Spec R \to \Spec S$ is ramified at two points, $x=1$ and $x=-1$, and unramified
elsewhere. Our results show that there are $4$ types of $\lambda$-involution for even-degree algebras and $1$ type in
odd degrees. Furthermore, the type of a $\lambda$-involution is determined by the types --- orthogonal or symplectic
--- obtained by specializing to $x=1$ and $x=-1$. For example, consider the $\lambda$-involution of
$\Mat_{2\times 2}(R)$ given by
\begin{equation}
 \label{eq:7}
 \tau: \begin{bmatrix}a(x) & b(x) \\ c(x) & d(x) \end{bmatrix}
 \mapsto \begin{bmatrix} d(x^{-1}) & x^{-1} b (x^{-1}) \\ xc(x^{-1}) &
a(x^{-1}) \end{bmatrix}.
\end{equation}
Evaluating at $x=1$, the involution of \eqref{eq:7} becomes orthogonal, whereas evaluating at $x=-1$ makes
it symplectic. Our generalization of Saltman's theorem implies that if $\alpha\in \Br R$ is represented
by an Azumaya $R$-algebra admitting a $\lambda$-involution, then each of the $4$ types of $\lambda$-involutions
is the type of a $\lambda$-involution of some representative of $\alpha$.

It seems likely that our results on types could be used
to extend the theory of involutive Brauer groups, intiated in \cite{parimala_92} (see also \cite{verschoren_98}),
to schemes carrying ramified involutions. We hope to address this in subsequent work.\uriyaf{Added
this paragraph. Do we hope to address this in subsequent work?
It might make sense of the $C_2$-equivariant cohomology of $\Gm$.}

We finally note that from the point of view of group schemes, the study of $\lambda$-involutions of Azumaya algebras in
the case where is $\lambda$ neither trivial nor unramified amounts to studying certain group schemes over $\Spec R$
which are generically reductive but degenerate on a divisor. Specifically, the projective unitary group of an Azumaya
algebra with a $\lambda$-involution is generically of type $A$ and degenerates to types $B$, $C$ or $D$ on the connected
components of the branch locus of $\Spec R\to\Spec S$. The study of degenerations of reductive groups have proved useful in
many instances. Recent examples include \cite{auel_parimala_suresh_2015} and \cite{bayer_17}, but this manifests even
more in the works of Bruhat and Tits on reductive groups over henselian discretely valued fields \cite{bruhat_I_72},
\cite{bruhat_II_84}, \cite{bruhat_III_87}. Broadly speaking, degenerations of reductive groups are encountered
naturally when one attempts to extend a group scheme defined on a generic point of an integral scheme to the entire
scheme, a process which is often considered in number theory.

\subsection{Outline}

Following is a detailed account of the contents of this paper, mostly in the order of presentation. While the majority
of this work applies to schemes without assuming $2$ is invertible, we make this assumption here in order to avoid certain
technicalities.

\medskip

Section \ref{sec:preliminaries} is devoted to technical preliminaries, largely to do with non-abelian cohomology in
the context of Gorthendieck topoi. 

\medskip

Let $X$ be a scheme and let $\lambda:X\to X$ be an involution. Our first concern is to specify an appropriate quotient
of $X$ by the group $C_2=\{1,\lambda\}$. There is an evident choice when $X=\Spec R$ with $R$ a ring, since one can take
the quotient to be $\Spec S$, where $S$ is the fixed ring of $\lambda:R\to R$. 
However, at the level of generality that we
consider, there is often more than one plausible option. For instance, if the action of $\lambda$ is not free, then $[X/
C_2]$, a Deligne--Mumford stack, might serve just as well as the scheme or algebraic space $X/C_2$. The difference
between these alternatives becomes particularly striking when $C_2$ acts trivially on $X$ --- the quotient $X \to
X/C_2 = X$ can be regarded as a degenerate case of a double covering, ramified everywhere, whereas $X \to [X/ C_2]$ is a
$C_2$-Galois covering, ramified nowhere. From the point of view of the first quotient, all involutions will appear to be of the
first kind, whereas with respect to the second quotient, all involutions will appear to be of the second kind.

We are therefore led to conclude that a chosen quotient $\pi: X \to Y$, in addition to $X$ and $\lambda$, is necessary
in order to discuss involutions in a way consistent with what is already done in the cases where $\lambda$ is an
involution of a ring.

We require a quotient to satisfy certain axioms, presented in Subsection \ref{subsec:quotients}, and prove that they are
satisfied in a number of important examples, notably when the categorical quotient $X/C_2$ exists in the category of
schemes and is a \emph{good quotient}. Such quotients exist for instance if $X$ is affine or projective, see
Theorem~\ref{TH:important-exact-quotients}. Thereafter in the development of the theory, we are usually agnostic about
the quotient chosen. In examples, we often return to the motivating case of a good quotient.

Consider, therefore, a good quotient $\pi: X \to Y=X/C_2$. It is technically easier to work on $Y$ than on
$X$. Specifically, by virtue of our Theorem \ref{TH:pi-star-equiv}, there is an equivalence between Azumaya algebras
with $\lambda$-involution on $X$ on the one hand and Azumaya algebras with $\pi_* \lambda $-involution over the sheaf of
rings $R:=\pi_*(\sh{O}_X)$ on the other. We therefore study Azumaya algebras over $R$. While $Y$ does not carry an
involution, the ring sheaf $R$ has an involution, namely, $\pi_*\lambda$, 
which we abbreviate to $\lambda$. A difficulty that we encounter
here is that the sheaf of rings $R$ is not a local ring object on $Y$, but rather a sheaf of rings with involution, the
fixed subsheaf of which is the local ring object $(\pi_*\sh O_X)^{C_2} = \sh{O}_Y$. We devote considerable work to the
study of commutative rings with involutions whose fixed subrings are local in Section \ref{sec:ringsWithInvolution}, and
conclude in Theorem \ref{thm:local-involution} that any such ring is a semilocal ring, so that the sheaf $R$ may be
viewed as making $Y$ a ``semilocally ringed'' space.

\medskip

In Section~\ref{sec:types}, we introduce and study \emph{types} of $\lambda$-involutions. Specifically, we define two
Azumaya $R$-algebras with a $\lambda$-involution, $(A,\tau)$ and $(B,\sigma)$, to be of the same \textit{type} if some
matrix algebra over $(A,\tau)$ is $Y_\et$-locally isomorphic to some matrix algebra over $(B,\sigma)$. We show in
Theorem~\ref{TH:ct-determines-local-iso} and Corollary~\ref{CR:coarse-type-determines-type} 
that the collection of types forms a $2$-torsion group whose product
rule is compatible with tensor products, and when $\deg A=\deg B$, the involutions $\tau$ and $\sigma$ have the same
type if and only if $(A,\tau)$ and $(B,\sigma)$ are $Y_\et$-locally isomorphic, without the need to pass to matrix algebras.
Thus, the definition given here agrees with the definition in Subsection~\ref{subsec:motivation}.
We then turn to the problem of calculating the group of types in specific cases.

Let $W \subset Y$ denote the \textit{branch locus} of $\pi :X \to Y$. Then, away from $W$, the $C_2$-action on
$V= X - \pi^{-1}(W)$ is unramified, hence there is only one possible type of $\lambda$-involution on $A|_V$, viz.\
unitary, and all involutions on $A|_V$ are locally isomorphic to the involution
$\Mat_{n \times n}(R) \to \Mat_{n\times n}(R)$ given by applying the involution $\lambda$ to each entry in the matrix
and then taking the transpose, i.e., $M \mapsto (M^\lambda)^\tr$. In contrast, over a connected
component $Z_1$ of $Z:= \pi^{-1}(W)$, regarded as a reduced closed subscheme of $X$, the involution $\lambda$ restricts
to the identity (Proposition~\ref{PR:Z-to-W-homeomorphism}), and so $\lambda$-involutions of $A|_{Z_1}$ fall into one of
two types --- orthogonal or symplectic. This suggests that the types of $\lambda$-involutions over $X$ should be in
bijection with $\Hoh^0(Z, \mu_2)$, where $\mu_2:=\{1,-1\}$ and $1$ and $-1$
represent orthogonal and symplectic involutions respectively, and that two $\lambda$-involutions are of the same type if and only if they are of the same type when
restricted to each connected component of $Z$. 
We prove the second statement in Theorem~\ref{TH:types-for-schemes}
and establish the first under the assumption that that $Y$ is noetherian and regular
in Corollary~\ref{CR:smooth-scheme-standard}. We do not know whether the first
statement holds in general.
Determining the type of a
given involution of a given algebra, $\tau : A \to A$, can now be carried out by considering the rank of the
sheaf of $\tau$-symmetric elements on the various components of $W$; see \cite[Prp.~2.6]{knus_book_1998-1}.

\medskip

In Section \ref{sec:Saltman}, we turn to the question of when a Brauer class $\alpha=[A] \in \Br(X)$ contains an algebra
$A'$ possessing a $\lambda$-involution. Saltman \cite[Thm.~3.1]{saltman_azumaya_1978} gave necessary and
sufficient conditions for this when $\lambda$ is trivial or unramified. Specifically, $A$ is equivalent to such an
algebra if $2[A]=0 \in \Br(X)$, in the case of a trivial action, or if $\cores_{X/Y}[A] = 0 \in \Br(Y)$, in the case of an
unramified action. We unify these two results, and generalize to the cases that are neither trivial nor unramified, by
defining a transfer map $\transf: \Br(X) \to \Hoh_\et^2(Y, \Gm)$, and deducing in Theorem~\ref{TH:Saltman-ramified} that
$A$ is equivalent to an algebra admitting a $\lambda$-involution of type $t$ if and only if $\transf([A]) = \Phi(t)$,
where $\Phi(t)\in\Hoh^2(Y,\Gm)$ is a cohomology class depending on the type. In both extreme cases of trivial and
unramified actions, and in fact whenever $Y$ is a nonsingular variety, $\Phi(t)$ is necessarily $0$. Moreover, in the
 case of a trivial action, $\transf([A]) = 2[A]$, and in the unramified case, $\transf([A]) = \cores_{X/Y}[ A ]$, so we
recover Saltman's theorem as a special case. We also show that if $A$ is equivalent to an algebra with involution, then
such an algebra can be constructed to have degree twice that of $A$, thus extending the analogous result of \cite[Thms.~4.1,
4.2]{knus_azumaya_1990}. We do not, however, follow \cite{saltman_azumaya_1978} and \cite{knus_azumaya_1990} in
considering the \textit{corestriction algebra} of $A$, taking instead a purely cohomological approach. In fact, it is
not clear whether a corestriction algebra of $A$ can be defined in a meaningful way when $\lambda:X\to X$ is ramified.
This problem was considered in \cite[\S5]{auel_parimala_suresh_2015}, where some positive results are given, and we
leave its pursuit in the current level of generality to a future work.

\medskip

Section \ref{sec:exapp} gives a number of examples of the workings out of the previous theory. In particular, we give
examples of schemes $X$ with involutions $\lambda: X\to X$ that are neither unramified nor trivial, along with a
classification of the various types of $\lambda$-involutions of Azumaya algebras, e.g., Examples \ref{ex:mixed} and
\ref{ex:semistandard}.

\medskip

While this overview has so far been written in the language of schemes, the majority of the results are established in
the setting of locally ringed Grothendieck topoi, of which the \'etale ringed topoi of a scheme is a special case. The
advantage of this generality is that all the results above also apply, essentially verbatim, to Azumaya algebras with
involution over a topological $C_2$-space, or to Azumaya algebras with involutions on algebraic stacks.
The applicability of our
results in the context of other 
sites associated with schemes, e.g., the Zariski site, the fppf site, the Nisnevich site
and some large sites, is discussed in Subsection~\ref{subsec:examples-of-exact-quo}.

Comparison of Azumaya algebras over schemes with topological Azumaya algebras has proved
useful in the
past, for instance in \cite{antieau_unramified_2014}, \cite{antieau_topology_2015}.
Having the previous theory available also in the topological context,
we consider a 
finite type, regular $\CC$-algebra $R$ with an unramified involution $\lambda$ and compare the theory of Azumaya
$R$-algebras with involutions restricting to $\lambda$ on the centre with the theory of topological Azumaya algebras
with involution on the complex manifold $(\Spec R)_{\text{an}}$. This is carried out in 
Subsection~\ref{subsec:morphisms}, specifically in
Example \ref{ex:complexRealization}.

By such comparison, we produce an example of an Azumaya algebra $A$ of degree $n$, over a ring $R$ with an unramified
involution $\lambda$, having the property that $A$ is Brauer equivalent to an algebra $A'$ with $\lambda$-involution,
but the least degree of such an $A'$ is $2n$, Theorem \ref{th:mainCounterexample}; 
the bound $2n$ is the lowest possible by
\cite[\S4]{knus_azumaya_1990}, which guarantees the existence of $A'$ of degree $2n$ in general. 
An analogous example in the case
where $\lambda$ is assumed to be trivial was given in \cite{asher_auel_azumaya_2017}.
The method of proof, which is carried out in Sections~\ref{sec:topology}
and~\ref{sec:the-example}, is by using existing study of
bundles with involution as a branch of equivariant homotopy theory, \cite{may_equivariant_1996}. In particular, we can
find universal examples of topological Azumaya algebras with involution, which are valuable sources of counterexamples.

\medskip

In an appendix, we give a proof that the stalks of the sheaf of continuous, complex-valued functions on a topological space
$X$ satisfy Hensel's lemma. This is used here and there in the body of the paper to treat this case at the same time as \'etale sites of schemes.

\subsection{Acknowledgments}

The authors would like to thank Zinovy Reichstein for introducing them to each other and recommending that they study
involutions of Azumaya algebras from a topological point of view. They would like to thank Asher Auel for helpful
conversations and good ideas, some of which appear in this paper. They owe an early form of an argument in
\ref{subsec:quotient-scheme} to Sune Precht Reeh. The second author would like to thank Omar Antol\'in, Akhil Mathew,
Mona Merling, Marc Stephan and Ric Wade for various conversations about equivariant classifying spaces, and Bert Guillou
for a reference to the literature on equivariant model structures. The second author would like to thank Ben Antieau for
innumerable valuable conversations about Azumaya algebras from the topological point of view, and would like to thank
Gwendolyn Billett for help in deciphering \cite{giraud_cohomologie_1971}.
We also thank the referees for many valuable suggestions.

\section{Preliminaries}
\label{sec:preliminaries}

This section recalls necessary facts and sets notation for the sequel. Throughout, $\bfX$ denotes a Grothendieck topos.
 We reserve the term ``ring'' for commutative unital rings, whereas algebras are assumed unital but not necessarily
 commutative.

\subsection{Generalities on Topoi}
\label{subsec:topos}

Recall that a Grothendieck topos is a category that is equivalent to the category of set-valued sheaves over a small
site, or equivalently, a category satisfying Giraud's axioms; see \cite[Chap.~0]{giraud_cohomologie_1971}. In this paper
we shall be particularly interested in the following examples:
\begin{enumerate}[label=(\roman*)]
\item $\bfX=\Sh(X_\et)$, the category of sheaves over the small \'etale site of a scheme $X$.
\item $\bfX=\Sh(X)$, the category of sheaves on a topological space $X$.
\end{enumerate} 
We will occasionally consider other sites associated with a scheme $X$. In particular, 
$X_{\Zar}$ and $X_\fppf$ will
denote the small Zariski and small fppf sites of
 $X$, respectively.

The topos of sheaves over a singleton topological space, which is nothing but the category of sets, will be denoted
$\mathbf{pt}$.

We note that every topos $\mathbf{X}$ can be regarded as a site relative to its canonical topology. In this case, a
 collection of morphisms $\{U_i\to V\}_{i\in I}$ is a covering of $V$ if and only if it is jointly surjective, and every
 sheaf over $\mathbf{X}$ is representable, so that $\mathbf{X}\iso \Sh(\mathbf{X})$. This allows us to define objects of $\bfX$
 by specifying the sheaf that they represent, and to define morphisms between objects by defining them on sections.

 The symbols $\emptyset_{\bfX}$ and $*_{\mathbf{X}}$ will be used for the initial and final objects of $\mathbf{X}$,
 respectively. When $\mathbf{X}=\Sh(X)$ for a site $X$, the sheaf $\emptyset_{\bfX}$ assigns an empty set to every
 non-initial object of $X$, and $*_{\bfX}$ is the sheaf assigning a singleton to every object in $X$. The subscript
 $\bfX$ will dropped when it may be understood from the context.

 For every object $A,U $ of $\bfX $ the $U$-sections of $A$ are
 \[ A(U):=\Hom_{\mathbf{X}}(U,A)\]
 and the global sections of $A$ are $\Hoh^0(\mathbf X, A) = \Gamma A=\Gamma_\bfX A:=A(*)$. We will write $A_U=A\times U$,
 and will regard $A_U$ as an object of the slice category $\mathbf{X}/U$.

 By a \textit{group} $G$ in $\mathbf{X}$ we will mean a group object in $\mathbf{X}$. In this case, the $U$-sections $G(U)$
 form a group for all objects $U$ of $\mathbf X$. Similar conventions will apply to abelian groups, rings, $G$-objects, and so
 on.

 If $R$ is a ring object in some topos, then $\mu_{2,R}$ will denote the object of square-roots of $1$ in $R$, that is, the
 object given section-wise by $\mu_{2,R}(U) = \{ x \in R(U) \suchthat x^2=1\}$. The bald notation $\mu_2$ will denote
 the constant sheaf $\{ +1 ,-1\}$.

\subsection{Torsors}

\begin{definition} \label{def:torsors}
 Let $X$ be a site and let $G$ be a sheaf of groups on $X$. A (right) \textit{$G$-torsor} is a sheaf $P$ on $X$ equipped
with a right action $P\times G\to P$ such that $P$ is locally isomorphic to $G$ as a right $G$-object.
\end{definition}

Equivalently, and intrinsically to the topos $\mathbf X:=\Sh(X)$, a (right) $G$-torsor is an object $P$ of $\mathbf X$
equipped with a (right) $G$-action $m: P \times G \to P$ such that the unique morphism $P \to \ast$ is an epimorphism
and such that the morphism $P \times G \overset{\pi_1 \times m}{\longrightarrow} P \times P$ is an isomorphism. %This definition is the original one,
See \cite[D\'ef.~III.1.4.1]{giraud_cohomologie_1971} where more general torsors over
objects $S$ of $\mathbf X$ are defined; our definition is that of torsors over the terminal object.

The equivalence of the two definitions of ``torsor'' is given by \cite[Prop.~III.1.7.3]{giraud_cohomologie_1971}.

 The category of $G$-torsors, with $G$-equivariant isomorphisms as morphisms, will be denoted
 \[ \Tors(\bfX,G). \]

A $G$-torsor $P$ is \emph{trivial} if $P\iso G$ as right $G$-objects, and an object $U$ is said to \textit{trivialize}
$P$ if $P_U\iso G_U$ as $G_U$-objects. 
The latter holds precisely when $P(U)\neq \emptyset$.
 
 Recall that if $P$ is a $G$-torsor and $X$ is a left $G$-object in $\bfX$, then $P\times^G X$ denotes the quotient of
 $P\times X$ by the equivalence relation $P\times G\times X\to (P\times X)\times (P\times X)$ given by $(p,g,x)\mapsto
 ((pg,x),(p,gx))$ on sections. We shall sometimes denote $P\times^G X$ by ${}^PX$ and call it the \emph{$P$-twist} of
 $X$. We remark that $X$ and ${}^PX$ are locally isomorphic in the sense that there exists a covering $U\to *$ in $\bfX$
 such that $X_U\iso {}^PX_U$ --- take any $U$ such that $G_U\iso P_U$. If $X$ posses some additional structure, for
 instance if $X$ is an abelian group, and $G$ respects this structure, then ${}^PX$ also posses the same structure and
 the isomorphism $X_U\iso {}^PX_U$ respects the additional structure. The general theory outlined here is established
 precisely in \cite[Chap.~III]{giraud_cohomologie_1971}.

\begin{remark}
 There is another plausible definition of ``torsor'' on a site $X$, particularly when the topology is subcanonical and
 when the category $X$ has finite products---i.e., $X$ is a \textit{standard site}. That is, one modifies the
 definition in \ref{def:torsors} by requiring the objects $G$ and $P$ to be objects of the site $X$. These are the
 \textit{representable torsors} as distinct from the \textit{sheaf torsors} defined above. We will not consider the
 question of representability in this paper beyond the following remark:
 Suppose $X$ is a scheme and $G$ is a group scheme over $X$. Then $G$ represents a group sheaf on
 the big flat site of $X$, also denoted
 $G$. If $G\to X$ is affine, then all sheaf $G$-torsors
 are representable by an $X$-scheme \cite[Thm.~III.4.3]{milne_etale_1980}.
\end{remark} 

\subsection{Cohomology of Abelian Groups}
\label{subsec:coh-of-ab-grps}

The functor $\Hoh^0$ sending an abelian group $A$ in $\mathbf{X}$ to its global sections is left exact. The $i$-th right
derived functor of $\Hoh^0$ is denoted $\Hoh^i(\mathbf{X},A)$, as usual. If $\mathbf{X}$ is clear from the context, we shall
simply write $\Hoh^i(A)$. When $\mathbf{X}=\Sh(X_\et)$ for a scheme $X$, we write $\Hoh^i(\mathbf{X},A)$ as
$\Hoh^i_{\et}(X,A)$, and likewise for other sites associated with $X$.

In the sequel, we shall make repeated use of Verdier's Theorem, quoted below, which provides a description of cohomology
classes in terms of hypercoverings. We recall some details, and in doing so, we set notation. One may additionally
consult \cite[Tag 01FX]{de_jong_stacks_2017}, \cite{dugger_hypercovers_2004} or \cite[Exp.~V.7]{artin_theorie_1972}.

\medskip

Let $\Simp$ denote the category having $\{\{0,\dots,n\}\where n=0,1,2,\dots\}$
as its objects and the non-decreasing functions as its morphisms.
Recall that a \textit{simplicial object} in $\bfX$ is a contravariant functor
$U_\bullet:\Simp\to \bfX$. 
For every $0\leq
i\leq n$, we write $U_n=U_\bullet(\{0,\dots,n\})$ and set $d^n_i=U_\bullet(\delta^n_i)$ and $s^n_i=U_\bullet(\sigma_i^n)$, where $\delta^n_i:\{0,\dots,n-1\}\to
\{0,\dots,n\}$ is the non-decreasing monomorphism whose image does not include $i$ and $\sigma^n_i:\{0,\dots,n+1\}\to
\{0,\dots,n\}$ is the non-decreasing epimorphism for which $i$ has two preimages. We shall write $d_i,s_i$ instead of
$d_i^n,s_i^n$ when $n$ is clear from the context. Since the morphisms $\{\sigma^n_i,\delta^n_i\}_{i,n}$ generate
$\Simp$, in order to specify a simplicial object $U_\bullet$ in $\bfX$, it is enough to specify objects $\{U_n\}_{n\geq
0}$ and morphisms $s^n_i:U_n\to U_{n+1}$, $d^n_i:U_n\to U_{n-1}$ for all $0\leq i\leq n$. Of course, the morphisms
$\{s^n_i,d^n_i\}_{i,n}$ have to satisfy certain relations, which can be found in \cite{may_simplicial_1992}, for
instance.

For $n\geq 0$, let $\Simp_{\leq n}$ denote the full subcategory of $\Simp$ whose objects are
$\{\{0\},\dots,\{0,\dots,n\}\}$. The restriction functor $U_\bullet\mapsto U_{\leq n}:\Fun(\Simp^\op,\bfX)\to
\Fun(\Simp^\op_{\leq n},\bfX)$ admits a right adjoint
called the $n$-th coskeleton and denoted $\cosk_n$. We also write $\cosk_n(U_\bullet)$ for
$\cosk_n(U_{\leq n})$. The simplicial object $U_\bullet$ is called a \emph{hypercovering} (of the terminal object) if $U_0\to *$ is a covering
and for all $n\geq 0$, the map $U_{n+1}\to \cosk_n(U_\bullet)_{n+1}$ induced by the adjunction is a covering. For
example, when $n=0$, the latter conditions means that $(d_0^1,d_1^1):U_1\to U_0\times U_0$ is a covering.

Hypercoverings form a category in the obvious manner, morphisms being natural transformations.

\begin{example}\label{EX:Cech-hypercovering}
 Let $U\to *$ be a morphism in $\bfX$. Define $U_n=U\times \dots\times U$ ($n+1$ times), let $d_i^n:U_{n}\to U_{n-1}$ be
 the projection omitting the $i$-th copy of $U$ and let $s_i^n:U_n\to U_{n+1}$ be given by $(u_0,\dots,u_n)\mapsto
 (u_0,\dots,u_i,u_i,\dots,u_n)$ on sections. These data determine a simplicial object $U_\bullet$ which is a
 hypercovering if $U\to *$ is a covering. In this case, the map $U_{n+1}\to \cosk_n(U_\bullet)_{n+1}$ is an isomorphism
 for all $n$. The hypercovering $U_\bullet$ is called the \emph{\v{C}ech hypercovering} associated to $U$. If
 $U_\bullet$ is an arbitrary hypercovering, then $\cosk_0(U_\bullet)$ is the \v{C}ech hypercovering associated to $U_0$.
\end{example}

The following lemma is fundamental.

\begin{lem}[{\cite[Lm.~24.7.3]{de_jong_stacks_2017} or \cite[Th.~V.7.3.2]{artin_theorie_1972}}]
\label{LM:hypercover-refinement} Let $U_\bullet$ be a hypercovering and let $V\to U_n$ be a covering. Then there exists
a hypercovering morphism $U'_\bullet\to U_\bullet$ such that ${U'_n\to U_n}$ factors through $V\to U_n$.
\end{lem}

Let $A$ be an abelian group object of $\mathbf X$. With any hypercovering $U_\bullet$ in $\bfX$ we associate a cochain
complex $C^\bullet(U_\bullet,A)$ defined by $C^n(U_\bullet,A)=A(U_n)$ for $n\geq 0$ and $C^n(U_\bullet,A)=0$
otherwise. The coboundary map $d^n:C^n(U_\bullet,A)\to C^{n+1}(U_\bullet,A)$ is given by $d^n(a)=\sum_{i=0}^{n+1} (-1)^i
d^{*}_i(a)$, as usual; here $d^{*}_i=(d^{n+1}_i)^*:A(U_n)\to A(U_{n+1})$ is the map induced by $d^{n+1}_i:U_{n+1}\to
U_n$. The cocycles, coboundaries, and cohomology groups of the complex are denoted $Z^n(U_\bullet,A)$,
$B^n(U_\bullet,A)$ and $\Hoh^n(U_\bullet,A)$. Any morphism of hypercoverings $U'_\bullet\to U_\bullet$ induces a
morphism $\Hoh^n(U_\bullet,A)\to \Hoh^n(U'_\bullet,A)$ in the obvious manner.

\begin{theorem}[{Verdier \cite[Th.~V.7.4.1]{artin_theorie_1972}}]\label{TH:Verdier} Let $\mathbf X$ be a topos and $A$ an
 abelian group object in $\mathbf X$. The functors
\[ A\mapsto \Hoh^n(\bfX,A)\qquad\text{and} \qquad A \mapsto \colim_{U_\bullet} \Hoh^n(U_\bullet,A)
\] from the category of abelian groups in $\bfX$ to the category of abelian groups are naturally isomorphic. Here, the
colimit is taken over the category of hypercoverings.
\end{theorem}

\begin{remark} If we were to take the colimit in the theorem over the category of the \v{C}ech hypercoverings, then the
result would be the \v{C}ech cohomology of $A$. Consequently, the \v{C}ech cohomology and the derived-functor
cohomology agree when every hypercovering admits a map from a \v{C}ech hypercovering. This is known to be the case when
$\bfX=\Sh(X)$ for a paracompact Hausdorff topological space $X$
 \cite[Th.~5.10.1]{godement_topologie_1973}, or $\bfX=\Sh(X_\et)$ for a noetherian scheme $X$ such that any finite subset
of $X$ is contained in an open affine subscheme \cite[\S4]{artin_joins_1971}.
\end{remark}

A short exact sequence of abelian groups $1 \to A'\to A\to A'' \to 1$ in $\bfX$ gives rise to a long exact sequence of
cohomology groups. By the second proof of \cite[Tag 01H0]{de_jong_stacks_2017}, quoted as Theorem~\ref{TH:Verdier}
here, the connecting homomorphism $\delta^n:\Hoh^n(A'')\to \Hoh^{n+1}(A')$ can be described as follows: Let
$\alpha''\in\Hoh^n(A'')$ be a cohomology class represented by a cocycle $a''\in Z^n(U_\bullet,A'')$ for some
hypercovering $U_\bullet$. Since $A\to A''$ is an epimorphism, we may find a covering $V\to U_n$ such that $a''\in
A''(U_n)$ is the image of some $a\in A(V)$. By 
%\cite[Lm.~24.7.3]{de_jong_stacks_2017} or
%\cite[Th.~V.7.3.2]{artin_theorie_1972}, 
Lemma~\ref{LM:hypercover-refinement} there exists a morphism of hypercoverings
$V_\bullet \to U_\bullet$ such that $V_n\to U_n$ factors through $V\to U_n$. We replace $a$ with its image in $A(V_n)$.
One easily checks that the image of $d^n(a)\in C^{n+1}(V_\bullet,A)$ in both $C^{n+1}(V_\bullet,A'')$ and
$C^{n+2}(V_\bullet,A)$ is $0$, and hence $d^n(a)\in Z^{n+1}(V_\bullet,A')$. Now, $\delta^n(\alpha'')$ is the cohomology
class determined by $d^n(a)\in Z^{n+1}(V_\bullet,A')$.

\subsection{Cohomology of Non-Abelian Groups}
\label{subsec:coh-of-non-ab-grp}

For a group object $G$ of $\mathbf X$, not necessary abelian, we define the pointed set $\Hoh^1(\mathbf{X},G)$ by
hypercoverings. Given a hypercovering $U_\bullet$ in $\bfX$,
let $Z^1(U_\bullet,G)$ be the set of elements $g\in G(U_1)$ satisfying
\begin{equation}\label{EQ:one-cocycle-cond}
d_2^* g\cdot d_0^*g\cdot d_1^*g^{-1}=1
\end{equation} in $G(U_2)$; here $d^*_i=(d_i^2)^*:G(U_1)\to G(U_2)$ is induced by $d^2_i:U_2\to U_1$. Two elements
$g,g'\in Z^1(U_\bullet,G)$ are said to be cohomologous, denoted $g\sim g'$, if there exists $x\in G(U_0)$ such that
$g=d_1^* x \cdot g'\cdot d_0^*x^{-1}$. We define the pointed set $\Hoh^1(U_\bullet,G)$ to be $Z^1(U_\bullet,G)/{\sim}$
with the equivalence class of $1_{G(U_1)}$ as a distinguished element. A morphism of hypercoverings $U'_\bullet\to
U_\bullet$ induces a morphism of pointed sets $\Hoh^1(U_\bullet,G)\to \Hoh^1(U'_\bullet,G)$. Now, following the
literature, we define
\[ \Hoh^1(\bfX,G):=\colim_{U_\bullet}\Hoh^1(U_\bullet,G),
\] where the colimit is taken over the category of all hypercoverings in $\bfX$. We note that some texts take the
colimit over the category of \v{C}ech hypercoverings, see \ Example~\ref{EX:Cech-hypercovering}, but this makes no
difference thanks to the following lemma.

\begin{lemma}\label{LM:Cech-hyp-vs-all-hyp} Let $U_\bullet$ be a hypercovering. Then the maps
$Z^1(\cosk_0(U_\bullet),G)\to Z^1(U_\bullet,G)$ and $\Hoh^1(\cosk_0(U_\bullet),G)\to \Hoh^1(U_\bullet,G)$, induced by
the canonical morphism $U_\bullet\to \cosk_0(U_\bullet)$, are isomorphisms.
\end{lemma}

\begin{proof} The proof shall require various facts about coskeleta. We refer the reader to
\cite[\S{}14.19]{de_jong_stacks_2017} or any equivalent source for proofs.

Recall from Example~\ref{EX:Cech-hypercovering} that $\cosk_0(U_\bullet)$ is nothing but the \v{C}ech hypercovering
associated to $U_0$. Since $U_\bullet$ is a hypercovering, $(d_0,d_1):U_1\to \cosk_0(U_\bullet)_1=U_0\times U_0$ is a
covering, and hence the induced map $G(U_0\times U_0)\to G(U_1)$ is injective. Since the map
$Z^1(\cosk_0(U_\bullet),G)\to Z^1(U_\bullet,G)$ is a restriction of the latter, it is also injective. This implies that
if two cocycles in $Z^1(\cosk_0(U_\bullet),G)$ become cohomologous in $Z^1(U_\bullet,G)$, then they are also
cohomologous in $Z^1(\cosk_0(U_\bullet),G)$, so $\Hoh^1(\cosk_0(U_\bullet),G)\to \Hoh^1(U_\bullet,G)$ is injective. It
is therefore enough to show that $Z^1(\cosk_0(U_\bullet),G)\to Z^1(U_\bullet,G)$ is surjective.

We first observe that the canonical map $Z^1(\cosk_1(U_\bullet),G)\to Z^1(U_\bullet,G)$ is an isomorphism. This
follows from the fact that $\cosk_1(U_\bullet)_1=U_1$ and $U_2\to \cosk_1(U_\bullet)_2$ is a covering, hence
\eqref{EQ:one-cocycle-cond} is satisfied in $G(U_2)$ if and only if it is satisfied in $G(\cosk_1(U_\bullet)_2)$. Since
$\cosk_0(\cosk_1 (U_\bullet))=\cosk_0 (U_\bullet)$, we may replace $U_\bullet$ with $\cosk_1(U_\bullet)$. In this case,
the construction of $\cosk_1$ implies that $U_2$ is characterized by
\begin{align*} U_2(X)=\{(e_{01},e_{02},e_{12},v_0,v_1,v_2)\in U_1(X)^3\times U_0(X)^3\suchthat\, &d_0e_{ij}=v_j,\\
&d_1e_{ij}=v_i~\text{for all legal}~i,j\}\ .
\end{align*} for all objects $X$ of $\mathbf X$. The maps $d_0,d_1,d_2:U_2\to U_1$ are then given by taking the
$e_{12}$-part, $e_{02}$-part, and $e_{01}$-part, respectively. Geometrically, $U_2$ is the object of simplicial
morphisms from the boundary of the $2$-simplex to $U_\bullet$.

Let $g\in Z^1(U_\bullet,G)\subseteq G(U_1)$. We claim that $g$ descends along $(d_0,d_1)$ to $g'\in G({U_0\times U_0})$.
Write $V=U_1\times_{U_0\times U_0} U_1$ and let $\pi_1$, $\pi_2$ denote the first and second projections from $V$ onto
$U_1$. We need to show that $\pi_1^*g=\pi_2^*g$ in $G(V)$. For an object $X$, the $X$-sections of $V$ can be described
by
\begin{align*} V(X)=\{(e_{01},e'_{01},v_0,v_1)\in X(U_1)^2\times X(U_0)^2\suchthat &d_0(e_{01})=d_0(e'_{01})=v_1,\\
&d_1(e_{01})=d_1(e'_{01})=v_0\}\ .
\end{align*} Define $\Psi: V\to U_2$ by
\[ \Psi(e_{01},e'_{01},v_0,v_1)=(e_{01},e'_{01},s_0v_1,v_0,v_1,v_1)
\] on sections. One readily checks that $d_0\Psi=s_0d_0\pi_1=d_0s_1\pi_1$, $d_1 \Psi=\pi_2$ and
 $d_2\Psi=\pi_1$. Now, applying $\Psi^*:G(U_2)\to G(V)$ to \eqref{EQ:one-cocycle-cond}, we arrive at the equation
$\pi_1^*g\cdot \pi_1^*d_0^*s_1^*g\cdot \pi_2^*g^{-1}=1$ in $G(V)$, and applying $s_1^*:G(U_2)\to G(U_1)$
 to \eqref{EQ:one-cocycle-cond}, we find that $g\cdot s_1^*d_0^*g\cdot g^{-1}=1$ in $G(U_1)$. Both equations taken
 together imply that $\pi_1^*g=\pi_2^*g$ in $G(V)$, hence our claim follows.

To finish the proof, it is enough to show that $g'$ is a $1$-cocycle. This will follow from the fact
 that $g$ is a $1$-cocycle if we show that the canonical map $U_2\to \cosk_0(U_\bullet)_2=U_0^3$ is a covering. The
 latter map is given by $(e_{01},e_{02},e_{12},v_0,v_1,v_2)\mapsto (v_0,v_1,v_2)$ on sections. Since $(d_0,d_1):U_1\to
 U_0$ is a covering, for every pair of vertices $v_0,v_1\in U_0(X)$, there exists a covering $X'\to X$ and an edge
 $e_{01}\in U_1(X')$ satisfying $d_0 e_{01}=v_1$, $d_1e_{01}=v_0$. This easily implies that $U_2\to U_0^3$ is locally
 surjective, finishing the proof.
\end{proof}

The following proposition summarizes the main properties of $\Hoh^1(\bfX,G)$.
As before,
we shall suppress $\bfX$, writing $\Hoh^1(G)$, when it is clear form the context.

\begin{prp}\label{PR:non-ab-coh-basic-prop} Let $1\to G'\to G\to G''\to 1$ be a short exact sequence of groups in $\bfX$.
\begin{enumerate}[label=(\roman*)]
 \item \label{item:PR:non-ab-coh-basic-prop:torsor} $\Hoh^1(G)=\Hoh^1(\bfX,G)$ is naturally isomorphic to the set of isomorphism
 classes of $G$-torsors; the distinguished element of $\Hoh^1(G)$ corresponds to the isomorphism class of the
 trivial torsor.
\item \label{item:PR:non-ab-coh-basic-prop:long-ex-seq} There is a long exact sequence of pointed sets
\[ 1\to \Hoh^0(G')\to \Hoh^0(G)\to \Hoh^0(G'')\xrightarrow{\delta^1} \Hoh^1(G')\to \Hoh^1(G)\to \Hoh^1(G'')\ .
\] This exact sequence is functorial in
 $1\to G'\to G\to G''\to 1$, i.e., a morphism from it to another short exact
sequence of groups gives rise a morphism between the corresponding long exact sequences.
\item \label{item:PR:non-ab-coh-basic-prop:longer-ex-seq} When $G'$ is central in $G$, one can extend the exact sequence
of \ref{item:PR:non-ab-coh-basic-prop:long-ex-seq} with an additional morphism $\delta^2:\Hoh^1(G'')\to \Hoh^2(G')$,
which is again functorial in $1\to G'\to G\to G''\to 1$.
\item When $G',G,G''$ are abelian, the exact sequence of \ref{item:PR:non-ab-coh-basic-prop:longer-ex-seq} is
canonically isomorphic to the truncation of the usual long exact sequence of cohomology groups associated to $1\to G'\to G\to
G''\to 1$. In particular, $\Hoh^1( G)$ defined here is naturally isomorphic to $\Hoh^1( G)$ defined in
\ref{subsec:coh-of-ab-grps} and regarded as a pointed set with distinguished element $1$.
\end{enumerate}
\end{prp}

 The proposition is well known, but 
 Giraud \cite[IV, 4.2.7.4, 4.2.10]{giraud_cohomologie_1971} 
 is the only source we are aware of that treats all
 parts in the generality that we require.
 Since the treatment in \cite{giraud_cohomologie_1971} is somewhat obscure,
 and since we shall need the definition of the maps $\delta^1$ and $\delta^2$
 in the sequel, 
 we include an outline of the proof here. Note that it is easier to prove
 \ref{item:PR:non-ab-coh-basic-prop:torsor} using the definition of $\Hoh^1(G)$ via \v{C}ech hypercoverings, while it is
 easier to prove \ref{item:PR:non-ab-coh-basic-prop:longer-ex-seq} using the definition of $\Hoh^1(G)$ via arbitrary
 hypercoverings, and these definitions are equivalent thanks to Lemma~\ref{LM:Cech-hyp-vs-all-hyp}.

\begin{proof}[Proof (sketch)]
\begin{enumerate}[label=(\roman*), align=left, leftmargin=0pt, itemindent=1ex, itemsep=0.3\baselineskip]
\item Let $P$ be a $G$-torsor. Choose a covering $U_0\to *_\bfX$ such that $P(U_0)\neq \emptyset$ and fix some $x\in P(U_0)$. Form the
\v{C}ech hypercovering $U_\bullet$ associated to $U_0$. Then there exists a unique $g\in G(U_1)$ such that
$d_1^*x\cdot g=d_0^*x$ in $P(U_1)$. We leave it to the reader to check that $g\in Z^1(U_\bullet,G)$ and the
construction $P\mapsto g$ induces a well-defined map from the isomorphism classes of $\Tors(G)$ to $\Hoh^1(G)$ taking
the trivial $G$-torsor to the special element of $\Hoh^1(G)$.

In the other direction, let $\alpha\in\Hoh^1(G)$. By Lemma~\ref{LM:Cech-hyp-vs-all-hyp}, $\alpha$ is represented by
some $g\in Z^1(U_\bullet,G)$ where $U_\bullet$ is a \v{C}ech hypercovering. Define $P$ to be the object of $\bfX$
characterized by $P(V)=\{x\in G(U_0\times V)\suchthat g_V\cdot d_0^*x=d_1^*x\}$; here, $d_i^*:G(U_0\times V)\to
G(U_1\times V)$ is induced by $d_i\times \id:U_1\times V\to U_0\times V$. There is a right $G$-action on $P$ given by
$(x,h)\mapsto x\cdot h_{U_0}$ on sections. We leave it to the reader to check that $P$ is indeed a $G$-torsor, and the
assignment $\alpha\mapsto P$ defines an inverse to the map of the previous paragraph. In particular, note that
$P\to *_\bfX$ is a covering because $g\in G(U_0)$; use the fact that $U_\bullet$ is a \v{C}ech hypercovering.

\item Define $\delta^1:\Hoh^0(G'')\to \Hoh^1(G')$ as follows: Let $g''\in \Hoh^0(G'')$. There is a covering $U_0\to *$ such
that $g''$ lifts to some $g\in G(U_0)$. One easily checks that $g':=d_0^*g\cdot d_1^*g^{-1}\in G(U_1)$ lies in
$Z^1(G',U_\bullet)$ where $U_\bullet$ is the \v{C}ech hypercovering associated to $U_0$. We define $\delta^1 g''$ to be
the cohomology class represented by $g'$, and leave it to the reader to check that this is well-defined. All other maps
in the sequence are defined in the obvious manner and the exactness is easy to check.

\item Define $\delta^2:\Hoh^1( G'')\to \Hoh^2(G')$ as follows: Let $\alpha\in\Hoh^1(G'')$ be represented by
$g''\in Z^1(U_\bullet,G'')$ where $U_\bullet$ is a hypercovering. There is a covering $V\to U_1$ such that $g''$ lifts
to some $g\in G(V)$. By Lemma~\ref{LM:hypercover-refinement}, there is a morphism of hypercoverings
$U'_\bullet\to U_\bullet$ such that $U'_1\to U_1$ factors through $V\to U_1$. We replace $g$ with its image in
$G(U'_1)$. Let $g':=d_2^*g\cdot d_0^*g\cdot d_1^*g^{-1}\in G(U'_1)$. It is easy to check that $g'$ lies in $G'(U'_1)$
and defines a $2$-cocycle of $G'$ relative to $U'_\bullet$. We define $\delta^2 \alpha$ to be the cohomology class
represented by $g'\in Z^2(U'_\bullet,G')$, and leave it to the reader to check that this is well-defined. The exactness
of the sequence at $\Hoh^1(G'')$ is straightforward to check.

\item Verdier's Theorem gives rise to an obvious isomorphism between $\Hoh^1(G)$ as defined here and $\Hoh^1(G)$ as
defined in \ref{subsec:coh-of-ab-grps}. It is immediate from the definitions that this isomorphism also induces an
isomorphism between the long exact sequences; see the proofs of \ref{item:PR:non-ab-coh-basic-prop:long-ex-seq},
\ref{item:PR:non-ab-coh-basic-prop:longer-ex-seq} and the comment at the end of \ref{subsec:coh-of-ab-grps}.
\qedhere
\end{enumerate}
\end{proof}

\subsection{Azumaya Algebras}

\label{subsec:Az-algs}

 Let $R$ be a ring object of $\mathbf X$ and let $n$ be a positive integer. Recall that an Azumaya $R$-algebra of
degree $n$ is an $R$-algebra $A$ in $\mathbf{X}$ that is locally isomorphic to $\nMat{R}{n\times n}$, i.e., there exists a
covering $U\to *$ such that $A_{U}\iso \nMat{R_U}{n\times n}$ as $R_U$-algebras.
 The Azumaya $R$-algebras of degree $n$ together with $R$-algebra isomorphisms form a category which we denote by
 \[ \Az_n(\bfX,R)\ .
 \]
 If $A'$ is another Azumaya $R$-algebra, we let $\sHom_{\text{$R$-alg}}(A,A')$ denote the subobject of the internal
 mapping object $(A')^A$ of $\mathbf{X}$ consisting of $R$-algebra homomorphisms. We define the group object
 $\sAut_{\textrm{$R$-alg}}(A)$ similarly.

 \begin{remark} \label{RM:non-constant-rank} We have defined here Azumaya algebras of constant degree only. When $\bfX$
 is connected, these are all the Azumaya algebras, but in general, one has to allow the degree $n$ to take values in
 the global sections of the sheaf $\mathbb{N}$ of positive integers on $\mathbf{X}$. For any such $n$ one can define
 $\nMat{R}{n\times n}$ and the definition of Azumaya algebras of degree $n$ extends verbatim. We ignore this
 technicality, both for the sake of simplicity, and also since it is unnecessary for connected topoi, which are the
 topoi of interest to us.
 \end{remark}

	Let $\calO_\bfX$ be a ring object in $\bfX$.
 Recall that $\bfX$ is \emph{locally ringed} by ${\sh{O}_{\mathbf X}}$, or ${\sh{O}_{\mathbf X}}$ is a \emph{local ring
 object} in $\bfX$, if for any object $U$ in $\bfX$
 and $\{r_i\}_{i\in I}\subseteq {\sh{O}_{\mathbf X}}(U)$ with
 ${\sh{O}_{\mathbf X}}(U)=\sum_i r_i{\sh{O}_{\mathbf X}}(U)$, there exists a covering $\{U_i\to U\}_{i\in I}$ such that
 $r_i\in\units{{\sh{O}_{\mathbf X}}(U_i)}$ for all $i\in I$. 
 In fact, 
 one can take $U_i=\emptyset_{\bfX}$ for almost all
 $i$. We remark that the condition should also hold when $I=\emptyset$,
 which implies that $\calO_\bfX(U)$ cannot be the zero ring when $U\ncong\emptyset_\bfX$. 
 When $\bfX$ has enough points, the condition 
 is equivalent to saying that for every point $i:\mathbf{pt} \to \bfX$, the ring $i^*
 {\sh{O}_{\mathbf X}}$ is local (the zero ring is not considered local).

 Suppose that ${\sh{O}_{\mathbf X}}$ is a local ring object. Then 
 the group homomorphism $\GL_n(\calO_\bfX)\to \sAut_{\text{${\sh{O}_{\mathbf X}}$-alg}}(\nMat{{\sh{O}_{\mathbf X}}}{n\times n})$
	given by $a\mapsto [x\mapsto axa^{-1}]$
	on sections is surjective, \cite[V.\S 4]{giraud_cohomologie_1971}.
 This induces an isomorphism
 $\PGL_n({\sh{O}_{\mathbf X}}):=\GL_n({\sh{O}_{\mathbf
 X}})/\units{\sh{O}_{\mathbf X}}\to \sAut_{\text{${\sh{O}_{\mathbf X}}$-alg}}(\nMat{{\sh{O}_{\mathbf X}}}{n\times n})$,
 which will be used to freely identify the source and target in the sequel.
 The following proposition is well established, again see \cite[V.\S
 4]{giraud_cohomologie_1971}.\uriyaf{There must be some reference which is not written in French...}

 \begin{proposition}\label{PR:equiv-Az-PGLnR} If ${\sh{O}_{\mathbf X}}$ is a local ring object, then there is an
 equivalence of categories
 \[ \Tors(\bfX,\PGL_n({\sh{O}_{\mathbf X}}))\overset{\sim}{\longrightarrow} \Az_n(\bfX,{\sh{O}_{\mathbf X}})
 \] given by the functors $P\mapsto P\times^{\PGL_n({\sh{O}_{\mathbf X}})}\nMat{{\sh{O}_{\mathbf X}}}{n\times n}$ and
 $A\mapsto\sHom_{\text{${\sh{O}_{\mathbf X}}$-alg}}(\nMat{{\sh{O}_{\mathbf X}}}{n\times n},A)$.
\end{proposition}

The proposition holds for any ring object ${\sh{O}_{\mathbf X}}$ of $\mathbf X$ if one replaces the group
object $\PGL_n({\sh{O}_{\mathbf X}})$ with $\sAut_{\textrm{${\sh{O}_{\mathbf X}}$-alg}}(\nMat{{\sh{O}_{\mathbf X}}}{n\times n})$.

\medskip

We continue to assume that $\calO_{\bfX}$ is a local ring object.
By Proposition~\ref{PR:non-ab-coh-basic-prop}\ref{item:PR:non-ab-coh-basic-prop:longer-ex-seq},
 the short exact sequence $1 \to \units{\sh{O}_{\mathbf X}}\to \GL_n({\sh{O}_{\mathbf X}})\to \PGL_n({\sh{O}_{\mathbf
 X}}) \to 1$ gives rise to a pointed set map $\Hoh^1(\PGL_n({\sh{O}_{\mathbf X}}))\to \Hoh^2(\units{\sh{O}_{\mathbf X}})$; here, $\Hoh^*(-)=\Hoh^*(\bfX,-)$. As usual,
 the \emph{Brauer group} of ${\sh{O}_{\mathbf X}}$ is
 \[\Br(\bfX,{\sh{O}_{\mathbf X}})=\Br({\sh{O}_{\mathbf X}}):=\bigcup_{n\in\N}\im\left(\Hoh^1(\PGL_n({\sh{O}_{\mathbf X}}))\to
 \Hoh^2(\units{\sh{O}_{\mathbf X}})\right)\]
 the addition being that inherited from the group $\Hoh^2(\units{\sh{O}_{\mathbf X}})$. Since Azumaya $\sh O_X$-algebras
 correspond to $\PGL_n(\sh O_X)$-torsors, which are in turn classified by $\Hoh^1(\mathbf X, \PGL_n (\sh O_X))$, any
 Azumaya ${\sh{O}_{\mathbf X}}$-algebra $A$ gives rise to an element in $\Br({\sh{O}_{\mathbf X}})$, denoted $[A]$ and called the
 \emph{Brauer class} of $A$. By writing $A'\in[A]$ or saying that $A'$ is \emph{Brauer equivalent} to $A$, we mean that
 $A'$ is an Azumaya ${\sh{O}_{\mathbf X}}$-algebra with $[A']=[A]$. For more details, see \cite{grothendieck_groupe_1968-1}
 or \cite[Chap.~V, \S 4]{giraud_cohomologie_1971}.

	\begin{example}
 Let $X$ be a topological space, let $\bfX=\Sh(X)$ and let $\calO_\bfX$ be the sheaf of continuous functions
 from $X$ to $\C$, denoted $\cont(X,\C)$. Then Azumaya $\calO_\bfX$-algebra are topological Azumaya algebras
 over $X$ as studied in \cite{antieau_period-index_2014}.
	\end{example}
	
	\begin{example}\label{EX:quasi-coh-sheaves}
 An Azumaya algebra $A$ of degree $n$ over a scheme $X$ is a sheaf of $\calO_X$-algebras that is locally, in
 the \'etale topology, isomorphic as an $\sh O_X$-algebra to $\Mat_{n \times n}(\sh O_X)$,
 \cite[Para.~1.2]{grothendieck_groupe_1968-1}
	\end{example}
	
	\begin{example}\label{EX:quasi-coh-sheaves-II}
 Let $R$ be a ring. An Azumaya $R$-algebra of degree $n$ is an $R$-algebra $A$ for which there exists a
 faithfully flat \'etale $R$-algebra $R'$ such that $A\otimes_RR'\cong \nMat{R'}{n\times n}$ as
 $R'$-algebras. This is equivalent to the definition of Example~\ref{EX:quasi-coh-sheaves} in the case where
 $X = \Spec R$ by \cite[Th.~5.1, Cor.~5.2]{grothendieck_groupe_1968-1}. 
 Consult \cite[III.\S5]{knus_quadratic_1991} for other equivalent
 definitions and cf.~Remark~\ref{RM:non-constant-rank}.
	\end{example}

\section{Rings with Involution}

\label{sec:ringsWithInvolution}

In this section we collect a number of results regarding involutions of rings that will be needed later in the paper. The
main result is Theorem \ref{thm:local-involution}, which gives the structure of those rings with involution $(R,\lambda)$
for which the fixed ring of $\lambda$ is local. It is shown that in this case, $R$ is a local ring in its own right, or
$R$ is a quadratic \'etale algebra over the fixed ring of $\lambda$. In particular, the ring $R$ is semilocal.

Throughout, involutions will be written exponentially and the Jacobson radical of a ring $R$ will be denoted $\Jac (R)$.

\subsection{Quadratic \'Etale Algebras}

\begin{definition} \label{def:etale_rank_n} Let $S$ be a ring. A commutative $S$-algebra $R$ is said to be
 \textit{finite \'etale of rank $n$} if $R$ is a locally free $S$-module of rank $n$, and the multiplication map
 $\mu:R\tensor_SR \to R$ may be split as a morphism of $R\tensor_SR$-modules, where $R$ is regarded as an
 $R\tensor_SR$-algebra via $\mu$.
 Finite \'etale $S$-algebras of rank $2$ will be called \emph{quadratic} \'etale algebras.
\end{definition}

\begin{remark}
 One common definition of \'etale for commutative $S$-algebras is that $R$ should be flat over $S$, 
 of finite presentation
 as an $S$-algebra, and unramified in 
 the sense that $\Omega_{R/S}$, the module of K\"ahler differentials, vanishes. 
 This is the definition in \cite[Sec.~17.6]{grothendieck_elements_1967} in the
 affine case.
 Our finite \'etale algebras of rank $n$ are precisely the \'etale algebras
 that are locally free of rank $n$.

Indeed, if $R$ is locally free of rank $n$ over $S$, then 
$R$ is also finitely presented and flat as an $S$-module
\cite[\href{https://stacks.math.columbia.edu/tag/00NX}{Tag 00NX}]{de_jong_stacks_2017}, and hence 
of finite presentation as an $R$-algebra \cite[Prop.~1.4.7]{grothendieck_elements_1964}.
Furthermore, $\mu:R\otimes_SR\to R$ admits a splitting $\psi$ if and only if the $R\otimes_SR$-ideal $\ker \mu$ is
 generated by an idempotent, namely $1-\psi(1_R)$. For finitely generated $S$-algebras $R$, the existence
 of such an idempotent is equivalent to saying $R$ is unramified over $S$ by \cite[Tag 02FL]{de_jong_stacks_2017}.
\end{remark}

\begin{example}\label{EX:quadratic-etale-alg}
 Let $f\in S[x]$ be a monic polynomial of degree $n$. It is well known that $S[x]/(f)$ is a finite \'etale $S$-algebra of
 rank $n$ if and only if the discriminant of $f$ is invertible in $S$. In particular, $S[x]/(x^2+\alpha x+\beta)$ is a
 quadratic \'etale $S$-algebra if and only if $\alpha^2-4\beta\in\units{S}$.
\end{example}

Every quadratic \'etale $S$-algebra $R$ admits a canonical $S$-linear involution $\lambda$ given by
$r^\lambda=\mathrm{Tr}_{R/S}(r)-r$,
see \cite[I.\S1.3.6]{knus_quadratic_1991}. 
The fixed ring of $\lambda$ is $S$ and $\lambda$ is the only $S$-automorphism of $R$
with this property. Moreover, when $S$ is connected, it is the only non-trivial $S$-automorphism of $R$.

\begin{proposition}

\label{PR:Hilbert-ninety} Let $S$ be a local ring with maximal ideal $\frakm$ and residue field $k$, 
let $R$ be a quadratic
 \'etale $S$-algebra, and let $\lambda$ be the unique non-trivial $S$-automorphism of $R$. Then:
 \begin{enumerate}[label=(\roman*)] 
\item \label{item:LM:Hilbert-ninety:Jac} $\Jac(R)=R\frakm$
\item \label{item:LM:Hilbert-ninety:structure} $\quo{R}:=R/R\frakm$ is either a separable
 quadratic field extension of $k$, or $\quo{R}\iso k\times k$. The automorphism that $\lambda$ induces on $\quo{R}$ is
 the unique non-trivial $k$-automorphism of $\quo{R}$.
\item \label{item:LM:Hilbert-ninety:norm}(``Hilbert 90'') For every $r\in \units{R}$ with $r^\lambda r=1$, there exists
$a\in \units{R}$ such that $r=a^{-1}a^\lambda$.
\end{enumerate}
\end{proposition}

\begin{proof} It is clear that $\quo{R}$ is a quadratic \'etale $k$-algebra, and hence a product of separable field
extensions of $k$ \cite[Th.~II.2.5]{demeyer_separable_1971}. This implies the first assertion of
\ref{item:LM:Hilbert-ninety:structure} as well as $\Jac(R)\subseteq R\frakm$. The inclusion $R\frakm\subseteq \Jac(R)$
holds because $R$ is a finite $S$-module \cite[Th.~6.15]{reiner_maximal_1975-1}, so we have proved
\ref{item:LM:Hilbert-ninety:Jac}. The last assertion of (ii) follows from the fact that $\lambda$ is given by $x\mapsto
\mathrm{Tr}_{R/S}(x)-x$.

Let $\quo{r}$ denote the image of $r\in R$ in $\quo{R}$. To prove \ref{item:LM:Hilbert-ninety:norm}, we first claim
that there is $\quo{x}\in \quo{R}$ with $\quo{x}+\quo{x}^{\lambda} \quo{r}\in\units{\quo{R}}$. This is easy to see if
$\quo{R}=k\times k$. Otherwise, $\quo{R}$ is a field and such $\quo{x}$ exists unless $\quo{x}=-\quo{x}^{\lambda}
\quo{r}$ for all $\quo{x}\in \quo{R}$. The latter forces $\quo{r}=-1$ (take $\quo{x}=1$) and
$\lambda_{\quo{R}}=\id_{\quo{R}}$, which is impossible by \ref{item:LM:Hilbert-ninety:structure}, so $\quo{x}$ exists.

Let $x\in R$ be a lift of $\quo{x}$. Then $t:=x+x^\lambda r$ is a lift of $\quo{x}+\quo{x}^{{\lambda}} \quo{r}$, which
implies $t\in\units{R}$ by~\ref{item:LM:Hilbert-ninety:Jac}. Since $r^\lambda =r^{-1}$, it is the case that $t^{\lambda} r=t$, and
so $r=a^\lambda a^{-1}$ with $a=t^{-1}$.
\end{proof}

\subsection{Quadratic \'{E}tale Algebras in Topoi}
\label{subsec:etale-algebras}

Our definition of a ``quadratic \'etale algebra'' extends directly to the case where $S$ is a local ring object in a topos $\bfX$.
\begin{definition}
 Given a ring object $S$ in a topos $\mathbf X$, we say an $S$-algebra $R$ is a \textit{ finite \'etale $S$-algebra
 of rank $n$} if $R$ is a locally free $S$-module of rank $n$ such that the multiplication map
 $\mu: R \tensor_S R \to R$ may be split as a morphism of $R \tensor_S R$-algebras. Finite \'etale $S$-algebras of rank $2$
 will be called \textit{quadratic} \'etale algebras.
\end{definition}
 
We alert the reader that if $R$ is a quadratic \'etale $S$-algebra, then it is not true in general that $R(U)$ is a
quadratic \'etale $S(U)$-algebra for all objects $U$ of $\bfX$. In fact $R(U)$ may not be locally free of rank $2$ over
$S(U)$. Rather, one can always find a covering $V\to U$ such that $R(V)$ is a quadratic \'etale $S(V)$-algebra; for
instance, one may take any
$V\to U$ such that $R_V$ is a free $S_V$-module of rank $2$. We further note that in general there is no covering
$U\to *$ such that $R_U\iso S_U\times S_U$ as $S_U$-algebras, e.g.\ 
let $\bfX=\mathbf{pt}$ (the topos of sets) and take $S$ and $R$ to be $\Q$ and $\Q[\sqrt{2}]$ respectively. While
$\mu: R\otimes_SR\to R$ is easily seen to be split, there is no covering $U\to *$ such that
$R_U\cong S_U\times S_U$.
 
 As one might expect, being a finite \'etale algebra of rank $n$ is a local property in that it may be tested on a covering.
 
\begin{lem}\label{LM:etale-is-local-property} Let $S$ be a ring in $\bfX$, let $R$ be an $S$-algebra and let $U\to *$ be
a covering. Then $R$ is a finite \'etale $S$-algebra of rank $n$ if and only if $R_U$ is a finite \'etale $S_U$-algebra of rank $n$
in $\bfX/U$.
\end{lem}

\begin{proof} Write $M=R\tensor_SR$. The only non-trivial thing to check is that if the multiplication map $\mu_U:
M_U\to R_U$ admits a splitting $\psi:R_U \to M_U$ in $\bfX/U$, then so does $\mu:M\to R$ in $\bfX$. Let
$\pi_1,\pi_2:U\times U\to U$ denote the first and second projections, and let $\pi_i^*\psi:R_{U\times U}\to M_{U\times
U}$ denote the pullback of $\psi$ along $\pi_i$. We claim that $\mu_{U\times U}:M_{U\times U}\to R_{U\times U}$ admits
at most one splitting. Provided this holds, we must have $\pi_1^*\psi=\pi_2^*\psi$ and so $\psi$ descends to a map
$\psi_0:R\to M$ splitting $\mu$ as required.

The claim can be verified on the level of sections, namely, it is enough to check that any ring surjection $\mu:A\to B$
admits at most one $A$-linear splitting. If $\psi$ is such a splitting and $e=\psi(1_B)$, then $\psi(1_B\cdot
\alpha)=\alpha e$ for all $\alpha\in A$, so $\psi$ is determined by the idempotent $e$. It is easy to check that
$(1-e)A=\ker \mu$ and that $e$ is the only idempotent with this property, hence $e$ is determined by $\mu$.
\end{proof}

\begin{example}\label{EX:quad-etale-morph-of-sch}
Let $\pi:X\to Y$ be a quadratic \'etale morphism of schemes. That is, $\pi$ is affine and $Y$ can be covered by
open affine subschemes $\{U_i\}_i$ such that the ring map corresponding to $\pi:\pi^{-1}(U_i)\to U_i$ is quadratic
\'etale for all $i$. Then $\pi_*\calO_X$ is a quadratic \'etale $\calO_Y$-algebra in both $\Sh(Y_\et)$ and
$\Sh(Y_{\Zar})$; this can be checked using Lemma~\ref{LM:etale-is-local-property}.
 \end{example}
\begin{example} Let $\pi:X\to Y$ be a double covering of topological spaces and let $\calC(X,\C)$ and $\calC(Y,\C)$ denote the
sheaves of continuous $\C$-values functions on $X$ and $Y$, respectively. Then $\pi_*\calC(X,\C)$ is a quadratic \'etale
$\calC(Y,\C)$-algebra in $\Sh(Y)$; again this can be checked with Lemma~\ref{LM:etale-is-local-property}.
\end{example}

\subsection{Rings with Involution}

Throughout, $R$ is an ordinary commutative ring, $\lambda:R\to R$ is an involution, and $S$ is the fixed ring of
$\lambda$. The purpose of this section is twofold. First, we show that the locus of primes $\frakp\in \Spec S$ such
that $R_\frakp$ is a quadratic \'etale over $S_\frakp$ is open in $\Spec S$. Second, we study the structure of $R$ when
$S$ is local, showing, in particular, that $R$ is quadratic \'etale over $S$, or $R$ is local.

There are two pitfalls in the study of $R$ over $S$. First of all, $R$ may not be finite over $S$.
\begin{example}\label{EX:pathological-rings-with-inv}
 Let $I$ be any set, let $R$ be the commutative $\C$-algebra freely generated by $\{x_i\}_{i\in I}$, and
 let $\lambda:R\to R$ be the $\C$-linear involution sending each $x_i$ to $-x_i$. Then the fixed ring of
 $\lambda$ is $S=\C[x_ix_j\where i,j\in I]$. 
 Let $\frakm=\langle x_ix_j\where i,j\in I\rangle$.
 Since $R/\frakm R\iso {\C[x_i\where i\in I]}/\langle x_ix_j\where i,j\in I\rangle$ and $S/\frakm\iso \C$, it follows that
 $R$ cannot be generated by fewer than $|I|$ elements as an $S$-algebra. Thus, when $I$ is infinite, $R$ is not finite
 over $S$. The same applies to the $S/\frakm$-algebra $R/\frakm R$, even though $S/\frakm$ is noetherian. We further
 note that when $1<|I|<\aleph_0$, the ring $R$ is a smooth affine $\C$-algebra, but $S$ is singular.
\end{example}

Second, the formation of fixed rings may not commute with extension of scalars. That is, if $S'$ is a commutative
$S$-algebra, then $S'$ need not be the subring of $\lambda$-fixed elements in $R':=R\otimes_SS'$. In fact, $S'\to R'$ is
{\it a priori} not one-to-one. Nevertheless, $S'\to R'$ restricts to an isomorphism
$S'\to \{r\in R'\suchthat r=r^\lambda\}$ if $S'$ is flat over $S$, or $2\in \units{S}$. To see this, consider the exact
sequence of $S$-modules $0\to S\to R\xrightarrow{\id_S-\lambda} R$. The statement amounts to showing that it remains
exact after tensoring with $S'$. This is clear if $S'$ is flat, and if $2\in \units{S}$, then it follows because
$S\to R$ is split by $r\mapsto \frac{1}{2}(r+r^\lambda)$.

\begin{remark}
 Voight \cite[Corollary~3.2]{voight_2011_standard_involution} showed that if $R$ is locally free of rank at least $3$
 over $S$ and $2$ is not a zerodivisor in $S$, then $R$ decomposes as $S\oplus M$ where $M$ is an ideal of $R$ such that
 $M^2=0$ and $\lambda|_M=-\id_M$. Voight calls such (commutative) rings with involution \emph{exceptional}. This shows
 that if $\lambda:R\to R$ is not exceptional then either $R$ is not locally free over $S$, or $\rank_S R\leq 2$.
 The case $R= \C[x_i\where i\in I]/\langle x_ix_j\where i,j\in I\rangle$ and $S=\C$ featuring in
 Example~\ref{EX:pathological-rings-with-inv} is an example of an exceptional ring with involution, and essentially the
 only one if $S=\C$.
	\end{remark}

 Having warned the reader of these pitfalls, we return to the main topic of the section, which is the study of
 $R$ over $S$.

	\begin{lemma}\label{LM:etale-criterion} 
 Assume that there exists $r\in R$ with $r-r^\lambda\in\units{R}$. Then $R$ is a quadratic \'etale $S$-algebra.
\end{lemma}

\begin{proof}
 For $a\in R$, write $t_a=a+a^\lambda$ and $n_a=a^\lambda a$, and observe that $t_a,n_a\in S$ and $a^2-t_aa+n_a=0$.
	
	Suppose that $R=S[r]$. Since $r^2-t_rr+n_r=0$, it follows that $R=S+Sr$. Furthermore, if $\alpha r\in Sr$ for
 $\alpha\in S$, then $0=(\alpha r)-(\alpha r)^\lambda=\alpha(r-r^\lambda)$, so $\alpha=0$ because
 $r-r^\lambda\in\units{R}$. It follows that the $S$-algebra map $R \to S[x]/(x^2-t_rx+n_r)$ sending $x$ to $r$ is an isomorphism. Since
 $t_r^2-4n_r=(r-r^\lambda)^2\in\units{S}$, we conclude that $R$ is a quadratic \'etale $S$-algebra.
	
 We now show that $R=S[r]$. Write $u=r-r^\lambda=2r-t_r\in S[r]$ and $a=u^{-1}r$. One verifies that
 $a=n_u^{-1}(n_r-r^2)$, 
 so that $a \in S[r]$. Since $u^\lambda=-u$, we have
 $a+a^\lambda=u^{-1}r-u^{-1}r^\lambda=u^{-1}u=1$. Let $b\in R$ and $b'=b-a(b+b^\lambda)$. Straightforward computation
 shows that $b'^\lambda=-b'$, hence $(u^{-1}b')^\lambda=u^{-1}b'$ and $u^{-1}b'\in S$. It follows that
 $b'\in uS\subseteq S[r]$ and thus, $b=b'+a(b+b^\lambda)\in S[r]+aS=S[r]$.
\end{proof}

	\begin{lemma}\label{LM:converse-etale-crit}	
 Suppose that $S$ is local and $R$ is a quadratic \'etale $S$-algebra. Then there exists $r\in R$ such that
 $r-r^\lambda\in\units{R}$.
	\end{lemma}
	
	\begin{proof}
		Let $\frakm$ be the maximal ideal of $S$.
		By Proposition~\ref{PR:Hilbert-ninety}\ref{item:LM:Hilbert-ninety:Jac},
		we may replace $R$ with $R/\frakm R$.
		The claim then follows easily from Proposition~\ref{PR:Hilbert-ninety}\ref{item:LM:Hilbert-ninety:structure}.
	\end{proof}
	
	\begin{corollary}\label{CR:etale-at-a-point}
 The set of prime ideals $\frakp\in\Spec S$ such that $R_\frakp$ is a quadratic \'etale $S_\frakp$-algebra is
 open in $\Spec S$. Equivalently for every $\frakp\in\Spec S$ such that $R_\frakp$ is a quadratic \'etale
 $S_\frakp$-algebra, there exists $s\in S- \frakp$ such that $R_s$ is a quadratic \'etale $S_s$-algebra.
\end{corollary}
 
\begin{proof}
 By Lemma~\ref{LM:converse-etale-crit}, there exists $r\in {R_\frakp}$ with $r-r^\lambda\in\units{R_\frakp}$. We can
 find $s_1\in S- \frakp$ and $a,b\in R$ such that $r=as_1^{-1}$ and $r-r^\lambda=bs_1^{-1}$ in $R_\frakp$.
 Since $b s_1^{-1}\in \units{R_\frakp}$, we can find $s_2\in S-\frakp$ such that $b s_1^{-1}$ is invertible in
 $R_{s_1s_2}$. Now take $s=s_1s_2$ and apply Lemma~\ref{LM:etale-criterion} to $R_{s}$ with $r=as_1^{-1}$.
\end{proof}

\begin{remark} Corollary \ref{CR:etale-at-a-point} is well known when $R$ is locally
free of finite rank over
$S$;
see, for instance, \cite[Tag 
\href{https://stacks.math.columbia.edu/tag/0C3J}{0C3J}]{de_jong_stacks_2017}.
 \end{remark}

	We momentarily consider an arbitrary finite group acting on $R$.

\begin{proposition}\label{PR:G-acts-on-R}
 Let $G$ be a finite group acting on a ring $R$ and let $S$ be the subring of 
 elements fixed under $G$. If $S$ is local
 then the maximal ideals of $R$ form a single $G$-orbit. In particular, $R$ is semilocal.
\end{proposition}

\begin{proof}
Let $\frakm$ denote the maximal ideal of $S$ and, for the sake of contradiction, suppose $\frakp$ and $\frakq$ are
 maximal ideals of $R$ lying in distinct $G$-orbits. 
 Let $\frakp'=\bigcap_{g\in G}g(\frakp)$ and $\frakq'=\bigcap_{g\in G}g(\frakq)$. 
 Since $g(\frakp)+h(\frakq)=R$ for all $g,h\in G$, we have $\frakp'+\frakq'=R$.
 Using the Chinese Remainder Theorem, choose $r\in \frakq'$ such that
 $r\equiv 1\!\!\mod \frakp'$. Replacing $r$ with $\prod_{g\in G} g(r)$, we may assume that $r\in S$. 
 Since $\frakp'\cap S$ and $\frakq'\cap S$ are both contained in $\frakm$, this
 means that $r$ lies both in $(1+\frakp')\cap S\subseteq 1+\frakm$ and in $\frakq'\cap S\subseteq \frakm$, which is absurd.
\end{proof}
 
 	We derive the main result of this section by 
	specializing Proposition~\ref{PR:G-acts-on-R} to the case of a group with $2$ elements.

\begin{theorem}\label{thm:local-involution} 
 Suppose $R$ is a ring and $\lambda: R \to R$ is an involution with fixed ring $S$ such that $S$ is local. Let $\quo R$ denote $R / \Jac(R)$ and $\quo \lambda$ the restriction of $\lambda$ to $\quo R$.
 \begin{enumerate}[label=(\roman*)]
 \item If $\quo \lambda \neq \id$, then $R$ is a quadratic \'etale algebra over $S$.
 \item If $\quo \lambda = \id$, then $R$ is a local ring that is not quadratic \'etale
 over $S$.
 \end{enumerate}
In either case, $R$ is semilocal.
\end{theorem}

\begin{proof}
 Let $\frakM$ be a maximal ideal of $R$. Taking $G=\{1,\lambda\}$ in Proposition~\ref{PR:G-acts-on-R}, we see that the
 maximal ideals of $R$ are $\{\frakM,\frakM^\lambda\}$. We consider the cases $\frakM\neq \frakM^\lambda$ and
 $\frakM=\frakM^\lambda$ separately.
	
 Suppose that $\frakM\neq \frakM^\lambda$. By the Chinese Remainder Theorem,
 $\quo{R}=R/(\frakM\cap \frakM^\lambda)\cong R/\frakM\times R/\frakM^\lambda$, and under this isomorphism,
 $\quo{\lambda}$ acts by sending $(a+\frakM,b+\frakM^\lambda)$ to $(b^\lambda+\frakM,a^\lambda+\frakM^\lambda)$. This
 implies that $\quo{\lambda}\neq \id$, so we are in the situation of (i). Furthermore, by taking $a=1$ and $b=0$, we see
 that there exists $\quo{r}\in R$ such that $\quo{r-r^\lambda}\in\units{\quo{R}}$, or equivalently,
 $r-r^\lambda\in \units{R}$. Thus, by Lemma~\ref{LM:etale-criterion}, $R$ is quadratic \'etale over $S$.
	
 Suppose now that $\frakM=\frakM^\lambda$. Then $R$ is local and $\quo{R}=R/\frakM$ is a field. If
 $\quo\lambda \neq\id$, then there exists $\quo{r}\in \quo{R}$ with $\quo{r-r^\lambda}\in\units{\quo{R}}$ and again we
 find that $R$ is quadratic \'etale over $S$. On the other hand, if $\quo{\lambda} = \id$, then $R$ cannot be quadratic
 \'etale over $S$ by Proposition~\ref{PR:Hilbert-ninety}\ref{item:LM:Hilbert-ninety:structure}.
	
 We have verified (i) and (ii) in both cases, so the proof is complete.
\end{proof}

\begin{notation}
 A \textit{henselian} ring is a local ring in which Hensel's lemma, \cite[Thm.~7.3]{eisenbud_commutative_1995},
 holds. A \textit{strictly henselian} ring is a henselian ring for which the residue field is separably closed.
 \end{notation}

 \begin{lemma} \label{LM:mR-subset-Jac} Let $G$ be a finite group acting on a ring $R$ and let $S$ be the subring of $R$
 fixed under $G$. If $S$ is local with maximal ideal $\frakm$, then $ \frakm R\subseteq \Jac(R)$. In particular,
 $\frakm R\cap S=\frakm$.
 \end{lemma}
 
 \begin{proof}
 Let $a\in \frakm$. To prove $ \frakm R\subseteq \Jac(R)$, we need to show that $aR\subseteq \Jac(R)$, or
 equivalently, that $1+aR$ consists of invertible elements. Let $b\in R$ and let $\{g_1,\dots,g_n\}$ denote the
 distinct elements of $G$. Then
 	\[\prod_{g\in G}g(1+ab)=1+\sigma_1(g_1b,\dots,g_nb)a+\sigma_2(g_1b,\dots,g_nb)a^2+
 \dots\] where $\sigma_i$ denotes the $i$-th elementary symmetric polynomial on $n$ letters. Since
 $a\in\frakm$, and since $\sigma_i(g_1b,\dots,g_nb)$ is invariant under $G$, the right hand side lies in
 $1+\frakm\subseteq \units{S}$. Thus, $1+ab\in\units{R}$.
 \end{proof}

 \begin{cor}\label{CR:henselian-ring-with-involution} Let $R$ be a ring, let $\lambda:R\to R$ be an involution, and let
 $S$ denote the fixed ring of $\lambda$. Suppose that $S$ is a strictly henselian ring with maximal ideal
 $\frakm$. Then $R$ is a finite product of strictly henselian rings.
 \end{cor}

 \begin{proof} By Theorem~\ref{thm:local-involution}, either $R$ is a quadratic \'etale $S$-algebra, or $R$ is local. In
 the former case, $R$ is finite over $S$, and the lemma follows from \cite[Tag 04GH]{de_jong_stacks_2017} so we
 assume $R$ is local. Write $\frakM$ for the maximal ideal of $R$, $k$ for its residue field, and denote by
 $r\mapsto\quo{r}$ the surjection $R\to k$.

 Observe first that $R$ is integral over $S$ since for all $r\in R$, we have $ r^2-(r^\lambda+r)r+r^\lambda r=0$ and
 $r^\lambda+r,r^\lambda r\in S$. This implies that $k$ is algebraic over the residue field of $S$, which is separably
 closed, hence $k$ is separably closed.

 Now, let $f\in R[x]$ be a monic polynomial such that $\quo{f}\in k[x]$ has a simple root $\eta \in k$. We will show
 that $f$ has root $y\in R$ with $\quo{y}=\eta$. Let $a_0,\dots,a_{n-1}\in R$ be the coefficients of $f$ and let
 $x\in R$ be any element with $\quo{x}=\eta$. Since $R$ is integral over $S$, there is a finite $S$-subalgebra
 $R_0\subseteq R$ containing $a_0,\dots,a_{n-1},x$. By \cite[Tag 04GH]{de_jong_stacks_2017}, $R_0$ is a product of
 henselian rings. Since $R$ has no non-trivial idempotents, this means $R_0$ is a henselian ring. Write $\frakM_0$ for
 the maximal ideal of $R_0$ and $k_0=R_0/\frakM_0$. Since $R_0$ is finite over $S$, there is a natural number $n$ such
 that $\frakM_0^n\subseteq \frakm R_0\subseteq \frakM_0$ by \cite[Th.~6.15]{reiner_maximal_1975-1}. This implies that
 $(R\frakM_0)^n=R\frakM_0^n\subseteq R\frakm$ and the latter is a proper ideal of $R$, by Lemma
 \ref{LM:mR-subset-Jac}. Therefore, $R\frakM_0\subseteq\frakM$ and we have a well-defined homomorphism of fields
 $k_0\to k$ given by $x+\frakM_0\mapsto x+\frakM$. Since $a_0,\dots,a_{n-1},x\in R_0$, we have $\quo{f}\in k_0[x]$ and
 $\eta\in k_0$. As $R_0$ is a henselian ring, there is $y\in R_0$ with $f(y)=0$ and $y+\frakM_0=\eta$. This completes the
 proof.
 \end{proof}

\section{Ringed Topoi with Involution} \label{sec:SitesWithInvolution}

Unless indicated otherwise, throughout this section, $(\mathbf{X}, \calO_{\mathbf{X}})$ denotes a locally 
ringed topos. Our interest is in the
following examples:
\begin{enumerate}
\item $\mathbf{X}=\Sh(X_\et)$ for a scheme $X$, and $\calO_{\mathbf{X}}$ is the structure
sheaf of $X$ which sends $(U\to X)$ to $\Gamma(U,\calO_U)$.
\item $\mathbf{X}=\Sh(X)$ for a topological space $X$, and $\calO_{\mathbf{X}}$ is the sheaf of continuous $\C$-valued
functions, denoted $\calC(X,\C)$.
\end{enumerate} We write the cyclic group with two elements as $C_2$
and, when applicable, the non-trivial element
of $C_2$ will always be denoted $\lambda$.

When there is no risk of confusion, we shall refer to $\bfX$ 
as a ringed topos, in which case the ring object is understood to be $\calO_\bfX$.

\subsection{Involutions of Ringed Topoi}
\label{subsec:involutions}

\begin{definition}
 Suppose $\mathbf X$ is a topos. An \textit{involution} of $\mathbf X$ consists of an equivalence of categories $\Lambda: \mathbf
 X \to \mathbf X$ and a natural isomorphism $\nu: \Lambda^2 \Rightarrow \id$ 
 satisfying the coherence condition that
 $\nu_{\Lambda X}=\Lambda \nu_X$ for all
 objects $X$ of $\mathbf X$.
\end{definition}

The natural equivalence $\nu$ will generally be suppressed from the notation.

\begin{definition}\label{DF:ringed-topos-with-inv}
 Suppose $(\mathbf X, \sh{O}_{\mathbf X})$ is a ringed topos. An \textit{involution} of $(\mathbf X, \calO_{\mathbf{X}})$ consists
 of an involution $(\Lambda, \nu)$ of $\mathbf X$ and an isomorphism $\lambda$ of ring objects $\lambda: \calO_{\mathbf{X}}
 \to \Lambda \calO_{\mathbf{X}}$ such that $ \Lambda\lambda \circ \lambda = \nu_{\calO_{\mathbf{X}}}^{-1}$.
\end{definition}

Suppressing $\nu$ from the last equation, we say $\Lambda \lambda \circ \lambda = \id$.

\begin{remark}\label{RM:C-two-action} 
	(i) The functor $\Lambda$ is left adjoint to itself with the unit and counit of
	the adjunction being $\nu^{-1}$ and $\nu$, respectively.
	Thus, if we write $\Lambda^*=\Lambda_*=\Lambda$, then
	the adjoint pair $(\Lambda^*,\Lambda_*)$ defines a geometric automorphism of $\mathbf X$.
	Moreover, $(\Lambda^*,\Lambda_*,\lambda^{-1}:\Lambda_*\calO_\bfX\to \calO_\bfX)$ defines
	an automorphism of the ringed topos $(\bfX,\calO_\bfX)$.

	(ii) Topoi and ringed topoi form $2$-categories in which there is a notion of a weak $C_2$-action. An
 involution of a topos, resp.\ ringed topos, as defined here induces such a weak $C_2$-action in which the
 non-trivial element $\lambda$ acts as the morphism $\Lambda=(\Lambda^*,\Lambda_*)$, resp.\
 $(\Lambda^*,\Lambda_*,\lambda^{-1})$, and the trivial element acts as the identity. It can be checked that all
 $C_2$-actions with the latter property arise in this manner. Since an arbitrary weak $C_2$-action is equivalent
 to one in which $1$ acts as the identity, specifying an involution is essentially the same as specifying a weak
 $C_2$-action.
\end{remark}

\begin{notation}
	Henceforth, when there is no risk of confusion, involutions of ringed topoi will always be denoted
	$(\Lambda,\nu,\lambda)$. In fact, we shall often
	abbreviate the triple $(\Lambda,\nu,\lambda)$ to $\lambda$.
\end{notation}

It is convenient to think of $\Lambda, \nu$ as the ``geometric data'' of the involution and of $\lambda$ as the
``arithmetic data'' of the involution. The following are the motivating examples.

 \begin{example} \label{ex:galois} 
	Let $X$ be a scheme and let $\sigma:X\to X$
	be an involution. 
	The direct image functor $\Lambda:=\sigma_*$ defines an involution
	of $\bfX:=\Sh(X_\et)$; the suppressed natural isomorphism $\nu:\Lambda^2\Rightarrow \id$
	is the identity.
	Let $\calO_\bfX$ be 
	the structure sheaf of $X$ on $X_\et$.
 	The involution of $X$ gives rise to an isomorphism $\lambda:\calO_\bfX\to \Lambda\calO_\bfX$
	as follows:
	For an \'etale morphism $U\to X$, define $U^\sigma\to X$ via the pullback diagram
	\[
	\xymatrix{
	U^\sigma \ar[r]\ar[d]^{\sigma_U} & X \ar[d]^{\sigma} \\
	U \ar[r] & X	
	}
	\] 
	By definition, $\Lambda\calO_\bfX(U\to X)=\sigma_*\calO_\bfX(U\to X)= \calO_\bfX(U^\sigma\to
 X)=\Gamma(U^\sigma,\calO_{U^\sigma})$, and we define $\lambda_{U\to X}:\Gamma(U,\calO_U)=\calO_\bfX(U\to X)\to
 \Lambda\calO_\bfX(U \to X)=\Gamma(U^\sigma,\calO_{U^\sigma})$ to be the isomorphism induced by $\sigma_U$. It
 is easy to check that $\Lambda\lambda\circ\lambda=\id$ and so $(\Lambda,\nu,\lambda)$ is an involution of
 $(\bfX,\calO_\bfX)$. 
 When $X=\Spec R$ for a ring $R$ and $\sigma$ is induced by an involution of $R$, we can
 recover that involution from $\lambda:\calO_\bfX\to \Lambda\calO_\bfX$ by taking global sections.
	
	The small \'etale site $X_\et$ can be replaced by other sites, for instance the site $X_\fppf$.
\end{example}

\begin{example} \label{ex:top} Let $X$ be topological space with a continuous involution $\sigma: X \to X$. Write
 $\mathbf{X}=\Sh(X)$ and let $\calO_{\mathbf{X}}=\calC(X,\C)$, the ring sheaf of continuous functions into $\C$. Then the
 direct image functor $\Lambda:=\sigma_*$ defines an involution of $\bfX$; the isomorphism
 $\nu:\Lambda^2\Rightarrow \id$ is the identity.

 In particular, $(\Lambda \calO_{\mathbf{X}}) (U) = \calO_{\mathbf{X}} (\sigma U) = \cont( \sigma U, \C)$. Let
 $\lambda:\calO_{\mathbf{X}}\to\Lambda \calO_{\mathbf{X}}$ be given sectionwise by precomposition with $\sigma$, namely
\begin{align*} \lambda_U : {C}(U, \mathbb{C}) & \to {C}(\sigma U, \mathbb{C}), \\ \phi &\mapsto\phi\circ \sigma.
\end{align*} Then $(\Lambda,\nu,\lambda)$ is an involution of $(\mathbf{X},\calO_{\mathbf{X}})$.
\end{example}

\begin{example}\label{EX:trivial-Lambda}
 Let $\bfX$ be any topos, let $R$ be a ring object in $\bfX$ and let $\lambda:R\to R$ be an involution. Then
 $(\id_{\bfX},\nu:=\id,\lambda)$ is an involution of $(\bfX,R)$.
\end{example}

\begin{definition}\label{DF:trivial-involution} The \textit{trivial involution} of $(\mathbf{X},\calO_{\mathbf{X}})$ is 
 $(\Lambda,\nu,\lambda)=(\id_{\mathbf{X}},\id,\id_{\calO_{\mathbf{X}}})$. Any other involution $(\Lambda, \nu, \lambda)$ of $(\mathbf{X},\calO_{\mathbf{X}})$
 will be said to be \textit{weakly trivial} if 
	it is equivalent to the trivial involution of $(\bfX,\calO_\bfX)$
	in the following sense: 
 There exists a natural isomorphism $\theta: \Lambda \Rightarrow \id$ such that
 $\theta_{\calO_X}=\lambda^{-1}$ and
 $\theta_X\circ \theta_{\Lambda X}=\nu_X$ 
 for all objects $X$ in $\bfX$.\uriyaf{
 The identity $\Lambda\theta_X =\theta_{\Lambda X}$
 is superfluous and follows from the equality
 $\Lambda\theta_X\circ\theta_X=\theta_{\Lambda X}\circ \theta_X$
 which holds since $\theta:\Lambda\Rightarrow\id$. 
}
\end{definition}

\begin{example} The involution of Example~\ref{ex:galois} (resp.\ Example~\ref{ex:top}) is trivial if and only if
$\sigma:X\to X$ is the identity.
\end{example}

Let $(\Lambda,\nu,\lambda)$ be an involution of $(\mathbf{X},\calO_{\mathbf{X}})$ and let $M$ be an $\calO_{\mathbf{X}}$-module. Then
$\Lambda M$ carries a $\Lambda \calO_{\mathbf{X}}$-module structure. We shall always regard $\Lambda M$ as an
$\calO_{\mathbf{X}}$-module by using the morphism $\lambda:\calO_{\mathbf{X}}\to \Lambda \calO_{\mathbf{X}}$, that is, by
twisting the structure.

Using the equality $\Lambda \lambda \circ \lambda = \nu_{\calO_{\mathbf{X}}}^{-1}$, one easily
checks that $\nu_M:\Lambda\Lambda M\to M$ is an isomorphism of
 $\calO_{\mathbf{X}}$-modules, which we suppress from the notation henceforth. 

\benw{The referee has objected to the vagueness of the following comment, asking for an example instead. This seems
 fair. I would like to give the example of a ``skyscraper module'', but now I don't know why I can't produce such a
 thing in $\Spec R$. Update! I can, it's just that the quasicoherent sheaves have a property of being determined by
 quasicoherence and their global sections, and such a quasicoherent object has the property asserted. One can get an
 example with something as simple as a non-quasicoherent $\sh O_X$-module.}
 
\uriyaf{If $\frakm$ is a maximal ideal of $R$, then $R/\frakm$
corresponds to a skyscraper sheaf on $\Spec R$. 
Anyhow, the example you gave is fine. There is not need to specify
a sheaf supported at $p$.}

 Notice that contrary
 to the case of $\lambda$-twisting of modules over ordinary rings, 
 it is not in general true that $M$ can be identified
 with its $\lambda$-twist as an abelian group object in $\mathbf{X}$.
 For instance, 
 in the context of Example~\ref{ex:galois},
 suppose that $X=\Spec R$ and $\sigma$ exchanges two maximal ideals
 $\frakm_1,\frakm_2\lhd R$, and let $M_1$, $M_2$
 denote the quasicoherent $\calO_\bfX$-modules corresponding to the $R$-modules
 $R/\frakm_1$ and $R/\frakm_2$, respectively. Then $M_2=\Lambda M_1$,
 but $M_1$ is not isomorphic to $M_2$ as abelian group objects
 in $\bfX$ because $M_1$ is supported at $\{\frakm_1\}$
 while $M_2$ is supported at $\{\frakm_2\}$.

 If $A$ is an $\calO_{\mathbf{X}}$-algebra, then $\Lambda A$, besides being an $\calO_{\mathbf{X}}$-module, carries an
 $\calO_{\mathbf{X}}$-algebra structure. Letting $A^\op$ denote the opposite algebra of $A$, we have $\Lambda(\Lambda
 A^\op)^\op=A$ up to the suppressed natural isomorphism $\nu_A$.

\begin{definition}\label{DF:alg-with-a-lambda-inv} A \emph{$\lambda$-involution} on an $\calO_{\mathbf{X}}$-algebra $A$ is a
morphism of $\calO_{\mathbf{X}}$-algebras $\tau:A\to \Lambda A^\op$ such that $\Lambda\tau\circ\tau=\id_A$. In this case,
$(A,\tau)$ is called an \emph{$\calO_{\mathbf{X}}$-algebra with a $\lambda$-involution}. If $A$ is an Azumaya algebra, it
will be called an \textit{Azumaya algebra with $\lambda$-involution}.

If $(A,\tau)$ and $(A',\tau')$ are $\calO_{\mathbf{X}}$-algebras with $\lambda$-involutions, then a morphism
 from $(A,\tau)$ to $(A',\tau')$ is a morphism of $\calO_{\mathbf{X}}$-algebras $\phi: A\to A'$ such that $\tau'\circ
 \phi=\Lambda\phi\circ \tau$.
\end{definition}

 Notice that applying $\Lambda$ to both sides of $\Lambda \tau \circ \tau = \id_A$ gives $\tau\circ \Lambda \tau =
 \id_{\Lambda A}$.

 \begin{notation}
 The category where the objects are degree-$n$ Azumaya $\calO_\bfX$-algebras with $\lambda$-involutions and where the
 morphisms are isomorphisms of $\calO_\bfX$-algebras with $\lambda$-invo\-lu\-tions shall be denoted
 $\Az_n(\bfX,\calO_\bfX,\lambda)$, or just $\Az_n(\calO_\bfX,\lambda)$.
\end{notation}

Given a ring with involution $(R,\lambda)$ and an $R$-algebra $A$, it is reasonable to define a $\lambda$-involution of
$A$ to be an involution $\tau:A\to A$ satisfying $(ra)^\tau=r^\lambda a^\tau$ for all $r\in R$, $a\in A$. Following
Gille \cite[\S1]{gille_gersten_2009}, this can be generalized to the context of schemes: Given a scheme $X$, an
involution $\sigma:X\to X$, and  an $\calO_X$-algebra $A\in\Sh(X_{\Zar})$, then a \textit{$\sigma$-involution} of $A$
is an $\calO_X$-algebra morphism $\tau:A\to \sigma_*A^\op$ such that $\sigma_*\tau\circ \tau=\id_A$. We reconcile these
elementary definitions with Definition~\ref{DF:alg-with-a-lambda-inv} in the following example.

\begin{example} 
 Let $(R,\lambda)$ be a ring with involution, let $X=\Spec R$ and write $\sigma:X\to X$ for the involution induced by
 $\lambda$. Abusing the notation, let $(\Lambda,\nu,\lambda)$ denote the involution induced by $\sigma$ 
 on the \'etale
 ringed topos of $X$, see Example~\ref{ex:galois}.
 
 Every $R$-module $M$ gives rise to an $\calO_\bfX$-module, also denoted $M$, the sections of which are given by
 $M(\Spec R'\to \Spec R)=M\otimes_RR'$ for any $(\Spec R'\to \Spec R)$ in $X_\et$. In fact, this defines an equivalence
 of categories between $R$-modules and quasicoherent $\calO_\bfX$-modules \cite[Tag 03DX]{de_jong_stacks_2017}.
 Straightforward computation shows that the $\calO_\bfX$-module $\Lambda M$ corresponds to $M^\lambda$, the $R$-module
 obtained from $M$ by twisting via $\lambda$.
 	
 Now let $A$ be an $R$-algebra and let $\tau:A\to A$ be an involution satisfying $(ra)^\tau=r^\lambda a^\tau$ for all
 $r\in R$, $a\in A$. Realizing $A$ as an $\calO_\bfX$-algebra in $\bfX$, the algebra $\Lambda A^\op$ corresponds to the
 $R$-algebra $(A^\lambda)^\op$, and so the involution $\tau$ induces a $\lambda$-involution $A \to \Lambda A ^\op$, also
 denoted $\tau$. By taking global sections, we see that all $\lambda$-involutions of the $\calO_\bfX$-algebra $A $ are
 obtained in this manner.
 	
 Likewise, if $X$ is a scheme, $\sigma:X\to X$ is an involution, and $A$ is a quasicoherent $\calO_X$-algebra in
 $\Sh(X_{\Zar})$\benw{Do we really mean Zariski here?}\uriyaf{For purpose of comparison, yes.}, then the
 $\lambda$-involutions of the $\calO_\bfX$-algebra associated to $A$ in $\bfX=\Sh(X_\et)$ (see \cite[Tag
 03DU]{de_jong_stacks_2017}) are in one-to-one correspondence with the $\sigma$-involutions of $A$ in the sense of Gille
 \cite[\S1]{gille_gersten_2009}. The quasicoherence assumption applies in particular when $A$ is an Azumaya algebra
 over $X$.
 \end{example}

 \begin{example}\label{EX:lambda-tr} 
 Let $\lambda=(\Lambda,\nu,\lambda)$ be an involution of $(\bfX,\calO_\bfX)$
 and
 let $n$ be a natural number. Then $\nMat{\calO_{\mathbf{X}}}{n\times n}$ admits a
$\lambda$-involution given by $(\alpha_{ij})_{i,j}\mapsto (\alpha_{ji}^{\lambda})_{i,j}$ on sections and denoted $\lambda\tr$.
If the sections $\alpha_{ij}$ lie in $\calO_{\mathbf{X}}(U)$, then the sections $\alpha_{ji}^{\lambda}$ lie in
$(\Lambda \calO_{\mathbf{X}})(U)$.
\end{example}

\subsection{Morphisms}
\label{subsec:morphisms}

We now give the general definition of a morphism of ringed topoi with involution. Only very few examples of these will
be considered in the sequel.

\medskip

 Recall that a geometric morphism of topoi $f:\bfX\to \bfX'$ consists of two functors $f_*:\bfX\to \bfX'$,
 $f^*:\bfX'\to \bfX$ together with an adjunction between $f^*$ and $ f_*$ and such that $f^*$ commutes with finite
 limits. We shall usually denote the unit and counit natural transformations associated to the adjunction by
 $\eta^{(f)}:\id_{\bfX'}\Rightarrow f_*f^*$ and $\veps^{(f)}:f^*f_*\Rightarrow \id_{\bfX}$, dropping the
 superscript $f$ when there is no risk of confusion. If $(\bfX,\calO)$ and $(\bfX',\calO')$ are ringed topoi, then
 a morphism $f:(\bfX,\calO)\to (\bfX',\calO')$ consists of a geometric morphism of topoi $f:\bfX\to \bfX'$ together
 with a ring homomorphism $f_\#:\calO'\to f_*\calO$, which then corresponds to a ring homomorphism
 $f^\#:f^*\calO'\to \calO$ via the adjunction.
	
	Now let $(\bfX,\calO)$ and $(\bfX',\calO')$ be ringed topoi with involutions $(\Lambda,\nu,\lambda)$ and
 $(\Lambda',\nu',\lambda')$, respectively. Regarding $\Lambda$ and $\Lambda'$ as geometric automorphisms of
 $\bfX$ and $\bfX'$, see Remark~\ref{RM:C-two-action}, a morphism of ringed topoi with involution
 $ (\bfX,\calO)\to (\bfX',\calO')$ should intuitively consist of a morphism of ringed topoi $f$ such that
 $f\circ \Lambda$ is ``equivalent'' to $\Lambda'\circ f$. The specifications of this equivalence, which we now
 give, are somewhat technical.
	
	\begin{definition}\label{DF:morphism-top-w-inv}
 With the previous notation, a morphism of ringed topoi with involution $(\bfX,\calO)\to (\bfX',\calO')$
 consists of a morphism of ringed topoi $f:(\bfX,\calO)\to (\bfX',\calO')$ together with natural isomorphisms
 $\alpha_*:f_*\Lambda\Rightarrow \Lambda'f_*$ and $\alpha^*:f^*\Lambda'\Rightarrow \Lambda f^*$ satisfying the
 following coherence conditions for all objects $X$ in $\bfX$ and $X'$ in $\bfX'$:
		\begin{enumerate}[label=(\arabic*)]
 \item \label{item:DF:alpha-to-alpha} The following diagram, the columns of which are induced by $\alpha^*$ and
 $\alpha_*$ and the rows of which are induced by the relevant adjunctions, commutes.
			\[
			\xymatrix{
			\Hom_{\bfX}(\Lambda f^* X',X) \ar[r]\ar[d] &
			\Hom_{\bfX'}(X',f_*\Lambda X) \ar[d] \\
			\Hom_{\bfX}(f^*\Lambda'X',X) \ar[r] &
			\Hom_{\bfX'}(X',\Lambda'f_*X)
			}
			\]
			\item \label{item:DF:lower-coherence}
			$f_*\nu_X=\nu'_{f_*X}\circ \Lambda'\alpha_{*, X}\circ \alpha_{*,\Lambda X}$
			\item \label{item:DF:lower-ring-coherence}
			$\Lambda' f_\#\circ \lambda'=\alpha_{*,\calO}\circ f_*\lambda \circ f_\#$
 \item \label{item:DF:upper-coherence}
 $f^*\nu'_{X'}=\nu_{f^*X'}\circ \Lambda\alpha^*_{X'}\circ \alpha^*_{\Lambda'X'}$
			\item \label{item:DF:upper-ring-coherence}
			$\lambda\circ f^{\#}=\Lambda f^{\#}\circ \alpha^*_{\calO'}\circ f^*\lambda'$
		\end{enumerate}
		We say that $f$ is \emph{strict} when $\alpha_*=\id$ and $\alpha^*=\id$, so that
 $f_*\Lambda=\Lambda'f_*$ and $\Lambda f^*=f^*\Lambda'$.
		
 We call $f$ an equivalence when $f_*$ is an equivalence of categories and $f_{\#}$ is an isomorphism, in which case the
 same holds for $f^*$ and $f^\#$.
	\end{definition}
	
	\begin{remark}
 Yoneda's lemma and condition~\ref{item:DF:alpha-to-alpha} imply that $\alpha_*$ and $\alpha^*$ determine each
 other. Explicitly, this is given as
		\[
		\alpha^*_{X'}=
		\veps^{(f)}_{\Lambda f^*X'}\circ f^*\nu'_{f_*\Lambda f^*X'}\circ
		f^*\Lambda'\alpha_{*,\Lambda f^*X'}\circ f^*\Lambda'f_*\nu^{-1}_{f^*X'}
		\circ f^*\Lambda'\eta_{X'}\ .
		\]
		Furthermore, provided \ref{item:DF:alpha-to-alpha} holds, conditions \ref{item:DF:lower-coherence} and
 \ref{item:DF:upper-coherence} are equivalent, and so are \ref{item:DF:lower-ring-coherence} and
 \ref{item:DF:upper-ring-coherence}. Thus, in practice, it is enough to either specify
 $\alpha_*:f_*\Lambda\Rightarrow \Lambda'f_*$ and verify \ref{item:DF:lower-coherence} and
 \ref{item:DF:lower-ring-coherence}, or specify $\alpha^*:f^*\Lambda'\Rightarrow \Lambda f^*$ and verify
 \ref{item:DF:upper-coherence} and \ref{item:DF:upper-ring-coherence}. 
	\end{remark}

	In accordance with Remark~\ref{RM:C-two-action}, we will sometimes
	call morphisms of ringed topoi with involution \emph{$C_2$-equivariant}
	morphisms.

	We will usually suppress $\alpha_*$ and $\alpha^*$
	in computations, identifying
	$f_*\Lambda$ with $\Lambda'f_*$ and $\Lambda f^*$ with $f^*\Lambda'$.
	The coherence conditions guarantee that this will not cause inconsistency
	or ambiguity.

	\begin{example}\label{EX:weakly-triv-inv-morphism}
 Let $(\Lambda,\nu,\lambda)$ be a weakly trivial involution of $(\bfX,\calO)$ and let
 $(\Lambda',\nu',\lambda')$ be the trivial involution on $(\bfX,\calO)$, see
 Definition~\ref{DF:trivial-involution}. Then there is $\theta:\Lambda\Rightarrow \id$ such that
 $\theta_{\calO}=\lambda^{-1}$ and $\theta_X\circ \theta_{\Lambda X}=\nu_X$ for all $X$ in $\bfX$, and one can
 readily verify that $(f,\alpha_*,\alpha^*):=(\id_{(\bfX,\calO )},\theta,\theta^{-1})$ defines an equivalence
 from $(\bfX,\calO,\Lambda,\nu,\lambda)$ to $(\bfX,\calO,\Lambda',\nu',\lambda')$, and that every such
 equivalence is of this form.
 		
 		More generally, if $(\Lambda,\nu,\lambda)$ is arbitrary and there exists an equivalence
 $(f,\alpha_*,\alpha^*)$ from $(\bfX,\calO,\Lambda,\nu,\lambda )$ to a ringed topos with a trivial
 involution, then $(\Lambda,\nu,\lambda)$ is weakly trivial in the sense of
 Definition~\ref{DF:trivial-involution}; take
 $\theta_X=\veps_X\circ f^*\alpha_{*,X}\circ \veps_{\Lambda X}^{-1}$.
	\end{example}

 If
 $f: (\mathbf X_1, \sh O_{1}, \Lambda_1, \nu_1, \lambda_1) \to (\mathbf X_2, \sh O_{2}, \Lambda_2, \nu_2,
 \lambda_2)$
 is a morphism of ringed topoi with involution, and if $A$ is an Azumaya $\calO_2$-algebra with a
 $\lambda_2$-involution $\tau$, then $ f^* A \tensor_{f^* \calO_2} \calO_1$ is an Azumaya algebra on
 $ (\mathbf X_1,\calO_1)$ with a $\lambda_1$-involution given by $f^* \tau \tensor \lambda_1$. This induces
 transfer functors
 \[\Az_n(\calO_2)\to \Az_n(\calO_1) \qquad\text{and}\qquad
 \Az_n(\calO_2,\lambda_2)\to \Az_n(\calO_1,\lambda_1)\ .\]
 We shall need a particular instance of these transfer maps
 later.

\begin{example} \label{ex:complexRealization}
 If $X$ is a complex variety with involution $\lambda: X \to X$, then 
 the \'etale ringed topos of $X$, denoted $\mathbf X_{\et}$, becomes a locally ringed topos with involution, as in
 Example \ref{ex:galois}. On the other hand, one may form the topological space $X(\CC)$, equipped with the analytic
 topology, which then has an associated site and consequently an associated topos $\cat X_\top$. One may endow the topos
 $\cat X_\top$ with several different local ring objects depending on the kind of geometry one wishes to carry
 out. There is the sheaf $\sh H$ of holomorphic functions, as set out in 
\cite{grothendieck_techniques_1960}, 
% \cite[Exp.~9]{1962}, 
and there is also the
 sheaf $\sh C$ of continuous $\CC$-valued functions.

 We claim that $(\cat X_\top, \sh C)$ is a locally ringed topos and that there is a ``realization'' morphism
 $(\cat X_\top, \sh C)\to (\cat X_\et, \sh O_X)$ of ringed topoi with involution. An outline of the argument follows.

For every complex variety $U$, there is a unique, functorially-defined \textit{analytic topology} on
 $U(\CC)$; this is established in \cite[Exp.~XII]{grothendieck_revetements_1971}. %an exercise in
 %\cite[Ch.~I.10]{mumford_red_1999}. 
 The functor $U \mapsto U(\CC)$ preserves finite limits. Moreover, if $\{ U_i \to X \}_{i \in I} $ is an \'etale
 covering of $X$, then the family of maps $\{U_i(\CC) \to X(\CC) \}_{i \in I}$ is a jointly surjective family of local
 homeomorphisms, and may therefore be refined by a jointly surjective family $\{V_i \to X(\CC)\}_{i \in I'}$ of open
 inclusions. Since families of this form generate the usual Grothendieck topology on the topological space $X(\CC)$, it
 follows that there is a morphism of sites $f: (X(\C), \top) \to X_\et$, and therefore a morphism of topoi, \cite[III.1
 and IV.4.9.4]{artin_theorie_1972-1}. Complex realization may be applied to $\A^1_\CC$, the representing object for
 $\sh O_X$, to obtain $\CC$, the representing object for $\sh C$ which is the local ring object on
 $X(\CC)$. Therefore, $f$ is a morphism of locally ringed topoi. It is routine to verify that since the involution on
 $X$ becomes the obvious involution on $X(\CC)$ after realization, the morphism $f: \mathbf X_\top \to \mathbf X_\et$
 extends to a strict morphism of locally ringed topoi with involution. In this instance, all the ``coherence'' natural
 isomorphisms appearing are, in fact, identities.
\end{example}

\subsection{Quotients by an Involution}
\label{subsec:quotients}

Let $\lambda=(\Lambda,\nu,\lambda)$ be an involution of 
a locally ringed topos 
$\bfX$.
We would like to consider a
quotient of 
$\bfX$ by the action
of $\lambda$, or equivalently, by the (weak) $C_2$-action it induces. 
In general, however, it is difficult to define a specific
quotient topos in a geometrically reasonable way. For example, if $(\mathbf{X},\calO_\bfX)=(\Sh(X_\et),\calO_X)$ for a scheme $X$ admitting a
$C_2$-action, then the \'etale ringed topoi of both the geometric quotient $X/C_2$, if exists, and the stack $[X/C_2]$
may \textit{a priori} serve as reasonable quotients of $\bfX$.

We therefore ignore the problem of constructing or specifying a quotient of a locally ringed topos with involution and
instead enumerate the properties that such a quotient should possess, declaring any
locally ringed topos possessing these
properties to be satisfactory.

To be precise, we ask for a locally ringed
topos 
$\bfY$, endowed with the trivial involution,
together with a $C_2$-equivariant morphism
$\pi:\bfX\to \bfY$ which satisfy 
certain axioms. 
Recall from Subsection~\ref{subsec:morphisms} that the data of $\pi$ consists of 
a geometric morphism of topoi $\pi=(\pi^*,\pi_*):\mathbf{X}\to \mathbf{Y}$, a ring 
homomorphism $\pi_\#:\calO_\bfY\to \pi_*\calO_\bfX$ (or equivalently, $\pi^\#:\pi^*\calO_\bfY\to \calO_\bfX$) and natural transformations $\alpha_*:\pi_*\Lambda\Rightarrow\pi_*$,
$\alpha^*:\pi^*\Rightarrow \Lambda\pi^*$ satisfying the relations of
Definition~\ref{DF:morphism-top-w-inv}.
We will often suppress $\alpha_*$ and $\alpha^*$, identifying
$\pi_*\Lambda$ with $\pi_*$ and $\Lambda\pi^*$ with $\pi^*$.
In fact, in many of our examples, $\alpha_*$ and $\alpha^*$ will both be the identity.

\begin{definition}\label{DF:exact-quotient} Let $(\mathbf{X},\calO_{{\mathbf{X}}})$ be a locally ringed topos with involution
$\lambda=(\Lambda,\nu,\lambda)$, let $(\mathbf{Y},\calO_{\bfY})$ be a locally ringed topos with a trivial involution, and let
$\pi:(\mathbf{X},\calO_{{\mathbf{X}}})\to (\mathbf{Y},\calO_{\bfY})$ be a $C_2$-equivariant morphism of ringed
topoi. We say that $\pi$ is an \emph{exact quotient} (of $(\mathbf{X},\calO_{{\mathbf{X}}})$ by the given 
$C_2$-action) if
\begin{enumerate}[label=(E\arabic*)]
\item \label{i:sq1} $\pi_\#: \calO_{\bfY}\to \pi_*\calO_{{\mathbf{X}}}$ is the equalizer of
$\pi_*\lambda:\pi_*\calO_{{\mathbf{X}}}\to \pi_*\Lambda \calO_{{\mathbf{X}}}= \pi_*\calO_{{\mathbf{X}}}$ and the identity map
$\id_{\pi_*\calO_\bfX}$,
\item \label{i:sq2} $\pi_*$ preserves epimorphisms.
\end{enumerate}
\end{definition}

\begin{remark}
An exact quotient is in particular a morphism of locally
	ringed topoi, i.e., a morphism of ringed topoi $\pi:\bfX\to \bfY$ satisfying
	the additional condition that $\units{\calO_\bfY}\to \calO_\bfY$ is the pullback
	of $\units{\pi_*\calO_\bfX}\to \pi_*\calO_\bfX$ along $\pi_\#$. 
	Indeed, given $V\in\bfY$, the $V$-sections of the pullback
	consist of pairs $(x,y)\in \pi_*\units{\calO_\bfX}(V)\times \calO_\bfY(V)$
	with $\pi_\# y=x$ in $\pi_*\calO_\bfX(V)$.
	By \ref{i:sq1}, we have $\pi_*\lambda(x)=x$ in $\calO_\bfX(V)$.
	Since $x\in \pi_*\units{\calO_\bfX}(V)=\units{\pi_*\calO_\bfX(V)}$,
	this means that $\pi_*\lambda(x^{-1})=x^{-1}$.
	Thus, again by \ref{i:sq1}, there exists unique $y'\in \calO_\bfY(V)$
	with $\pi_\# y'=x^{-1}$. In particular,
	$\pi_\#(yy')=xx^{-1}=1$ in $\pi_*\calO_\bfX(V)$.
	Since $\calO_\bfY$ is a subobject of $\pi_*\calO_\bfX$
	via $\pi_\#$ (again, by \ref{i:sq1}), this means that $yy'=1$
	in $\calO_\bfY(V)$, so $y\in\units{\calO_\bfY(V)}=\units{\calO_\bfY}(V)$.
	As $x=\pi_\#y$, it follows that $(x,y)$ is the image of a (necessarily unique)
	$V$-section of $\units{\calO_\bfY}$ 
	under the natural map $\units{\calO_\bfY}\to \pi_*\units{\calO_\bfX}\times_{\pi_*\calO_\bfX}\calO_\bfY$,
	as required.
\end{remark}

The name ``exact'' comes from condition~\ref{i:sq2}, which implies in particular that $\pi_*$ preserves exact sequences
of groups. We shall see below that this condition is critical for transferring cohomological data from $\bfX$ to $\bfY$.
Condition~\ref{i:sq1} informally means that $\calO_\bfY$ behaves as one would expect from the subring of $\calO_\bfX$
fixed by $\lambda$ --- such an object cannot be defined in $\bfX$ because the source and target of $\lambda$ are not, in
general, canonically isomorphic.

\medskip

	To motivate Definition~\ref{DF:exact-quotient},
	we now give two fundamental examples of exact quotients.
	However, in order not to digress, we postpone the proof of their exactness
	to Subsection~\ref{subsec:examples-of-exact-quo}, where
 	further examples and non-examples are exhibited.

\begin{example}\label{EX:important-exact-quo-scheme}
	Let $X$ be a scheme and let $\lambda:X\to X$ be an involution.
	A morphism of schemes $\pi:X\to Y$ is called a \emph{good quotient}
	of $X$ relative to the action of $C_2=\{1,\lambda\}$
	if $\pi$ is affine, $C_2$-invariant, and
	$\pi_\#:\calO_Y\to\pi_*\calO_X$ 
	defines an isomorphism of $\calO_Y$ with $(\pi_*\calO_X)^{C_2}$.
	By \cite[Prp.~V.1.3]{grothendieck_revetements_1971} 
	and
	the going-up theorem, these conditions imply
	that $\pi$ is universally surjective 
	and so this agrees
	with the more general definition in \cite[Tag 04AB]{de_jong_stacks_2017}.
	A good $C_2$-quotient of $X$ exists if and only if
	every $C_2$-orbit in $X$ is contained in an affine open subscheme,
	in which case it is also a categorical quotient in the category
	of schemes, hence unique up to isomorphism \cite[Prps.~V.1.3, V.1.8]{grothendieck_revetements_1971}.

	Let $({\mathbf{X}},\calO_{\mathbf{X}})=(\Sh(X_\et),\calO_X)$, and let $\lambda=(\Lambda,\nu,\lambda)$ be the
 involution of $(\bfX,\calO_\bfX)$ induced by $\lambda:X\to X$, see Example~\ref{ex:galois}. Given a good
 $C_2$-quotient, $\pi:X\to Y$, we define an exact quotient $\pi:(\bfX,\calO_\bfX)\to (\bfY,\calO_\bfY)$ relative
 to $\lambda$ by taking $(\bfY,\calO_\bfY)=(\Sh(Y_\et),\calO_Y)$, letting
 $\pi=(\pi^*,\pi_*):\Sh(X_\et)\to \Sh(Y_\et)$ and defining $\pi_\#:\calO_\bfY\to \pi_*\calO_\bfX$ to be the
 canonical extension of $\pi_\#:\calO_Y\to \pi_*\calO_X$ in $\Sh(X_{\Zar})$ to the corresponding ring objects in
 $\Sh(X_{\et})$. The suppressed natural transformations $\alpha_*$ and $\alpha^*$ are both the identity.
 \end{example}

\begin{example}\label{EX:important-exact-quo-top}
 Let $X$ be a Hausdorff topological space, let $\lambda:X\to X$ be a continuous involution, let $Y=X/\{1,\lambda\}$ and
 let $\pi:X\to Y$ be the quotient map. Let $(\bfX,\calO_\bfX)=(\Sh(X),\calC(X,\C))$ and let 
 $\lambda=(\Lambda,\nu,\lambda)$ be the
 involution induced by $\lambda:X\to X$, see Example~\ref{ex:top}. We define an exact quotient
 $\pi:(\bfX,\calO_\bfX)\to (\bfY,\calO_\bfY)$ relative to $\lambda$ by taking
 $(\bfY,\calO_\bfY)=(\Sh(Y),\cont(Y,\C))$, letting $\pi=(\pi^*,\pi_*):\Sh(X)\to \Sh(Y)$ be the geometric morphism
 induced by $\pi:X\to Y$, and defining $\pi_\#:\calC(Y,\C)\to \pi_*\calC(X,\C)$ to be the morphism sending a section
 $f\in \cont(U,\C)$ to
 $f\circ \pi\in \cont(\pi^{-1}(U),\C)=\pi_*\cont(X,\C)(U)$.
 Again, the suppressed natural transformations $\alpha_*$ and $\alpha^*$ are both
 the identity.
\end{example}

	We also record the following trivial example.
	
	\begin{example}\label{EX:trivial-exact-quotient}
		Suppose that the involution $\lambda=(\Lambda,\nu,\lambda)$
		on $(\bfX,\calO_\bfX)$
		is weakly trivial, namely, there is a natural isomorphism
		$\theta:\Lambda\Rightarrow \id$ such that $\theta_X\circ\theta_{\Lambda X}=\nu_X$,
		and $\lambda=\theta^{-1}_{\calO_\bfX}$. 
		Then 
		the identity morphism $\id:(\bfX,\calO_\bfX)\to(\bfX,\calO_\bfX)$
		defines an exact quotient upon taking $\alpha_*=\theta$
		and $\alpha^*=\theta^{-1}$. 
		We call it the
		\emph{trivial quotient} of $(\bfX,\calO_\bfX)$.
		
		More generally, an arbitrary
		exact $C_2$-quotient 
		$\pi:\bfX\to \bfY$		
		will be called trivial when $\pi$ is an equivalence of ringed topoi
		with involution. As noted in Example~\ref{EX:weakly-triv-inv-morphism},
		such a quotient can only exist when the involution
		of 
		$\bfX$ is weakly trivial. 
	\end{example}

	We shall see below (Remark~\ref{RM:C-two-quotients-not-unique}) that
	a locally ringed topos with involution 
	may admit several non-equivalent exact quotients.

\medskip

We turn to establish some properties of exact quotients that will arise in the sequel. The most crucial of these will
be the fact that when 
$\pi:\bfX\to \bfY$ 
is an exact $C_2$-quotient, $\pi_*$ induces an
equivalence between the Azumaya $\calO_\bfX$-algebras and the Azumaya $\pi_*\calO_\bfX$-algebras, and similarly for
Azumaya algebras with a $\lambda$-involution.

\medskip

 The following theorem is a consequence of condition~\ref{i:sq2}.

\begin{theorem}
\label{TH:vanishing}
 Let $\pi: \mathbf X \to \mathbf Y$ be a geometric morphism of topoi such that $\pi_*$ preserves epimorphisms, and let
 $G$ be a group in $\mathbf X$. Then:
 \begin{enumerate}[label=(\roman*)] 
 \item \label{item:TH:vanishing:cat-equiv-of-tors} $\pi_*$ induces an equivalence of categories $\Tors(\mathbf{X}, G) \to
 \Tors(\mathbf{Y}, \pi_*G)$.
 \item \label{item:TH:vaninshin:can-isom} There is a canonical isomorphism $\Hoh^i(\mathbf Y, \pi_* G) \iso \Hoh^i(\mathbf X,
 G)$ when $i=0,1$, and for all $i \ge 0$ when $G$ is abelian. For $i=0$, this is the canonical isomorphism
 $\Hoh^0(\bfY,\pi_*G)=\Hoh^0( \bfX, G)$. For $i=1$, this isomorphism agrees with the one induced by
 \ref{item:TH:vanishing:cat-equiv-of-tors} and
 Proposition~\ref{PR:non-ab-coh-basic-prop}\ref{item:PR:non-ab-coh-basic-prop:torsor}.
 \item If $0 \to G'\to G\to G'' \to 0$ is a short exact sequence of abelian groups then the isomorphism of
 \ref{item:TH:vaninshin:can-isom} gives rise to an isomorphism between the cohomology long exact sequence of $ G'\to
 G\to G''$ and the cohomology long exact sequence of $\pi_*G'\to \pi_*G\to\pi_*G''$.
The same holds for the truncated long exact sequence of parts \ref{item:PR:non-ab-coh-basic-prop:long-ex-seq} and
\ref{item:PR:non-ab-coh-basic-prop:longer-ex-seq} of Proposition~\ref{PR:non-ab-coh-basic-prop} when $G',G,G''$ are not
assumed to be abelian.
 \end{enumerate}
\end{theorem}
\begin{proof}
 \begin{enumerate}[label=(\roman*), align=left, leftmargin=0pt, itemindent=1ex, itemsep=0.3\baselineskip]
 \item We treat the topos as a site in its own canonical topology. In this language,
 \cite{giraud_cohomologie_1971}*{Chap.~V, Sect.~3.1.1.1} says that $\pi_*$ induces an equivalence $\Tors(\mathbf{X},
 G)^{\mathbf{Y}} \to \Tors(\mathbf{Y}, \pi_*G)$, where the source is the category of $G$-torsors $P$ for which there exists a
 covering $U \to *_{\mathbf{Y}}$ such that $P_{\pi^*U}\iso G_{\pi^*U}$. We claim that this applies to all $G$-torsors, and
 so $\pi_*$ induces an equivalence $\Tors(\mathbf{X}, G) \to \Tors(\mathbf{Y}, \pi_*G)$.

 Since a $G$-torsor $P$ is trivialized by itself, it is also trivialized by any object mapping to $P$, for instance by
 $\pi^{*}\pi_* P$. It is therefore sufficient to show that the map $\pi_*P\to *_{\bfY}$ is an epimorphism. Since $P\to
 *_{\bfX}$ is an epimorphism, our assumption implies that $\pi_*P\to \pi_*(*_\bfX)=*_{\bfY}$ is also an epimorphism, so
 the claim is verified.
\item Suppose first that $G$ is abelian. The fact that $\pi_*$ is exact implies that the family of functors
 $\{G\mapsto \Hoh^i(\bfY,\pi_*G)\}_{i\geq 0}$ from abelian groups in $\bfX$ to abelian groups forms a
 $\delta$-functor. Thus, the universality of derived functors implies that the canonical isomorphism
 $\Hoh^0(\bfX, G)\xrightarrow{\sim}\Hoh^0(\bfY, \pi_*G)$ gives rise to a unique natural transformation
 $\Hoh^i(\bfX,G)\to \Hoh^i(\bfX,\pi_*G)$ for any abelian $G$. Since $\pi_*$ takes injective abelian groups to injective
 abelian groups, $\{G\mapsto \Hoh^i(\bfY,\pi_*G)\}_{i\geq 0}$ is an effaceable $\delta$-functor, hence
 universal. Applying the universality of the latter to the natural isomorphism
 $\Hoh^0(\bfY, \pi_*G)\xrightarrow{\sim}\Hoh^0(\bfX, G)$ implies that $\Hoh^i(\bfX,G)\to \Hoh^i(\bfX,\pi_*G)$ is an
 isomorphism.

 We note that if we use Verdier's Theorem, see Subsection~\ref{subsec:coh-of-ab-grps}, to describe $\Hoh^i(\bfX,G)$ and
 $\Hoh^i(\bfY,\pi_*G)$, then the isomorphism is given by sending the cohomology class represented by $g\in
 Z^i(U_\bullet,G)$ to the cohomology class represented by $\pi_*g\in Z^i(\pi_*U_\bullet,\pi_*G)$. Notice that
 $\pi_*U_\bullet$, which is just $\pi_*\circ U_\bullet:\Simp\to \bfY$, is a hypercovering since $\pi_*$ preserves
 epimorphisms and commutes with $\cosk_n$ for all $n$. This isomorphism coincides with the one in the previous paragraph
 because they coincide on the $0$-th cohomology.

 When $i=1$, the map we have just described is defined for an arbitrary group $G$, and we take it to be the canonical
 morphism $\Hoh^1(\bfX,G)\to \Hoh^1(\bfY,\pi_*G)$.
 The construction in the proof of
 Proposition~\ref{PR:non-ab-coh-basic-prop}\ref{item:PR:non-ab-coh-basic-prop:torsor} implies 
 that this map agrees with the one induced by \ref{item:TH:vanishing:cat-equiv-of-tors}
 and Proposition~\ref{PR:non-ab-coh-basic-prop}\ref{item:PR:non-ab-coh-basic-prop:torsor},
 and thus $\Hoh^1(\bfX,G)\to \Hoh^1(\bfY,\pi_*G)$ is an isomorphism.
 
\item In the abelian case, this follows from the argument given in \ref{item:TH:vaninshin:can-isom}, 
	which shows that
 $\{G\mapsto \Hoh^i(\bfX,G)\}_{i\geq 0}$ and $\{G\mapsto \Hoh^i(\bfY,\pi_*G)\}_{i\geq 0}$ are isomorphic
 $\delta$-functors. In the nonabelian case, this follows from the proof of Proposition~\ref{PR:non-ab-coh-basic-prop},
 parts \ref{item:PR:non-ab-coh-basic-prop:long-ex-seq} and \ref{item:PR:non-ab-coh-basic-prop:longer-ex-seq}.
 \qedhere
 \end{enumerate}
 \end{proof}
 
 \begin{remark} When $\pi:\bfX\to \bfY$ is a geometric morphism of topoi, with no further assumptions, one still has
 natural transformations $\Hoh^i(\bfY,\pi_*G)\to \Hoh^i(\bfX,G)$ for all $G$ abelian and $i\geq 0$, or $G$ non-abelian
 and $i=0,1$, 
 and if $\pi_*$ preserves epimorphisms, they coincide with the inverses of the isomorphisms of
 Theorem~\ref{TH:vanishing}. In the abelian case, the construction is given as follows: Using the exactness of $\pi^*$,
 one finds that the canonical map $\Hoh^0(\bfY, A)\to \Hoh^0(\bfX,\pi^* A)$ gives rise to natural transformations
 $\Hoh^i(\bfY,A)\to \Hoh^i(\bfX,\pi^*A)$. Taking $A=\pi_*G$ and composing with the map $\Hoh^i(\bfX,\pi^*\pi_*G)\to
 \Hoh^i(\bfX,G)$, induced by the counit $\pi^*\pi_*G\to G$, one obtains a natural transformation $\Hoh^i(\bfY,\pi_*G)\to
 \Hoh^i(\bfX,G)$. This map can be written explicitly on the level of cocycles, using Verdier's Theorem, and be adapted
 to the non-abelian case when $i=0,1$.
 \end{remark}

 Henceforth, let 
 $\pi:\bfX\to \bfY$ be an exact quotient relative to an
 involution $\lambda=(\Lambda,\nu,\lambda)$ on 
 $\bfX$. 
 We write
 \[R=\pi_*\calO_{{\mathbf{X}}}\qquad\text{and}\qquad S=\calO_\bfY\] for brevity, and, abusing the notation, we let
 $\lambda:R\to R$ denote the involution $\pi_*\lambda: \pi_*\calO_{{\mathbf{X}}}\to\pi_*\Lambda\calO_{{\mathbf{X}}}$.

 We shall use the following lemma freely to identify $\pi_*\GL_n(\calO_{\mathbf{X}})$ with $\GL_n(R)$ and
 $\pi_*\PGL_n(\calO_{\mathbf{X}})$ with $\PGL_n(R)$.

 \begin{lemma}\label{LM:PGLn-preserved-under-pi} For all $n\in\N$, there are canonical isomorphisms
 $\pi_*\nMat{\calO_\bfX}{n\times n}\iso \nMat{R}{n\times n}$, $\pi_*\GL_n(\calO_{{\mathbf{X}}})\iso \GL_n(R)$ and
 $\pi_*\PGL_n(\calO_{{\mathbf{X}}})\iso \PGL_n(R)$.
 \end{lemma}

 \begin{proof} 
 Let $U\in \bfY$. Thanks to the adjunction
 between $\pi^*$ and $\pi_*$, we have a natural
 isomorphism
 $\pi_*\nMat{\calO_\bfX}{n\times n}(U)\cong \nMat{\calO_\bfX}{n\times n}(\pi^*U)
 =\nMat{\calO_\bfX(\pi^*U)}{n\times n}\iso\nMat{\pi_*\calO_\bfX(U)}{n\times n}$.
 This establishes the first isomorphism.

	The second isomorphism is obtained in the same manner.

 The last isomorphism is deduced from the following ladder of short exact sequences
 \begin{equation*} \xymatrix{ 1 \ar[r] & R^\times \ar[r] \ar^\iso[d] & \GL_n(R) \ar[r] \ar^\iso[d] & \PGL_n(R)\ar[r]
 \ar@{.>}^\iso[d] & 1 \\ 1 \ar[r] & \pi_* (\calO_{\mathbf X}^\times) \ar[r] & \pi_*( \GL_n(\calO_{\mathbf{X}}) ) \ar[r]&
 \pi_*\PGL_n( \calO_{\mathbf{X}})\ar[r] & 1.}
 \end{equation*} 
	Here, the left and middle
 isomorphisms follow from the previous paragraph,
 and the bottom row is exact since $\pi_*$ preserves epimorphisms.
\end{proof}

 We shall use the following lemma to identify
 $\PGL_n(R)$ with $\sAut_{\text{$R$-alg}}(\nMat{R}{n\times n})$
 henceforth.

 \begin{lemma}\label{LM:inner-automorphisms}
 The canonical 
 group homomorphism $\PGL_n(R)\to \sAut_{\text{$R$-alg}}(\nMat{R}{n\times n})$ 
 is an isomorphism.
 \end{lemma}
 
 The lemma would follow immediately from the discussion in Subsection~\ref{subsec:Az-algs} if $R$ were a local ring
 object, but this is not the case in general. Using Theorem~\ref{thm:local-involution}, we shall see that $R$
 functions as a ``semilocal'' ring object and therefore the isomorphism still holds.

\begin{proof}
 We need to show that for all $U\in\bfY$, any $R_U$-automorphism $\psi$ of $\nMat{R_U}{n\times n}$ becomes an inner
 automorphism after passing to a covering $V\to U$. Let $\frakp\in \Spec S(U)$. Then $B:=S(U)_\frakp$ is the
 $\lambda$-fixed subring of $A:=R(U)_\frakp$. By Theorem~\ref{thm:local-involution}, $A$ is semilocal. It is then well
 known that $\psi_A$ is an inner automorphism, \cite[III.\S5.2]{knus_quadratic_1991}. Write
 $\psi_A(x)=a_\frakp xa_\frakp^{-1}$ for some $a_\frakp\in \GL_n(A)$. There exists $f_\frakp\in S(U)- \frakp$ such
 that $a_\frakp$ is the image of an element in $\GL_n(R(U)_{f_\frakp})$, also denoted $a_\frakp$, and such that
 $\psi_{R(U)_{f_\frakp}}$ agrees with $x\mapsto a_\frakp xa_\frakp^{-1}$ on $\nMat{R(U)_{f_\frakp}}{n\times n}$, e.g.\
 if they agree on an $R(U)$-basis of $\nMat{R(U)}{n\times n}$. Since $S=\calO_\bfY$ is a local ring object, and since
 $S(U)=\sum_{\frakp\in\Spec S(U)} S(U)f_\frakp$, there exists a covering $\{V_\frakp\to U\}_{\frakp}$ such that
 $f_\frakp\in\units{S(V_\frakp)}$. By construction, $R(U)\to R(V_\frakp)$ factors through $R(U)_{f_\frakp}$, and thus
 $\psi$ is inner on $V_\frakp$ for all $\frakp$, as required.
 \end{proof}

\begin{lem}\label{LM:equiv-relation-quotient} 
 Let $\pi:{\mathbf{X}}\to \bfY$ be a geometric morphism of topoi such that $\pi_*$ preserves epimorphisms. Then $\pi_*$
 preserves quotients by equivalence relations. In particular, for any group object $G$, any $G$-torsor $P$, and any
 $G$-set $X$, there is a canonical isomorphism $\pi_*(P\times^GX)\iso \pi_*P\times^{\pi_*G}\pi_*X$.
 \end{lem}

 \begin{proof} Recall that in a topos ${\mathbf{X}}$, any equivalence relation $Q\to A\times A$ is effective, meaning that
 $Q=A\times_BA$ for some epimorphism $A\to B$---in fact, $A \to B$ must be isomorphic to $A \to A/Q$.

 Since $\pi_*$ preserves epimorphisms and limits, this means that $\pi_*(A/Q)$ is canonically isomorphic to
 $\pi_*A/\pi_*Q$.
 \end{proof}

 We now come to the main result of this section, which allows passage from Azumaya algebras in the locally ringed topos
 $(\mathbf X, \sh{O}_{\mathbf X})$ to Azumaya algebras in the ringed topos $(\mathbf Y, R)$.

 \begin{thm}\label{TH:pi-star-equiv} Suppose 
 $\bfX$ is a locally ringed topos with involution $\lambda=(\Lambda,\nu, \lambda)$, and that $\pi:\bfX\to \bfY$ is an
 exact quotient relative to $\lambda$. Let $R=\pi_* \sh O_{\mathbf X}$ and $n\in\N$. Then the following categories
 are equivalent:
 \begin{enumerate}[label=(\alph*)]
 \item \label{i:a} $\Az_n({\mathbf{X}},\calO_{\mathbf{X}})$, the category of Azumaya $\sh O_{\mathbf X}$-algebras of degree $n$.
 \item \label{i:b} $\Tors({\mathbf{X}},\PGL_n(\calO_{\mathbf{X}}))$, the category of $\PGL_n(\calO_{\mathbf{X}})$-torsors on
 $\mathbf X$.
 \item \label{i:c} $\Az_n(\bfY,R)$, the category of Azumaya $R$-algebras of degree $n$.
 \item \label{i:d} $\Tors(\bfY,\PGL_n(R))$, the category of $\PGL_n(R)$-torsors on $\mathbf Y$.
 \end{enumerate} The equivalence between \ref{i:a} and \ref{i:b}, resp.\ \ref{i:c} and \ref{i:d}, is the one given in
 Proposition~\ref{PR:equiv-Az-PGLnR}, and the equivalence between \ref{i:a} and \ref{i:c}, resp.\ \ref{i:b} and
 \ref{i:d}, is given by applying $\pi_*$.
 \end{thm}

 In the context of Example~\ref{EX:important-exact-quo-scheme}, where our exact quotient is induced from a good
 $C_2$-quotient of schemes $\pi:X\to Y$, the theorem says that every Azumaya algebra $A$
 over $X$ admits an \'etale covering $U$ of $Y$ (not of $X$) such that $A$ becomes a matrix algebra after base change
 to $X\times_YU$, and every automorphism $\psi$ of $A$ admits an \'etale covering $V$ of $Y$ (again, not of $X$) such
 that $\psi$ becomes an inner automorphism after passing to $X\times_YV$. Theorem~\ref{TH:pi-star-equiv} can be regarded as a 
 generalization of this fact.

 \begin{proof} 
 The equivalence between \ref{i:a} and \ref{i:b}, resp.\ \ref{i:c} and \ref{i:d}, is
 Proposition~\ref{PR:equiv-Az-PGLnR}; here we identified $\PGL_n(R)$ with $\sAut_R(\nMat{R}{n\times n})$ as in
 Lemma~\ref{LM:inner-automorphisms}. The equivalence between \ref{i:b} and \ref{i:d} is
 Theorem~\ref{TH:vanishing}\ref{item:TH:vanishing:cat-equiv-of-tors}, together with the isomorphism
 $\pi_*\PGL_n(\calO_{{\mathbf{X}}})\cong \PGL_n(R)$ of Lemma~\ref{LM:PGLn-preserved-under-pi}. It now follows from
 Lemma~\ref{LM:equiv-relation-quotient} that the induced equivalence between \ref{i:a} and \ref{i:c} is given by
 applying $\pi_*$.
 \end{proof}
 
 	\begin{remark}
 The exact quotient $\pi:\bfX\to \bfY$ gives rise to a morphism of ringed topoi with involution
 $\hat{\pi}:(\bfX,\calO_\bfX)\to (\bfY,R)$ by setting $\hat{\pi}_\#:R\to \pi_*\calO_\bfX $ to be the identity.
 The induced transfer functor $\Az_n(R)\to\Az_n(\calO_\bfX)$ is then an inverse to the equivalence
 $\pi_*:\Az_n(\calO_\bfX)\to \Az_n(R) $.
	\end{remark}

	\begin{remark}\label{RM:lambda-stable-degree}
 As phrased, Theorem~\ref{TH:pi-star-equiv} addresses Azumaya $\calO_\bfX$-algebras of constant degree $n$
 only. These constitute all Azumaya algebras when $\bfX$ is connected, but not in general. If we replace $n$
 with a global section of the constant sheaf $\N$ on $\bfX$, then Theorem~\ref{TH:pi-star-equiv} still holds,
 provided that $n$ is fixed by $\Lambda:\Hoh^0(\bfX,\N)\to \Hoh^0(\bfX,\Lambda \N)= \Hoh^0(\bfX,\N)$, in which
 case $n$ can be understood as an element of $\Hoh^0(\bfY,\N)$. Since Theorem~\ref{TH:pi-star-equiv} is used
 throughout, we tacitly assume henceforth that all Azumaya $\calO_\bfX$-algebras have degrees that are fixed
 under $\Lambda$. This makes little difference in practice, because any Azumaya $\calO_\bfX$-algebra is Brauer
 equivalent to another Azumaya $\calO_\bfX$-algebra of degree which is fixed by $\Lambda$.
\end{remark}

We define an involution on the ringed topos $(\bfY,R)$ by setting $\Lambda=\id_\bfY$, $\nu=\id$ and
$\lambda_{(\bfY,R)}=\pi_*\lambda_{(\bfX,\calO_\bfX)}$. Since $\pi_*$ preserves products, for any
$\calO_{\mathbf{X}}$-algebra $A$, we have $\pi_*\Lambda A^\op=\Lambda\pi_*A^\op$ as $R$-algebras. However, we alert the
reader that while $\Lambda\pi_*A^\op=\pi_*A^\op$ as noncommutative rings, the $R$-algebra structure of
$\Lambda\pi_*A^\op$ is obtained from the $R$-algebra structure of $\pi_*A^\op$ by twisting via
$\lambda:R\to \Lambda R=R$, as explained in Subsection~\ref{subsec:involutions}. Theorem~\ref{TH:pi-star-equiv} now
implies:

 \begin{cor}\label{CR:equiv-Az-with-inv} For all $n\in\N$, the functor $\pi_*$ induces an equivalence between 
 $\Az_n(\calO_\bfX,\lambda)$, the category of degree-$n$ Azumaya $\calO_{\mathbf{X}}$-algebras with
 $\lambda$-involution, and $\Az_n(R,\lambda)$, the category of degree-$n$ Azumaya $R$-algebras with
 $\lambda$-involution.
 \end{cor}

 At this point we conclude that results proved so far allow one to shift freely between $(\bfX,\calO_\bfX)$ and
 $(\bfY,R)$, at least when Azumaya algebras, possibly with a $\lambda$-involution, are concerned. Of these two
 contexts, we shall work most often in the second, since this is technically easier. That said, the starting point is
 always a locally ringed topos $(\bfX,\calO_\bfX)$ with involution $\lambda=(\Lambda,\nu,\lambda)$, and the choice of
 the corresponding $\bfY$, $R$ and $\lambda_{(\bfY,R)}$ is not in general uniquely determined by the initial data.

\subsection{Examples of Exact Quotients}
\label{subsec:examples-of-exact-quo}

We now turn to providing various examples of exact $C_2$-quotients. In particular, we will prove that
Examples~\ref{EX:important-exact-quo-scheme} and~\ref{EX:important-exact-quo-top} are exact $C_2$-quotients. It will
also be shown that any locally ringed topos with involution admits an exact quotient.

\medskip

Of the two conditions of Definition~\ref{DF:exact-quotient}, condition~\ref{i:sq2} is harder to establish. The following
lemma is our main tool in proving it.

\begin{lemma}\label{LM:finite-morph} Let $\mathbf{X}$, $\mathbf{Y}$ be topoi and let $\pi:\mathbf{X}\to \mathbf{Y}$ be a geometric
 morphism. Suppose that $\mathbf{Y}$ has a conservative family of points $\{p_i:\mathbf{pt}\to \bfY\}_{i \in I}$ with
 the property that for each $i \in I$, there exists a set of points $\{j_n:\mathbf{pt} \to \bfX\}_{n\in {N_i}}$ such
 that the functors $U\mapsto p_i^*\pi_*U$ and $ U\mapsto \prod_{n\in N_i} j_n^*U$ from $\bfX$ to $\mathbf{pt}$ are
 isomorphic. Then $\pi_*$ preserves epimorphisms.
\end{lemma}

\begin{proof}
 If $p_i:\mathbf{pt}\to \bfY$ is a point as in the lemma, then $p_i^*\pi_*$ preserves epimorphisms because each $j_n^*$
 does. By assumption, a morphism $\psi$ in $\bfY$ is an epimorphism if and only if $p_i^*\psi$ is an epimorphism for
 any $p_i:\mathbf{pt}\to \bfY$ as in the statement, so $\pi_*$ preserves epimorphisms.
\end{proof}

Informally, a geometric morphism satisfying the conditions of Lemma~\ref{LM:finite-morph} can be regarded as having
``discrete fibres''. It can also be thought of as a generalization of a finite morphism in algebraic geometry, thanks to
the following corollary.

\begin{cor}\label{CR:finite-etale-morphism} Let $\pi:X\to Y$ be a finite morphism of schemes. Then the direct image
 functor $\pi_*:\Sh(X_\et)\to \Sh(Y_\et)$ preserves epimorphisms.
\end{cor}

This arises in the proof that the higher direct images vanish for cohomology with abelian coefficients, \cite[Tag
03QN]{de_jong_stacks_2017}. We have included a proof here in order to present a modification later.
\begin{proof}

 Recall (\cite[Exp.~VIII, \S3--4]{artin_theorie_1972}) that the points of $\Sh(Y_\et)$ are constructed as follows:
 Given $y\in Y$, choose a cofiltered system of \'etale neighbourhoods $\{(U_\alpha,u_\alpha)\to (Y,y)\}_\alpha$ such
 that $\lim_\alpha U_\alpha=\Spec B$, where $B$ is a strictly henselian ring, necessarily isomorphic to the strict
 henselization of $\calO_{Y,y}$. Then the functor $i^*:F\mapsto \colim_\alpha F(U_\alpha)$ from $\Sh(Y_\et)$ to
 $\mathbf{pt}$ and its right adjoint $i_*$ define a point $i:\mathbf{pt} \to\Sh(Y_\et)$, and these points form a
 conservative family,
 \cite[Thm.~VIII.3.5]{artin_theorie_1972}.
 %every point is obtained in this manner \cite[Lm.~54.18.13]{de_jong_stacks_2017}.

 Let $i:\mathbf{pt}\to \Sh(Y_\et)$ be such a point and write $V_\alpha=U_\alpha\times_Y X$. For all sheaves $F$ on
 $X_\et$, we have $i^*\pi_*F=\colim_\alpha F(V_\alpha)$. Note that
 $\lim_\alpha V_\alpha=\lim U_\alpha\times_YX=\Spec B\times_Y X$. Since $\pi$ is finite, $\Spec B\times_Y X=\Spec A$
 where $A$ is a finite $B$-algebra.

 By \cite[Tag 04GH]{de_jong_stacks_2017}, $A=\prod_{n=1}^t A_n$ where each $A_n$ is a henselian ring. Since the residue
 field of each $A_n$ is finite over the separably closed residue field of $B$, each $A_n$ is strictly henselian.
 Letting $e_1,\dots,e_t$ be the primitive idempotents of $A$, we may assume, by appropriately thinning the family
 $\{U_\alpha\to Y\}_{\alpha}$, that $e_1,\dots,e_t$ are defined as compatible global sections on each $V_\alpha$. This
 allows us to write $V_\alpha$ as $\bigsqcup_n V_{n,\alpha}$ such that $\lim_\alpha V_{n,\alpha}=\Spec A_n$ for all
 $n$. Let $x_n$ and $v_{n,\alpha}$ denote the images of the closed point of $\Spec A_n$ in $X$ and $V_{n,\alpha}$,
 respectively, and let $j_n:\mathbf{pt}\to \Sh(X_\et)$ denote the point corresponding to the filtered system
 $\{(V_{n,\alpha},v_{n,\alpha})\to (X,x_n)\}_{\alpha}$. Since we can commute a directed colimit past a finite limit, we
 have shown that
\[ i^*\pi_*F=\colim_\alpha F(V_\alpha)=\colim_\alpha\prod_n F(V_{n,\alpha})=\prod_n j_n^*F\ ,
\] and the result now follows from Lemma~\ref{LM:finite-morph}.
\end{proof}

\begin{cor}\label{CR:finite-top-morphism}
 Let $\pi:X\to Y$ be a continuous morphism of topological spaces such that:
 \begin{enumerate}[label=(\arabic*)]
 \item \label{item:CR:finite-top-morphism:inverse-basis} For any $y\in Y$ and any open neighbourhood $U \supseteq
 \pi^{-1}(y)$, there exists an open neighbourhood $V$ of $y$ such that $\pi^{-1}(V) \subseteq U$.
 \item \label{item:CR:finite-top-morphism:very-disc-fibre} For any $y \in Y$, the fibre $\pi^{-1}(y)$ is finite and,
 letting $x_1,\dots, x_t\in X$ denote the points lying over $y$, there exist disjoint open sets $\{U_i\}_{i=1}^t$ such
 that $x_i \in U_i$. 
 \end{enumerate} Then $\pi_*:\Sh(X)\to \Sh(Y)$ 
 preserves epimorphisms.
\end{cor}

 The hypotheses are satisfied when $\pi: Y \to X$ is 
 a finite covering space map of Hausdorff spaces, or a closed embedding
 of Hausdorff spaces, for instance.
 
 It is also easy to see that condition~\ref{item:CR:finite-top-morphism:inverse-basis} is equivalent to $\pi$ being
 closed, and condition~\ref{item:CR:finite-top-morphism:very-disc-fibre} is equivalent to $\pi$ having finite fibres and
 being separated in the sense that the image of the diagonal map $X\to X\times_Y X$ is closed. \benw{Check whether this
 is the same as $G \times X \to X \times X$ being proper---in the category of Hausdorff spaces to saying the inverse
 image of a compact set is compact.}
	
\begin{proof}
 Again, we use Lemma~\ref{LM:finite-morph}. For $\Sh(Y)$, the points are induced by inclusion maps $i: \{y\} \to Y$ as
 $y$ rages over $Y$, \cite[Chap.~VII, \S 5]{mac_lane_sheaves_1992}. Fix such an inclusion, let $x_1,\dots,x_t$ denote
 the points in $\pi^{-1}(y)$, and let $j_n: \{x_n\} \to X$ denote the inclusion maps. The corresponding morphisms on
 the topoi of sheaves will be denoted by the same letters.

 By definition, for any sheaf $F$ on $X$, we have
 \[ i^* \pi_* \sh F=\colim_{U \ni y} \sh{F}(\pi^{-1}(U))
 \]
 where the colimit is taken over all open neighbourhoods of $y$. Using
 condition~\ref{item:CR:finite-top-morphism:very-disc-fibre}, choose disjoint open neighbourhoods $\{V_n\}_{n=1}^t$ with
 $x_n\in V_n$. Condition~\ref{item:CR:finite-top-morphism:inverse-basis} implies that the family
 $\{{V_n\cap \pi^{-1}(U)}\where\text{$U$ is an open neighbourhood of $y$}\}$ is a basis of open neighbourhoods of $x_n$,
 hence
 \[ \colim_{U \ni y} \sh{F}(\pi^{-1}(U))=\prod_{n=1}^t\colim_{U \ni y} \sh{F}(\pi^{-1}(U) \cap V_n)= \prod_{n=1}^t
 j_n^*\sh F \ .
 \] It follows that $i^*\pi_*=\prod_{n=1}^t j_n^*$, so the proof is complete.
\end{proof}

 \begin{theorem}\label{TH:important-exact-quotients} Suppose that
 \begin{enumerate}[label=(\roman*)]
 \item \label{item:TH:important-exact-quotients:schemes} $X$ is a scheme, $\lambda:X\to X$ is an involution,
 $\pi:X\to Y$ is a good quotient relative to $\{1,\lambda\}$, $({\mathbf{X}},\calO_{\mathbf{X}})=(\Sh(X_\et),\calO_X)$
 and $(\bfY,\calO_\bfY)=(\Sh(Y_\et),\calO_Y)$ (Example~\ref{EX:important-exact-quo-scheme}), or
 \item \label{item:TH:important-exact-quotients:top-spaces} $X$ is a Hausdorff topological space, $\lambda:X\to X$ is an
 involution, $Y=X/\{1,\lambda\}$ and $\pi:X\to Y$ is the quotient map,
 $({\mathbf{X}},\calO_{\mathbf{X}})=(\Sh(X),\cont(X,\C))$ and $(\bfY,\calO_\bfY)=(\Sh(Y),\cont(Y,\C))$
 (Example~\ref{EX:important-exact-quo-top}).
 \end{enumerate} Then the morphism $\pi:(\mathbf{X},\calO_{{\mathbf{X}}})\to (\mathbf{Y},\calO_{\bfY})$ induced by $\pi:X\to Y$
 is an exact quotient relative to the involution induced by $\lambda$.
 \end{theorem}
 
\begin{proof} 
\begin{enumerate}[label=(\roman*), align=left, leftmargin=0pt, itemindent=1ex, itemsep=0.3\baselineskip]
\item The fact that $\pi_*:\Sh(X_\et)\to \Sh(Y_\et)$ preserves epimorphisms is shown as in the proof of
 Corollary~\ref{CR:finite-etale-morphism}, except one has to replace \cite[Tag 04GH]{de_jong_stacks_2017} by
 Corollary~\ref{CR:henselian-ring-with-involution}. Checking that $\calO_\bfY$ is the coequalizer of
 $\lambda,\id:\pi_*\calO_{\bfX}\to \pi_*\calO_\bfX$, amounts to showing that for any \'etale morphism $U\to Y$,
 $\Hoh^0 (U,\calO_U)$ is the fixed ring of $\lambda$ in $\Hoh^0(U\times_YX,\calO_{U\times_YX})$. In fact, it is enough
 to check this after base changing to an open affine covering $\{Y_i\to Y\}$, so we may assume that $U\to Y$ factors as
 $U\to Y_0\to Y$ with $Y_0$ open and affine. Write $Y_0=\Spec B$ and $U=\Spec B'$. Since $X\to Y$ is affine, we may
 further write $Y_0\times_YX=\Spec A$. The assumption that $X\to Y$ is a good quotient relative to $\{1,\lambda\}$
 implies that the sequence of $B$-modules $0\to B\to A\xrightarrow{a\mapsto a-a^\lambda} A$ is exact. Since $B'$ is
 flat over $B$, the sequence $0\to B'\to {A\tensor_B B'}\xrightarrow{a\mapsto a-a^\lambda} A\tensor_B B'$ is exact, and
 hence $B'=\Hoh^0 (U,\calO_U)$ is the fixed ring of $\lambda$ in $A\tensor_BB'=\Hoh^0(U\times_YX,\calO_{U\times_YX})$.
 
\item Conditions \ref{item:CR:finite-top-morphism:inverse-basis} and \ref{item:CR:finite-top-morphism:very-disc-fibre}
 of Corollary \ref{CR:finite-top-morphism} are easily seen to hold, hence $\pi_*$ preserves epimorphisms. It remains to
 show that $\sh{O}_{Y}$ is the equalizer of $\pi_*\sh{O}_X$ under the action of $C_2=\{1,\lambda\}$. To this end, let
 $U \subseteq Y$ be an open set. The $C_2$-action on $X$ restricts to an action on $\pi^{-1}(U)$ and
 $\pi^{-1}(U)/C_2 = U$. In particular, $\sh{O}_Y(U)=C(U,\C)$ is in natural bijection with the set of functions in
 $\pi_*\sh{O}_X(U)=\sh{O}_X(\pi^{-1}(U))=C(\pi^{-1}(U),\C)$ that are fixed under the $C_2$-action. This means that
 $\sh{O}_Y$ is the fixed subsheaf of $\pi_*\sh{O}_X$ under the action of $C_2$.\qedhere
 \end{enumerate}
 \end{proof}
 
 \begin{remark} The proofs of Theorem~\ref{TH:important-exact-quotients}\ref{item:TH:important-exact-quotients:schemes}
 and Corollary~\ref{CR:finite-etale-morphism} can be modified to work for the Nisnevich site of a scheme --- simply
 replace \'etale neighbourhoods by Nisnevich neighbourhoods and strictly henselian rings by henselian rings.
 Disregarding set-theoretic problems, the large \'etale and Nisnevich sites can be handled similarly, using suitable
 conservative families of points, provided one assumes in
 Theorem~\ref{TH:important-exact-quotients}\ref{item:TH:important-exact-quotients:schemes} that
 $\calO_Y\to \pi_*\calO_X\xrightarrow{\lambda-\id} \pi_*\calO_X$ splits in the middle, which is the case when
 $2\in\units{\calO_X}$. This assumption guarantees that the sequence remains exact after base-change to any $Y$-scheme
 $U$, not necessarily flat.

Likewise, the proofs Theorem~\ref{TH:important-exact-quotients}\ref{item:TH:important-exact-quotients:top-spaces}
 and Corollary~\ref{CR:finite-top-morphism} can be modified to work for the large site of a topological space.
\end{remark}

	The next examples bring several situations
	where condition \ref{i:sq1} of Definition~\ref{DF:exact-quotient}
	is satisfied while condition~\ref{i:sq2} is not.
	They also show that some
	of the assumptions made in Theorem~\ref{TH:important-exact-quotients}
	cannot be removed in general.

\begin{example}\label{EX:not-exact-I}
	Let $R$ be a strictly henselian discrete valuation ring
	with fraction field $K$. Let $X$ denote the scheme obtained
	by gluing two copies of $Y:=\Spec R$ along $\Spec K$, and let
	$\lambda:X\to X$ denote the involution exchanging these two copies.
	The morphism $\pi:X\to Y$ which restricts to the identity
	on each of the copies of $\Spec R$ is a geometric quotient relative
	to $C_2=\{1,\lambda\}$ in the sense of \cite[Tag 04AD]{de_jong_stacks_2017}, 
	namely, $\calO_Y=(\pi_*\calO_X)^{C_2}$, $Y=X/C_2$ 
	as topological spaces, and the latter property
	holds after base change. In particular, $\pi:X\to Y$ is the $C_2$-quotient
	of $X$ in the category of schemes. However, the induced $C_2$-equivariant 
	morphism
	$\pi:(\Sh(X_\et),\calO_X)\to (\Sh(Y_\et),\calO_Y)$ is not 
	an exact quotient, the reason being that $\pi_*$
	does not preserve epimorphisms.
	
	To see this, fix a non-trivial abelian group $A$, which will be regarded as a constant sheaf on the appropriate
 space, and let $i,j:\Spec R\to X$ denote the inclusions of the two copies of $\Spec R$ in $X$. Since $i$ is an
 open immersion, we can form the extension-by-$0$ functor $i_!:\Sh((\Spec R)_\et)\to \Sh(X_\et)$, which is left
 adjoint to $i^*$. Let $F=i_!i^*A\oplus j_!j^*A$. The counit maps $\veps^{(i)}:i_!i^*A\to A$ and
 $\veps^{(j)}:j_!j^*A\to A$ give rise to a morphism $\psi:F\to A$ given by
 $(x\oplus y)\mapsto (\veps^{(i)}x+\veps^{(j)}y)$ on sections. This morphism is surjective, as can be easily
 seen by checking the stalks. However $\pi_*\psi:\pi_*F\to \pi_*A$ is not surjective, as can be seen by noting
 that $\pi_*F(Y)=0$, $\pi_*A(Y)=A\neq 0$, and any \'etale covering of $Y$ has a section, because $R$ is strictly
 henselian.
	
	This example does not stand in contradiction to
 Theorem~\ref{TH:important-exact-quotients}\ref{item:TH:important-exact-quotients:schemes} because $\pi:X\to Y$
 is not affine, and hence not a good quotient.
\end{example}

\begin{example}\label{EX:not-exact-II}
 Let $X$ be an infinite set endowed with the cofinite topology, let $\lambda:X\to X$ be an involution acting freely on
 $X$, and let $\pi:X\to Y=X/C_2$ be the quotient map. Then the induced $C_2$-equivariant morphism
 $\pi:(\Sh(X),\cont(X,\C))\to (\Sh(Y),\cont(Y,\C))$ is not an exact $C_2$-quotient, because $\pi_*$ fails to preserve
 epimorphisms. This is shown as in Example~\ref{EX:not-exact-I}, except here one chooses $x\in X$ and uses the open
 embeddings $i:X-\{x\}\to X$ and $j:X-\{\lambda(x)\}\to X$. We conclude that in
 Theorem~\ref{TH:important-exact-quotients}\ref{item:TH:important-exact-quotients:top-spaces}, the assumption that $X$
 is Hausdorff in cannot be removed in general, even when $\lambda$ acts freely on $X$.
\end{example}

\begin{example}
 Let $R$ be a principal ideal domain admitting exactly two maximal ideals, $\fraka$ and $\frakb$. Suppose that there
 exists an involution $\lambda:R\to R$ exchanging $\fraka$ and $\frakb$, and moreover, that the fixed ring of
 $\lambda$, denoted $S$, is a discrete valuation ring. Let $X=\Spec R$, $Y=\Spec S$ and let $\pi:X\to Y$ be the
 morphism adjoint to the inclusion $S\to R$. Then $\pi:X\to Y$ is a good quotient relative to $\lambda:X\to X$, but the
 induced $C_2$-equivariant morphism $\pi:(\Sh(X_{\Zar}),\calO_X)\to (\Sh(Y_{\Zar}),\calO_Y)$ is not an exact quotient,
 because, yet again, $\pi_*$ does not preserve epimorphisms. Again, this is checked as in Example~\ref{EX:not-exact-I}
 by using the open embeddings $i:\Spec R_\fraka \to \Spec R$ and $j:\Spec R_\frakb \to \Spec R$. This shows that we
 cannot, in general, replace the \'etale site with the Zariski site in
 Theorem~\ref{TH:important-exact-quotients}\ref{item:TH:important-exact-quotients:schemes}, even when $\pi:X\to Y$ is
 quadratic \'etale.
\end{example}

\begin{remark}\label{RM:fppf-site}
 Let $X$ be a scheme, let $\lambda:X\to X$ be an involution and let $\pi:X\to Y$ be a good quotient relative to
 $\{1,\lambda\}$. Then the associated morphism of $\fppf$ topoi
 $\pi:(\Sh(X_{\fppf}),\calO_X)\to (\Sh(Y_{\fppf}),\calO_Y)$ is not exact in general, even when $\pi$ is an fppf
 morphism.

 For example, let $k$ be a field characteristic $\neq 2$, and consider $X=\Spec k[\veps\where \veps^2=0]$, $Y=\Spec k$
 and the $k$-involution $\lambda$ sending $\veps$ to $-\veps$. Then $x\mapsto x^2: \calO_X \to \calO_X $ is surjective
 as morphism in $\Sh(X_\fppf)$, but its pushforward to $\Sh(Y_{\fppf})$ is not, because $\veps$ is not in the image of
 $x\mapsto x^2: A[\veps\where \veps^2=0]\to A[\veps\where\veps^2=0]$ for all commutative $k$-algebras $A$.
	%\footnote{
	%This counterexample is inspired by an example of Kestutis Cesnavicius on MathOverflow,
	%\href{https://mathoverflow.net/q/322653}{https://mathoverflow.net/q/322653}.}
	
 Nevertheless, when $\pi:X\to Y$ is finite and locally free, Theorem~\ref{TH:pi-star-equiv} and
 Corollary~\ref{CR:equiv-Az-with-inv} still hold, the reason being that $\PGL_n(\calO_X)$ and
 $\PGL_n(R)=\pi_*\PGL_n(\calO_X)$ are both represented by smooth affine group schemes over $X$ (use
 \cite[Prop.~7.6.5(h)]{bosch_1990_Neron_models}), and hence their \'etale and fppf cohomologies coincide
 \cite[Thm.~11.7, Rmk.~11.8(3)]{grothendieck_1968_groupe_de_Brauer_III}. As a result, some theorems in the next
 sections, e.g.\ Theorems \ref{TH:ct-determines-local-iso} and~\ref{TH:Saltman-ramified}, also hold in the context of
 $\fppf$ ringed topoi associated to a finite locally free good $C_2$-quotient of schemes $\pi:X\to Y$.
\end{remark}

We finish with demonstrating that every locally ringed topos 
$\bfX$ with involution 
$\lambda=(\Lambda,\nu,\lambda)$
admits a canonical exact quotient, sometimes called the ``homotopy fixed points'', as in \cite[Section
2]{merling_equivariant_2017}. We denote this exact quotient by
$\pi:\bfX\to [\bfX/C_2]$.

As the notation suggests, when $(\bfX,\calO_\bfX)=(\Sh(X_\et),\calO_X)$ for a scheme $X$, the ringed topos $[\bfX/C_2]$
will be equivalent to the \'etale ringed topos of the Deligne--Mumford stack $[X/C_2]$. Indeed, the objects of
$[\bfX/C_2]$ will be \emph{$C_2$-equivariant sheaves}, the data of which are equivalent to specifying a sheaf on the
\'etale site of $[X/C_2]$; this is explained for coherent sheaves in
\cite[Example~7.21]{vistoli_1989_intersection_theory}, but the principle works for set-valued sheaves (in the sense of
\cite[{\href{https://stacks.math.columbia.edu/tag/06TN}{Tag 06TN}}]{de_jong_stacks_2017}) as well.

\begin{construction}\label{CN:stacky-quotient}
 Define the category $[\bfX/C_2]$ as follows: The objects of $[\bfX/C_2]$ consist of pairs $(U,\tau)$, where $U$ is an
 object of $\bfX$ and $\tau:U\to \Lambda U$ is a morphism satisfying $\Lambda\tau\circ \tau=\id_U$. In other words, the
 objects of $[\bfX/C_2]$ are objects of $\bfX$ equipped with an involution, or a $C_2$-action. Morphisms in
 $[\bfX/C_2]$ are defined as commuting squares
\[ \xymatrix{ U \ar^\tau[r] \ar^f[d] & \Lambda U \ar^{\Lambda f}[d] \\ U' \ar^{\tau'}[r] & \Lambda U'. } \]

Define $\pi^*:[\bfX/C_2]\to\bfX$ to be the forgetful functor $(U,\tau)\mapsto U$, and define $\pi_*:\bfX\to [\bfX/C_2]$
to be the functor sending $U$ to $(U\times \Lambda U, \tau_U)$ where
$\tau_U:U\times \Lambda U\to \Lambda(U\times \Lambda U)=\Lambda U\times U$ is the interchange morphism. For a morphism
$\phi:U\to V$ in $\bfX$, let $\pi_*\phi=\phi\times \Lambda\phi$. The functor $\pi^*$ is easily seen to be left adjoint
to $\pi_*$ with the unit and counit of the adjunction given by
$u\mapsto (u,\tau u):(U,\tau)\to \pi_*\pi^*(U,\tau)=(U\times \Lambda U,\tau_U)$ and
$(v,v')\mapsto v: \pi^*\pi_*V=V\times \Lambda V\to V$ on the level of sections (in $\bfX$).
 
 	For objects $V$ in $\bfX$ and $(U,\tau)$ in $[\bfX/C_2]$,
 	let $\alpha_{*,V}$ denote
	the interchange morphism $(\Lambda V\times V,\tau_{\Lambda V})\to
	(V\times \Lambda V,\tau_V)$ and
	let $\alpha^*_{(U,\tau)}$ denote
	$\tau: U\to \Lambda U$.
	Then $\alpha_*$ is a natural isomorphism
	$\pi_*\Lambda\Rightarrow \pi_*$ and $\alpha^*$ is a natural isomorphism
	$\pi^*\Rightarrow \Lambda\pi^*$. We alert
 the reader that these
 natural isomorphisms are in general not the identity
 transformations, even when the involution $\lambda$
 is trivial.
 
 Define the ring object $\calO_{[\bfX/C_2]}$ in $[\bfX/C_2]$ to be $(\calO_\bfX,\lambda)$ with the obvious ring
 structure. Finally, define $\pi_\#:\calO_{[\bfX/C_2]}\to \pi_*\calO_\bfX$ to be
 $(\calO_\bfX,\lambda)\to(\calO_\bfX\times\Lambda \calO_\bfX,\tau_{\calO_\bfX})$, where the underlying morphism
 $\calO_\bfX\to \calO_\bfX\times \Lambda\calO_\bfX$ is given by $x\mapsto (x,x^\lambda)$ on sections.
\end{construction}

\begin{prp}\label{PR:stacky-quotient}
	In Construction~\ref{CN:stacky-quotient}, the following hold:
\begin{enumerate}[label=(\roman*)] 
		\item $[\bfX/C_2]$ is a Grothendieck topos.
		\item $\pi:=(\pi^*,\pi_*):\bfX\to [\bfX/C_2]$ is an essential geometric morphism of topoi.
		\item A family of morphisms 
		$\{(U_i,\tau_i)\to (U,\tau)\}_{i\in I}$ in $[\bfX/C_2]$
 is a covering if and only if $\{U_i\to U\}_{i\in I}$ is a covering in $\bfX$.
 		\item $\calO_{[\bfX/C_2]}$ is a local ring object in $[\bfX/C_2]$.
 		\item 
		$(\pi^*,\pi_*,\pi_\#,\alpha^*,\alpha_*)$
		defines an exact quotient		
 		$\pi:(\bfX,\calO_\bfX)\to ([\bfX/C_2],\calO_{[\bfX/C_2]})$ 
 		relative to $\lambda$.
 	\end{enumerate}
\end{prp}

\begin{proof}
	We first introduce the functor
	$\pi_!:\bfX\to [\bfX/C_2]$ given by sending an object $U$ to $(U\sqcup \Lambda U, \sigma_U)$ where $\sigma_U:U\sqcup \Lambda U\to \Lambda(U \sqcup
 \Lambda U)=\Lambda U\sqcup U$ is the interchange morphism.
 It is routine to check that $\pi_!$ is left adjoint to $\pi^*$;
 the unit map is the canonical embedding $V\to \pi^*\pi_!V=V\sqcup \Lambda V$ 
 and the counit map is the map $\pi_!\pi^*(U,\tau)=(U\sqcup \Lambda U,\sigma_U)\to (U,\tau)$
 restricting to $\id_U$ on $U$ and to $\tau^{-1}$ on $\Lambda U$.
	The existence of adjoints implies formally that $\pi^*$ is continuous and cocontinuous, and
	that $\pi^*$ and $\pi_!$ preserve epimorphisms.	We now turn to the proof itself.
	
\begin{enumerate}[label=(\roman*), align=left, leftmargin=0pt, itemindent=1ex, itemsep=0.3\baselineskip]
 \item 
	We verify Giruad's axioms for $[\bfX/C_2]$.
	Briefly, if $\{G_i\}_{i\in I}$ is a set of generators for $\bfX$,
	then $\{\pi_!G_i\}_{i\in I}$ is a set of generators for $[\bfX/C_2]$.
	Indeed, let $f,g:U\to V$ be distinct morphisms in $[\bfX/C_2]$.
	Since $\pi^*$ is faithful, $\pi^*f,\pi^*g:\pi^*U\to \pi^*V$ are distinct in $\bfX$,
	hence there exist $i\in I$ and $h:G_i\to \pi^*U\to \pi^*V$
	such that $ \pi^*f\circ h\neq \pi^* g\circ h $.
	By the adjunction between $\pi_!$ and $\pi^*$, the morphism
	$h$ corresponds to a morphism $h':\pi_!G_i\to U$ in $[\bfX/C_2]$
	such that $ f\circ h'\neq g\circ h'$, as required.
	
	That sums are disjoint
	and equivalence relations are effective in $[\bfX/C_2]$ can
	be checked with the help of the forgetful functor $\pi^*$ and the fact
	that these properties hold in $\bfX$. Finally, the existence of
	colimits and the fact that they commute with fiber products can be checked directly.

 \item This is immediate from the adjunctions between $\pi_!$, $\pi^*$ and $\pi_*$
 noted above.

 \item 
 We may replace $\{(U_i,\tau_i)\}_{i\in I}$ with their disjoint union,
 denoted
 $(U',\tau')$,
 to assume that $I$ consists of a single element. We need
 to show that $U'\to U$ is an epimorphism in $\bfX$
 if and only if $(U',\tau')\to (U,\tau)$ is an epimorphism in $[\bfX/C_2]$.
 The ``if''
 part follows from the fact that $\pi^*$ preserves epimorphisms, being a left adjoint.
 To see the converse, it is enough to show that the composition
 $\pi_!U'=\pi_!\pi^*(U',\tau')\xrightarrow{\text{counit}} (U',\tau')\to (U,\tau) $
 is an epimorphism.
 Let $(V,\sigma)\in [\bfX/C_2]$.
 Then $\Hom_{[\bfX/C_2]}(\pi_!U',(V,\sigma))\cong \Hom_{\bfX}(U',\pi^*(V,\sigma))=
 \Hom_{\bfX}(U',V)$, and under this isomorphism the map 
 \[\Hom_{[\bfX/C_2]}((U,\tau),(V,\sigma))\to \Hom_{[\bfX/C_2]}(\pi_!U',(V,\sigma))\]
 is the composition $\Hom_{[\bfX/C_2]}((U,\tau),(V,\sigma))\hookrightarrow
 \Hom_{\bfX}(U,V)\to \Hom_{\bfX}(U',V)$, which is injective since $U'\to U$
 an epimorphism. 
 Thus, $(U',\tau')\to (U,\tau)$ is an epimorphism.

\item 	 Let $\{r_i\}_{i\in I}$ be $(U,\tau)$-sections of
 $(\calO_\bfX,\lambda)$ generating the unit ideal. 
 Then $\{\pi^*r_i\}_{i\in I}$ generate the unit ideal in $\calO_\bfX(U)$,
 and hence there exists a covering $\{\alpha_i:U_i\to U\}_{i\in I}$ 
 such that $r_i\in\units{\calO_\bfX(U_i)}$ for all $i$. (We remark that $\calO_\bfX(\emptyset)$ is the $0$-ring, in
 which the unique element is invertible; it is possible that some of the $U_i$ called for in this covering are initial objects.)

	Fix $i\in I$.
	The adjunction between $\pi_!$ and $\pi^*$
	gives rise to a morphism 	
	$\beta_i:\pi_!U_i\to (U ,\tau)$,
	adjoint to $\alpha_i:U_i\to U=\pi^*(U,\tau)$,
	and an isomorphism of rings
	$\calO_\bfX(U_i)=\Hom_\bfX(U_i,\calO_\bfX)\cong\Hom_{[\bfX/C_2]}(\pi_!U_i,(\calO_\bfX,\lambda))= 
	\calO_{[\bfX/C_2]}(\pi_!U_i)$.
	Unfolding the definitions, one finds that
	under this isomorphism, $r_i|_{\pi_!U_i}=r_i\circ \beta_i\in 
	\calO_{[\bfX/C_2]}(\pi_!U_i)$ 
	corresponds to $\pi^*r_i|_{U_i}=\pi^*r_i\circ \alpha_i$, which is invertible.
	Thus, $r_i$ is invertible in $\calO_{[\bfX/C_2]}(\pi_!U_i)$.
	By (iii), the collection $\{\beta_i:\pi_!U_i\to (U,\tau)\}_{i\in I}$ is a covering,
	so we have shown that $[\bfX/C_2]$ is locally ringed by $\calO_{[\bfX/C_2]}$.

\item 	One checks that
 $\pi_*\lambda:\pi_*\calO_\bfX\to \pi_*\calO_\bfX$ is the morphism \[ (\calO_\bfX\times\Lambda
 \calO_\bfX,\tau_{\calO_\bfX})\xrightarrow{(x,y)\mapsto (y^\lambda, x^\lambda)}(\calO_\bfX\times\Lambda
 \calO_\bfX,\tau_{\calO_\bfX}),\] and so $\calO_{[\bfX/C_2]}$ is the equalizer of $\id,\pi_*\lambda:\pi_*\calO_\bfX\to
 \pi_*\calO_\bfX$. That $\pi_*$ preserves epimorphisms can be checked
 directly using the definitions and (iii).
	The verification of the coherence conditions in Definition~\ref{DF:morphism-top-w-inv}
	is routine.\uriyaf{To do: Verify them.} 
 \qedhere
 \end{enumerate}
\end{proof}

\subsection{Ramification}
\label{subsec:ramification}

Let 
$\bfX$ be a locally
ringed topos with an involution $\lambda=(\Lambda,\nu,\lambda)$, and let 
$\pi:\bfX\to \bfY$
be an exact quotient, see Definition~\ref{DF:exact-quotient}. For brevity, write
\[ R=\pi_*\calO_{\mathbf{X}}\qquad\text{and}\qquad S=\calO_\bfY\ .
\] As in Subsection~\ref{subsec:quotients}, 
we write $\pi_*\lambda:R\to \pi_*\Lambda\calO_{\mathbf{X}}=R$ as $\lambda$. The
automorphism $\lambda:R\to R$ is an involution the fixed ring of which is $S$.

\begin{definition}\label{DF:ramification}
 Let $V$ be an object of $\bfY$. We say that $\pi$ is \emph{unramified} (relative to $\lambda$) along $V$ if $R_V$ is a
 quadratic \'etale $S_V$-algebra in $\bfY/V$, see Subsection~\ref{subsec:etale-algebras}. Otherwise, $\pi$ is said to
 be \emph{ramified} along $V$.

 The morphism $\pi$ is said to be \emph{unramified} if it is unramified along $*_\bfY$, and \emph{ramified}
 otherwise. It is \emph{everywhere ramified} if $\pi$ is ramified along every non-initial object of $\bfY$.
\end{definition}

We alert the reader that, contrary to the common use of the term ``ramification'',
we consider trivial $C_2$-quotients as everywhere ramified.

\begin{example}\label{EX:triv-quotient-is-tot-ramified}
 Suppose $(\bfX,\calO_\bfX)$ is given a weakly trivial involution and $\pi$ is the trivial $C_2$-quotient, namely, the
 identity map $\id:(\bfX,\calO_\bfX)\to (\bfX,\calO_\bfX)$ (Example~\ref{EX:trivial-exact-quotient}). 
 Then $\pi$ is everywhere ramified. Indeed, in this
 case $R=S=\calO_\bfX$ and and $\lambda=\id_R$.
	Since $\calO_\bfX$ is a local ring object, for any $V\ncong \emptyset$ in $\bfX$,
	the ring $\calO_\bfX(V)$ is nonzero, and so $R_V$ cannot be locally free of rank $2$ over $S_V$.
\end{example}

For any object $V$ of $\bfY$, define $U(V)$ to be a singleton if $\pi$ is unramified along $V$, and an empty set
otherwise. It follows from Lemma~\ref{LM:etale-is-local-property} that $V\mapsto U(V)$ defines a sheaf (the action of
$U$ on morphisms in $\bfY$ is uniquely determined), which is then represented by an object of $\bfY$, denoted $U=U_\pi$. We
call $U$ the \emph{unbranched locus} of $\pi$. It is a subobject of $*_\bfY$.
Clearly, $\pi$ is unramified if and only if $U=*_\bfY$,
and $\pi$ is everywhere ramified if and only if $U=\emptyset_\bfY$.

The following propositions give a more concrete description of the unbranched locus when
$\pi:\bfX\to \bfY$ is induced by a $C_2$-quotient of schemes or topological spaces, see
Examples~\ref{EX:important-exact-quo-scheme} and~\ref{EX:important-exact-quo-top}.

\begin{prp} \label{PR:ramification-in-schemes} In the situation of Example~\ref{EX:important-exact-quo-scheme}, i.e.,
 when 
 $\pi:\bfX\to \bfY$ 
 is obtained from a good $C_2$-quotient of schemes $\pi:X\to Y$ by
 taking \'etale ringed topoi, the unbranched locus of $\pi$, defined above, is represented by an open subscheme
 $U\subseteq Y$. The latter can be defined in any of the following equivalent ways:
\begin{enumerate}[label=(\alph*)]
	\item \label{item:PR:ram:max-etale} $U$ is the largest open subset of $Y$ such that $\pi_U:\pi^{-1}(U)\to U$
	is quadratic \'etale.
 \item \label{item:PR:ram:Zariski} The (Zariski) points of $U$ are those points $y\in Y$ such that
 $X\times_Y\Spec\calO_{Y,y}\to \Spec \calO_{Y,y}$ is quadratic \'etale.
 \item \label{item:PR:ram:stalk} The (Zariski) points of $U$ are those points $y\in Y$ such that
 $X\times_Y\Spec\calO_{Y,y}^{\mathrm{sh}}\to \Spec \calO_{Y,y}^{\mathrm{sh}}$ is quadratic \'etale;
 here, $\calO_{Y,y}^{\mathrm{sh}}$ is the strict henselization of $\calO_{Y,y}$.
 \item \label{item:PR:ram:lambda} The (Zariski) points of $Y- U$ are those points $y\in Y$ such that 
 the set $\pi^{-1}(y)$ is a singleton $\{x\}$, and
 $\lambda$ induces the identity on $k(x)$.
\end{enumerate}
Consequently, 
$\pi:\bfX\to \bfY$ is unramified if and only if
$\pi:X\to Y$ is quadratic \'etale.
\end{prp}

\begin{proof}
 By virtue of Lemma~\ref{LM:etale-is-local-property}, if $V\to Y$ is an \'etale morphism having image
 $U$ in $Y$, then $\pi_V:X\times_YV\to V$ is quadratic \'etale if and only if $\pi_U:\pi^{-1}(U)\to U$ is quadratic
 \'etale. It follows that there exists a maximal open subset $U$ of $Y$ with the property that $\pi_U$ is quadratic
 \'etale, and any $V\to X$ as above factors through the inclusion $U\subseteq X$. We also let $U$ denote the sheaf it
 represents in $\bfY=\Sh(Y_\et)$.
	
 Since $U$ is a subobject of $*_\bfY$, the set $U(V)$ is a singleton or empty for all $V\in\bfY$. To show that $U$ is
 the unbranched locus of 
 $\pi:\bfX\to \bfY$, it is enough to show that $\pi$ is
 unramified along an object $V$ of $\bfY$ if and only if there exists a morphism $V\to U$. For any such $V$, we can find a covering
 $\{V_i\to V\}_i$ with each $V_i$ represented by some $ (\tilde{V}_i\to Y)$ in $ Y_\et$. By
 Lemma~\ref{LM:etale-is-local-property}, $\pi$ unramified along $V$ if and only if $\pi$ is unramified along each
 $V_i$. Furthermore, if for each $i\in I$ there is a morphism $V_i\to U$ (in $\bfY$), then these morphisms must patch to a
 morphism $V\to U$, because $U(V_i\times_V V_j)$ is a singleton. It is therefore 
 enough to show that $\pi$ is unramified
 along $V_i$ if and only if there exists a morphism $\tilde{V}_i\to U$ in $Y_\et$. The latter holds precisely when
 $\im( \tilde{V}_i\to Y)\subseteq U$, and so the claim follows from the definition of $U$.
	
 We finish by showing that the different characterizations of $U$ are equivalent. The equivalence of 
 \ref{item:PR:ram:max-etale} and \ref{item:PR:ram:Zariski}
 follows from Corollary~\ref{CR:etale-at-a-point}, and the equivalence of \ref{item:PR:ram:Zariski} and 
 \ref{item:PR:ram:lambda} follows from
 Theorem~\ref{thm:local-involution}.
 That the condition in \ref{item:PR:ram:Zariski} implies the condition in \ref{item:PR:ram:stalk}
 is clear. It remains to prove the converse.
 Since $\calO_{Y,y}^{\mathrm{sh}}$ is faithfully flat over $\calO_{Y,y}$,
 this is a consequence of Lemma~\ref{LM:etale-is-local-property} applied to the fpqc site
 of $\Spec \calO_{Y,y}$.
\end{proof}

\begin{prp} \label{PR:ramification-top} In the situation of Example~\ref{EX:important-exact-quo-top}, i.e., when
 $\pi:\bfX\to \bfY$ 
 is induced by a $C_2$-quotient of Hausdorff topological spaces
 $\pi:X\to Y=X/C_2$, the unbranched locus is represented by an open subset ${U}\subseteq Y$. Specifically,
 ${U}=\{x\in X\suchthat x^\lambda\neq x\}/C_2$. Consequently, 
 $\pi:\bfX\to \bfY$ is
 unramified if and only if $C_2$ acts freely on $X$.
\end{prp}

\begin{proof}
	This is similar to the previous proof and is left to the reader.
\end{proof}

We refer to the situations of Examples~\ref{EX:important-exact-quo-scheme}
and~\ref{EX:important-exact-quo-top} as
the scheme-theoretic case and topological case, respectively.
In both cases, we define
the \emph{branch locus} of $\pi$ to be the complement $W:=Y- U$,
where $U$ is as in Proposition~\ref{PR:ramification-in-schemes} or Proposition~\ref{PR:ramification-top}.
The \emph{ramification locus} of $\pi$ is $Z:=\pi^{-1}(W)$. In the scheme-theoretic case, we also 
endow $W$ and $Z$ with the
reduced induced closed subscheme structure.

Notice that in the topological case, $\pi_U:\pi^{-1}(U)\to U$ is a double covering, while $\pi_W:\pi^{-1}(W)\to W$ is a
homeomorphism. With slight modification, a similar statement holds for schemes.

\begin{prp}\label{PR:Z-to-W-homeomorphism}
 In the notation of Proposition~\ref{PR:ramification-in-schemes}, let $W=U- Y$, and regard $W$ and $\pi^{-1}(W)$ as
 reduced closed subschemes of $Y$ and $X$, respectively. Then:
	\begin{enumerate}[label=(\roman*)]
		\item $\pi_U:\pi^{-1}(U)\to U$ is quadratic \'etale.
		\item $\lambda:X\to X$ restricts to the identity morphism on $\pi^{-1}(W)$.
		\item $\pi|_{\pi^{-1}(W)}:\pi^{-1}(W)\to W$ is a homeomorphism,
		and when $2$ is invertible on $Y$, it is an isomorphism of schemes.
	\end{enumerate}
\end{prp}

\begin{proof}
\begin{enumerate}[label=(\roman*), align=left, leftmargin=0pt, itemindent=1ex, itemsep=0.3\baselineskip]
	\item This immediate from condition \ref{item:PR:ram:max-etale} in Proposition~\ref{PR:ramification-in-schemes}.
	
	\item 
	Condition \ref{item:PR:ram:lambda} of Proposition~\ref{PR:ramification-in-schemes}	
	implies that $\lambda$ fixes the (Zariski) points of $\pi^{-1}(W)$.
	Let $f$ be a (Zariski) section of $\calO_{\pi^{-1}(W)}$ and let $S\subseteq \pi^{-1}(W)$
	be the largest open subset on which $f-f^\lambda$ is invertible.
	Let $s\in S$. Then $f-f^\lambda$ is invertible in $k(s)$, which is impossible
	by condition \ref{item:PR:ram:lambda} of Proposition~\ref{PR:ramification-in-schemes}.
	Thus, $S=\emptyset$. Since $\pi^{-1}(W)$ is reduced, we conclude that $f-f^\lambda=0$.
	
	\item It is enough to prove the claim after restricting to an open affine covering
	of $Y$, so assume $X=\Spec A$, $Y=\Spec B$, $W=\Spec B/I$ with $I$ a
	radical ideal of $B$, and $\pi^{-1}(W)=\Spec A/\sqrt{IA}$,
	where $\sqrt{IA}$ denotes the radical of $IA$.
	The morphism $\pi|_{\pi^{-1}(W)}:\pi^{-1}(W)\to W$ is adjoint to the evident homomorphism 
	$B/I\to A/\sqrt{IA}$, and
	$\lambda:X\to X$ induces an involution $\lambda:A\to A$ having fixed ring $B$.
	
	We know that $\pi:{\pi^{-1}(W)}\to W$ is continuous and it is a set bijection
	since its set-theoretic fibers consist of singletons by 
	condition \ref{item:PR:ram:lambda} of Proposition~\ref{PR:ramification-in-schemes}.
	Thus, proving that $\pi|_{\pi^{-1}(W)}$ is a homeomorphism
	amounts to checking that it is closed. 
	Since any $a\in A$ satisfies
	$a^2-(a^\lambda+a)a+(a^\lambda a)=0$, the morphism
	 $\Spec A/\sqrt{IA}\to \Spec B/I$ is integral, and therefore
	closed by
	\cite[Tag 01WM]{de_jong_stacks_2017}.

	Suppose now that $2\in\units{B}$. 
	We need to show that $B/I\to A/\sqrt{IA}$ is bijective.
	Let $a\in A$. 
	By virtue of (ii), $\lambda$ induces the identity
	involution on $A/\sqrt{IA}$, and thus $a\equiv \frac{1}{2}(a+a^\lambda)\mod \sqrt{IA}$.
	Since $a+a^\lambda\in B$, we have established the surjectivity
	of $B/I\to A/\sqrt{IA}$. 
	Next, write $J=\ker(B/I\to A/\sqrt{AI})$.
	Since $\Spec A/\sqrt{AI}\to \Spec B/I$ is bijective, 
	$J$ is contained in every prime ideal of $B/I$,
	and
	since $B/I$ is reduced, $J=0$.
	\qedhere
\end{enumerate}
\end{proof}
 
\begin{remark}
\begin{enumerate}[label=(\roman*), align=left, leftmargin=0pt, itemindent=1ex, itemsep=0.3\baselineskip]
\item In Proposition~\ref{PR:Z-to-W-homeomorphism}(iii), 
it is in general necessary to assume that $2$ is invertible in order
to conclude that $\pi:\pi^{-1}(W)\to W$ is an isomorphism. 
Consider, for example, a DVR $S$ with $2\neq 0$ having a non-perfect residue field
$K$ of characteristic $2$, let $\alpha\in \units{S}$ be an element
such that its image in $K$ is not a square, let $R=S[x\,|\,x^2=\alpha]$,
and let $\lambda:R\to R$ be the $S$-involution sending $x$ to $-x$.
Taking $X=\Spec R$ and $Y=\Spec S$, the set
$W$ consists of the closed
point $y$ of $Y$, but the induced map $k(y)\to k(\pi^{-1}(y))$
not an isomorphism.

\item Let $\pi:X\to Y$ be a good $C_2$-quotient
of schemes which is everywhere ramified, and suppose $2$ is invertible
on $Y$. Then Proposition~\ref{PR:Z-to-W-homeomorphism}
implies that the induced morphism $\pi:X_{\mathrm{red}}\to Y_{\mathrm{red}}$
is an isomorphism. However, in general, and in contrast
to Proposition~\ref{PR:ramification-top}, it can happen that $\pi:X\to Y$
is not an isomorphism. For example, take $X=\Spec \C[\veps]/(\veps^2)$
and let $\lambda$ be the $\C$-involution taking $\veps$ to $-\veps$.
\end{enumerate}
\end{remark}

\begin{remark}
	For a general exact $C_2$-quotient $\pi:\bfX\to \bfY$ with unbranched
	locus $U$,
	it is possible to define the ``branch topos'' $\bfW$ of $\pi$ as
 the full subcategory of $\bfY$ consisting of objects $W$ such that the projection $U\times W\to U$ is an isomorphism.
 In the situation of Examples~\ref{EX:important-exact-quo-scheme}
 and~\ref{EX:important-exact-quo-top}, this turns out to give the topos of sheaves over the
 set-theoretic or scheme-theoretic branch locus of $\pi:X\to Y$
 defined above. We omit the details as they will not be needed
 in this work.
\end{remark}

We finish with showing that the exact quotient of Construction~\ref{CN:stacky-quotient}
is unramified. Thus, any locally ringed topos with involution admits an unramified exact quotient.
When $\bfX$ is the \'etale ringed topos of a scheme $X$, this
generalizes the well-known fact that the morphism from $X$ to its quotient stack $[X/C_2]$ is quadratic \'etale.

\begin{prp} \label{PR:stacky-quotient-ramification}
The exact quotient 
$\pi:\bfX\to [\bfX/C_2]$ of
Construction~\ref{CN:stacky-quotient} is unramified.
\end{prp}

\begin{proof} Recall from Construction~\ref{CN:stacky-quotient} 
that $S=(\calO_\bfX,\lambda)$ and $R=(\calO_\bfX\times
\Lambda\calO_\bfX,\tau_{\calO_\bfX})$, where $\tau_{\calO_\bfX}$ is the interchange involution, and the morphism $\pi_\#:S\to
R$ is given by $x\mapsto (x,x^\lambda)$ on sections (in $\bfX$). 
We shall make use of $\pi_!:\bfX\to [\bfX/C_2]$, the left adjoint of $\pi^*$
constructed in the proof of Proposition~\ref{PR:stacky-quotient}.

Write $D:=\pi_!(*_\bfX)=(*\sqcup \Lambda
*,\sigma_{*})$ and observe that the unique map $D\to (*_\bfX,\id)=*_{[\bfX/C_2]}$ is a covering
by Proposition~\ref{PR:stacky-quotient}(iii). By Lemma~\ref{LM:etale-is-local-property}, 
it is enough to show that $R_{D}$ is a quadratic \'etale $S_{D}$-algebra. In fact, we will show that
$R_D\iso S_D\times S_D$.

We first observe that the slice category $[\bfX/C_2]/D$ is equivalent to $\bfX$; the equivalence is
given by mapping $(U,\tau)\to D$ in $ [\bfX/C_2]/D$ to $*\times_{(*\sqcup \Lambda *)} U$, and by $\pi_!$ in the other
direction. Now, consider $R_D$ and $S_D$ as sheaves on $\calO_{[\bfX/C_2]}/D$. Then $R'=R_D\circ \pi_!$ and
$S'=S_D\circ \pi_!$ are sheaves of rings on the equivalent topos $\bfX$. Since $\pi_!$ is left adjoint to $\pi^*$, for
all objects $V$ of $\bfX$, we have natural isomorphisms of rings
\begin{align*} R'(V) &=R(\pi_!V)= \Hom_{[\bfX/C_2]}(\pi_!V, (\calO_\bfX\times \Lambda\calO_\bfX,\tau_{\calO_\bfX}))
 \iso\calO_\bfX(V)\times \Lambda\calO_\bfX(V)\\ S'(V) &=S(\pi_!V)= \Hom_{[\bfX/C_2]}(\pi_!V, (\calO_\bfX,\lambda))
 \iso\calO_\bfX(V)
\end{align*} and under these isomorphisms, the embedding $S'\to R'$ induced by $\pi_\#:S\to R$ is given by $x\mapsto (x,
x^\lambda)$ on sections. From this it follows that $R'\iso S'\times S'$ as $S'$-algebras, and hence $R_D\iso
S_D\times S_D$ as $S_D$-algebras.
\end{proof}

\begin{remark}\label{RM:C-two-quotients-not-unique}
 Propositions~\ref{PR:ramification-in-schemes} and~\ref{PR:ramification-top} provide plenty of examples where a locally
 ringed topos with involution admits a ramified exact quotient. However,
 Proposition~\ref{PR:stacky-quotient-ramification} shows that these examples also admit unramified exact quotients.
 It follows that exact quotients are not unique in general.
\end{remark}

\section{Classifying Involutions into Types}
\label{sec:types}

The purpose of this section is to classify 
involutions of Azumaya algebras into types in such a way which generalizes
the familiar classification of involutions of central simple algebras over fields as orthogonal, symplectic or unitary.

Throughout, $\bfX$ denotes a locally ringed topos with an involution $\lambda=(\Lambda,\nu,\lambda)$ and
$\pi:\bfX\to \bfY$ is an exact quotient relative to $\lambda$, see
Definition~\ref{DF:exact-quotient}. For brevity, we write
\[ R=\pi_*\calO_{\mathbf{X}} \qquad\text{and}\qquad S=\calO_\bfY ,\]
and, abusing the notation, denote the involution $\pi_*\lambda:R\to R$ by $\lambda$. Theorem~\ref{TH:pi-star-equiv} and
Corollary~\ref{CR:equiv-Az-with-inv} imply that Azumaya $\calO_{\mathbf{X}}$-algebras with a $\lambda$-involution are
equivalent to Azumaya $R$-algebras with a $\lambda$-involution, and the latter are easier to work with.

In fact, most of the results of this section can be phrased with no direct reference to 
$\bfX$ or the quotient map $\pi$, assuming only a locally ringed topos $(\bfY,S)$, an $S$-algebra $R$,
and an involution $\lambda:R\to R$ having fixed ring $S$.

We remind the reader that Azumaya $\calO_\bfX$-algebras are always assumed
to have a degree which is fixed by $\Lambda$, see Remark~\ref{RM:lambda-stable-degree}.
This is automatic when $\bfX$ is connected.

\subsection{Types of Involutions}
\label{subsec:types}

Suppose $K$ is a field and $\lambda:K\to K$ is an involution, the fixed subfield of which is $F$. Classically, when
$\lambda=\id_K$, the $\lambda$-involutions of central simple $K$-algebras are divided into two types, \emph{orthogonal} and
\emph{symplectic}, whereas in the case $\lambda\neq \id_K$, they are simply called \emph{unitary}; see
\cite[\S2]{knus_book_1998-1}. This classification satisfies the following two properties:
\begin{enumerate}[label=(\roman*)]
	\item \label{item:type-cond-local-iso} 
	If $(A,\tau)$ and $(A',\tau')$ are 
	central simple 
	$K$-algebras with $\lambda$-involutions such that $\deg A=\deg A'$,
	then $\tau$ and $\tau'$ are of the same type
	if and only if $(A,\tau)$ and $(A',\tau')$
	become isomorphic as algebras with involution over an algebraic closure of $F$.
	\item \label{item:type-cond-matrix} If $(A,\tau)$ is a central simple $K$-algebra with $\lambda$-involution,
	then $\tau$ has the same type as $\tau\tr:\nMat{A}{n\times n}\to \nMat{A}{n\times n}$
	given by $(a_{ij})_{i,j}\mapsto (a_{ji}^{\tau})_{i,j}$.
\end{enumerate}
Of these two properties, it is mainly the first that motivates the classification into types. The second property
should not be disregarded as it guarantees, at least when $2\in\units{K}$, that involutions adjoint to symmetric
bilinear forms (resp.\ alternating bilinear forms, hermitian forms) of arbitrary rank all have the same type, see
\cite[\S4]{knus_book_1998-1}.

Our aim in this section is to partition the $\lambda$-involutions of Azumaya $R$-algebras into equivalence classes,
called types, so that properties analogous to \ref{item:type-cond-local-iso} and \ref{item:type-cond-matrix} hold. To
this end, we simply take the minimal equivalence relation forced by 
the ``if'' part of condition
\ref{item:type-cond-local-iso} and condition~\ref{item:type-cond-matrix}.

\begin{definition}\label{DF:types}
 Let $(A,\tau)$, $(A',\tau')$ be two Azumaya $R$-algebras with a $\lambda$-involution. Let $\tau\tr$ denote the
 involution $\tau\tensor \lambda\tr$ of $\nMat{A}{n\times n} \iso A \tensor \nMat{R}{n\times n}$. On sections, it is
 given by $(a_{ij})_{i,j}\mapsto (a_{ji}^{\tau})_{i,j}$.

 We say that $\tau$ and $\tau'$ are of the same \emph{$\lambda$-type} or \emph{type} if there exist $n,n'\in \N$ and a
 covering $U\to \ast$ in $\mathbf{Y}$ such that
 \[ (\nMat{A_U}{n\times n},\tau_U\tr)\iso (\nMat{A'_U}{n'\times n'},\tau'_U\tr) \] as $R_U$-algebras with involution.

 Being of the same $\lambda$-type is an equivalence relation. The equivalence classes will be called
 \emph{$\lambda$-types} or just \emph{types}, and the type of $ \tau$ will be denoted
 \[ \tp(\tau)\qquad\text{or}\qquad \tp(A,\tau)\ . \]
 The set of all $\lambda$-types will be denoted
 $\Types{\lambda}$. Tensor product of Azumaya algebras with involution endows $\Types{\lambda}$ with a monoid
 structure. We write the neutral element, represented by $(R,\lambda)$, as $1$.
\end{definition}

We shall shortly see that our definition gives the familiar types in the case of fields, as well as in a number of other
cases. It is no longer clear whether two involutions of the same degree and type are locally isomorphic, however, and
the majority of this section will be dedicated to showing that this is indeed the case under mild assumptions. Another
drawback of the definition is that it not clear how to enumerate the types it yields, and nor does it provide a way to
test whether two involutions are of the same type. These problems will also be addressed, especially in
the situation of Theorem~\ref{TH:important-exact-quotients}, namely, in cases induced by a good $C_2$-quotient of
schemes on which $2$ is invertible, or by a $C_2$-quotient of Hausdorff topological spaces.

\begin{remark}\label{RM:type-depends-on-quotient} 
\label{RM:types-over-OX}
Let $(A,\tau)$, $(A',\tau')$ be two Azumaya $\calO_\bfX$-algebras in $\mathbf X$ with a $\lambda$-involution. Using
Corollary~\ref{CR:equiv-Az-with-inv}, we say that $\tau$ and $\tau'$ have the same \textit{$\lambda$-type (relative to
 $\pi$)} when the same holds after applying $\pi_*$. The equivalence class of $(A,\tau)$ is denoted $\tp_\pi(\tau)$ or
$\tp_\pi(A,\tau)$ and the monoid of types is denoted $\Types[\pi]{\lambda}$.

We warn the reader that the $\lambda$-type of a $\lambda$-involution 
of an Azumaya $\calO_\bfX$-algebra depends on the choice of the quotient $\pi: \mathbf X \to \mathbf Y$, which
is why we include $\pi$ in the notation.

For instance, we shall see below in Theorem~\ref{TH:basic-properties-of-types} that in the case where $\bfX$ is
 given the trivial involution and $\bfX$ is connected, then under mild assumptions, 
 taking $\pi$ to be the trivial exact quotient $\id:\bfX\to\bfX$ results in two $\lambda$-types, whereas taking $\pi$ to be the exact
 quotient $\bfX\to [\bfX/C_2] $ of Construction~\ref{CN:stacky-quotient} results in only one
 $\lambda$-type.
\end{remark}

\begin{example}\label{EX:types-for-schemes}
 \begin{enumerate}[label=(\roman*), leftmargin=0pt, align=left, itemindent=1ex, itemsep=0.6\baselineskip]
 \item \label{EX:types-for-schemesi} Let $X$ be a connected scheme on which $2$ is invertible, let $\lambda:X\to X$ be
 the trivial involution and let $Y :=X/C_2=X$. Consider the exact quotient obtained from $\pi:X\to Y$ by taking
 \'etale ringed topoi, see Example~\ref{EX:important-exact-quo-scheme}. In this case,
 $\mathbf X = \mathbf Y=\Sh(X_\et)$, $R = S=\sh O_X $ and $\lambda=\id_R$. Thus, an Azumaya $R$-algebra with a
 $\lambda$-involution is simply an Azumaya $\calO_X$-algebra with an involution of the first kind on $X$. It is well
 known that there are two $\lambda$-types: \emph{orthogonal} and \emph{symplectic}. The orthogonal type is represented
 by $(R,\id_R)$ and the symplectic type is represented by $(\nMat{R}{2\times 2},\syp)$, where $\syp$ is given by
 $x \mapsto h_2 x^\tr h_2^{-1}$ on sections and $h_2= [\begin{smallmatrix} 0 & 1 \\ -1 & 0
\end{smallmatrix}]$.
 Moreover, every Azumaya $R$-algebra of degree $n$ with an involution of the first kind is locally isomorphic either to
 $(\nMat{R}{n},\tr)$ or to $(\nMat{R}{n},\syp)$, where in the latter case $n=2m$ and $\syp$ is given sectionwise
 by
	$x\mapsto h_n x^{ \tr} h_n^{-1}$ with
	$h_n= [\begin{smallmatrix}
 0 & I_{m}\\ -I_m & 0 
 \end{smallmatrix}]$;
 see \cite[III.\S8.5]{knus_quadratic_1991} or \cite[\S1.1]{parimala_92}\uriyaf{added this.}. In this
 case, $\Types{\lambda}$ is isomorphic to the group $\{\pm 1\}$.

 \item \label{EX:types-for-schemesii} Let $\pi:X\to Y$ be a quadratic \'etale morphism,
 and let $\lambda:X\to Y$ denote the canonical $Y$-involution of $X$.
 Again, let $\pi:\bfX\to \bfY$ denote the exact quotient
 obtained from $\pi:X\to Y$ by taking \'etale ringed topoi. 
In this case,
 $R=\pi_*\calO_X$ is quadratic \'etale over $S=\calO_Y$, and 
 $\lambda$-involutions are known as \emph{unitary} involutions. There is only
 one type in this situation, and moreover, any Azumaya $R$-algebra of degree $n$ with a $\lambda$-involution is
 locally isomorphic to $(\nMat{R}{n},\lambda\tr)$; 
 these well-known facts can be found in
 \cite[\S1.2]{parimala_92} without proof, but they follow from the
 results in the sequel. We were unable to find a source providing complete proofs. 
\end{enumerate}
\end{example}

\begin{example}\label{EX:types-in-char-two} 
 Let $K$ be a perfect field of characteristic $2$, and consider the case of the trivial involution on $K$. 
 As in Example~\ref{EX:types-for-schemes}\ref{EX:types-for-schemesi}, this corresponds
 to taking
 $\mathbf X = \mathbf Y = \Sh((\Spec K)_\et)$, $R=S=\calO_{\Spec K}$ and
 $\lambda=\id$. Azumaya $R$-algebras with a $\lambda$-involution are therefore 
 central simple $K$-algebras with an
 involution of the first kind. There are again two types in this case, again called orthogonal and symplectic, but
 $\Types{\lambda}$ is isomorphic to the multiplicative monoid $\{0,1\}$ with $0$ corresponding to the symplectic type;
 see \cite[\S2]{knus_book_1998-1}. This shows that the theory in characteristic $2$ is substantially different from
 that in other characteristics.
 
 The assumption that $K$ is
 perfect can be dropped if one replaces the \'etale site with the fppf site
 (consult Remark~\ref{RM:fppf-site}).
\end{example}

\subsection{Coarse types}
\label{subsec:coarse-types}

In Subsection~\ref{subsec:types}, we defined the type of a $\lambda$-involution of an Azumaya $R$-algebra in terms of
the entire class of algebras and not as an intrinsic invariant of the involution. We now introduce another invariant of
$\lambda$-involutions, called the \emph{coarse type}, which, while in general coarser than the type, will enjoy an
intrinsic definition. It will turn out that under mild assumptions the invariants are equivalent, and this will be
used to address the questions raised in Subsection~\ref{subsec:types}. Apart from that, coarse types will also be needed in proving the main results of Section~\ref{sec:Saltman}.

\medskip

We begin by defining the abelian group object $N$ of $\bfY$ via the exact sequence
\begin{equation}\label{EQ:N-dfn} 1\to N\to \units{R}\xrightarrow{x\mapsto x^\lambda x} \units{S}
\end{equation} and the abelian group object $T$ of $\bfY$ via the short exact sequence
\begin{equation}\label{EQ:T-dfn} 1\to \units{R}/\units{S}\xrightarrow{x\mapsto x^\lambda x^{-1}} N\to T\to 1 \ .
\end{equation} The group $N$ should be regarded as the group of elements of $\lambda$-norm $1$. We call the global
sections of $T$ \emph{coarse $\lambda$-types} and write
\[ \CTypes{\lambda}=\Hoh^0(T)\ .
\]

The following example and propositions give some hints about the structure of $T$.

\begin{example}\label{EX:T-tot-ramified} If the involution of $\bfX$ is 
trivial and the quotient map is the
 identity, then $N=\mu_{2,R}$ and the map $x\mapsto
x^{-1}x^\lambda:\units{R}/\units{S}\to N$ is trivial, hence $T=\mu_{2,R}$ and $\CTypes{\lambda}=\Hoh^0(\mu_{2,R})=\{x\in
\Hoh^0(R)\suchthat x^2=1\}$.
\end{example}

\begin{proposition}\label{PR:T-structure-unramified-pi} If $\pi$ is unramified along an object $U$ of $\bfY$, 
 see Subsection \ref{subsec:ramification},
then $T_U=1$ in $\bfY/U$. In particular, when $\pi$ is unramified, $T=1$ and
$\CTypes{\lambda}=\{1\}$.
\end{proposition}

When $\bfY$ has enough points, it is possible to argue at stalks, and therefore the proposition follows 
from our version of Hilbert's Theorem 90, Proposition~\ref{PR:Hilbert-ninety}\ref{item:LM:Hilbert-ninety:norm}. The
following argument applies even without the assumption of enough points.
 
 \begin{proof} We must show that for all objects $U$ of $\bfY$ and all $r\in R(U)$ satisfying $r^\lambda r=1$, there
 is a covering $V\to U$ and $a\in \units{R}(V)$ such that $a^{-1}a^\lambda=r$ in $R(V)$.

 Refining $U$ if necessary, we may assume that $R(U)$ is a quadratic \'etale $S(U)$-algebra, see Subsection
 \ref{subsec:etale-algebras}. By Proposition~\ref{PR:Hilbert-ninety}\ref{item:LM:Hilbert-ninety:norm}, for all
 $\frakp\in\Spec S(U)$, there is $ a_\frakp\in \units{R(U)_\frakp}$ such that $a_\frakp^{-1} a_\frakp^\lambda=r$. For
 each $\frakp$, choose some $f_\frakp\in S(U)- \frakp$ such that $a_\frakp$ is the image of an element in
 $\units{R(U)_{f_\frakp}}$, also denoted $a_\frakp$, which satisfies $a_\frakp^{-1} a_\frakp^\lambda=r$ in
 $S(U)_{f_\frakp}$. The set $\{f_\frakp\}_\frakp$ is not contained in any proper ideal of $S(U)$ and therefore generates
 the unit ideal. Since $S$ is a local ring object, there exists a covering $\{U_\frakp\to U\}_{\frakp\in\Spec S(U)}$
 such that the image of $f_{\frakp}$ is invertible in $S(U_\frakp)$ for all $\frakp$. Now take
 $V=\bigsqcup_{\frakp} U_\frakp$ and $a$ to be the image of $(a_{\frakp})_{\frakp}$ in
 $R(V)=\prod_{\frakp} R(U_\frakp)$.
\end{proof}

\begin{proposition} \label{PR:T-is-two-torsion}
	$T$ is a $2$-torsion abelian sheaf.
\end{proposition}

\begin{proof} Let $U$ be an object of $\bfY$ and $t\in T(U)$. By passing to a covering of $U$, we may assume that $t$ is
the image of some $r\in N(U)$. Since $r^\lambda r=1$, we have $r^2=(r^{-1})^\lambda (r^{-1})^{-1}$, and thus $t^2=1$ in
$T(U)$.
\end{proof}

Let $n=\deg A$. Using Lemma~\ref{LM:inner-automorphisms}, we shall freely identify the group $\PGL_n(R)$ with
 $\sAut_{\textrm{$R$-alg}}(\nMat{R}{n\times n})$. 
 We denote by 
 \[-\lambda\tr\] 
 the automorphism of $\GL_n(R)$ given by
 $x\mapsto (x^{-1})^{\lambda\tr}$ on sections. This automorphism induces an automorphism on $\PGL_n(R)$, which is also
 denoted $-\lambda\tr$. We need the following lemma.

 \begin{lemma}\label{LM:lambda-tr-relation} Let $g$ be a section of $\PGL_n(R)=\Aut_{\text{\rm
 $R$-alg}}(\nMat{R}{n\times n}) =\Aut_{\text{\rm $R$-alg}}(\Lambda\nMat{R}{n\times n}^\op)$ and view $\lambda\tr$ as an
 $R$-algebra isomorphism $\nMat{R}{n\times n}\to \Lambda\nMat{R}{n\times n}^\op$. Then $g\circ \lambda\tr=\lambda\tr
 \circ g^{-\lambda\tr}$.
 \end{lemma}

 \begin{proof} Suppose $g\in \PGL_n(R)(U)$ for some object $U$ of $\bfX$. It is enough to prove the equality after
 passing to a covering $V\to U$. We may therefore assume that $g$ is inner, and the lemma follows by computation.
 \end{proof}

\begin{construction}\label{CN:coarse-type}
Let $(A,\tau)$ be a degree-$n$ Azumaya $R$-algebra with a $\lambda$-involution. We now construct an element
\[\ct(\tau)\in \CTypes{\lambda}=\Hoh^0(T)\] and call it the \emph{coarse $\lambda$-type} of $\tau$. This construction,
which is concluded in Definition~\ref{DF:coarse-types},
will play a major role in the sequel.

Choose a covering $U\to *_\bfY$ such that there exists an isomorphism of $R_U$-algebras 
$\psi:A_U\xrightarrow{\sim}
\nMat{R_U}{n\times n}$. The isomorphism $\psi$ gives rise to a $\lambda_U$-involution
\[ \sigma=\psi \circ \tau_U\circ \psi^{-1}:\nMat{R_U}{n\times n}\to \nMat{R_U}{n\times n}\ .
\] From $\sigma$ and the involution $\lambda \tr$, we construct
\[g:=\lambda\tr\circ \sigma\in \Aut_{\text{$R_{U}$-alg}}(\nMat{R_{U}}{n\times n})=\PGL_n(R)(U)\ .\] By
Lemma~\ref{LM:lambda-tr-relation}, we have $\id = \sigma \circ \sigma=\lambda\tr\circ g\circ \lambda\tr \circ g=
\lambda\tr\circ \lambda\tr \circ g^{-\lambda\tr}\circ g $, hence $g^{-\lambda\tr} g=1$. Replacing $U$
by a covering $U'\to U$ if necessary, we may assume that $g\in \PGL_n(R)(U)$ lifts to a section
\[ h\in \GL_n(R)(U)\ .
\] From $g^{-\lambda\tr} g=1$, we get
\begin{equation}\label{EQ:eps-def} \veps:=h^{-\lambda\tr} h\in \units{R}(U)\ .
\end{equation} Note that $\veps^{\lambda \tr}\veps=\veps^{\lambda \tr}h^{-\lambda\tr} h=h^{-\lambda\tr} \veps^{\lambda
\tr} h= h^{-\lambda\tr}(h^{\lambda \tr}h^{-1})h=1$, hence $\veps\in N(U)$. Let $\quo{\veps}$ be the image of $\veps$ in
$T(U)$.
\end{construction}

\begin{lem} \label{LM:coarse-type-descends}
The section $\quo{\veps}\in T(U)$ determines a global section of $T$. It is independent of the choices made
in Construction~\ref{CN:coarse-type}.
\end{lem}

\begin{proof}
 Let $U_\bullet$ denote the \v{C}ech hypercovering associated to $U$---for the definition see
 Example~\ref{EX:Cech-hypercovering}. In particular, $U_0=U$, $U_1=U\times U$ and $d_0,d_1:U_1\to U_0$ are given by
 $d_i(u_0,u_1)=u_{1-i}$ on sections. Proving that $\quo{\veps}$ determines a global section of $T$ amounts to showing
 that there exists a covering $V\to U_1=U\times U$ and $\beta\in \units{R}(V)$ such that $d_1^*\veps^{-1}\cdot
 d_0^*\veps = \beta^{-1}\beta^\lambda$ holds in $\units{R}(V)$.

For $i\in\{0,1\}$, let $\psi_i$ denote the pullback of $\psi:A_U\to \nMat{R_U}{n\times n}$ along $d_i:U_1\to
U_0=U$. Define $\sigma_i:\nMat{R_{U_1}}{n\times n}\to \nMat{R_{U_1}}{n\times n}$ similarly. Let
\[ a:=\psi_{1}\circ \psi_{0}^{-1}:\nMat{R_{U_1}}{n\times n}\to \nMat{R_{U_1}}{n\times n} .
\] and regard $a$ as an element of $\PGL_n(R)(U_1)$. The fact that $\psi_i^{-1}\sigma_i\psi_i=\tau_{U_1}$ for $i=0,1$ implies
that $\sigma_1\circ a=a\circ \sigma_0$. Therefore, using Lemma~\ref{LM:lambda-tr-relation}, we get
\begin{eqnarray*} d_1^*g\cdot a &= \lambda\tr\circ \sigma_1\circ a =\lambda\tr\circ a\circ \sigma_0=
a^{-\lambda\tr}\circ \lambda\tr\circ \sigma_0=a^{-\lambda \tr} \cdot d_0^*g\ ,
\end{eqnarray*} or equivalently, $ a^{-\lambda\tr} \cdot d_0^*g \cdot a^{-1} \cdot d_1^* g^{-1}=1 $ in 
$\PGL_n(R)(U_1)$.

There exists a covering $V\to U_1$ such that the image of $a$ in $\PGL_n(R)(V)$, lifts to
\[b\in \GL_n(R)(V)\ .\] The relation $ a^{-\lambda\tr} \cdot d_0^*g \cdot a^{-1} \cdot d_1^* g^{-1}=1 $ now implies that
\begin{equation}\label{EQ:beta-b-v-relation} \beta:=b^{-\lambda\tr} \cdot d_0^*h \cdot b^{-1}\cdot d_1^*h^{-1}\in
\units{R}(V)\ .
\end{equation} Using \eqref{EQ:eps-def} and \eqref{EQ:beta-b-v-relation} we get
\begin{align*} \beta^{-1}\beta^\lambda &=\beta^{-1}(b^{-\lambda\tr} \cdot d_0^*h \cdot b^{-1}\cdot d_1^*h^{-1})^{\lambda
\tr}\\ &=\beta^{-1}\cdot d_1^*h^{-\lambda\tr} \cdot b^{-\lambda\tr} \cdot d_0^*h^{\lambda\tr} \cdot b^{-1} \\
&=d_1^*h^{-\lambda\tr} \cdot (b^{-\lambda\tr} \cdot d_0^*h \cdot b^{-1}\cdot d_1^*h^{-1})^{-1}\cdot b^{-\lambda\tr}
\cdot d_0^*h^{\lambda\tr} \cdot b^{-1} \\ &=d_1^*h^{-\lambda\tr} \cdot d_1^*h \cdot b \cdot d_0^*h^{-1} \cdot
b^{\lambda\tr} \cdot b^{-\lambda\tr} \cdot d_0^*h^{\lambda\tr} \cdot b^{-1}= d_1^*\veps\cdot d_0^*\veps^{-1}
\end{align*} in $\GL_n(R)(V)$. This establishes the first part of the lemma.

Let $t$ denote the global section determined by $\quo{\veps}$. The construction of $t$ involves choosing $U$, $\psi$
and $h\in \GL_n(U)$ above. Suppose that $t'\in\Hoh(T)$ was obtained by replacing these choices with $U'$, $\psi'$ and
$h'\in \GL_n(U')$. We need to show that $t=t'$.

Define $g',\sigma',\veps'$ as above using $U'$, $\psi'$, $h'$ in place of $U$, $\psi$, $h$. It is clear that refining
the covering $U\to *$ does not affect $t$. Therefore, refining $U\to *$ and $U'\to *$ to ${U\times U'}\to *$, we may
assume that $U=U'$. Write $\psi'=u\circ \psi$ with $u\in \Aut_{R_U}(\nMat{R_U}{n\times n})=\PGL_n(R)(U)$. Then
$\sigma'=\psi'\circ \tau_U\circ\psi'^{-1}=u\circ \sigma \circ u^{-1}$, and using Lemma~\ref{LM:lambda-tr-relation}, we
get $g'=\lambda\tr \circ \sigma'=u^{-\lambda\tr}g u^{-1}$. Refining $U\to *$ further, if necessary, we may assume that
$u$ lifts to $v\in\GL_n(R)(U)$. The relation $g'=u^{-\lambda\tr} gu^{-1}$ implies that there is $\alpha\in \units{R_U}$
such that $h'=\alpha^{-1}v^{-\lambda\tr} hv^{-1}$. Thus,
\begin{align*} \veps'& =h'^{-\lambda\tr}h' =(\alpha^{-1}v^{-\lambda\tr} hv^{-1})^{-\lambda\tr}\alpha^{-1}v^{-\lambda\tr}
hv^{-1} =\alpha^\lambda vh^{-\lambda\tr}v^{\lambda\tr}\alpha^{-1}v^{-\lambda\tr} hv^{-1} =\alpha^\lambda
\alpha^{-1}\veps
\end{align*} and $\quo{\veps}=\quo{\veps'}$ in $T(U)$. This completes the proof.
\end{proof}

\begin{definition} \label{DF:coarse-types} 
Let $(A,\tau)$ be an Azumaya $R$-algebra with a $\lambda$-involution. The \emph{coarse
$\lambda$-type} or \emph{coarse type} of $\tau$ is the global section of $T$ determined by $\quo{\veps}\in T(U)$
constructed above. It shall be denoted $\ct(\tau)$ or $\ct(A,\tau)$.
\end{definition}

\begin{remark}
	In accordance with 
	Remark~\ref{RM:types-over-OX},
	we shall write the coarse type of a $\lambda$-involution $\tau$
	of an Azumaya $\calO_\mathbf{X}$-algebra, defined to be $\ct(\pi_*\tau)$, as
	$\ct_\pi(\tau)$.
\end{remark}
 
\begin{prp}\label{PR:type-determines-coarse-type} Let $(A,\tau)$, $(A',\tau')$ be Azumaya $R$-algebras with
$\lambda$-involutions. Then:
\begin{enumerate}[label=(\roman*)]
\item $\ct(A,\tau)=\ct(\nMat{A}{m\times m},\tau\tr)$ for all $m$.
\item $\ct(\tau\tensor_R\tau')=\ct(\tau)\cdot \ct(\tau')$ in $\Hoh^0(T)$. 
\item If there is a covering $V\to *$ such that $(A_V,\tau_V)\iso (A'_V,\tau'_V)$, then $\ct(\tau)=\ct(\tau')$.
\end{enumerate} Consequently, the map $\tp(\tau)\mapsto \ct(\tau):\Types{\lambda}\to \CTypes{\lambda}$ is a
well-defined morphism of monoids.
\end{prp}

\begin{proof} Write $n=\deg A$ and $n'=\deg A'$. Define $U,\psi,g,\sigma,h,\veps$ as in
Construction~\ref{CN:coarse-type}, and analogously, define
$U',\psi',g',\sigma',h',\veps'$ using $(A',\tau')$ in place of $(A,\tau)$.
\begin{enumerate}[label=(\roman*), align=left, leftmargin=0pt, itemindent=1ex, itemsep=0.6\baselineskip]
\item The isomorphism $\psi:A_U\to \nMat{R_U}{n\times n}$ gives rise
to an isomorphism
\[ \psi_m:\nMat{A}{m\times m}_U\to \nMat{\nMat{R_U}{n\times n}}{m\times m}= \nMat{R_U}{nm\times nm}.\] Let
$\sigma_m=\psi_m\circ \lambda\tr \circ \psi_m^{-1}$, let $g_m:=\lambda\tr\circ \sigma$ and let $h_m=(h\oplus \dots\oplus
h)\in \GL_{nm}(R)(U)$. Straightforward computation shows that the image of $h_m$ in $\PGL_{nm}(R)(U)$ is $g_m$.
This means that $\veps_m:=h_m^{-\lambda\tr}h_m$ coincides with $\veps=h^{-\lambda\tr}h$ in $N(U)$, and thus
$\ct(A,\tau)=\ct(\nMat{A}{m\times m},\tau\tr)$.

\item Consider $\tilde{\psi}=\psi\tensor \psi':A_U\tensor A'_U\to \nMat{R_U}{n\times n}\tensor \nMat{R_U}{n'\times n'}=
\nMat{R_U}{nn'\times nn'}$, let $\tilde{\sigma}=\tilde{\psi}\circ(\tau\tensor \tau')\circ \tilde{\psi}^{-1}=
\sigma\tensor \sigma'$, and let $\tilde{g}=\lambda\tr\circ \sigma=g\tensor g'$. Define $\tilde{h}:=h\tensor h'\in
\GL_{nn'}(R)(U)$. Then $\tilde{h}$ maps onto $\tilde{g}$, and we have
$\tilde{h}^{-\lambda\tr}\tilde{h}=(h^{-\lambda\tr}h)\tensor (h'^{-\lambda\tr}h')$, which means $\ct(\tau\tensor_R\tau')=\ct(\tau)\cdot \ct(\tau')$.

\item Fix an isomorphism $\eta:(A'_V,\tau'_V)\to (A_V,\tau_V)$ and, in the construction of $\ct(\tau)$, choose a
covering $U\to *$ factoring through $V\to *$. Taking $U'=U$ and $\psi':=\psi\circ \eta_U: A'_U\to \nMat{R}{n\times
n}$ in the construction of $\ct(A',\tau')$, we find that
\[\sigma'=\psi'\tau'_{U}\psi'^{-1}=\psi\eta_{U}\tau'_{U}\eta_{U}^{-1}\psi^{-1}= \psi\tau\psi^{-1}=\sigma\ , \] so
$\ct(\tau)=\ct(\tau')$. \qedhere
\end{enumerate}
\end{proof}

\begin{remark}\label{RM:type-to-ct-not-injective}
There are examples where $\Types{\lambda}\to \CTypes{\lambda}$ is not injective. For example, 
by Proposition~\ref{PR:T-is-two-torsion},
the image of $\Types{\lambda}\to \CTypes{\lambda}$
is a subgroup, so this map is not injective when $\Types{\lambda}$ is not a group, 
e.g.\ Example~\ref{EX:types-in-char-two}.
\end{remark}

\begin{definition}
 An abelian group object $G$ of $\bfY$ is said to have \emph{square roots locally} if the
 the squaring map $x\mapsto x^2:G\to G$ is an epimorphism. That is, for any object $U$ of $\bfY$ and $x\in G(U)$, there
 exists a covering $V\to U$ and $y\in G(V)$ such that $y^2=x$.
 \end{definition}

 \begin{example} \label{EX:SHL_implies_SRL}
 If $\mathbf Y$ is the topos of a topological space with
 the ring sheaf of continuous functions into $\C$ or the \'etale ringed topos of a scheme on which $2$ is
 invertible, then the group $\units{\sh O_\bfY}$ has square roots locally.
 Indeed, this holds at the stalks, because the stalks of ${\sh O}_\bfY$ 
 are strictly
 henselian rings in which $2$ is invertible---this is well known in the case of an \'etale ringed
 topos of a scheme, or proved in Appendix \ref{AX:C_is_s_henselian} in the case of a topological space.
 Furthermore, if $\bfY$ is the fppf ringed topos of an arbitrary scheme $Y$, then
 $\units{\sh O_\bfY}$ has square roots locally, because any $U$-section 
 has a square root over a degree-$2$ finite flat covering of $U$. 
 % (Consult Remark~\ref{RM:fppf-site}
 %regarding the exactness of $\pi:\bfX\to \bfY$ in this case).
 \end{example}

The main result of this section is the following theorem, which shows that under mild assumptions, 
$\lambda$-involutions of Azumaya algebras of the same degree having the same coarse
type are locally isomorphic.
As a consequence, an analogue of the desired property \ref{item:type-cond-local-iso}
from Section~\ref{subsec:types} holds under the same assumptions.
The theorem holds in particular when $\pi:\bfX\to \bfY$ is induced by a
good $C_2$-quotient of schemes on which $2$ is invertible (see
Example~\ref{EX:important-exact-quo-scheme}), 
or by a $C_2$-quotient of
Hausdorff topological spaces (see
Example~\ref{EX:important-exact-quo-top}).

\begin{theorem}\label{TH:ct-determines-local-iso}
	Let $\bfX$ be a locally ringed topos with involution $\lambda$,
	let $\pi:\bfX\to \bfY$ be an exact quotient relative to $\lambda$,
	and
	write $R=\pi_*\calO_\bfX$ and $S=\calO_\bfY$.
 Suppose that $\units{S}$ has square roots locally and at least one of the following
 conditions holds:
 \begin{enumerate}[label=(\arabic*)]
 \item $2 \in S^\times$.
 \item \label{item:TH:ct-determines-local-iso:quad-et} 
 $\pi: \mathbf X \to \mathbf Y$ is unramified.
 \item \label{item:TH:ct-determines-local-iso:odd} $n$ is odd.
 \end{enumerate}
 Suppose $(A, \tau)$ and
 $(A',\tau')$ are two degree-$n$ Azumaya $R$-algebras with $\lambda$-involutions.
 Then the following are equivalent:
 \begin{enumerate}[label=(\alph*)]
 \item \label{pr51} $(A, \tau)$ and $(A', \tau')$ are locally isomorphic as $R$-algebras with involution.
 \item \label{pr52} $(A, \tau)$ and $(A', \tau')$ have the same type.
 \item \label{pr53} $(A, \tau)$ and $(A', \tau')$ have the same coarse type.
 \end{enumerate}
\end{theorem}
\begin{proof}
 Statement \ref{pr52} is implied by \ref{pr51} by virtue of the definition of ``type'', Definition \ref{DF:types}. Then the
 implication of \ref{pr53} by \ref{pr52} is Proposition \ref{PR:type-determines-coarse-type}. The final implication,
 that of \ref{pr51} by \ref{pr53}, is somewhat technical and it is given by Proposition \ref{PR:completion} below.
\end{proof}

\begin{cor}\label{CR:coarse-type-determines-type} 
 Suppose the assumptions of Theorem \ref{TH:ct-determines-local-iso} hold and $2 \in S^\times$.
 Then the map $\tp(\tau)\mapsto \ct(\tau):\Types{\lambda}\to \CTypes{\lambda}$ is injective. 
\end{cor}
\begin{proof}
 Suppose $\ct(A,\tau)=\ct(A',\tau')$ and write $n=\deg A$, $n'=\deg A'$. 
 By Proposition~\ref{PR:type-determines-coarse-type}, we may replace
 $(A,\tau)$ with $(\nMat{A}{n'\times
 n'},\tau\tr)$ and $(A',\tau')$ with $(\nMat{A'}{n\times n},\tau'\tr)$, and assume that $\deg A=\deg
 A'$. Now, by Theorem~\ref{TH:ct-determines-local-iso}, $(A,
 \tau)$ is locally isomorphic to $(A', \tau')$ as an $R$-algebra with involution, and \textit{a fortiori} it has the
 same type.
\end{proof}

\begin{corollary}\label{CR:Types-are-two-torsion}
 With the assumptions of Corollary \ref{CR:coarse-type-determines-type}, $\Types{\lambda}$ is a $2$-torsion group.
\end{corollary}
\begin{proof}
 We know $\CTypes{\lambda}$ is a $2$-torsion group by Proposition~\ref{PR:T-is-two-torsion}, 
 and $\Types{\lambda}$ is a submonoid by Corollary~\ref{CR:coarse-type-determines-type}.
\end{proof}

\begin{remark}\label{RM:are-all-coarse-types-realizable}
We do not know whether the assumptions
of Corollary~\ref{CR:coarse-type-determines-type}
imply that the map $\Types{\lambda}\to \CTypes{\lambda}$ is surjective.
A more extensive discussion of this and some positive results
will be given in
Subsection~\ref{subsec:ker-of-transfer}. 
\end{remark}

The following theorem shows that the properties exhibited in Example~\ref{EX:types-for-schemes} extend to our general
 setting under some assumptions. 
 
\begin{thm}\label{TH:basic-properties-of-types} 
	With the assumptions of Theorem~\ref{TH:ct-determines-local-iso}, the following hold:
\begin{enumerate}[label=(\roman*)] 
\item \label{item:TH:basic-properties-of-types:tot-ram} If $2\in\units{S}$ and 
$\pi:\bfX\to \bfY$ is a trivial quotient (Example~\ref{EX:trivial-exact-quotient}), 
then $\Types{\lambda}$ is isomorphic to the group $\Hoh^0(\mu_{2,R})$.
\item \label{item:TH:basic-properties-of-types:unram} 
If $\pi:\bfX\to \bfY$ is unramified, then $\Types{\lambda}=\{1\}$.
\item \label{item:TH:basic-properties-of-types:odd} Let $(A,\tau)$ be an Azumaya $R$-algebra with a $\lambda$-involution. If
$\deg A$ is odd, then $\tp(\tau) =1$.
\end{enumerate}
\end{thm}

We deduce Theorem~\ref{TH:basic-properties-of-types} mostly as a corollary of Theorem~\ref{TH:ct-determines-local-iso}.

\begin{proof}
\ref{item:TH:basic-properties-of-types:tot-ram} This follows from Theorem~\ref{TH:ct-determines-local-iso} and Example~\ref{EX:T-tot-ramified} if we show that for every
$\veps\in \Hoh^0(\mu_{2,R})=\Hoh^0(T)$, there is an involution of coarse type $\veps$.
To see that, let $h=[\begin{smallmatrix} 0 & \veps \\ 1 & 0 \end{smallmatrix}]$
and take $\tau:\nMat{R}{2\times 2}\to \nMat{R}{2\times 2}$
defined by $x\mapsto (h xh^{-1})^{\lambda\tr}$. That $\ct(\tau)=\veps$
follows by applying Construction~\ref{CN:coarse-type} with $U=*$ and $h$, $\veps$
just defined.

\ref{item:TH:basic-properties-of-types:unram} This follows from
Theorem~\ref{TH:ct-determines-local-iso} and Proposition~\ref{PR:T-structure-unramified-pi}.

\ref{item:TH:basic-properties-of-types:odd} It is enough to show that
$\ct( \tau)=1$ whenever $\deg A=2m+1$. 
Define $U$, $h$ and $\veps$ as in Construction~\ref{CN:coarse-type}.
From \eqref{EQ:eps-def}, we have $h=\veps h^{\lambda\tr}$. 
Taking the determinant of both sides yields
$\det h = \veps^{2m+1} (\det h)^\lambda$. Since $\veps^\lambda=\veps^{-1}$, this implies that
$\veps=\beta^{-1}\beta^\lambda$ for $\beta=\veps^{-m}\det h$. 
This means that $\quo{\veps}$, the image of $\veps$ in $T(U)$, is trivial,
so $\ct( \tau)=1$.
\end{proof}

 We note that part \ref{item:TH:basic-properties-of-types:tot-ram} applies in the case 
 where $\pi:\bfX\to \bfY$ is induced by a scheme $X$ on which $2$ is invertible endowed 
 with the trivial 
 involution $\lambda=\id :X\to X$; see
 Example~\ref{EX:types-for-schemes}\ref{EX:types-for-schemesi}.

 Part \ref{item:TH:basic-properties-of-types:unram} applies to the case where $\pi:\bfX\to \bfY$ is induced by a
 quadratic \'etale morphism of schemes $\pi: X \to Y$ where $X$ is given the canonical $Y$-involution; see
 Example~\ref{EX:types-for-schemes}\ref{EX:types-for-schemesii}.

\medskip

Our last application of Theorem~\ref{TH:ct-determines-local-iso} provides a concrete realization of the first cohomology
set of the projective unitary group of an Azumaya $R$-algebra with a $\lambda$-involution $(A,\tau)$. As usual, the
unitary group of $(A,\tau)$ is the group object $\U(A,\tau)$ in $\bfY$ whose $V$-sections are
$\{a\in A(V)\suchthat a^\tau a=1\}$, and the projective unitary group of $(A,\tau)$ is the quotient
	\[
	\PU(A,\tau)=\U(A,\tau)/N
	\]
	where $N$, the group of $\lambda$-norm $1$ elements in $R$ defined above.

	If $(A',\tau')$ is another Azumaya $R$-algebra with a $\lambda$-involution
	such that $\deg A=\deg A'$, we further define
	$\sHom_R((A,\tau),(A',\tau'))$ to be the sheaf
	of $R$-linear isomorphisms from $(A,\tau)$
	to $(A',\tau')$, and $\sAut_R(A,\tau)$
	to be the group sheaf of $R$-linear automorphisms 
	of $(A,\tau)$. 
	
	\begin{lemma}\label{LM:PU-isomorphism}
		Suppose $\units{S}$ has square roots locally.
	Let $(A,\tau)$ be a degree-$n$ Azumaya $R$-algebra 
	with a $\lambda$-involution. 
	Then the map $\U(A,\tau)\to \sAut_R(A,\tau)$
	sending a section $x$ to conjugation
	by $x$ is an epimorphism with kernel $N$.
	Consequently, it induces an isomorphism
	$\PU(A,\tau)\cong \sAut_R(A,\tau)$.
	\end{lemma}
	
	\begin{proof}
		That the kernel is $N$ follows easily from the fact
		that the centre of $A$ is $R$.
		We turn to proving that the map is an epimorphism.
		
		Let $V\in\bfY$ and $\psi\in \sAut_R(A,\tau)(V)=\Aut_{R_V}(A_V,\tau_V)$. By replacing $V$ with suitable
 covering, we may assume that $A_V=\nMat{R_V}{n\times n}$ and that $\psi_V\in \PGL_n(R)(V)$ is given
 section-wise by $\psi_V(x)=hxh^{-1}$ for some $h\in \GL_n(R)(V)$. Since
 $\psi_V\circ \tau_V=\tau_V\circ \psi_V$, for any section $x\in A(V)$, we have
 $h^{-1}x^\tau h=h^{\tau}x^\tau (h^{-1})^\tau$, and thus $hh^{\tau}\in \units{R}(V)$. In fact, since
 $(h h^\tau)^\lambda=h h^\tau$, we have $h h^\tau\in\units{S}(V)$. By assumption, we can replace $V$
 with a suitable covering to assume that there is $\alpha\in \units{S}(V)$ with $\alpha^2=h h^\tau$.
 Replacing $h$ with $h\alpha^{-1}$ yields $h h^\tau=1$. We have therefore shown that over a covering of
 $V$, $\psi$ lifts to a section of $\U(A,\tau)$.
	\end{proof}

\begin{corollary}\label{CR:torsors-of-PU}
 With the assumptions of Theorem~\ref{TH:ct-determines-local-iso}, let $(A,\tau)$ be a degree-$n$ Azumaya $R$-algebra
 with a $\lambda$-involution and identify $\sAut_R(A,\tau)$ with $\PU(A,\tau)$ as in Lemma~\ref{LM:PU-isomorphism}.
 Then the functor $(A',\tau')\mapsto \sHom_R((A,\tau),(A',\tau'))$
 defines an equivalence between 
 the full subcategory
 of
 $\Az_n(\bfY,R,\lambda)$ consisting of $R$-algebras with a $\lambda$-involution of the same type as $\tau$
 and
 $\Tors(\bfY,\PU(A,\tau))$.
 Consequently, $\Hoh^1(\bfY,\PU(A,\tau))$ is in canonical bijection with isomorphism classes of the aforementioned
 algebras with involution.
\end{corollary}

	\begin{proof}
 Theorem~\ref{TH:ct-determines-local-iso} implies that an $R$-algebra with a $\lambda$-involution $(A',\tau')$
 is locally isomorphic to $(A,\tau)$ if and only if $A$ is Azumaya of degree $n$ and $\tau$ is of the same type
 as $\tau'$. With this fact at hand, the equivalence is standard; see \cite[V.\S
 4]{giraud_cohomologie_1971}\benw{Right book, wrong chapter. Will
 fix.}\benw{Actually, I think no reference is better than any reference here.}. The last statement follows from
 Proposition~\ref{PR:non-ab-coh-basic-prop}\ref{item:PR:non-ab-coh-basic-prop:torsor}.
	\end{proof}

 Many of the previous results require that $2\in\units{S}$. Indeed, our arguments build on using the coarse
 type, which is too coarse if $2$ is not invertible---see Remark~\ref{RM:type-to-ct-not-injective}. Nevertheless,
 we ask:
 
 \begin{question}
	Does the equivalence between \ref{pr51} and \ref{pr52} in 
	Theorem~\ref{TH:ct-determines-local-iso}
	hold without assuming any of the conditions (1), (2), (3)? 
 \end{question}
 
	A particularly interesting case is 
	the morphism of \emph{fppf} ringed topoi associated to
	a finite locally free good $C_2$-quotient 
 	$\pi:X\to Y$, where $X$ is a scheme on which $2$ is not invertible
 	(consult Remark~\ref{RM:fppf-site}).

\subsection{Proof of Theorem~\ref{TH:ct-determines-local-iso}}
\label{subsec:technical-proofs}

In this subsection we complete the proof of Theorem~\ref{TH:ct-determines-local-iso}
by showing that condition \ref{pr53} implies condition \ref{pr51}. This result
is given as Proposition~\ref{PR:completion} below. The reader can skip this subsection without loss
of continuity.

\medskip

In the following lemmas, unless otherwise specified, $A$ will be a ring, $\lambda: A \to A$ an involution, 
and $B$ will be
the fixed ring of $\lambda$. We will write $\quo A $ for $A/\Jac(A)$, and $\quo \lambda: \quo A \to \quo A$ for the
involution induced on $\quo A$ by $\lambda$. Let $\veps \in \{ \pm 1 \}$ and let $h \in \GL_n(A)$ be an
$(\veps, \lambda \tr)$-hermitian matrix, which is to say
\[ h=\veps h^{\lambda \tr} . \]

Let $H:A^n\times A^n\to A$ be the $(\veps,\lambda)$-hermitian form associated to $h$; it is given by
$h(x,y)=x^{\lambda\tr} hy$ where $x,y\in A^n$ are written as column vectors. Let $\quo{H}$ denote the reduction of $H$ to $\quo{A}$.

\begin{lem}\label{LM:decomposition-I}
Assume $B$ is local.
If $\quo{\lambda}\neq\id_{\quo{A}}$, or $\veps\neq -1$ in $\quo{A}$, or $n$ is
odd, then there exists $v\in \GL_n(A)$ such that $v^{\lambda\tr}hv$ is a diagonal matrix.
\end{lem}

\begin{proof} 
Proving the lemma is equivalent to showing that $H$ is diagonalizable, i.e., has an orthogonal basis.

We first claim that $\quo{H}$ has an orthogonal basis. This is
well known when $\quo{A}$ is a field; see \cite[Th.~7.6.3]{scharlau_quadratic_1985} for the case where
$\quo{\lambda}\neq \id$ or $\veps\neq -1$ in $\quo{A}$, and \cite[Thm.~6]{albert_symmetric_1938} for the case where $\quo{\lambda}=\id_A$,
$\veps=-1$ in $\quo{A}$ and $n$ is odd. We note in passing that the second case can occur only when the characteristic of $\quo{A}$ is $2$. If
$\quo{A}$ is not a field, then Theorem \ref{thm:local-involution} and
Proposition~\ref{PR:Hilbert-ninety}\ref{item:LM:Hilbert-ninety:structure} imply that $\quo{A}\iso k\times k$, where
$k$ is the residue field of $B$, and $\quo{\lambda}$ acts by interchanging the two copies of $k$. In this case
$\quo{H}$ is hyperbolic and the easy proof is left to the reader.

We now claim that any nondegenerate $(\veps,\lambda)$-hermitian form $H$ whose reduction $\quo{H}$ admits a
diagonalization is diagonalizable, thus proving the lemma. Let $\{\quo{x}_1,\dots,\quo{x}_n\}\subseteq \quo{A}^n$ be an
orthogonal basis for $\quo{H}$ and let $x_1\in A^n$ be an arbitrary lift of $\quo{x}_1$. Then $H(x_1,x_1)\in\units{A}$ and
hence $A^n=x_1A\oplus x_1^{\perp}$. Write $P=x_1^{\perp}$ and $H_1=H|_{P\times P}$. The $A$-module $P$ is free because
$A$ is semilocal and $P$ is projective of constant $A$-rank $n-1$. Furthermore, since $\quo{P}=\sum_{i=2}^{n}
\quo{x}_i\quo{A}$, the form $\quo{H_1}$ is diagonalizable by construction. We finish by applying induction to $H_1$.
\end{proof}

\begin{lem}\label{LM:decomposition-II} 
 Assume $B$ is local, and suppose $\quo{\lambda}=\id_{\quo{A}}$, that $\veps = -1$ and that $2\in\units{A}$. Then
 there exists $v\in \GL_n(A)$ such that $v^{\lambda\tr}hv$ is a direct sum of $2\times 2$ matrices in
 $[\begin{smallmatrix} 0 & 1 \\ -1 & 0\end{smallmatrix}]+\nMat{\Jac(A)}{2\times 2}$. In particular, $n$ is even.
\end{lem}

\begin{proof} 
We need to show that $A^n$
admits a basis $\{x_1,y_1,x_2,y_2,\dots\}$ such that $x_iA+y_iA$ is orthogonal to $x_jA+y_jA$ whenever $i\neq j$ and such
that $[\begin{smallmatrix} H(x_i,x_i) & H(x_i,y_i) \\ H(y_i,x_i) & H(y_i,y_i)\end{smallmatrix}] \in [\begin{smallmatrix}
0 & 1 \\ -1 & 0\end{smallmatrix}]+\nMat{\Jac(A))}{2\times 2}$.

By Theorem~\ref{thm:local-involution}, the
assumption $\quo{\lambda}=\id_{\quo{A}}$ implies that $A$ is local, hence $\quo{A}$ is a field of characteristic different from $2$ and
$\quo{H}$ is a nondegenerate alternating bilinear form. This means that $n$ must be even. Choose a nonzero
$\quo{x}\in \quo{A}^n$. Since $\quo{H}$ is nondegenerate, there is $\quo{y}$ such that
$\quo{H}(\quo{x},\quo{y})=1$. Since $\quo{H}$ is alternating, we also have $\quo{H}(\quo{y},\quo{x})=-1$ and
$\quo{H}(\quo{x},\quo{x})=\quo{H}(\quo{y},\quo{y})=0$. Let $x,y\in A^n$ be lifts of $\quo{x}$ and $\quo{y}$. The previous
equations imply that $M:=[\begin{smallmatrix} H(x,x) & H(x,y) \\ H(y,x) & H(y,y)\end{smallmatrix}] \in
[\begin{smallmatrix} 0 & 1 \\ -1 & 0\end{smallmatrix}]+\nMat{\Jac(A)}{2\times 2}$. In particular, $M$ is invertible,
and 
so $A^n=(xA\oplus yA)\oplus \{x,y\}^\perp$. 
We proceed by
induction on the restriction of $H$ to $\{x,y\}^\perp$.
\end{proof}

\begin{lem}\label{LM:JacR-squared}
	Assume $B$ is local and $2\in\units{B}$. Then $\Jac(A)^2\subseteq \Jac(B)A$.
\end{lem}

\begin{proof}
	Write $\frakm =\Jac(B)$
	and let $x,y\in \Jac(A)$. Then $x^\lambda+x,x^\lambda x\in \Jac(A)\cap B\subseteq \frakm$.
	The equality $x^2-(x^\lambda+x)x+(x^\lambda x)=0$ implies $x^2\in \frakm A$.
	Likewise, $y^2,(x+y)^2\in \frakm A$.
	We finish by noting that $xy=\frac{1}{2}((x+y)^2-x^2-y^2)$.
\end{proof}

In the following lemmas, given a ring $A$ and $a\in A$, we write $A[\sqrt{a}]$ to denote the ring $A[T]/(T^2-a)$ and let
$\sqrt{a}$ denote the image of $T$ in $A[\sqrt{a}]$. By induction, we define
$A[\sqrt{a_1},\sqrt{a_2},\dots]=A[\sqrt{a_1}][\sqrt{a_2},\dots] \cong A[T_1,T_2,\dots]/(T_1^2-a_1,T_2^2-a_2,\dots)$. If
$\lambda:A\to A$ is an involution with fixed ring $B$ and $a\in B$, then $\lambda$ extends to $A[\sqrt{a}]$ by setting
$(\sqrt{a})^\lambda=\sqrt{a}$, and the fixed ring of $\lambda:A[\sqrt{a}]\to A[\sqrt{a}]$ is $B[\sqrt{a}]$.

\begin{lem}\label{LM:tough-lemma} 
Assume $B$ is local and 
suppose $\quo{\lambda}=\id_{\quo{A}}$, that $\veps = -1$ and 
 $2\in\units{A}$. Suppose that $h$ lies in $[\begin{smallmatrix} 0 & 1 \\ -1 &
0\end{smallmatrix}]+\nMat{\Jac(A)}{2\times 2}$. Then there are
$s,t\in\units{B}$, $f_1,\dots,f_r\in {B[\sqrt{s},\sqrt{t}]}$ and $v_i\in \GL_2(A[\sqrt{s},\sqrt{t}]_{f_i})$
($i=1,\dots,r$) such that $\sum_i f_iB[\sqrt{s},\sqrt{t}]=B[\sqrt{s},\sqrt{t}]$ and
$v_i^{\lambda\tr}hv_i=[\begin{smallmatrix} 0 & 1 \\ -1 & 0\end{smallmatrix}]$ in $A[\sqrt{s},\sqrt{t}]_{f_i}$ for all
$i$.
\end{lem}

\begin{proof} 
Again, by Theorem~\ref{thm:local-involution}, $A$ is local.
We write $\frakm=\Jac(B)$, $\frakM=\Jac(A)$ and let $\{x,y\}$ denote
the standard $A$-basis of $A^2$. Once $s,t\in S$ and $f_1,\dots,f_r\in B[\sqrt{s},\sqrt{t}]$ above have been chosen, we
need to show that for all $i$, there are $\tilde{x},\tilde{y}\in A[\sqrt{s},\sqrt{t}]_{f_i}^2$ such that
$\big[\begin{smallmatrix} H(\tilde{x},\tilde{x}) & H(\tilde{x},\tilde{y}) \\ H(\tilde{y},\tilde{x}) &
H(\tilde{y},\tilde{y})\end{smallmatrix}\big]= [\begin{smallmatrix} 0 & 1 \\ -1 & 0\end{smallmatrix}]$.

Our assumption on $h$ implies that $H(x,y)\in 1+\Jac(A)\subseteq \units{A}$. Replacing $y$ with $H(x,y)^{-1}y$, we may
assume that $H(x,y)=1$ and so $H(y,x)=-1$. We further write $\alpha=H(x,x)$ and $\beta=H(y,y)$. Since
$h=-h^{\lambda\tr}$, we have $\alpha^\lambda=-\alpha$ and $\beta^\lambda=-\beta$, and since $\quo{\lambda}=\id_{\quo{A}}$,
it follows that $\alpha,\beta\in \frakM$ and hence $\alpha^2,\beta^2\in B\cap \frakM\subseteq \frakm$.

Observe that the polynomial $ p(x)=-2x^2+(2-2\alpha)x+\alpha\in A[x] $ has discriminant
\[s:=(2-2\alpha)^2-4(-2\alpha)=4+4\alpha^2\] and that $s \in 4+\frakm\subseteq \units{B}$. 
The roots of $p(x)$ in $A[\sqrt{s}]$ are $c_1=\frac{1}{4}(\sqrt{s}+2-2\alpha)$ and
$c_2=\frac{1}{4}(-\sqrt{s}+2-2\alpha)$ and we note that
\begin{align}\label{EQ:c-invertablity} 16(c_1^\lambda c_1+c_2^\lambda c_2)&=
(\sqrt{s}+2)^2-4\alpha^2+(\sqrt{s}-2)^2-4\alpha^2= 16\in\units{B}\ .
\end{align}

We repeat this construction with $\beta$ in place of $\alpha$, denoting the elements corresponding to $s$, $c_1$, $c_2$
by $t$, $d_1$, $d_2$.

Fix a maximal ideal $\frakp \lhd B[\sqrt{s},\sqrt{t}]$ and write
$B':=B[\sqrt{s},\sqrt{t}]_\frakp$, $A':=A[\sqrt{s},\sqrt{t}]_\frakp$,
$\frakm'=\Jac(B')$ and $\frakM'=\Jac(A')$.

It is clear that $B'$ is local. We claim that $A'$ is also local and $\frakM A'\subseteq \frakM'$. Indeed, by
\cite[Th.~6.15]{reiner_maximal_1975-1}, we have
$\frakm B[\sqrt{s},\sqrt{t}]\subseteq \Jac(B[\sqrt{s},\sqrt{t}])\subseteq \frakp$, and hence $\frakm B'
\subseteq \frakp_\frakp=\frakm'$, which in turn implies $\frakm A'\subseteq \frakm'A'$.
By Lemma~\ref{LM:mR-subset-Jac}, we have $\frakm'A'\subseteq \frakM'$, and by Lemma~\ref{LM:JacR-squared},
$\frakM^2\subseteq \frakm A$. Using the last three inclusions, we get
$(\frakM A')^2=\frakM^2 A'\subseteq \frakm A'\subseteq \frakm' A'\subseteq \frakM'$, and
since $\frakM'$ is semiprime, 
$\frakM A'\subseteq \frakM'$. The latter implies that $A'\to A'/\frakM'$ factors through $A'/\frakM A'\iso \quo{A}\tensor_B B'$, and
hence the specialization of $\lambda$ to $A'/\frakM'$ is the identity. Since $B'$
is flat over $B$, the local ring $B'$ is the fixed ring of
$\lambda:A'\to A'$ and Theorem~\ref{thm:local-involution} implies that $A'$ is local.

Now, the inclusion $\frakM A'\subseteq \frakM'$ 
implies $\alpha,\beta\in \frakM'$. By equation \eqref{EQ:c-invertablity},
there is $i\in\{1,2\}$ such that $c_i^\lambda c_i\in \units{B'}$, and hence $c_i\in\units{A'}$. In the same way, there
is $j\in\{1,2\}$ such that $d_j\in\units{A'}$.

Working in $A'$, we have
\begin{align}\label{EQ:ci-relation} c_i^{-1}+(c_i^{-1})^\lambda
&=\frac{4}{\pm\sqrt{s}+2-2\alpha}+\frac{4}{\pm\sqrt{s}+2+2\alpha} \\
&=\frac{8(\pm\sqrt{s}+2)}{(\pm\sqrt{s}+2)^2-4\alpha^2} \nonumber \\ &=\frac{8(\pm\sqrt{s}+2)}{s\pm
4\sqrt{s}+4-4\alpha^2} \nonumber \\ &=\frac{8(\pm\sqrt{s}+2)}{\pm 4\sqrt{s}+8}=2 \nonumber
\end{align} and likewise for $d_j$. Write $u=c_i^{-1}-1$. Since
$\alpha(c_i^{-1})^2+(2-2\alpha)c_i^{-1}-2=c_i^{-2}p(c_i)=0$, we have
\begin{equation}\label{EQ:alpha-relation} \alpha u^2+2u-\alpha=0\ ,
\end{equation} and since $\alpha\in\frakM'$, this implies $u\in \frakM'$. We further have $u^\lambda=-u$ by
\eqref{EQ:ci-relation}. In the same way, writing $w=d_j^{-1}-1$, we have
$\beta w^2+2w-\beta=0$, $w\in \frakM'$
and $w^\lambda=-w$. Let $\tilde{x}=(1+u)x+(1-w)y$. Then, using \eqref{EQ:alpha-relation},
we get
\begin{align*} H(\tilde{x},\tilde{x}) &=(1+u)^\lambda (1+u)\alpha + (1+u)^\lambda (1-w)-(1-w)^\lambda(1+u) +
(1-w)^\lambda(1-w)\beta\\ &=(1-u^2)\alpha+(1-u)(1-w)-(1+w)(1+u)+(1-w^2)\beta\\ &=\alpha-\alpha u^2-2u +\beta -\beta w^2
- 2w=0
\end{align*} Likewise, $\tilde{y}:=(1-u)x-(1+w)y$ satisfies $H(\tilde{y},\tilde{y})=0$. Since $u,w\in
\frakM'=\Jac(A')$, the vectors $\tilde{x},\tilde{y}$ span $A'^2$. Since $H$ is nondegenerate, this forces
$H(\tilde{x},\tilde{y})\in\units{A'}$. Replacing $\tilde{y}$ with $H(\tilde{x},\tilde{y})^{-1}\tilde{y}$, we find that
$\big[\begin{smallmatrix} H(\tilde{x},\tilde{x}) & H(\tilde{x},\tilde{y}) \\ H(\tilde{y},\tilde{x}) &
H(\tilde{y},\tilde{y})\end{smallmatrix}\big]= [\begin{smallmatrix} 0 & 1 \\ -1 & 0\end{smallmatrix}]$.

Finally, for every maximal ideal $\frakp \lhd B[\sqrt{s},\sqrt{t}]$, choose $f_\frakp\in B[\sqrt{t},\sqrt{s}]-
\frakp$ such that the coefficients of $\tilde{x}$, $\tilde{y}$ constructed above are defined in
$B[\sqrt{t},\sqrt{s}]_{f_\frakp}$ and such that the identity $\big[\begin{smallmatrix} H(\tilde{x},\tilde{x}) &
H(\tilde{x},\tilde{y}) \\ H(\tilde{y},\tilde{x}) & H(\tilde{y},\tilde{y})\end{smallmatrix}\big]= [\begin{smallmatrix} 0
& 1 \\ -1 & 0\end{smallmatrix}]$ holds in $B[\sqrt{s},\sqrt{t}]_{f_\frakp}$. Then $\sum_\frakp f_\frakp
B[\sqrt{s},\sqrt{t}]=B[\sqrt{s},\sqrt{t}]$. Since $B[\sqrt{s},\sqrt{t}]$ is a finite algebra over a local ring, it has
only finitely many maximal ideals. The result follows.
\end{proof}

\begin{lem}\label{LM:point-to-neighbourhood} Maintaining the 
 assumptions made at the beginning of this subsection, suppose that $h'\in \GL_n(A)$ is another
 $(\veps,\lambda\tr)$-hermitian matrix. Suppose further we are given a prime ideal $\frakp\in \Spec B$, units
 $s_1,\dots,s_\ell\in \units{B_\frakp}$, elements $f_1,\dots,f_r\in B_\frakp[\sqrt{s_1},\dots,\sqrt{s_\ell}]$ and
 $v_i\in \GL_n(A_\frakp[\sqrt{s_1},\dots,\sqrt{s_\ell}]_{f_i})$ ($i=1,\dots,r$) such that $f_1,\dots,f_r$ generate the
 unit ideal in $B_\frakp[\sqrt{s_1},\dots,\sqrt{s_\ell}]$ and such that $v_i^{\lambda\tr}hv_i=h'$ for all
 $1\leq i\leq r$. Then there is $b\in B- \frakp$ for which the previous condition holds upon replacing $B_\frakp$,
 $A_\frakp$ with $B_b$, $A_b$.
\end{lem}

\begin{proof} There is $b\in B$ such that $s_1,\dots,s_\ell$ are in the image of $\units{B_b}\to \units{B_\frakp}$. We
 may replace $B$, $\frakp$ with $B_b$, $\frakp_b$ and assume $s_1,\dots,s_\ell\in \units{B}$ henceforth.

 Write $B'=B[\sqrt{s_1},\dots,\sqrt{s_\ell}]$ and $A'=A[\sqrt{s_1},\dots,\sqrt{s_\ell}]$, and choose
 $g_1,\dots,g_r\in B'_\frakp$ such that $\sum_i f_i g_i=1$. Then there is $b\in B- \frakp$ such that
 $f_1,\dots,f_r,g_1,\dots,g_r$ are images of elements in $B'_b$, also denoted $f_1,\dots,f_r,g_1,\dots,g_r$, and such
 that $\sum_i g_if_i=1$ in $B'_b$. Again, we replace $B$, $\frakp$ with $B_b$, $\frakp_b$ and assume
 $f_1,\dots,f_r\in B'$.

Fix $1\leq i\leq r$. There are $v'_i\in \nMat{A'}{n\times n}$, $b\in B- \frakp$ and $m\in \N\cup \{0\}$ such
that $v_i=v'_ib^{-1}f_i^{-m}$ in $\nMat{(A'_\frakp)_{f_i}}{n\times n}$. Since $v_i^{\lambda \tr}hv_i=h'$, we have
$v'^{\lambda\tr}_i h v'_i=b^2f_i^{2m} h'$ in $\nMat{(A'_\frakp)_{f_i}}{n\times n}$, and hence there is $k\in \N\cup 0$
such that $f_i^kv'^{\lambda\tr}_i h v'_i=b^2f_i^{2m+k} h'$ in $\nMat{A'_\frakp}{n\times n}$. This in turn implies that
there is $b'\in B-\frakp $ such that $b'f_i^kv'^{\lambda\tr}_i h v'_i=b'b^2f_i^{2m+k}h'$ in $\nMat{A'}{n\times
n}$. Replacing $B$, $\frakp$ with $B_{bb'}$, $\frakp_{bb'}$, we may assume $b,b'\in\units{B}$. Let
$\tilde{v}_i=v'_ib^{-1}f_i^{-m}\in \nMat{A'_{f_i}}{n\times n}$. Then the image of $\tilde{v}_i$ in
$\nMat{(A'_\frakp)_{f_i}}{n\times n}$ is $v_i$ and the equality $b'f_i^kv'^{\lambda\tr}_i h v'_i=b'b^2f_i^{2m+k}h'$
implies that $\tilde{v}_i^{\lambda\tr}h\tilde{v}_i=h'$ in $\nMat{A'_{f_i}}{n\times n}$. Taking determinants, we see
that $\tilde{v}_i\in \GL_n(A'_{f_i})$.

The lemma follows by applying the previous paragraph to all $1\leq i\leq r$.
\end{proof}

We now return to the context of ringed topoi. 

\begin{lem}\label{LM:one-plus-unitary} Assume $2\in\units{S}$, and let $U\in\bfY$, $\veps\in N(U)$ and $t=\quo{\veps}\in
T(U)$. Then:
\begin{enumerate}[label=(\roman*)]
\item There is a covering $\{{V_i\to U}\}_{i=1,2}$ of $U$ such that $1+\veps\in\units{R}(V_1)$ and $1-\veps\in
\units{R}(V_2)$.
\item \label{item:LM:one-plus-unitary:beta}
There exists a covering $\{V_i\to U\}_{i=1,2}$ of $U$ such that $t|_{V_1}=1$ and $t|_{V_2}=-1$. Equivalently,
there exists a covering $\{V_i\to U\}_{i=1,2}$ of $U$ and $\beta_i\in R(V_i)$ ($i=1,2$) such that
$\beta_1^{-1}\beta_1^\lambda=\veps|_{V_1}$ and $-\beta_2^{-1}\beta_2^\lambda=\veps|_{V_2}$.
\end{enumerate}
\end{lem}

\begin{proof} 
\begin{enumerate}[label=(\roman*), align=left, leftmargin=0pt, itemindent=1ex, itemsep=0.3\baselineskip]
\item We note that this statement requires proof because $R$ is not a local ring object in general. Observe
that $\veps^{-1}(1\pm\veps)^2=\veps^\lambda\pm 2+\veps$ and hence $\veps^{-1}(1\pm\veps)^2\in S(U)$. Since
$\veps^{-1}(1+\veps)^2- \veps^{-1}(1-\veps)^2=4$ and $S$ is a local ring object, the assumption $2\in\units{S}$ implies
that there exists a covering $\{V_i\to U\}_{i=1,2}$ of $U$ such that $\veps^{-1}(1+\veps)^2\in\units{S}(V_1)$ and
$\veps^{-1}({1-\veps})^2\in\units{S}(V_2)$. It follows that $1+\veps\in\units{R}(V_1)$ and $1-\veps\in \units{R}(V_2)$.

\item Choose a covering $\{V_i\to U\}_{i=1,2}$ as in (i) and let $\beta_1=(1+\veps)|_{V_1}$, $\beta_2=(1-\veps)|_{V_2}$.
Since $\veps^\lambda = \veps^{-1}$, we have $\beta_1^\lambda\veps=\beta_1$, and hence
$\beta_1^{-1}\beta_1^\lambda=\veps$. Likewise, $\beta_2^{-1}\beta_2^\lambda=-\veps$. \qedhere
\end{enumerate}
\end{proof}

The following lemma is known when $\pi:\bfX\to \bfY$ is unramified or trivial, i.e., when $R$ is a quadratic \'etale $S$-algebra or
$R=S$. The ramified situation that we consider appears not to have been considered before in the literature.

\begin{lem}\label{LM:hermitina-local-congruence} 
Let
$U\in \bfY$, let $\veps\in N(U)$, and let $h,h'\in \GL_n(R)(U)$ be two $(\veps,\lambda\tr)$-hermitian matrices,
i.e., $h=\veps h^{\lambda \tr}$ and $h'=\veps h'^{\lambda\tr}$.
Assume $\units{S}$ has square roots locally
and that $2\in\units{S}$, or $\pi$ is unramified, or $n$ is odd. Then
there exists a covering $V\to U$ and $v\in \GL_n(R)(V)$ such that $v^{\lambda\tr}hv=h'$ in $\GL_n(R)(V)$.
\end{lem}

\begin{proof} Suppose first that $2\in\units{S}$. Let $\{V_i\to U\}_{i=1,2}$ and $\beta_1,\beta_2$ be as in
Lemma~\ref{LM:one-plus-unitary}\ref{item:LM:one-plus-unitary:beta}. 
We may replace $h,h',U$ with $(\beta_i h,\beta_i h',V_i)_{i=1,2}$ and
assume that $\veps\in\{\pm 1\}$ henceforth.

We claim that it is enough to show that for all $\frakp\in\Spec S(U)$, there are
$s_{1},\dots,s_{\ell}\in \units{S(U)_\frakp}$, $f_{1},\dots,f_{r}\in S(U)_\frakp[\sqrt{s_1},\dots,\sqrt{s_\ell}]$ and
$v_j\in\GL_n(R(U)_\frakp[\sqrt{s_1},\dots,\sqrt{s_\ell}]_{f_j})$ ($j=1,\dots,r$) such that $f_1,\dots,f_r$ generate the
unit ideal in $S(U)_\frakp[\sqrt{s_1},\dots,\sqrt{s_\ell}]$ and $v_j^{\lambda\tr}hv_j=h'$. If this holds, then
Lemma~\ref{LM:point-to-neighbourhood} implies that for all $\frakp$, we can find $b_\frakp\in S(U)- \frakp$ such
that the previous condition holds upon replacing $S(U)_\frakp,R(U)_\frakp$ with $S(U)_{b_\frakp},R(U)_{b_\frakp}$.
Since $\sum_\frakp b_\frakp S(U)=S(U)$ and $S$ is a local ring object, there is a covering $\{V_\frakp\to U\}_\frakp$
such that $b_\frakp\in\units{S(V_\frakp)}$. Fix some $\frakp$ and let
$b_\frakp,s_1,\dots,s_\ell,f_1,\dots,f_r,v_1,\dots,v_r$ be as above. By construction, the images of $s_1,\dots,s_\ell$
in $S(V_\frakp)$ are invertible in $S(V_\frakp)$. Since $\units{S}$ has square roots locally, we can replace $V_\frakp$
with a suitable covering such that $s_1,\dots,s_\ell$ have square roots in $S(V_\frakp)$. In particular,
$S(U)\to S(V_\frakp)$ factors through $S(U)\to S(U)_{b_\frakp}[\sqrt{s_1},\dots,\sqrt{s_{\ell}}]$. Applying the fact
that $S$ is a local ring object again, we see that there is a covering $\{V_{\frakp,i}\to V_\frakp\}_{i=1}^r$ such that
the image of $f_i$ in ${S(V_{\frakp,i})}$ is invertible. It follows that $R(U)\to R(V_{\frakp,i})$ factors through
$R(U)_{b_\frakp}[\sqrt{s_1},\dots,s_{s_\ell}]_{f_i}$ and hence there is $v_{\frakp,i}\in \GL_n(R(V_{\frakp,i}))$ --- the
image of $v_i$ --- such that $v_{\frakp,i}^{\lambda\tr}hv_{\frakp,i}=h'$. Finally, let
$V=\bigsqcup_{\frakp,i} V_{\frakp,i}$ and take
$v=(v_{\frakp,i})_{\frakp,i}\in \GL_n(R(V))=\prod_{\frakp,i}\GL_n(R(V_{\frakp,i}))$.

Let $\frakp\in \Spec S(U)$. We now prove the existence of $s_1,\dots,s_\ell,f_1,\dots,f_r,v_1,\dots,v_r$ above. Write
$B=S(U)_\frakp$ and $A=R(U)_\frakp$. Then $B$ is local and it is the fixed ring of $\lambda:A\to A$. We write
$\quo{A}=A/\Jac(A)$ and let $\quo{\lambda}:\quo{A}\to \quo{A}$ denote the involution induced by $\lambda$.

Suppose $\veps=1$ or $\quo{\lambda}\neq\id$. By Lemma~\ref{LM:decomposition-I}, we may assume that
$h$ and $h'$ are diagonal, say $h=\mathrm{diag}(\alpha_1,\dots,\alpha_n)$ and $h'\in
\mathrm{diag}(\alpha'_1,\dots,\alpha'_n)$. Since $h$ and $h'$ are $(\veps,\lambda\tr)$-hermitian, we have
$(\alpha_i^{-1}\alpha'_i)^\lambda=\alpha_i^{-1}\veps\alpha'_i\veps^{-1}=\alpha_i^{-1}\alpha_i'$,
hence $\alpha_i^{-1}\alpha'_i\in \units{B}$ for all $i$. Writing $s_i=\alpha_i^{-1}\alpha'_1$ and
$v=\mathrm{diag}(\sqrt{s_1},\dots,\sqrt{s_n})\in \GL_n(A[\sqrt{s_1},\dots,\sqrt{s_n}])$, we have $v^{\lambda\tr}hv=h'$, as required
(take $f_1=1$).

Suppose now that $\veps=-1$ and $\quo{\lambda}=\id$. Applying Lemma~\ref{LM:decomposition-II} to $h$ and $h'$, we may
assume that $h$ and $h'$ are direct sums of $2\times 2$ matrices in $[\begin{smallmatrix} 0 & 1 \\ -1 &
0 \end{smallmatrix}]+\nMat{\Jac(A)}{2\times 2}$,
say $h=h_1\oplus\dots\oplus h_m$, $h'=h_{m+1}\oplus\dots\oplus h_{n}$ with $m=n/2$. Applying Lemma~\ref{LM:tough-lemma}
to $h_j$, we obtain $s_j,t_j\in \units{B}$, $f_{j1},\dots,f_{jr_j}\in B[\sqrt{s_j},\sqrt{t_j}]$, and
$v_{ji}\in\GL_2(A[\sqrt{s_j},\sqrt{t_j}]_{f_{ji}})$ such that $v_{ji}^{\lambda\tr} h_j v_{ji}= [\begin{smallmatrix} 0 &
1 \\ -1 & 0 \end{smallmatrix}]$.

Let $B'=B[\sqrt{s_1},\sqrt{t_1},\dots,\sqrt{s_n},\sqrt{t_n}]$, $A'=A[\sqrt{s_1},\sqrt{t_1},\dots,\sqrt{s_n},\sqrt{t_n}]$
and regard $\{f_{ji}\}_{j,i}$ as elements of $B'$. For every tuple $I=(i_1,\dots,i_n)\in\prod_j\{1,\dots,r_j\}$, let
$f_I=\prod_j f_{ji_j}$, $v_I=v_{1i_1}\oplus\dots\oplus v_{mi_m}$ and $v'_I=v_{(m+1)i_{m+1}}\oplus\dots\oplus v_{ni_n}$,
where $v_I,v'_I$ are regarded as elements of $\GL_n(A'_{f_I})$. Then $\sum_I f_IB'=\prod_j(\sum_{i=1}^{r_j}
f_{ji}B')=B'$ and $(v_I v'^{-1}_I)^{\lambda\tr}h(v_I v'^{-1}_I)=h'$,
which is what we want. This establishes the lemma
when $2\in\units{S}$.

To prove the remaining cases, we observe that when $\pi$ is unramified, or $n$ is odd, the use of
Lemmas~\ref{LM:decomposition-II}, \ref{LM:tough-lemma} and~\ref{LM:one-plus-unitary} can be avoided, and hence the
assumption $2\in\units{S}$ is unnecessary.

When $\pi$ is unramified, we apply Proposition~\ref{PR:T-structure-unramified-pi} instead of
Lemma~\ref{LM:one-plus-unitary}(ii) and assume $\veps=1$ hereafter. Since in this case $A$ is a quadratic \'etale
$B$-algebra, Proposition~\ref{PR:Hilbert-ninety}\ref{item:LM:Hilbert-ninety:structure} implies that $\quo{\lambda}\neq\id$,
and so the case $\quo{\lambda}=\id$ does not occur.

Suppose now that $n$ is odd, say $n=2m+1$. Taking the determinant of both sides of $h=\veps h^{\lambda\tr} $ yields
$\det h = \veps^{2m+1} (\det h)^\lambda$. Since $\veps^\lambda=\veps^{-1}$, this implies that
$\veps=\beta^{-1}\beta^\lambda$ for $\beta=\veps^{m}(\det h)^\lambda$. Replacing $h,h'$ with $\beta h,\beta h'$, we may
assume $\veps=1$. Since $n$ is odd, we can now apply Lemma~\ref{LM:decomposition-I} even when $\quo{\lambda}=\id$ and
finish the proof without using Lemmas~\ref{LM:decomposition-II}, \ref{LM:tough-lemma} or~\ref{LM:one-plus-unitary}.
\end{proof}

We can finally complete the proof of Theorem~\ref{TH:ct-determines-local-iso}.

\begin{proposition} \label{PR:completion} 
 Let $n$ be a positive integer. Suppose $\units{S}$ has square roots locally
 and at least one of the following holds:
 \begin{enumerate}[label=(\arabic*)]
 \item $2 \in S^\times$.
 \item $\pi: \mathbf X \to \mathbf Y$ is unramified.
 \item $n$ is odd.
 \end{enumerate}
Let $(A, \tau)$ and
 $(A',\tau')$ be two degree-$n$ Azumaya $R$-algebras with $\lambda$-involutions having the same coarse type. Then
 $(A,\tau)$ and $(A', \tau')$ are locally isomorphic as $R$-algebras with involution.
\end{proposition}
\begin{proof}
 Following the construction of $\ct(A,\tau)$ in \ref{subsec:coarse-types}, define $U$, $\psi$, $\sigma$, $g$, $h$ and
$\veps=h^{-\lambda\tr}h\in N(U)$ so that $\quo{\veps}\in T(U)$ induces $\ct(A,\tau)\in\Hoh^0(T)$. Repeating the
construction with $(A',\tau')$ in place of $(A,\tau)$, we define $U',\psi',\sigma',g',h',\veps'$ analogously. By
refining both $U$ and $U'$, we may assume $U=U'$.

Since $\ct(A,\tau)=\ct(A',\tau')$, $\veps$ and $\veps'$ determine the same section in $T(U)$. Thus, there exists a
covering $V\to U$ and $\beta\in R(V)$ such that $\veps'=\beta^{-1}\beta^\lambda\veps$. Replacing $U$ with $V$, and $h'$
with $\beta h'$, we may assume that $\veps'=\veps$.

Now, by Lemma~\ref{LM:hermitina-local-congruence}, there exists a covering $V\to U$ and $v\in \GL_n(R(V))$ such that
$v^{\lambda\tr} h v=h'$. Again,
replace $U$ with $V$. 
Letting $u$ denote the image of $v$ in $\PGL_n(R)(U)$, we deduce $g u=u^{-\lambda\tr} g'$.
Unfolding the construction in \ref{subsec:coarse-types}, one finds that $\tau_U=\psi^{-1}\circ \sigma \circ \psi=
\psi^{-1}\circ \lambda\tr \circ g\circ \psi$, and likewise $\tau'_U=\psi'^{-1}\circ \lambda\tr \circ g'\circ \psi'$.
Let $\theta:=\psi^{-1}\circ u\circ \psi':A'_U\to A_U$. Then $\theta$ is an isomorphism of $R$-algebras, and since
$gu=u^{-\lambda\tr}g'$, we have
\begin{align*} \theta\circ \tau'_U &=\psi^{-1}\circ u\circ \psi'\circ \psi'^{-1}\circ \lambda\tr\circ g' \circ \psi'\\
&=\psi^{-1} \circ \lambda \tr \circ u^{-\lambda\tr}g' \circ \psi'\\ &=\psi^{-1}\circ \lambda\tr \circ gu\circ \psi'\\
&=\psi^{-1}\circ \lambda\tr \circ g\circ \psi\circ \psi^{-1}\circ u\circ \psi' =\tau_U\circ \theta\ .
\end{align*} Thus, $\theta$ defines an isomorphism of algebras with involution $(A'_U,\tau'_U)\xrightarrow{\sim}
(A_U,\tau_U)$.
\end{proof}

\subsection{Determining types in specific cases} 
\label{subsec:quotient-scheme}

Under mild assumptions,
Theorem~\ref{TH:ct-determines-local-iso} provides a cohomological criterion
to determine whether two $\lambda$-involutions of Azumaya algebras
have the same type, and Corollary~\ref{CR:coarse-type-determines-type} embeds
the possible $\lambda$-types in $\CTypes{\lambda}=\Hoh^0(T)$. 
We finish this section by making this criterion and the realization of the types
even more explicit in case the exact
quotient $\pi:\bfX\to \bfY$ is induced by
a $C_2$-quotient of schemes or topological spaces.

\begin{notation}\label{NT:types-in-specific-cases}
	Throughout, we assume one of the following:
	\begin{enumerate}[label=(\arabic*)]
		\item $X$ is a scheme on which $2$ is invertible,
		$\lambda:X\to X$ is an involution and $\pi:X\to Y$
		is a good quotient relative to $C_2=\{1,\lambda\}$, see
		Example~\ref{EX:important-exact-quo-scheme}.
		\item $X$ is a Hausdorff topological space,
		$\lambda:X\to X$ is a continuous involution,
		and 
		$\pi:X\to Y=X/\{1,\lambda\}$ is the quotient map.
	\end{enumerate}
	We will usually treat both cases simultaneously,
	but when there is need to distinguish them,
	we shall address them as the scheme-theoretic case and the topological
	case, respectively.
	
	In the scheme-theoretic case, the terms sheaf, cohomology
	and covering should be understood as \'etale sheaf, \'etale cohomology
	and \'etale covering, whereas in the topological case,
	they retain their ordinary meaning relative to the relevant topological space.
	Furthermore, in the topological case, $\calO_X$ stands
	for $\cont(X,\C)$, the sheaf of continuous functions into $\C$, 
	and likewise for all topological spaces.
	
	As in Subsection~\ref{subsec:coarse-types},
	write $S=\calO_Y$ and $R=\pi_*\calO_X$,
	and define $N$ to be the kernel of the $\lambda$-norm
	$x\mapsto x^\lambda x:\units{R}\to \units{S}$
	and $T$ to be the cokernel of $x\mapsto x^{-1}x^\lambda:\units{R}\to N$.
	By means of Theorem~\ref{TH:important-exact-quotients},
	the results of the previous subsections can
	be applied, essentially verbatim, to
	Azumaya $\calO_X$-algebras with $\lambda$-involution.

	Recall from Propositions~\ref{PR:ramification-in-schemes} and~\ref{PR:ramification-top} that there is a maximal open subscheme, resp.\ subset, $U\subseteq Y$ such that
$\pi_U:\pi^{-1}(U)\to U$ is unramified, i.e.\ a quadratic \'etale morphism
or a double covering of topological spaces.
We write
\[ W=Y- U \qquad \text{and} \qquad Z=\pi^{-1}(W)\ .\]
Then $W$ and $Z$ are the branch locus and
the ramification
locus of $\pi:X\to Y$, respectively. We endow $Z$ and $W$ with the
subspace topologies. In the scheme-theoretic case, we further endow them with the reduced induced closed subscheme
structures in $X$ and $Y$, respectively.
Recall from Proposition~\ref{PR:Z-to-W-homeomorphism} and the preceding
comment that $\pi$ induces an isomorphism of schemes, resp.\ topological
spaces, $Z\to W$, and $\lambda$ restricts to the identity map on $Z$.
In particular, $W=Z/C_2$.

	Recall that $\mu_{2,\calO_W}$ denotes the sheaf of square roots of $1$
	in $\calO_W$; we abbreviate this sheaf as $\mu_{2,W}$. Since $2$ is invertible
	in $\calO_W$, the sheaf $\mu_{2,W}$
	is just the constant sheaf $\{\pm 1\}$. Similar notation applies to $Z$.
\end{notation}

Let $(A,\tau)$ be an Azumaya $\calO_X$-algebra with $\lambda$-involution and let $z\in Z$ be a point of the
the ramification locus. Propositions~\ref{PR:ramification-in-schemes} and~\ref{PR:ramification-top} imply that
$\lambda(z)=z$ and the specialization of $\lambda$ to $k(z)$, denoted $\lambda_{k(z)}$, is the
identity. 
Thus, the specialization of $(A,\tau)$ to $k(z)$, denoted $(A_{k(z)},\tau_{k(z)})$, is a central simple
$k(z)$-algebra with an involution of the first kind.

\begin{lemma} \label{lem:defftau}
 With the above notation, the function $f_\tau: Z \to \{ 1, -1\}$ determined by 
 \begin{equation*}
 f_\tau(z) = \begin{cases} \phantom{-}1 \quad \text{ if $\tau_{k(z)}$ is orthogonal} \\ -1 \quad \text{ if $\tau_{k(z)}$ is symplectic} \end{cases}
 \end{equation*}
 is locally constant, and therefore determines a global section $f_\tau \in \Hoh^0(Z, \mu_{2,Z})$.
\end{lemma}

\begin{proof}
	We may assume that $\deg A$ is constant; otherwise, decompose $X$ into
	a disjoint union of components
	on which this holds and work component-wise.

	Let $(A_Z,\tau_Z)$ denote the base change of $(A,\tau)$ from $X$ to $Z$,
	namely $(i^*A,i^*\tau)\otimes_{i^*\calO_X}(\calO_Z, \lambda_Z^\#=\id_{\calO_Z})$, where $i:Z\to X$
	denotes the inclusion map.
 For any point $z\in Z$, the type of $\tau$ at $k(z)$ may be calculated 
 relative to $(A_Z, \tau_Z)$.
 We may therefore replace $X$ and $(A,\tau)$ with $Z$ and $(A_Z,\tau_Z)$
 to assume that $Z =X$ and $\lambda:X\to X$ is the trivial involution.

	Let $A_+=\ker(\id_A-\tau)$ and $A_-=\ker(\id_A+\tau)$.
	That is, $A_+$ and $A_-$ are the sheaves and $\tau$-symmetric
	and $\tau$-antisymmetric elements in $A$.
	Since $\lambda=\id$, both $A_+$ and $A_-$
	are $\calO_X$-modules, and since $2$ is invertible
	on $X$, the sequence $0\to A_-\hookrightarrow A\xrightarrow{\id+\tau} A_+\to 0$
	is split exact. Consequently, the sequence remains exact after
	base changing to $k(z)$ for all $z\in X$, and so we
	may identify $(A_+)_{k(z)}$ with the $\tau_{k(z)}$-symmetric
	elements of $A_{k(z)}$.
	
	It is well known \cite[\S2A]{knus_book_1998-1} that $\dim (A_+)_{k(z)}$ 
	equals $\frac{1}{2}n(n+1)$ when $\tau_{k(z)}$ is orthogonal
	and $\frac{1}{2}n(n-1)$ when $\tau_{k(z)}$ is symplectic.
	Since $A_+$ is an $\calO_X$-summand of $A$, it is locally free.
	Thus, the rank of $A_+$ is locally constant, and {\it a fortiori} so is $f_\tau$.
\end{proof}

We will prove, after a number of lemmas, that the element $f_\tau$ determines the type of $\tau$. In the
course of the proof, we shall see that the sheaf $T$ introduced in Subsection~\ref{subsec:coarse-types} is nothing but
the pushforward to $Y$ of the sheaf $\mu_{2,W}$ on $W$.

\begin{lemma}\label{LM:morphism-between-quotients}
	Consider a commutative diagram
	\[
	\xymatrix{
	X' \ar[r]^u \ar[d]^{\pi'} &
	X \ar[d]^\pi \\
	Y' \ar[r]^v &
	Y	
	}
	\]
	in which $\pi:X'\to Y'$ is a 
	good $C_2$-quotient of schemes, resp.\
	a $C_2$-quotient of Hausdorff topological spaces,
	and $u$ is $C_2$-equivariant.
	Let $\lambda'$ denote
	the involution of $X'$ and let $S',R',N',T'$ denote the sheaves corresponding to $S,R,N,T$
	and constructed with $\pi':X'\to Y'$ in place of $\pi:X\to Y$.
	Then:
	\begin{enumerate}[label=(\roman*)]
	\item There are commutative squares of ring sheaves on $Y$
	and $Y'$, respectively:
	\[
	\xymatrix{
	v_*R' \ar@{=}[r] &
	\pi_*u_*\calO_{X'} &
	R \ar[l]_-{\pi_*u_\#} \\
	v_*S' \ar[u]_{v_*\pi'_\#} &
	&
	S \ar[ll]_{v_\#} \ar[u]_{\pi_\#} 
	}
	\qquad
 \xymatrix{
 R' &
 v^*R \ar[l] \\
 S' \ar[u]_{\pi'_\#} &
 v^*S \ar[l]_{v^\#} \ar[u]_{v^*\pi_\#}
 }
	\]
	Here, the horizontal arrows of the right square are the adjoints
	of the horizontal arrows of the left square relative to the adjuntion between
	$v^*$ and $v_*$. Furthermore, in both squares,
	the top horizontal arrows are morphism of rings with involution.
	
	\item
	The left square of (i) induces
	morphisms
	$N\to v_*N'$, $T\to v_*T'$ and $\Hoh^0(T)\to\Hoh^0(T')$.
	Furthermore, if $(A,\tau)$ is an Azumaya $\calO_X$-algebra
	with a $\lambda$-involution and $(A',\tau')$ denotes
	the base change of $(A,\tau)$ to $X'$, namely
	$(u^*A,u^*\tau)\otimes_{u^*\calO_X}(\calO_{X'}, \lambda')$,
	then the image of $\ct_\pi(\tau)\in \Hoh^0(T)$ in $\Hoh^0(T')$
	is $\ct_{\pi'}(\tau')$.
	\end{enumerate}
\end{lemma}

\begin{proof}
	Part (i) and the first sentence of (ii)
	are straightforward from the definitions.
	We turn to prove the last statement of (ii).

	We first claim that
	\begin{equation}\label{EQ:base-change-w-inv}
	(\pi'_*A',\pi'_*\tau')\cong (v^*\pi_*A,v^*\pi_*\tau)\otimes_{v^*R}(R',\lambda')\ .
	\end{equation}
	To see this, observe that the relevant counit maps induce a ring homomorphism
	\begin{align*}
	\pi'^* (v^*\pi_*A \otimes_{v^*R} R') 
	&= \pi'^*v^*\pi_*A \otimes_{\pi'^*v^*\pi_*\calO_X} \pi'^*\pi'_*\calO_{X'} \\
	&=u^*\pi^*\pi_*A \otimes_{u^*\pi^*\pi_*\calO_X} \pi'^*\pi'_*\calO_{X'} 
	\to u^*A\otimes_{u^*\calO_X}\calO_{X'}=A'
	\end{align*}
	which respects the relevant involutions. This morphism is adjoint to a morphism
	\begin{equation}\label{EQ:base-change}
	v^*\pi_*A \otimes_{v^*R} R'\to \pi'_*A'
	\end{equation}
	which we claim to be the desired isomorphism.
	This is easy to see when $A=\nMat{\calO_X}{n\times n}$.
	In general, by Theorem~\ref{TH:pi-star-equiv},
	there exists a covering $U\to Y$ such that $A$ becomes a matrix
	algebra after pulling back to $X_U$. Thus, $(v^*\pi_*A \otimes_{v^*R} R')_{Y'_U}\to (\pi'_*A')_{Y'_U}$
	is an isomorphism, and we conclude that so does \eqref{EQ:base-change}.
	
	With \eqref{EQ:base-change-w-inv} at hand,
	let $U,\psi,\sigma,g,h,\veps$ be as in
	Construction~\ref{CN:coarse-type}, applied to $(A,\tau)$. 
	We may assume that $U$ is represented
	by a covering of $Y$, denoted $U\to Y$. 
	Let $U'\to Y'$ be the pullback of $U\to Y$ along $v:Y'\to Y$,
	which corresponds to the sheaf $v^*U$ in $\bfY'$.
	Let $\psi'=v^*\psi\otimes_{v^*R_U}\id_{R'_{U'}}$
	and let $\sigma'=\psi'\tau'\psi'^{-1}=v^*\sigma\otimes_{v^*R}\lambda'$.
	The right square of (i) induces canonical maps 	
	$v^*\PGL_n(R)=\PGL_n(v^*R)\to \PGL_n(R')$,
	$v^*\GL_n(R)=\GL_n(v^*R)\to \GL_n(R')$ and $v^*N\to N'$ (notice that $v^*$ is exact).
	Let $g'$ be the image of $v^*g\in v^*\PGL_n(R)(U')$ in $\PGL_n(R')(U')$,
	and define $h'\in \GL_n(R')(U')$ and $\veps'\in N'(U')$ similarly.
	It is easy to check that we can apply Construction~\ref{CN:coarse-type}
	to $(A',\tau')$ using $U',\psi',\sigma',g',h',\veps'$.
	Consequently, the image of $\veps'$ in $T'(U')$
	agrees with the image of $v^*\veps$, which is exactly what we need
	to prove.
\end{proof}

Endowing $Z$ with the trivial involution, we can apply Lemma~\ref{LM:morphism-between-quotients}
with the square
	\begin{equation}\label{EQ:ramification-square}
	\xymatrix{
	Z \ar@{^(->}[r]^i \ar[d]^{\pi'} &
	X \ar[d]^\pi \\
	W \ar@{^(->}[r]^j &
	Y	
	}
	\end{equation}
where $\pi'$ is the restriction of $\pi$ to $Z$.
By Example~\ref{EX:T-tot-ramified}, the sheaf $T'$ is just $\mu_{2,W}$
and hence Lemma~\ref{LM:morphism-between-quotients}(ii) gives
rise to a morphism
\[
\Psi: T\to j_*\mu_{2,W} \ .
\]

\begin{lem}\label{LM:structure-of-T-interesting-cases}
$\Psi: T\to j_*\mu_{2,W}$ is an isomorphism of abelian sheaves on $Y$.
\end{lem}

\begin{proof}
	To show that $\Psi$ is an isomorphism, it is enough to check the stalks.
	The topos-theoretic points of $\bfY$ 
	are recalled in the proofs of Corollaries~\ref{CR:finite-etale-morphism}
	and~\ref{CR:finite-top-morphism}; they are in correspondence
	with the set-theoretic points of $Y$.
	
	Let $p:\mathbf{pt}\to \bfY$ be a point, corresponding to $y\in Y$.
	Since $p^*$ is exact, $p^*N$ is the kernel
	of $x\mapsto x^\lambda x:\units{p^*R}\to \units{p^*R}$
	and $p^*T$ is the cokernel
	of $x\mapsto x^{-1}x^\lambda:\units{p^*R}\to p^*N$.
	
	Suppose that $y\notin W$. Then, since $j:W\to Y$
	is a closed embedding, $p^*j_*\mu_{2,W}=0$.
	On the other hand, since $\pi$ is unramified at $y$, it is unramified
	at a neighborhood of $y$ and hence $p^*T=0$ by Proposition~\ref{PR:T-structure-unramified-pi}.
	Thus, $p^*\Psi:p^*T\to p^*j_*\mu_{2,W}$
	is an isomorphism.
	
	Suppose henceforth that $y\in W$. Then $\pi$ is ramified at $y$.
	We claim that $p^*R$ is local and $\lambda$ induces the identity
	map on its residue field. This is evident from the definitions
	in the topological case, see Proposition~\ref{PR:ramification-top}.
	In the scheme-theoertic case, this follows from condition~\ref{item:PR:ram:stalk}
	in Proposition~\ref{PR:ramification-in-schemes} and Theorem~\ref{thm:local-involution}
	after noting that $\Spec p^*R=X\times_Y \Spec \calO_{Y,y}^{\mathrm{sh}}$.

	Now, we have $p^*j_*\mu_{2,W}=\{\pm 1\}$.
	With the notation of 	Lemma~\ref{LM:morphism-between-quotients},
	applied to the square \eqref{EQ:ramification-square}, 
	the morphism $N\to j_*N'$
	is just a restriction of the morphism 
	$R\to j_*R'=j_*\calO_W$. 
	This implies that the images of $-1,1\in p^*N$ in $p^*T$
	are mapped under $p^*\Psi$
	to $-1,1\in p^*j_*\mu_{2,W}$, respectively, so $p^*\Psi$ is surjective.
	
	To finish, we show that $p^*T$ consists of at most $2$ elements.
	Every $t \in p^*T$ is
	represented by some $\veps\in p^*N$.
	Since $2\in \units{p^*R}$ and $p^*R$ is local,
	either $1+\veps$ or $1-\veps$ is invertible.
	Suppose $\beta:=1+\veps\in \units{p^*R}$.
	Since $\veps^\lambda=\veps^{-1}$, we have
	$\veps\beta^\lambda=\beta$, or rather, $\veps=(\beta^{-1})^\lambda\beta$,
	which implies $t=\quo{1}$.
	Similarly, when $1-\veps\in\units{p^*R}$, we find that ${t}=\quo{-1}$.
	It follows that $p^*T=\{\quo{1},\quo{-1}\}$ and the proof is complete.	
\end{proof}

We finally prove the main result of this subsection.

\begin{theorem}\label{TH:types-for-schemes}
	With Notation~\ref{NT:types-in-specific-cases},
	Let $(A,\tau)$ and $(A',\tau')$ be Azumaya $\calO_X$-algebras
	with $\lambda$-involutions, and let $f_\tau,f_{\tau'}\in \Hoh^0(Z,\mu_{2,Z})$
	be as in Lemma~\ref{lem:defftau}. Then:
	\begin{enumerate}[label=(\roman*)]
		\item $\tau$ and $\tau'$ have the same type if and only if $f_\tau=f'_\tau$.
		\item 
		There is a group isomorphism $\Phi: \Hoh^0(T)\to \Hoh^0(Z,\mu_{2,Z})$
		such that $\Phi(\ct_\pi(\tau))=f_\tau$ for all $(A,\tau)$.
	\end{enumerate}
\end{theorem}

\begin{remark}
	We do not know whether every $f\in \Hoh^0(Z,\mu_{2,Z})$ arises as $f_\tau$
	for some $(A,\tau)$, see Remark~\ref{RM:are-all-coarse-types-realizable}.
\end{remark}

\begin{proof}
	By Theorem~\ref{TH:ct-determines-local-iso} and Example~\ref{EX:SHL_implies_SRL}, in order prove (i), it is enough to
 	prove that $\ct_\pi(\tau)=\ct_\pi(\tau')$, and this follows if we prove (ii).
 	
 	Apply Lemma~\ref{LM:morphism-between-quotients}
 	and its notation to the square \eqref{EQ:ramification-square}.
 	The lemma gives rise to a morphism of sheaves
 	$T\to j_*T'=j_*\mu_{2,W}$, which is an isomorphism
 	by Lemma~\ref{LM:structure-of-T-interesting-cases}. This in turn induces
 	an isomorphism
 	\begin{equation*} \label{EQ:H-zero-T-iso} 
 	\Hoh^0(Y,T)\to \Hoh^0(Y,j_*T')=\Hoh^0(W,T')\ ,
 	\end{equation*} 
 	such that $\ct_{\pi}(A,\tau)$
 	is mapped to $\ct_{\pi'}(A_Z,\tau_Z)$, where $(A_Z,\tau_Z)$
 	denotes the base change of $(A,\tau)$ to $Z$.
 	Since $\pi':Z\to W$ is an isomorphism, 
 	this gives rise to an isomorphism
 	\[
 	\Hoh^0(Y,T)\to \Hoh^0(W,T')=\Hoh^0(W,\mu_{2,W})\cong \Hoh^0(Z,\mu_{2,Z})\ ,
 	\]
 	which we take to be $\Phi$.
 	It remains to show that $\Phi(\ct_\pi(A,\tau))=f_\tau$.
 	
 	Since $f_{\tau_Z}=f_\tau$, and since the image of $\ct_{\pi}(A,\tau)$
 	in $\Hoh^0(T')$ is $\ct_{\pi'}(A_Z,\tau_Z)$,
 	it is enough to show that the image of $ \ct_{\pi'}(A_Z,\tau_Z)\in \Hoh^0(W,T') $
 	in $\Hoh^0(Z,\mu_{2,Z})$
 	is $f_{\tau_Z}$. 
 	To this end, we replace $\pi:X\to Y$ and $(A,\tau)$
 	with $\pi':Z\to W$ and $(A_Z,\tau_Z)$.
 	Now, $\lambda$ is the trivial involution and we may assume that $Y=X$
 	and $\pi$ is the identity map.
 	The map $\Hoh^0(W,T')\to \Hoh^0(Z,\mu_{2,Z})$ is just
 	the identity
 	map $\Hoh^0(X,\mu_{2,X})\to \Hoh^0(X,\mu_{2,X})$, see Example~\ref{EX:T-tot-ramified},
 	and the proof reduces to showing that $\ct(\tau)=f_\tau$.

 	Let $x\in X$, and let $U,\sigma,h,\veps$ be as in Construction~\ref{CN:coarse-type},
 	applied to $(A,\tau)$.
 	We may assume that the sheaf $U$ in $\bfY=\bfX$
 	is represented by a covering $U\to X$.
 	Since $\lambda$ is the trivial involution, $\veps\in N(U)=\mu_{2,X}(U)$,
 	and since
 	$\mu_{2,X}$ is the constant sheaf $\{\pm 1\}$ on $X$,
 	there is a covering $U_1\sqcup U_{-1}\to U$ 
 	such that $\veps|_{U_{-1}}=-1$ and $\veps|_{U_{1}}=1$.
 	
	There is $c\in \{\pm 1\}$ and $u\in U_c$ such that $x$
	is the image of $u$ under $U_c\to X$. It is immediate from
	the definition of $t:=\ct(\tau)$ that $t(x)=c$.	
	Let $k(u)$ denote
	the residue field of $u$.
	By construction, $(A_{U_c},\tau_{U_c})\cong( \nMat{\calO_X}{n\times n},\sigma)$,
	where $\sigma$ is given section-wise by $x\mapsto (hx h^{-1})^\tr$
	and $c h^{\tr}= h$. Thus, $\tau_{U_c}$ is orthogonal when $c=1$
	and symplectic when $c=-1$.
	The same applies to $\tau_{k(u)}:A_{k(u)}\to A_{k(u)}$.
	Since $(A_{k(u)},\tau_{k(u)})=(A_{k(x)},\tau_{k(x)})\otimes_{k(x)}(k(u),\id)$,
	it follows that $\tau_{k(x)}:A_{k(x)}\to A_{k(x)}$
	is orthogonal when $t(x)=1$ and symplectic when $t(x)=-1$.
	This means $t=f_\tau$, so we are done.
\end{proof}

\section{Brauer Classes Supporting an Involution}
\label{sec:Saltman}

\subsection{Introduction}

Let $K$ be a field and let $\lambda:K\to K$ be an involution
with fixed field $F$.
The central simple $K$-algebras
admitting a $\lambda$-involution were characterized 
by Albert, Riehm and Scharlau, see for instance \cite[Thm.~3.1]{knus_book_1998-1}, who proved:

\begin{thm*}
	Let $A$ be a central simple $K$-algebra. Then:
	\begin{enumerate}[label=(\roman*)]
	\item (Albert) When $\lambda=\id$, $A$ admits a $\lambda$-involution
	if and only if $2[A]=0$ in $\Br(K)$.
	\item (Albert--Riehm--Scharlau) 
	When $\lambda\neq \id$, $A$ admits a $\lambda$-involution
	if and only if $[\cores_{K/F}(A)]=0$ in $\Br(F)$. 
	\end{enumerate}
\end{thm*}

Here, $\cores_{K/F}(A)$ is the \emph{corestriction algebra} of $A$,
whose definition we recall below.

\medskip

The Albert--Riehm--Scharlau Theorem does not, in general, hold if we replace $K$ with an arbitrary ring. However, in
\cite{saltman_azumaya_1978}, Saltman showed that the Brauer classes admitting a representative with a
$\lambda$-involution can still be characterized similarly.

\begin{thm*}[Saltman {\cite[Th.~3.1]{saltman_azumaya_1978}}] 
 Let $R$ be a ring, let $\lambda:R\to R$ be an involution and let $S$ be the fixed ring of $\lambda$. Let $A$ be an
 Azumaya $R$-algebra. Then:
 \begin{enumerate}[label=(\roman*)]
 \item When $\lambda=\id$, there exists $A'\in [A]$ such that $A'$ admits a
 $\lambda$-involution if and only if $2[A]=0$ in $\Br(R)$.
 \item When $R$ is quadratic \'etale over $S$,
 there exists $A'\in [A]$ such that $A'$ admits a
 $\lambda$-involution if and only if $[\cores_{R/S}(A)]=0$ in $\Br(S)$. 
 \end{enumerate}
\end{thm*}

A later proof by Knus, Parimala and Srinivas \cite[Thms.~4.1, 4.2]{knus_azumaya_1990} 
applies in the generality of schemes and also implies that the representative
$A'$ can be chosen such that $\deg A'=2\deg A$.

\medskip

In this section, we extend Saltman's theorem 
to locally ringed topoi with involution.
We note that our generalization implies in particular that Salman's theorem
applies to topological Azumaya algebras. Furthermore,
while Saltman's theorem assumes that $\lambda=\id$, or $R$ is quadratic \'etale over
the fixed ring of $\lambda$,
our result will apply without any restriction on the involution. 
Finally, we also characterize the possible types, or more precisely,
coarse types, of the involutions of the various representatives $A'\in [A]$.

\begin{notation} \label{NOT:BCSI}
 Throughout this section,
let 
$\bfX$ 
be a locally ringed topos with ring object $\calO_\bfX$
and involution $\lambda=(\Lambda,
\nu,\lambda)$, and let 
$\pi:\bfX\to \bfY$
be an exact quotient relative to $\lambda$, see~\ref{subsec:quotients}. 
Recall that such quotients arise, for instance, from $C_2$-quotients
of schemes or Hausdorff topological spaces as explained in Examples~\ref{EX:important-exact-quo-scheme}
and~\ref{EX:important-exact-quo-top}. In such cases, we shall
work with the original schemes, resp.\ topological spaces, denoted $X$ and $Y$, rather than the
associated ringed topoi.

As in Section~\ref{sec:types}, we write $S=\calO_\bfY$ and $R=\pi_*\calO_\bfX$.

We sometimes omit bases when evaluating
cohomology; the base will always be clear from the context.
If $A$ is an abelian group in $\bfX$, we shall freely
identify $\Hoh^i(\mathbf X, A)$, written $\Hoh^i(A)$, with $\Hoh^i(\mathbf Y, \pi_* A)$, written $\Hoh^i(\pi_*A)$, using Theorem~\ref{TH:vanishing}.
\end{notation}

\subsection{The Cohomological Transfer Map}
\label{subsec:coh-transfer}

The corestriction map considered in the aforementioned
theorems of Albert--Riehm--Scharlau and Saltman
is a special case of the cohomological
transfer map, which will feature in our generalization
of Saltman's theorem.

\begin{definition} \label{DEF:coho-cores} The \textit{cohomological $\lambda$-transfer} map
 $\transf_\lambda:\Hoh^2(\mathbf X, \units{\sh O_{\mathbf X}}) \to \Hoh^2(\mathbf Y, \units{\calO_\bfY})$ is the composite of
 the isomorphism $\Hoh^2(\bfX,\units{\calO_\bfX})\xrightarrow{\sim} \Hoh^2(\bfY,\units{\pi_*\calO_\bfX})$ 
 induced by $\pi_*$, see
 Theorem~\ref{TH:vanishing}, and the morphism $\Hoh^2(\bfY,\units{\pi_*\calO_\bfX})\to \Hoh^2(\bfY,\units{\calO_\bfY})$ induced by the $\lambda$-norm
 map $x\mapsto x^\lambda x:\units{\pi_*\calO_\bfX}\to \units{\calO_\bfY}$. When no confusion can arise, we
 shall omit $\lambda$, simply writing $\transf$ for $\transf_\lambda$, and calling it the transfer map.
\end{definition}

\begin{example} \label{EX:mult_by_two} 
If the involution 
$\lambda$ of 
$\bfX$ is weakly trivial and
$\pi:\bfX\to \bfY$
is the trivial quotient,
 see Example~\ref{EX:trivial-exact-quotient}, then 
	the $\lambda$-norm is the
 squaring map $x\mapsto x^2:\units{\pi_*\calO_\bfX}\to \units{\pi_*\calO_\bfX}=\units{\calO_\bfY}$,
 and so
 $\transf_\lambda : \Hoh^2(\sh O_\bfX^\times) \to \Hoh^2(\sh O_\bfY^\times)\cong 
 \Hoh^2(\units{\calO_\bfX})$
 is multiplication by $2$. 
\end{example}

\begin{example}\label{EX:transfer-unramified-schemes}
	Let $\pi:X\to Y$ be a quadratic \'etale morphism of schemes, and let 
	$\lambda:X\to X$ be the canonical $Y$-involution of $X$, given
	section-wise by $x^\lambda=\mathrm{Tr}_{X/Y}(x)-x$.
	We consider the exact quotient obtained from $\pi$ and $\lambda$
	by taking \'etale ringed topoi, see Example~\ref{EX:important-exact-quo-scheme}.
 In this case,
 the transfer map
 $\transf_{\lambda}:\Hoh^2_{\et}(X,\units{\calO_X})\to \Hoh^2_{\et}(Y,\units{\calO_Y})$ is, by definition, the
 corestriction map $\cores_{X/Y}:\Hoh^2_{\et}(X,\units{\calO_X})\to \Hoh^2_{\et}(Y,\units{\calO_Y})$.
	Moreover, $\cores_{X/Y}$ restricts to a map $\cores_{X/Y}:\Br(X)\to \Br(Y)$ which can be described explicitly on
 the level of Azumaya algebras: Let $A$ be an Azumaya $\calO_X$-algebra. The corestriction algebra
 $\cores_{X/Y}(A)$ is an Azumaya $\calO_Y$-algebra defined as the $\calO_Y$-subalgebra of
 $\pi_*(A \otimes_{\calO_X}\lambda^*A)$ fixed by the exchange automorphism, given by $x\otimes y\mapsto y\otimes x$ on
 sections. The map $\cores_{X/Y}:\Br(X)\to \Br(Y)$ is then given by $[A]\mapsto [\cores_{X/Y}(A)]$, see
 \cite[p.~68]{knus_azumaya_1990} (the diagram on that page contains a misprint, on the right column, both `$S$'s should
 be `$R$'s).
\end{example}

\begin{remark}\label{RM:no-corestriction}
 In contrast to the situation in Examples~\ref{EX:mult_by_two}
 and~\ref{EX:transfer-unramified-schemes}, we do not know whether
 \[\transf_\lambda:\Hoh^2(\mathbf X, \units{\sh O_{\mathbf X}}) \to \Hoh^2(\mathbf Y, \units{\calO_\bfY}) \] restricts to a map
 between the Brauer groups $\Br(\bfX,\calO_\bfX)\to \Br(\bfY,\calO_\bfY)$, even in the cases 
 induced by a good $C_2$-quotient of schemes $\pi:X\to Y$. 
	Some
 positive results appear in \cite[Lem.~5.1, Rmk.~5.2]{auel_parimala_suresh_2015}. 
 Also, when $R$ is locally free of rank $2$ over $S$,
 Ferrand \cite{ferrand_1998_norm_functors} constructs a universal norm functor taking $R$-algebras
 to $S$-algebras, which coincides with $\cores_{R/S}$
 when $R$ is quadratic \'etale over $S$, but it is {\it a priori}
 not clear whether it takes Azumaya $R$-algebras to Azumaya $S$-algebras in general.
 We hope to address this problem in a subsequent work.
	
 We further note that without assuming that $\pi$ is unramified, the construction of 
 Example~\ref{EX:transfer-unramified-schemes} may produce an algebra which is not Azumaya. For example, 
 it can be checked directly that $\cores_{R/S}(\nMat{R}{2\times 2})$ is
 not Azumaya over $S$ when $S=\C$, $R=\C[x]/(x^2)$, and $\lambda:R\to R$ is the $\C$-involution
 taking $x$ to $-x$.
\end{remark}

\begin{example} \label{EX:topological-free-transfer}
 In the case where $X$ is a Hausdorff topological space with a free $C_2$-action and $\pi: X \to Y:=X/C_2$ is the 
 corresponding $2$-sheeted covering,
 the construction \[\transf: \Hoh^2(X ,S^1)\iso \Hoh^2(X , \sh O_{ X}^\times) \to \Hoh^2(Y, \sh O_{ Y}^\times)
 \iso \Hoh^2(Y ,S^1)\]
 is a special case of the usual transfer map for a $2$-sheeted cover. This can be proved by considering $\transf$ on the level
 of $2$-cocycles. See also \cite[Sec.~3.3]{piacenza_transfer_1984} and note that $\pi^*$ takes
 $\units{\sh O_Y}$, the sheaf of nonvanishing continuous complex-valued functions on $Y$, to $\sh O_X$ on
 $\mathbf X$.
\end{example}

\begin{remark}
 There is a notion of transfer for ramified covers $X \to X/G$ where $G$ is a finite group, in particular, when
 $G=C_2$. This may be found in \cite{aguilar_transfer_2010}. It seems likely, that map $\transf_{\lambda}$ given here
 is a special case of that construction, but we do not pursue this further.
\end{remark}

\subsection{Brauer Classes Supporting a $\lambda$-Involution}
\label{subsec:Saltman-ramified}

In this subsection, we characterize
those Brauer classes in $\Br(\bfX, \calO_\bfX)$ %, or equivalently $\Br(R)$, 
admitting a representative with a $\lambda$-involution, 
thus generalizing Saltman's theorem \cite[Th.~3.1]{saltman_azumaya_1978}.

We remind the reader that the notational conventions of Notation~\ref{NOT:BCSI} are still in effect. In particular,
$S:=\calO_\bfY$ is a local ring object in $\bfY$ and $R:=\pi_*\calO_\bfX$ is a commutative $S$-algebra with involution $\lambda$ such that
the fixed ring of $\lambda$ is $S$. 

\medskip

As in Subsection~\ref{subsec:coarse-types},
we define $N$ to be the kernel of the $\lambda$-norm $x\mapsto x^\lambda x:\units{R}\to \units{S}$ 
and let
$T$ be the quotient of $N$ by the image of the map $x \mapsto x^{\lambda}x^{-1}: R^\times \to N$. 
Recall that $\CTypes{\lambda}:=\Hoh^0(T)$ is the group of coarse $\lambda$-types
and there is a map $(A,\tau)\mapsto \ct_{\pi}(A,\tau)\in \Hoh^0(T)$ associating an Azumaya
$\calO_\bfX$-algebra with a $\lambda$-involution to its coarse type, see Subsection~\ref{subsec:coarse-types}.

The short exact sequence 
$ 1 \to R^\times/S^\times \xrightarrow{x \mapsto x^{\lambda}x^{-1}} N \to T \to 1 $ induces the connecting homomorphism
\[ \delta^0:\Hoh^0(T) \to \Hoh^1(\units{R}/\units{S}) \] 
and the short exact sequence 
$ 1 \to \units S \to \units R \to \units R / \units S \to 1 $
induces a connecting homomorphism
\[ \delta^1:\Hoh^1(\units{R}/\units{S}) \to \Hoh^2(\units S). \]
\begin{notation}\label{NT:Phi}
 We denote the composite morphism $\delta^1\circ\delta^0$ by $\Phi$,
\[ \Phi:\CTypes{\lambda}= \Hoh^0(T) 
\to \Hoh^2(\units S). \]
\end{notation}

\begin{proposition} \label{PR:Phi-trivial} The map $\Phi$ is the $0$-map in the following cases:
\begin{enumerate}[label=(\roman*)] 
\item \label{pr:phitriv-1} 
When $\pi:\bfX\to \bfY$ is a trivial quotient (Example~\ref{EX:trivial-exact-quotient}), i.e.\ $R=S$.
\item \label{pr:phitriv-4} When $\pi$ is everywhere ramified (Definition~\ref{DF:ramification}),
$2\in\units{S}$ and $\units{S}$ has square roots locally. 
\item \label{pr:phitriv-2} When $\pi$ is unramified (Definition~\ref{DF:ramification}), i.e.\ $R$ is a quadratic \'etale
$S$-algebra.
\item \label{pr:phitriv-3} When $\pi:X\to Y$ is a good $C_2$-quotient of schemes, 
$Y$ is noetherian and regular, and $\pi$ is unramified at the generic points of $Y$;
the corresponding exact quotient is obtained
by taking \'etale ringed topoi as in Example~\ref{EX:important-exact-quo-scheme}.
\end{enumerate}
\end{proposition}

\begin{proof}
 \begin{enumerate}[label=(\roman*), align=left, leftmargin=0pt, itemindent=1ex, itemsep=0.3\baselineskip]
 \item 
	In this case, $\units{R}/\units{S}$ is trivial. 
 Since $\Phi$ factors through $\Hoh^1( \units R/\units S) = 0$, the result follows.
 
\item We claim that squaring induces an automorphism of $\units{R}/\units{S}$, and hence of the group
 $\Hoh^1(\units{R}/\units{S})$. Since $\Hoh^0(T)$ is a $2$-torsion group (Proposition~\ref{PR:T-is-two-torsion}), this
 forces $\delta^0:\Hoh^0(T)\to \Hoh^1(\units{R}/\units{S})$ to vanish, implying $\Phi$ vanishes as well.
 
 We show the surjectivity of $x\mapsto x^2:\units{R}/\units{S}\to \units{R}/\units{S}$ by checking that $\units{R}$ has
 square roots locally. Let $U$ be an object of $\bfY$ and $r\in \units{R}(U)$.
 Since $\units{S}$ has square roots locally, there a covering $V\to U$ and $s\in \units{S}(V)$ such that
 $r^\lambda r=s^2$. Replacing $r$ with $rs^{-1}$ and $U$ with $V$, we may assume $r^\lambda r=1$. Now, by
 Lemma~\ref{LM:one-plus-unitary}, there is a covering $\{V_i\to U\}_{i=1,2}$ and $\beta_i\in \units{R}(V_i)$ such that
 $r=\beta_1^{-1}\beta_1^\lambda$ in $\units{R}(V_1)$ and $r=-\beta_2^{-1}\beta_2^\lambda$ in $\units{R}(V_2)$. We may
 refine $V_2\to V$ to assume that there is $a\in\units{S}(V_2)$ such that $-\beta_2^\lambda \beta_2=a^2$ and get
 $r=a^2\beta_2^{-2}$. Similarly, we refine $V_1$ to find a square root of $r$ in $\units{R}(V_1)$ and conclude that $r$
 has a square root on $V_1\sqcup V_2$.
 
 Next, let $K$ denote the kernel of $x\mapsto x^2:\units{R}/\units{S}\to \units{R}/\units{S}$. A section of $K$ is
 represented by some $a\in \units{R}(U)$ such that $a^2\in \units{S}(U)$, or rather, $a^2=(a^\lambda)^2$. 
 Since $(a-a^\lambda)^2+(a+a^\lambda)^2=4a^2\in\units{S}(U)$ and $(a-a^\lambda)^2,(a+a^\lambda)^2\in S(U)$, 
 and since $S$
 is a local ring object, there is a covering $\{U_i\to U\}_{i=1,2}$ such that $a-a^\lambda\in \units{R}(U_1)$ and
 $a+a^\lambda \in \units{R}(U_2)$. By virtue of Lemma~\ref{LM:etale-criterion}, $R_{U_1}$ is a quadratic \'etale over $S_{U_1}$,
 so our assumption that $\pi$ is everywhere ramified forces $U_1=\emptyset$. Thus, $U_2\to U$ is a covering, implying that
 $a+a^\lambda$ is invertible in $R(U)$. Since $(a+a^\lambda)(a-a^\lambda)=a^2-(a^\lambda)^2=0$, we must have
 $a-a^\lambda=0$, so $a\in \units{S}(U)$. It follows that $a$ represents the $1$-section in $\units{R}/\units{S}$, and
 thus $K=0$.

 \item In this case, a version of Hilbert's Theorem 90 applies in the form of Proposition~\ref{PR:T-structure-unramified-pi}, and
 $\Hoh^0(\mathbf Y, T) = 0$. \textit{A fortiori}, $\Phi$ is $0$.
 
 \item We may assume that $Y$ is connected and therefore integral, otherwise we may work component by component.
 
 Let $\xi:\Spec K\to Y$ denote the generic point of $Y$. Since
 $\xi$ is flat, $\pi_\xi:X_\xi\to \Spec K$ is a good
 $C_2$-quotient relative to the action induced by $\lambda$;
 denote the sheaves corresponding to $S,R,N,T$ by $S',R',N',T'$.
 We now apply Lemma~\ref{LM:morphism-between-quotients} to
 the square
 \[
 \xymatrix{
 X_\xi \ar[r]^{\xi_\pi} \ar[d]^{\pi_{\xi}}&
 X \ar[d]^\pi \\
 \Spec K \ar[r]^{\xi} &
 Y
 }
 \]
 which gives rise to maps $\xi^*N\to N'$, $\xi^*T\to T'$,
 adjoint to the maps in the lemma.
 The exactness of $\xi^*$ together with the natural homomorphism
 $\Hoh^*(Y,-)\to \Hoh^*(K,\xi^*(-))$ now give
 rise to a commutative diagram
 \begin{equation*}
 \xymatrix{ 
 \Hoh_\et^0( Y, T) \ar[d] \ar^\Phi[r] & \Hoh^2_\et(Y, \units{\calO_Y}) \ar[d] \\
 \Hoh^0_\et(K, T') \ar^{\Phi_\xi
 }[r] & \Hoh^2_\et(K, \units{\calO_{\Spec K}}) }
 \end{equation*}
 By \cite[Cor.~1.8]{grothendieck_groupe_1968}, the right vertical morphism
 is injective (here we need $Y$ to be regular),
 and
 by \ref{pr:phitriv-2}, $\Phi_\xi=0$. Therefore, $\Phi=0$.
	\qedhere
 \end{enumerate}
\end{proof}

We are now ready to state our generalization of Saltman's theorem. Whereas Saltman's original proof
\cite{saltman_azumaya_1978} and the later proof by Knus, Parimala and Srinivas \cite{knus_azumaya_1990} make use of the
corestriction of an Azumaya algebra, we cannot employ this construction, as demonstrated in
Remark~\ref{RM:no-corestriction}. Rather, our proof is purely cohomological, phrased in the language set in
Subsections~\ref{subsec:coh-of-ab-grps} and~\ref{subsec:coh-of-non-ab-grp}. We remind the reader of our standing
assumption from Remark~\ref{RM:lambda-stable-degree} that the degrees of all Azumaya $\calO_\bfX$-algebras considered
are fixed under $\Lambda$, which is automatic when $\bfX$ is connected.

\begin{theorem}\label{TH:Saltman-ramified} 
 Let 
 $\bfX$
 be a locally ringed topos with involution $\lambda$, let 
 $\pi:\bfX\to \bfY$ be an
 exact quotient relative to $\lambda$, 
 and consider the map ${\transf}_\lambda:\Br(\bfX,\calO_\bfX)\to\Hoh^2(\bfY,\units{\calO_\bfY})$ of
 Definition~\ref{DEF:coho-cores} 
 and the map $\Phi:\CTypes{\lambda}=\Hoh^0(T)\to \Hoh^2(\bfY,\units{\calO_\bfY})$
 of Notation~\ref{NT:Phi}. Let $A$ be an Azumaya $\calO_\bfX$-algebra of degree $n$, and let $t\in \Hoh^0(T)$. Then
 there exists $A'\in [A]$ admitting a $\lambda$-involution of coarse type $t $ if and only if $\transf([A]) = \Phi(t)$
 in $\Hoh^2(\bfY,\units{\calO_\bfY})$. The algebra $A'$ can be chosen such that $\deg A'=2n$.
\end{theorem}

We recover Saltman's original theorem \cite[Th.~3.1]{saltman_azumaya_1978} and the improvement of Knus, Parimala and
Sinivas \cite[Thms.~4.1, 4.2]{knus_azumaya_1990} from Theorem~\ref{TH:Saltman-ramified} by taking
$\pi:\bfX\to \bfY$
to be the exact quotient associated to a good $C_2$-quotient of schemes
$\pi:X\to Y$ such that $\pi$ is an isomorphism or quadratic \'etale, see
Examples~\ref{EX:important-exact-quo-scheme}. In this case, $\Phi=0$ by Proposition~\ref{PR:Phi-trivial}, and the
transfer map coincides with multiplication by $2$ when $\pi=\id$, or with the corestriction map when $\pi$ is quadratic
\'etale, as demonstrated in Examples~\ref{EX:mult_by_two} and~\ref{EX:transfer-unramified-schemes}.

The relation between the type and the coarse type of an involution, as well as the question of when two involutions of
the same type are locally isomorphic, had been studied extensively in Subsections~\ref{subsec:coarse-types}
and~\ref{subsec:quotient-scheme}.

\begin{proof}
Thanks to Theorems~\ref{TH:vanishing} and~\ref{TH:pi-star-equiv}, we may replace
$A$ with $\pi_*A$ and work with $R$-algebras,
rather than $\calO_\bfX$-algebras. We abuse the notation and 
denote the map $\Hoh^2(\units{R})\to \Hoh^2(\units{S})$ induced by
$x\mapsto x^\lambda x:\units{R}\to\units{S}$ as $\transf_\lambda$.

\medskip

Suppose first that there exists $[A']\in A$ admitting a $\lambda$-involution $\tau$ of coarse type $t$. 
We may replace $A$ with $A'$.
We now invoke all the notation of Construction~\ref{CN:coarse-type} and 
the proof of Lemma~\ref{LM:coarse-type-descends} through which $t$ is constructed from $(A,\tau)$.
Specifically: 
\begin{itemize} 
\item $U\to *_\bfY$ is a covering such that there exists an isomorphism of $R_U$-algebras
$\psi:A_U\to \nMat{R_U}{n\times n}$,
\item
$\sigma:=\psi\circ\tau_U\circ \psi^{-1}$ is an involution of $\nMat{R_U}{n\times n}$,
\item
$g:=\lambda\tr\circ \sigma$ is an element of $\PGL_n(R)(U)$,
\item
$h\in \GL_n(R)(U)$ is a lift of $g$ (refine $U$ if necessary),
\item 
$\veps:=h^{-\lambda\tr}h$ is an element of $N(U)$, embedded diagonally in $\GL_n(R)(U)$,
\item
$U_\bullet$ is the \v{C}ech hypercovering corresponding to $U\to *$, see Example~\ref{EX:Cech-hypercovering},
\item 
$a=\psi_1\circ\psi_0^{-1}\in \PGL_n(R)(U_1)$, where $\psi_i$ is the pullback of $\psi$
along $d_i:U_1\to U_0$,
\item $b$ is a lift of $a$ to $\GL_n(R)(V)$, where $V\to U_1$ is some covering,
\item $\beta:=b^{-\lambda\tr} \cdot d_0^*h \cdot b^{-1} \cdot d_1^*h^{-1}$ is an element $\units{R}(V)$,
embedded diagonally in $\GL_n(R)(V)$,
\item $t=\ct(\tau)$ is the image of $\veps$ in $T(U)$; it descends
to a global section of $T$ since $d_1^*\veps \cdot d_0^*\veps^{-1}=\beta^{-1}\beta^\lambda$.
\end{itemize}
By Lemma~\ref{LM:hypercover-refinement}, there is a hypercovering
morphism $V_\bullet\to U_\bullet$ such that $V_1\to U_1$ factors
through $V\to U_1$. We replace $V$ with $V_1$.

Recall from Theorem~\ref{TH:pi-star-equiv} that $A$ corresponds to a $\PGL_n(R)$-torsor,
which in turn corresponds to a cohomology class in $\Hoh^1(\PGL_n(R))$.
We claim that $a$ is a $1$-cocycle in $Z^1(U_\bullet,\PGL_n(R))$ which represents
this cohomology class. Indeed, the $\PGL_n(R)$-torsor corresponding to $a$ 
is $P:=\sAut_R(\nMat{R}{n\times n},A)$, and $\psi^{-1}\in P(U)=P(U_0)$
by construction. By the isomorphism given in the proof of 
Proposition~\ref{PR:non-ab-coh-basic-prop}\ref{item:PR:non-ab-coh-basic-prop:torsor},
the cohomology class corresponding to $P$ is represented
by $ d_1^*(\psi^{-1})^{-1} \cdot d_0^*(\psi^{-1})=\psi_1\circ \psi_0^{-1}=a$.

Consider the short exact sequence
$1 \to \units{R}\to \GL_n(R)\to \PGL_n(R) \to 1$ and its associated
$7$-term cohomology exact sequence, see Proposition~\ref{PR:non-ab-coh-basic-prop}\ref{item:PR:non-ab-coh-basic-prop:longer-ex-seq}.
It follows from the definition of 
$\delta^2:\Hoh^1(\PGL_n(R))\to \Hoh^2(\units{R})$,
see the proof of Proposition~\ref{PR:non-ab-coh-basic-prop}\ref{item:PR:non-ab-coh-basic-prop:longer-ex-seq}, 
that $[A]=\delta^2(a)\in \Hoh^2(\units{R})$ 
is represented
by 
\begin{equation}\label{EQ:alpha-definition}
\alpha := d_2^*b \cdot d_0^*b \cdot d_1^*b^{-1}\in Z^2(V_\bullet,\units{R})\ .
\end{equation}
and thus, $\transf([A])$ is represented by $\alpha^\lambda \alpha \in Z^2(V_\bullet, \units{S})$.

On the other hand, by the definition of $\delta^0:\Hoh^0(T)\to \Hoh^1(\units{R}/\units{S})$,
see the beginning of this subsection and the end of~\ref{subsec:coh-of-ab-grps},
$\delta^0(t)$ is represented by the image of $\beta^{-1}\in \units{R}(V_1)$
in $(\units{R}/\units{S})(V_1)$, since $d_0^*\veps\cdot d_1^*\veps^{-1}=(\beta^{-1})^\lambda\beta$.
Likewise,
by the definition of $\delta^1:\Hoh^1(\units{R}/\units{S})\to \Hoh^2(\units{S})$,
the class $\Phi(t)=\delta^1\delta^0(t)$ is represented
by $d_0^*\beta^{-1}\cdot d_1^*\beta\cdot d_2^*\beta^{-1} \in Z^2(V_\bullet,\units{S})$.

In order to show that $\transf([A])=\Phi(t)$, we check
that $\alpha^\lambda \alpha=d_0^*\beta^{-1}\cdot d_1^*\beta\cdot d_2^*\beta^{-1}$ in
$\units{S}(V_2)$.
For the computation, 
we shall make use of $d_0^*d_0^*=d_1^*d_0^*$, $d_0^*d_1^*=d_2^*d_0^*$,
$d_1^*d_1^*=d_2^*d_1^*$ and the fact that if $xyz$ is central in a group $G$,
then $xyz=zxy=yzx$.
\begin{align*}
	d_0^*\beta^{-1}\cdot d_1^*\beta\cdot d_2^*\beta^{-1}
	&=d_0^*\beta^{-1}(d_1^*b^{-\lambda\tr} \cdot d_1^*d_0^*h\cdot d_1^*b^{-1}\cdot d_1^*d_1^*h^{-1})\\
	&\phantom{=}\quad\cdot 
	(d_2^*d_1^*h\cdot d_2^*b\cdot d_2^*d_0^*h^{-1}\cdot d_2^*b^{\lambda\tr})\\
	&=d_1^*b^{-\lambda\tr} \cdot d_1^*d_0^*h\cdot d_1^*b^{-1}\cdot
	d_2^*b\cdot d_2^*d_0^*h^{-1}\cdot(d_0^*\beta^{-1})\cdot d_2^*b^{\lambda\tr}\\
	&=d_1^*b^{-\lambda\tr} \cdot d_1^*d_0^*h\cdot d_1^*b^{-1}\cdot
	d_2^*b\cdot d_2^*d_0^*h^{-1} \\
	&\phantom{=}\quad \cdot (d_0^*d_1^*h\cdot d_0^*b\cdot d_0^*d_0^*h^{-1}\cdot d_0^*b^{\lambda\tr})
	\cdot d_2^*b^{\lambda\tr}\\
	&=d_1^*b^{-\lambda\tr} \cdot d_1^*d_0^*h\cdot (d_1^*b^{-1}\cdot
	d_2^*b\cdot d_0^*b)\cdot d_0^*d_0^*h^{-1}\cdot d_0^*b^{\lambda\tr}
	\cdot d_2^*b^{\lambda\tr}\\
	&=d_1^*b^{-\lambda\tr} \cdot d_1^*d_0^*h\cdot \alpha \cdot d_0^*d_0^*h^{-1}\cdot d_0^*b^{\lambda\tr}
	\cdot d_2^*b^{\lambda\tr}\\
	&= (d_2^*b \cdot d_0^*b \cdot d_1^*b^{-1})^{\lambda\tr} \cdot \alpha \\
	&=\alpha^\lambda \alpha
\end{align*}
This completes the proof of the ``only if'' statement.

\medskip

Suppose now that $\transf([A])=\Phi(t)$.
Define $U\to *$, $U_\bullet$, $a$, $b$, $V_\bullet$ and $\alpha$ as before.
Using Lemma~\ref{LM:hypercover-refinement} twice, we can refine $V_\bullet$
to assume that $t$ lifts to some $\veps\in N(V_0)$
and there is $\beta\in \units{R}(V_1)$
such that
\begin{equation}\label{EQ:first-cond}
d_0^*\veps\cdot d_1^*\veps^{-1}=(\beta^{-1})^{\lambda}\beta 
\end{equation}
in $N(V_1)$.
As explained above, $\transf([A])$ is represented
by $\alpha^\lambda\alpha\in Z^2(V_\bullet,\units{S})$
and $ \Phi(t)$ is represented
by $d_0^*\beta^{-1}\cdot d_1^*\beta\cdot d_2^*\beta^{-1}$.
The assumption $\Phi(t)=\transf([A])$ therefore means that,
after refining $V_\bullet$,
there exists $\gamma\in \units{S}(V_1)$
such that 
\[
d_0^*\gamma\cdot d_1^*\gamma^{-1}\cdot d_2^*\gamma\cdot d_0^*\beta^{-1}\cdot d_1^*\beta\cdot d_2^*\beta^{-1}=\alpha^\lambda \alpha\ .
\]
We replace $\beta$ with $\beta\gamma^{-1}\in\units{R}(V_1)$,
which does not affect \eqref{EQ:first-cond}
and allows us to assume
\begin{equation}\label{EQ:second-cond} 
d_0^*\beta^{-1}\cdot d_1^*\beta\cdot d_2^*\beta^{-1}=\alpha^\lambda \alpha\ .
\end{equation}

Writing in block form, define the $2n\times 2n$ matrices
 \[ h=\begin{bmatrix} 0 & 1 \\ \veps & 0
 \end{bmatrix}\in \GL_{2n}(R)(V_0) \qquad \text{and} \qquad b'=\begin{bmatrix} b & 0 \\ 0 &
 \beta^{-1} b^{-\lambda\tr}
 \end{bmatrix}\in \GL_{2n}(R)(V_1)
 \] 
and let $\sigma : \nMat{R_{V_0}}{2n\times 2n}\to \nMat{R_{V_0}}{2n\times 2n}$
be the involution given by $x\mapsto (hxh^{-1})^{\lambda\tr} =h^{-\lambda\tr}x h^{\lambda\tr} $ on
sections. 
Also, let $a'$ be the image of $b'$ in $\PGL_{2n}(R)(V_1)$, namely, $a'\in \PGL_{2n}(R)(V_1)$
is the automorphism of $\nMat{R_{V_1}}{2n\times 2n}$ given
by $x\mapsto b' xb'^{-1}$ on sections.

We first observe that $a'\in Z^1(V_\bullet,\PGL_{2n}(R))$.
Indeed, working in $\GL_{2n}(R)(V_2)$ and using \eqref{EQ:alpha-definition}
and \eqref{EQ:second-cond},
we find that
\begin{equation}\label{EQ:why-a-prime-is-cocycle}
	d_2^*b'\cdot d_0^*b'\cdot d_1^*b'^{-1} =
	\begin{bmatrix} \alpha & 0 \\ 0 &
	( d_0^*\beta \cdot d_2^*\beta \cdot d_1^*\beta ^{-1})^{-1} \alpha^{-\lambda\tr}
	\end{bmatrix}
	=
	\begin{bmatrix} \alpha & 0 \\ 0 &
	\alpha
	\end{bmatrix}\in\units{R}(V_1)\ .
\end{equation}
Let $\tilde{V}_\bullet$ denote the \v{C}ech hypercovering associated to $V_0\to *$, see
Example~\ref{EX:Cech-hypercovering}. By Lemma~\ref{LM:Cech-hyp-vs-all-hyp}, $a'\in Z^1(V_\bullet,\PGL_{2n}(R))$ descends
uniquely to a cocycle $\tilde{a}'\in Z^1(\tilde{V}_\bullet,\PGL_{2n}(R))$.\uriyaf{ Maybe the introduction of $\tilde{a}'$ can
 be saved if we verify that morphisms descend along ``hyper'' descent data. This should be known, in fact. } The
\v{C}ech $1$-cocycle $\tilde{a}'$ defines descent data for $\nMat{R_{V_0}}{2n\times 2n}$ along $V_0\to *$, giving rise
to an Azumaya $R$-algebra $A'$ of degree $2n$ and an isomorphism $\psi:A'_{V_0}\to \nMat{R_{V_0}}{2n\times 2n}$ such
that $\tilde{a}'=\psi_1\circ \psi_0^{-1}$, where $\psi_i$ is the pullback of $\psi$ along
$d_i:\tilde{V}_1\to \tilde{V}_0=V_0$. Note that by
construction, $a'$ represents the class in $\Hoh^1(\PGL_{2n}(R))$ corresponding to $A'$, hence
\eqref{EQ:why-a-prime-is-cocycle} implies that $[A']=\alpha=[A]$.

We now claim that $\sigma$ descends to an involution $\tau:A'\to A'$.
Letting $\sigma_i$ denote the pullback of $\sigma$ along $d_i:V_1\to V_0$,
and noting that $(d_0,d_1):V_1\to V_0\times V_0$ is a covering, see Subsection~\ref{subsec:coh-of-ab-grps},
this amounts to showing that $\sigma_1a'=a'\sigma_0$. 
To see this, we first note that \eqref{EQ:first-cond} and $\veps^\lambda \veps=1$ imply
that
\[
b'^{-\lambda\tr} \cdot d_0^*h \cdot b'^{-1}\cdot
d_1^*h^{-1}=
\begin{bmatrix}
\beta & 0 \\ 0 & \beta^\lambda \cdot d_0^*\veps \cdot d_1^*\veps^{-1}
\end{bmatrix}
=
\begin{bmatrix}
\beta & 0 \\ 0 & \beta
\end{bmatrix},
\]
or equivalently,
\[
b'^{-1}\cdot d_1^*h^{-1}=
\beta \cdot d_0^*h^{-1} \cdot b'^{\lambda\tr}. 
\]
Using this, for any section $x$ of $\nMat{R_{V_1}}{2n\times 2n}$, we have
\begin{align*}
\sigma_1(a'(x))&= d_1^*h^{-\lambda\tr}(b'x b'^{-1})^{\lambda\tr} d_1^*h^{\lambda\tr}
=(b'^{-1} d_1^*h^{-1} )^{\lambda\tr} x( b'^{-1} d_1^*h^{-1})^{-\lambda\tr}\\
&=(\beta \cdot d_0^*h^{-1} \cdot b'^{\lambda\tr})^{\lambda\tr}x
(\beta \cdot d_0^*h^{-1} \cdot b'^{\lambda\tr})^{-\lambda\tr}
=a'(\sigma_0(x))\ ,
\end{align*}
which is what we want.

We finish by checking that $\ct(\tau)=t$.
To see this, apply Construction~\ref{CN:coarse-type}
to $(A',\tau)$ using $U:=V_0$, $\psi$, $\sigma$ and $h$ defined above
and note that $h^{-\lambda\tr}h=[\begin{smallmatrix} \veps & 0 \\ 0 & \veps
\end{smallmatrix}]$.
\end{proof}

We now specialize Theorem~\ref{TH:Saltman-ramified}
to Azumaya algebras over schemes and over topological spaces.

It is worth recalling at this point that in the situation of 
a good $C_2$-quotients of schemes $\pi:X\to Y$ such that $2$ is invertible on $Y$
(Example~\ref{EX:important-exact-quo-scheme}),
or a $C_2$-quotient of Hausdorff topological spaces $\pi:X\to Y$
(Examples~\ref{EX:important-exact-quo-top}),
the sheaf $T$ is isomorphic to $i_*\mu_{2,W}$,
where $i:W\to Y$ is the embedding of the branch locus of $\pi$ in $Y$.
Under this isomorphism, the coarse type of an involution $\tau:A\to A$
is the unique global section $f\in \Hoh^0(W,\mu_{2,W})=\cont(W,\{\pm 1\})$ 
such that $f(w)=1$ if $\tau_{k(\pi^{-1}(w))}:A_{k(\pi^{-1}(w))}\to A_{k(\pi^{-1}(w))}$
is orthogonal, and $f(w)=-1$ if $\tau_{k(\pi^{-1}(w))}:A_{k(\pi^{-1}(w))}\to A_{k(\pi^{-1}(w))}$
is symplectic, for all $w\in W$; see Subsection~\ref{subsec:quotient-scheme}.
Furthermore, in these situations, two $\lambda$-involutions of the same coarse type
have the same type, and they are locally isomorphic if the degrees of their underlying Azumaya
algebras agree; this follows from Theorem~\ref{TH:ct-determines-local-iso}
and Corollary~\ref{CR:coarse-type-determines-type}.

\begin{corollary}\label{CR:Saltman-ramified-for-schemes}
	Let $X$ be a scheme, let $\lambda:X\to X$ be an involution
	and let $\pi:X\to Y$ be a good quotient relative to $C_2:=\{1,\lambda\}$.
	Let $A$ be an Azumaya $\calO_X$-algebra
	and let $t\in \CTypes{\lambda}$ be a coarse type.
	Then there exists $A'\in [A]$ admitting a $\lambda$-involution of coarse type
	$t$ if and only if $\Phi(t)=\transf_{\lambda}([A])$ in $\Hoh^2_{\et}(Y,\units{\calO_Y})$.
	The algebra $A'$ can be chosen such that $\deg A'=2\deg A$.
\end{corollary}

\begin{proof}
	This is a special case of Theorem~\ref{TH:Saltman-ramified}, see
	Example~\ref{EX:important-exact-quo-scheme} and Theorem~\ref{TH:important-exact-quotients}.
\end{proof}

\begin{corollary}\label{CR:Saltman-regular-schemes}
	In the situation of Corollary~\ref{CR:Saltman-ramified-for-schemes},
	suppose that
	\begin{enumerate}[label=(\arabic*)]
	\item $\lambda=\id$, or 
	\item $\pi:X\to Y$ is quadratic \'etale, or 
	\item $Y$ is noetherian
	and regular, and $\pi$ is unramified at the generic points of $Y$.
	\end{enumerate}
	Then there exists $A'\in [A]$ admitting a $\lambda$-involution 
	if and only if $\transf_{\lambda}([A])=0$.
	In this case, $A'$ can be chosen to have a $\lambda$-involution of any prescribed coarse type
	(or any prescribed type, when $2$ is invertible on $Y$)
	and to satisfy $\deg A'=2\deg A$.
\end{corollary}

\begin{proof}
	This follows from Corollary~\ref{CR:Saltman-ramified-for-schemes} 
	and Propositions~\ref{PR:Phi-trivial} and~\ref{PR:ramification-in-schemes}.
\end{proof}

\begin{corollary}\label{CR:Saltman-ramified-for-top}
	Let $X$ be a Hausdorff topological space, let $\lambda:X\to X$
	be a continuous involution, and let $\pi$ denote the quotient
	map $X\to Y:=X/\{1,\lambda\}$.
	Let $A$ be an Azumaya $\calO_X$-algebra
	and let $t\in \CTypes{\lambda}$ be a coarse type.
	Then there exists $A'\in [A]$ admitting a $\lambda$-involution of coarse type
	$t$ if and only if $\Phi(t)=\transf_{\lambda}([A])$ in $\Hoh^2_{\et}(Y,\units{\calO_Y})$.
	The algebra $A'$ can be chosen such that $\deg A'=2\deg A$.
\end{corollary}

\begin{proof}
	This is a special case of Theorem~\ref{TH:Saltman-ramified}, see
	Examples~\ref{EX:important-exact-quo-top} and Theorem~\ref{TH:important-exact-quotients}.
\end{proof}

\begin{corollary}
	In the situation of Corollary~\ref{CR:Saltman-ramified-for-top},
	if
	$\lambda=\id$, or $\lambda$ acts freely on $X$, 
	then there exists $A'\in [A]$ admitting a $\lambda$-involution 
	if and only if $\transf_{\lambda}([A])=0$.
	In this case $A'$ can be chosen to have a $\lambda$-involution of any prescribed type
	and to satisfy $\deg A'=2\deg A$.
\end{corollary}

\begin{proof}
	This follows from Corollary~\ref{CR:Saltman-ramified-for-top} 
	and Propositions~\ref{PR:Phi-trivial} and~\ref{PR:ramification-top}.
\end{proof}

\begin{remark}
	Let $R$ be a connected semilocal ring, and let $\lambda:R\to R$
	be an involution with fixed ring 
	$S$.
	When $R=S$
	or $R$ is quadratic \'etale over $S$, it
	was observed by Saltman \cite[Th.~4.4]{saltman_azumaya_1978}
	%and before him by Albert \cite[Chp.~X]{albert_61_structure_of_algs} when $R$ is a field,
	that an Azumaya $R$-algebra $A$ that is Brauer equivalent to an
	algebra with 
	a $\lambda$-involution already possesses a $\lambda$-involution.
	Otherwise said, in this special situation, we can choose $A'=A$
	in Corollary~\ref{CR:Saltman-ramified-for-schemes}.
	
	We do not know whether this statement continues to hold if the assumption that $R=S$ or $R$ is quadratic \'etale
 over $S$ is dropped. In this case, the fact that $R$ is semilocal implies that two Azumaya algebras of the same
 degree are isomorphic \cite{ojanguren_71}. With this in hand, Corollary~\ref{CR:Saltman-ramified-for-schemes}
 implies that if $A$ is equivalent to an Azumaya $R$-algebra admitting a $\lambda$-involution, then
 $\nMat{A}{2\times 2}$ has a $\lambda$-involution. The problem therefore reduces to the question of whether the
 existence of a $\lambda$-involution on $\nMat{A}{2\times 2}$ implies the existence of a $\lambda$-involution on
 $A$. The same question was asked for arbitrary non-commutative semilocal rings $A$ in \cite[\S12]{first_15},
 where it was also shown that counterexamples, if any exist, are restricted. In particular, returning to the
 case of Azumaya algebras, it follows from \cite[Thm.~7.3]{first_15} that if $\deg A$ is even, then $A$ does
 posses a $\lambda$-involution when $\nMat{A}{2\times 2}$ has one.
\end{remark}

\subsection{The Kernel of the Transfer Map}
\label{subsec:ker-of-transfer}

We continue to use $R$, $S$, $N$, $T$ defined in Subsection~\ref{subsec:Saltman-ramified}.

Saltman's theorem can also be regarded as a result characterizing the kernel of the transfer map in terms of existence
of certain involutions. We now use Theorem~\ref{TH:Saltman-ramified} to generalize this particular aspect, namely,
describing the kernel of $\transf_\lambda:\Br(\bfX,\calO_\bfX)\to \Hoh^2(\bfY,\units{\calO_\bfY})$ in terms of the
involutions that the Brauer classes support. For that purpose, we introduce the following families of
$\lambda$-involutions.

\begin{definition}\label{DF:standard-inv}
 Let $A$ be an Azumaya $\calO_\bfX$-algebra. A $\lambda$-involution $\tau:A\to \Lambda A$ is called
 \emph{semiordinary} if there exists a split Azumaya $\calO_\bfX$-algebra $A'$ and a $\lambda$-involution
 $\tau':A'\to \Lambda A'$ such that $(\pi_*A,\pi_*\tau)$ and $(\pi_*A',\pi_*\tau')$ are locally isomorphic. If $A'$
 can moreover be chosen to be $\nMat{\calO_\bfX}{n\times n}$ and there is $a\in \Hoh^0(A')$ such that $\tau'$ is given
by $x\mapsto (axa^{-1})^{\lambda\tr}$ on sections, we say that $\tau$ is \emph{ordinary}.

When $\tau$ is not semiordinary, we shall say it is \emph{extraordinary}.
\end{definition}

\begin{theorem}\label{TH:standard-Azumaya}
	With notation as in Theorem~\ref{TH:Saltman-ramified},
	let $A$ be an Azumaya $\calO_\bfX$-algebra of degree $n$.
	Then the following conditions are equivalent:
	\begin{enumerate}[label=(\alph*)]
		\item \label{item:TH:standard-Azumaya:trnaf} $\transf_\lambda([A])=0$,
		\item \label{item:TH:standard-Azumaya:semi} 
		there exists $A'\in [A]$ admitting a semiordinary $\lambda$-involution,
		\item \label{item:TH:standard-Azumaya:standard}
		there exists $A'\in [A]$ admitting a ordinary $\lambda$-involution.
	\end{enumerate}
	In \ref{item:TH:standard-Azumaya:semi}, the algebra $A'$ can be chosen to satisfy $\deg A'=2\deg A$
	and to have
	a semiordinary involution of any prescribed coarse type in $\ker(\Phi:\Hoh^0(T)\to \Hoh^2(\units{S}))$.
	In \ref{item:TH:standard-Azumaya:standard}, 
	the algebra $A'$ can be chosen to satisfy $\deg A'=2\deg A$
	and to have
	an ordinary 
	involution of any prescribed coarse type in $\im( \Hoh^0(N)\to \Hoh^0(T))$.
\end{theorem}

We shall see below (Corollary~\ref{CR:all-invs-are-standard}) that in the situation of 
a scheme on which $2$ is invertible and a trivial involution, or a quadratic \'etale covering 
of schemes with its canonical involution, 
all $\lambda$-involutions are ordinary.
Thus,
Theorem~\ref{TH:standard-Azumaya} recovers 
Saltman's Theorem when $2$ is invertible. 

More generally, it will turn out that under mild assumptions,
all involutions are ordinary when 
$\pi:\bfX\to \bfY$
is unramified or everywhere ramified.

\begin{proof}
	As in the proof of Theorem~\ref{TH:Saltman-ramified}, we switch
	to Azumaya
	$R$-algebras by applying $\pi_*$.

	(c)$\implies$(b) is clear.
	
	(b)$\implies$(a): Suppose $A'\in [A]$ 
	admits a semiordinary involution $\tau$ and let $(B,\theta)$
	be a split Azumaya $R$-algebra with a $\lambda$-involution that is locally isomorphic
	to $(A',\tau)$.
	Then by
	Theorem~\ref{TH:Saltman-ramified} and Proposition~\ref{PR:type-determines-coarse-type},
	$\transf_\lambda([A])=\Phi(\ct(\tau))=\Phi(\ct(\theta))=
	\transf_\lambda([B])=0$.

	(a)$\implies$(c): Let $t\in \im( \Hoh^0(N)\to \Hoh^0(T))$. Then $t$ is the image of some $\veps\in \Hoh^0(N)$.
 We revisit the proof of the ``if'' part in Theorem~\ref{TH:Saltman-ramified} and apply it with our $t$ and
 $\veps$ to obtain an Azumaya $R$-algebra with involution $(A',\tau)$ such that $A'\in [A]$, $\deg A'=2n$,
 $\ct(\tau)=t$ and $(A'_{V_0},\tau_{V_0})$ is isomorphic to $(\nMat{R_{V_0}}{2n\times 2n},\sigma)$ with $\sigma$
 being given by $x\mapsto ([\begin{smallmatrix} 0 & 1 \\ \veps & 0 \end{smallmatrix}]x [\begin{smallmatrix}
 0 & 1 \\ \veps & 0 \end{smallmatrix}]^{-1})^{\lambda\tr}$ on sections. Since $\veps\in\Hoh^0(N)$, the involution
 $\sigma$ descends to an involution on $\nMat{R}{2n\times 2n}$, defined by the same formula as $\sigma$, hence
 $\tau$ is ordinary.
	
\smallskip	
	
It remains to show that we can choose $A'$ to have a semiordinary involution with a prescribed coarse type $t\in \ker
\Phi$. Let $V\to \pt$ be a covering such that $t$ lifts to some $\veps\in N(V)$. Again, we apply the proof of the
``if'' part of Theorem~\ref{TH:Saltman-ramified} with $t$, $\veps$ and $A$ to obtain an Azumaya $R$-algebra with
involution $(A',\tau )$ satisfying $A'\in [A]$ and $\ct(\tau)=t$. 
We then reapply the proof with
$\nMat{R}{n \times n}$ in place of $A$ to obtain another Azumaya $R$-algebra with involution $(A'_1,\tau_1)$ such that
$A'_1$ is split and $\ct(\tau_1)=t$. By construction, after suitable refinement of $V\to *$, both $(A'_V,\tau_V)$ and
$(A'_{1,V},\tau_{1,V})$ are isomorphic to $(\nMat{R_{V}}{2n\times 2n},\sigma)$, where is $\sigma$ given by $x\mapsto
([\begin{smallmatrix} 0 & 1 \\ \veps & 0 \end{smallmatrix}] x [\begin{smallmatrix} 0 & 1 \\ \veps &
 0 \end{smallmatrix}]^{-1})^{\lambda\tr}$ on sections. Consequently, $(A',\tau)$ and $(A'_1,\tau_1)$ are locally isomorphic
and therefore $\tau$ is semiordinary.
\end{proof}

We now shift our attention from the involutions $\tau$ to the coarse types $t$.

\begin{definition}
	Let $t\in \Hoh^0(T)$ be a coarse $\lambda$-type.
	We say that $t$ is \emph{realizable} if 
	there exists
	some Azumaya $\calO_\bfX$-algebra 
	$A$ and some $\lambda$-involution $\tau:A\to\Lambda A$ with coarse type $t$.
	We also say that $t$ is {realizable} in degree $n$ when $A$ can be chosen so that $n=\deg A$.
	When $\tau$ can be chosen to be ordinary, resp.\ semiordinary,
	we call $t$ \emph{ordinary}, resp.\ \emph{semiordinary}.
\end{definition}

The following theorem characterizes the realizable,
semiordinary, and ordinary coarse types in cohomological terms.

\begin{theorem}\label{TH:standard-types}
	With notation as in Theorem~\ref{TH:Saltman-ramified},
	let $t\in \Hoh^0(T)$ be coarse type and let 
	$\delta^0:\Hoh^0(T)\to\Hoh^1(\units{R}/\units{S})$,
	$\delta^1:\Hoh^1(\units{R}/\units{S})\to \Hoh^2(\units{S})$ and 
	$\Phi=\delta^1\circ\delta^0$ be as in Subsection~\ref{subsec:Saltman-ramified}. Then:
	\begin{enumerate}[label=(\roman*)]
		\item $t$ is realizable if and only if $\Phi(t)\in \im(\transf_\lambda:\Br(\bfX,\calO_\bfX)\to \Hoh^2(\bfY,\units{\calO_\bfY}))$.
		\item \label{item:TH:standard-types:semi} 
		$t$ is semiordinary if and only if $\Phi(t)=\delta^1\delta^0(t)=0$. 
		\item \label{item:TH:standard-types:stan} $t$ is ordinary if and only if $\delta^0(t)=0$,
		or equivalently $t\in\im(\Hoh^0(N)\to \Hoh^0(T))$.
	\end{enumerate}
	When (ii) or (iii) hold, $t$ is realizable in degree $2$, and hence in all even degrees.
\end{theorem}

\begin{proof}
\begin{enumerate}[label=(\roman*), align=left, leftmargin=0pt, itemindent=1ex, itemsep=0.3\baselineskip]
	\item 
	This follows form Theorem~\ref{TH:Saltman-ramified}.
	
	\item 
	The ``only if'' part follows from Theorems~\ref{TH:standard-Azumaya} and~\ref{TH:Saltman-ramified}.
	The ``if'' part and the assertion that $t$
	can be realized in degree $2$ follow by applying Theorem~\ref{TH:standard-Azumaya} with $A=R$.
	
	\item Suppose $t$ is ordinary, say
	$t=\ct(\nMat{R}{n},\tau)$ with $\tau$ given by $x\mapsto (hxh^{-1})^{\lambda\tr}$
	on sections. 
	Then we can apply Construction~\ref{CN:coarse-type}
	by with $U=*$, $\psi=\id$, and $h$ as above, resulting in $\veps\in \Hoh^0(N)$,
	which then maps onto $t\in \Hoh^0(T)$.
	
	The reverse implication follows by applying Theorem~\ref{TH:standard-Azumaya} with $A=R$.
	\qedhere
\end{enumerate}
\end{proof}

\begin{corollary}\label{CR:standard-inv-vs-standard-types}
 With the notation of Theorem~\ref{TH:Saltman-ramified},
 suppose $ \units{\calO_\bfY}$ has square roots locally, and assume further that $2\in\units{\calO_\bfY}$ or $\pi:\bfX\to \bfY$ is
 unramified. Let $(A,\tau)$ be an Azumaya $\calO_\bfX$-algebra with a $\lambda$-involution. Then $\tau$ is ordinary,
 resp.\ semiordinary, if and only if its coarse type is.
\end{corollary}

\begin{proof}
	The ``only if'' part is clear, so we turn to the ``if'' part.
	We replace $(A,\tau)$ with $(\pi_*A,\pi_*\tau)$, 
	see Theorem~\ref{TH:pi-star-equiv} and Corollary~\ref{CR:equiv-Az-with-inv},
	and write $t=\ct(\tau)$.
	In case $\bfY$ is not connected, we express
	$*_\bfY$ as $\bigsqcup_{n\in\N} Y_n$ such that $A_{Y_n}$ has degree 
	$n$, and work with each component separately. We may therefore assume
	that $n:=\deg A$ is constant.
	
	By Theorem~\ref{TH:ct-determines-local-iso}, it is enough to 
	find an Azumaya $R$-algebra $A'$ with an ordinary, resp.\ semiordinary,
	involution $\tau'$ such that $\deg A=\deg A'$ and $\ct(\tau)=\ct(\tau')$.
	If $n$ is odd, then $t=1$ by
	Theorem~\ref{TH:basic-properties-of-types}\ref{item:TH:basic-properties-of-types:odd},
	and we can take $(A',\tau')=(\nMat{R}{n\times n},\lambda\tr)$.
	Otherwise, $n=2m$, and applying Theorem~\ref{TH:standard-Azumaya}
	to $\nMat{R}{m\times m}$ yields an algebra with an ordinary, resp.\ semiordinary, involution
	$(A',\tau')$ such that $\ct(\tau')=\ct(\tau)$ and $\deg A'=\deg A$; 
	here we used parts (ii) and (iii) of Theorem~\ref{TH:standard-types}. 
\end{proof}

\begin{corollary}\label{CR:extreme-cases-standard-types}
	With the notation of Theorem~\ref{TH:Saltman-ramified}, suppose that 
	\begin{enumerate}[label=(\arabic*)] 
	\item $\pi:\bfX\to \bfY$ is a trivial quotient (Example~\ref{EX:trivial-exact-quotient}), or
	\item $\pi:\bfX\to \bfY$ is everywhere
	ramified, $2\in\units{\calO_\bfY}$ and $\units{\calO_\bfY}$ has square roots locally, or
	\item $\pi:\bfX\to \bfY$ is unramified.
	\end{enumerate}
	Then all coarse $\lambda$-types are realizable and ordinary.
\end{corollary}

\begin{proof}
	It follows from the proof of Proposition~\ref{PR:Phi-trivial}
	that in all three cases, $\delta^0:\Hoh^0(T)\to\Hoh^0(\units{R}/\units{S})$
	is the $0$ map. Now apply Theorem~\ref{TH:standard-types}\ref{item:TH:standard-types:stan}.
\end{proof}

\begin{corollary}\label{CR:all-invs-are-standard}
	With the notation of Theorem~\ref{TH:Saltman-ramified},
	suppose that $\units{\calO_\bfY}$ has square roots locally and moreover
	\begin{enumerate}[label=(\arabic*)] 
	\item $\pi:\bfX\to\bfY$ is everywhere
	ramified and $2\in\units{\calO_\bfY}$, or
	\item $\pi:\bfX\to \bfY$ is unramified.
	\end{enumerate}
	Then all $\lambda$-involutions are ordinary.
\end{corollary}

\begin{proof}
	This follows from Corollaries~\ref{CR:standard-inv-vs-standard-types}
	and~\ref{CR:extreme-cases-standard-types}.
\end{proof}

\begin{corollary}\label{CR:smooth-scheme-standard}
	Let $X$ be a scheme, let $\lambda:X\to X$ be an involution,
	and let $\pi:X\to Y$ be a good quotient relative to $\{1,\lambda\}$.
	Assume $Y$ is noetherian and regular, and $\pi$ is quadratic \'etale
	on the generic points of $Y$. 
	Then all coarse $\lambda$-types are realizable and semiordinary.
	If moreover $2$ is invertible on $Y$, then all $\lambda$-involutions are semiordinary.
\end{corollary}

\begin{proof}
	The first assertion follows from Proposition~\ref{PR:Phi-trivial}\ref{pr:phitriv-3}
	and Theorem~\ref{TH:standard-types}\ref{item:TH:standard-types:semi}.
	The second assertion then follows from Corollary~\ref{CR:standard-inv-vs-standard-types}.
\end{proof}

We conclude this section with two problems, both of which are open both in the context of varieties over fields of
characteristic different from $2$ with (ramified) involutions and in the context of topological spaces with (non-free)
$C_2$-actions.

\begin{problem}
 Is there an element $t \in \CTypes{\lambda}=\Hoh^0(T)$ that is not the coarse type of any Azumaya algebra with
 $\lambda$-involution?
\end{problem}

\begin{problem}
 Is there an Azumaya algebra $A$ with a $\lambda$-involution $ \tau $ that is extraordinary
 (i.e.\ not semiordinary)?
\end{problem}

	By Theorem~\ref{TH:types-for-schemes}, the first
	problem can be phrased as follows: 
	Suppse that
	$X$ is a scheme with involution $\lambda$ admitting
	a good quotient relative to $\{1,\lambda\}$,
	or $X$ is a Hausdorff topological space with involution $\lambda$.
	Let $Z$ be the locus of points where $\lambda$ ramifies
	and let $Z=Z_{-1}\sqcup Z_{1}$ be a partition of $Z$ into two closed
	subsets. Is it always possible to find an Azumaya algebra $A$ over $X$ admitting
	a $\lambda$-involution $\tau$ such that the specialization of $\tau$
	to $k(z)$ is orthogonal if $z\in Z_1$ and symplectic if $z\in Z_{-1}$? 

\section{Examples and Applications}
\label{sec:exapp}

\setcounter{subsection}{1}

\begin{example} \label{ex:split_model_example}
 Fix an exact quotient $\pi: (\mathbf X,\calO_\bfX)\to (\mathbf Y,\calO_\bfY)$ 
 and write $R = \pi_* \sh O_{\mathbf X}$,
 $S=\calO_\bfY$. 
 We assume that $\units S$ has square roots locally and $2 \in \units S$.
 This assumption allows us to drop the distinction between
 types and coarse types for the most part (Corollary~\ref{CR:coarse-type-determines-type}). 
 As in the previous section, we use the
 notation $N$ for 
 the kernel the $\lambda$-norm map ${x\mapsto x^\lambda x}:\units R\to \units S$, 
 and $T$ for the quotient of $N$ by the image of 
 the map $\units R \to N$
 given by $r \mapsto r^{-1} r^{\lambda}$. The coarse types are then $\Hoh^0(\mathbf Y, T)$.

 Suppose $t$ is an ordinary coarse type. By Theorem \ref{TH:standard-types}, this is equivalent to saying there exists
 some $\varepsilon$ in $\Hoh^0(\mathbf Y, N)$ mapping to $t$ under the map
 $\Hoh^0(\mathbf Y, N) \to \Hoh^0(\mathbf Y, T)$.
 Such an $\varepsilon$ can always be found if $\Hoh^1(\mathbf Y, \units R/ \units S)$ vanishes, for instance.

 Let $n$ be a natural number. Consider the matrix
\[ h=h_{2n}(\veps) =
\begin{bmatrix}
 0 & I_n \\ \varepsilon I_n & 0
\end{bmatrix}.
\]
It is immediate that $\veps h^{\lambda\tr} = h$. This equality implies that the map
$\tau_\veps:\Mat_{2n \times 2n}(R)\to \Mat_{2n \times 2n}(R)$ given on sections by
 \[ M \mapsto (h_{2n} (\varepsilon)\, M\, h_{2n}(\varepsilon)^{-1})^{\lambda \tr} . \]
 is a $\lambda$-involution.
 The (coarse) type of $\tau$ is easily seen to be $t$,
 the image of $\veps$ in $\Hoh^0(T)$. This follows
 from Construction~\ref{CN:coarse-type}.

 In this case, any algebra of degree $2n$ with involution of coarse type $t$ is locally isomorphic to $(\Mat_{2n \times 2n}(R), \tau)$, by
 Theorem~\ref{TH:ct-determines-local-iso}. Thanks
 to Corollary~\ref{CR:torsors-of-PU}, we may therefore place such 
 algebras in bijective correspondence with $G$-torsors on
 $\mathbf Y$ where 
 $G = \PU(\Mat_{2n \times 2n}(R),\tau_\veps)\cong \sAut_R( \Mat_{2n \times 2n}(R), \tau_\veps)$.
\end{example}

\begin{example} \label{ex:symplectic}
 As a special case of the previous example, we describe the Azumaya algebras with \textit{symplectic involution} on a
 scheme or topological space with trivial involution.
 See Theorem \ref{TH:important-exact-quotients} for the specific hypotheses on the underlying geometric object,
 and note that we assume $2$ is invertible. 
 
 In this case,
 $\pi = \id$, $R=S$, and $N = \mu_{2,R}$. 
 By Theorem
 \ref{TH:basic-properties-of-types}\ref{item:TH:basic-properties-of-types:tot-ram},
 the group of coarse types is $\Hoh^0(\mu_{2,R})$, which
 is just $\{1, -1\}$ when $X$ is connected.
 We consider the (coarse) type
 $-1$, called the \emph{symplectic} type. 

 Any Azumaya algebra with involution having this type must be of even degree, $2n$, by
 Theorem~\ref{TH:basic-properties-of-types}\ref{item:TH:basic-properties-of-types:odd}, and is locally isomorphic to the
 split degree-$2n$ algebra with symplectic involution 
 \[ \syp: M \mapsto (h_{2n}(-1) M h_{2n}(-1)^{-1})^{\tr}. \]
 
 The unitary group of $(\nMat{R}{2n\times 2n},\syp)$
 is the familiar symplectic group $\Sp_{2n}(R)$,
 and it follows from Lemma~\ref{LM:PU-isomorphism} that the automorphism
 group of $(\nMat{R}{2n\times 2n},\syp)$ is
 \[
 \PSp_{2n}(R):=\Sp_n(R)/\mu_{2,R}\ .
 \]
 In particular, as noted
 in Corollary~\ref{CR:torsors-of-PU}, 
 the set of isomorphism classes of degree-$2n$ Azumaya algebras 
 with symplectic involution is in canonical bijection with 
 \[ \Hoh^1 (\mathbf X, \PSp_{2n}(R) ). \]

	Since the symplectic type is ordinary,
 by Theorem \ref{TH:standard-Azumaya} and Example \ref{EX:mult_by_two},
 we derive the well known fact that an Azumaya
 algebra $A$ on $X$ is Brauer equivalent to one having a symplectic involution if and only if the Brauer class of $A$ is
 $2$-torsion.
\end{example}

\begin{example} \label{ex:general-trivial-example}
 Fix an exact quotient $\pi: \mathbf X \to \mathbf Y$ with ring objects $\sh O_{\mathbf X}$,
 and let $R = \pi_* \sh O_{\mathbf X}$ and
 $S=\sh O_{\bfY}$. Let $n$ be a natural number and assume that the hypotheses of 
 Theorem~\ref{TH:ct-determines-local-iso} hold, namely $\units{S}$ has
 square roots locally, and either $2 \in \units S$,
 or $\pi$ is unramified, or $n$ is odd. We consider the trivial type, $1$. This is the type of the involution
 \[ M \mapsto M^{\lambda \tr} \]
 on the split algebra $\Mat_{n \times n}(R)$.

 Any algebra with involution of the trivial type is locally isomorphic to this one, 
 and therefore, as summarized in Corollary~\ref{CR:torsors-of-PU}, 
 these are classified
 by $G$ torsors where $G=\sAut_R(\nMat{R}{n\times n},\lambda\tr)\cong \PU(\nMat{R}{n\times n},\lambda\tr)$.
 We write the latter group as 
 \[\PU_n(R, \lambda)\] 
 and call it the \textit{projective unitary group} of rank
 $n$ for the involution $\lambda$.
 In accordance with this notation,
 the unitary group of $(\nMat{R}{n\times n},\lambda\tr)$
 will be denoted $\U_n(R,\lambda)$.
\end{example}

\begin{example} \label{ex:orthogonal}
 Consider the case of a scheme or a topological space $X$ with trivial involution, as in the case of
 Example~\ref{ex:symplectic}. The theory of Azumaya algebras with involution of type $1$ can be established along the
 same lines as that of type $-1$. These algebras are called \textit{orthogonal}. The automorphism group 
 of $(\nMat{R}{2n\times 2n},\tr)$ is the quotient
 group $O_{2n}(R)/\mu_{2,R}$, which we denote by $\PO_{2n}(R)$, the \textit{projective orthogonal group}. This is
 special notation for the group $\PU_{2n}(R, \id)$ of Example \ref{ex:general-trivial-example}.

Again, by reference to Theorem \ref{TH:standard-Azumaya} and Example \ref{EX:mult_by_two}, an algebra is Brauer equivalent to one having involution of this type
 if and only if the Brauer class is $2$-torsion.
\end{example}

\begin{example} \label{ex:unitary-schemes} In this example we discuss \textit{unitary involutions}.
As a special case of Example \ref{ex:general-trivial-example}, we consider
 the case of an unramified double covering $\pi: X \to Y$ of schemes or topological spaces. Again, we refer to Theorem
 \ref{TH:important-exact-quotients} for the specific hypotheses on the underlying geometric object.

 In this case, the ring object $R$ is a quadratic \'etale extension of $S$, see
 Propositions~\ref{PR:ramification-in-schemes} and~\ref{PR:ramification-top}. Since $\pi$ is unramified,
 Theorem~\ref{TH:basic-properties-of-types}\ref{item:TH:basic-properties-of-types:unram} implies that there is only one
 type of involution on Azumaya algebras, the trivial one, which is called \emph{unitary} in this context. In
 particular, we are in a special case of Example~\ref{ex:general-trivial-example}.

 The structure of the groups $\U_n(R,\lambda)$ and $\PU_n(R, \lambda)=\U_n(R,\lambda)/N$ depends on the nature of
 $\lambda$, so a complete description in the abstract is not possible. We can, however, find an \'etale, resp.\ open,
 covering $U\to Y$ such that $R_U\cong S_U\times S_U$. After specializing to $U$, the algebra $\nMat{R}{n\times n}$
 becomes $\nMat{S}{n\times n}\times \nMat{S}{n\times n}$ and the involution $\lambda\tr$ is becomes the involution given
 sectionwise by $(x,y)\mapsto (y^{\tr},x^{\tr})$. From this one verifies that
 $\U_n(R,\lambda)_U\cong \GL_n(S)_U$ and $\PU_n(R,\lambda)_U\cong \PGL_n(S)_U$. Consequently, for a general degree-$n$
 Azumaya $R$-algebra with involution $(A,\tau)$, the groups $\U(A,\tau)$ and $\PU(A,\tau)$ are locally isomorphic to
 $\GL_n(S)$ and $\PGL_n(S)$, respectively.
 
 This agrees with the well established fact that projective unitary group schemes of unitary involutions are of type
 $A$.
\end{example}

\begin{example} \label{ex:mixed}
 In another instance of Example \ref{ex:split_model_example}, we can produce an example of an Azumaya algebra with an
 involution that mixes the various classical types. This
 example also featured in the introduction.
 
 We work with \'etale sheaves and \'etale cohomology throughout, see Example~\ref{EX:important-exact-quo-scheme}. Let
 $k$ be an algebraically closed field and let $X = \Spec k[x, x^{-1}]$ with the $k$-linear involution $\lambda$ sending
 $ x $ to $ x^{-1}$. The good quotient of $X$ by this involution exists, and is given by $Y= \Spec k[y]$ where
 $y = x + x^{-1}$.

 Here, the ring object $R$ is the ring $k[x,x^{-1}]$ viewed as a sheaf of rings on $Y$, 
 and the ring object $S$ is the
 structure sheaf of $Y$. The sheaf $N$ is the sheaf of norm-$1$ elements in $R$, where the norm map sends a
 Laurent polynomial $p(x)$ to $p(x)p(x^{-1})$. Both the Picard and the Brauer groups of $X$ and $Y$
 vanish, so that we can calculate $\Hoh^1(Y, \units R/ \units S) =0$. The following sequence is
 therefore exact
 \[ 1 \to \Hoh^0(Y, \units R/ \units S) \xrightarrow{\psi} \Hoh^0(Y, N) \to \Hoh^0(Y, T) \to 1.\]
 Explicitly, we calculate that $\Hoh^0(Y, \units R/ \units S)$ consists of classes of monomials $x^i$ for $i \in \ZZ$, and $\Hoh^0(Y,
 N)$ consists of monomials of the form $\pm x^i$ for $i \in \ZZ$, but $\psi$ maps
 the class of $x^i$ to
 $x^i/ x^{-i} = x^{2i}$. Therefore, the group of (coarse) types is isomorphic to the Klein $4$-group:
 $\Hoh^0(Y, T) = \{ \quo{1}, \quo{-1}, \quo{x}, \quo{-x}\}$.
 
 Since $\Hoh^1(Y, \units R/\units S) = 0$, we are in the circumstance of 
 Example~\ref{ex:split_model_example} which provides 
 models for each of the four types on even-degree split algebras. 
 For instance, on $\Mat_{{2}\times {2}}(R)$, we
 have the involution given by conjugating by $[\begin{smallmatrix} 0 & 1 \\ x & 0 \end{smallmatrix}]$
 and then applying $\lambda \tr$, namely
 \[ \begin{bmatrix} a(x) & b (x) \\ c(x) &d(x) \end{bmatrix}
 \mapsto \begin{bmatrix} d(x^{-1}) & x^{-1} b (x^{-1}) \\ x c(x^{-1}) &
 a(x^{-1}) \end{bmatrix}. \]
 Away from the fixed locus of $\lambda: X \to X$, namely, the points $x=1$ and $x=-1$, this involution is unitary,
 whereas at $x=1$ it specializes to be orthogonal and at $x=-1$ to be symplectic.
 
 More generally, it follows from Theorem~\ref{TH:types-for-schemes}
 that the type of any Azumaya $X$-algebra with involution $(A,\tau)$
 is determined by the types seen upon specializing to
 $x=1$ and $x=-1$. 
\end{example}

\begin{example} \label{ex:semistandard} We now demonstrate that there exist involutions that are not locally isomorphic
 to involutions of the form exhibited in Example~\ref{ex:split_model_example}. Specifically, we will show that there
 are involutions which are not ordinary in the sense of Definition~\ref{DF:standard-inv}.
	
 We consider a complex hyperelliptic curve $X$ of genus $g$ and a double covering $X\to Y=:{\mathbb{P}}^1_{\C}$.
 Explicitly: Let $a_1,\dots,a_{2g+1}$ be distinct complex numbers, and let
 \[X_0=\Spec \C[x,y]/(y^2-\prod_i(x-a_i)).\] We complete $X_0$ by gluing it to $X_1:=\Spec \C[u,v]/(v^2-u\prod_i(1-a_iu))$
 by mapping $(x,y)$ to $(u,v):=(\frac{1}{x},\frac{y}{x^{g+1}})$, and denote the resulting smooth complete curve by $X$.
 View $Y={\mathbb{P}}^1_{\C}$ as the gluing of $Y_0:=\Spec \C[x]$ to $Y_1:=\Spec \C[u]$ via
 $(x:1)\leftrightarrow (1:u^{-1})$. Projection onto the $x$ or $u$ coordinate
 induces a double covering $\pi:X\to Y$ with ramification at the points $(a_1\!:\!0),\dots,(a_{2g+1}\!:\!0)$ and
 $(0\!:\!1)$. The map $\lambda:X\to X$ given by $(x,y)\mapsto (x,-y)$, resp.\ $(u,v)\mapsto (u,-v)$, on the charts is an
 involution and $\pi$ is a good quotient relative to $C_2:=\{1,\lambda\}$. Indeed, working with the affine covering
 $Y=Y_0\cup Y_1$, we see that $\C[x]$ is the fixed ring of
 \[ \lambda^\#: \frac{\C[x,y]}{(y^2-\prod_i(x-a_i))} \to \frac{\C[x,y]}{(y^2-\prod_i(x-a_i))}, \qquad x \mapsto x, \quad
 y \mapsto -y, \] and similarly on the other chart.

 By Corollary \ref{CR:smooth-scheme-standard}, all coarse $\lambda$-types in $\Hoh^0(T)$ are realizable and semiordinary. Since
 the branch locus of $\pi$ consists of $2g+2$ points, it follows from Theorem~\ref{TH:types-for-schemes} that
 there are $2^{2g+2}$ $\lambda$-types. 
 Theorem~\ref{TH:standard-types} also says that the number of ordinary types is the cardinality of the image of the map
 $\Hoh^0(N)\to \Hoh^0(T)$, where $N$ is the sheaf of sections of $\lambda$-norm $1$ in $R:=\pi_*\calO_X$. Let
 $\veps\in\Hoh^0(N)$. Then $\veps\in \Hoh^0(Y, \pi_*\calO_X)=\Hoh^0(X, \calO_X)$. Since $X$ is a complete complex curve,
 the global sections of the structure sheaf are constant functions, meaning that $\veps\in\units{\C}$. Since
 $\veps^\lambda\veps=1$, it follows that $\Hoh^0(N)=\{\pm 1\}$. The images of $1,-1\in\Hoh^0(N)$ in $\Hoh^0(T)$ are
 therefore the ordinary types. Thus, of the $2^{2g+2}$ possible coarse types, only $2$ are ordinary, and the remaining
 $2^{2g+2}-2$ are merely semiordinary.
	
 We remark that we have reached the latter conclusion without actually constructing Azumaya algebras with involution
 realizing any of the non-ordinary types. A construction is given in the proof of
 Theorem~\ref{TH:standard-types}, and it can be made explicit in our setting with further work.
	
 This example can also be carried with the affine models $X_0$ and $Y_0$. One can check directly that
 $\Hoh^0(X_0,\units{\calO_{X_0}})=\units{\C}$ and thus it is still the case that $\Hoh^0(Y,N)=\{\pm 1\}$. Since the
 ramification point $(0:1)\in {\mathbb{P}}^1_{\C}=Y$ was removed, in this case, there are $2^{2g+1}$ coarse types, all
 semiordinary, of which only $2$ are ordinary.
\end{example}

\begin{example}
 A surprising source of examples comes from Clifford algebras of quadratic forms with simple degeneration. We refer
 the reader to \cite[\S1]{auel_fibrations_2014} or \cite[\S1]{auel_parimala_suresh_2015} for all relevant definitions.
	
 Let $Y$ be a scheme on which $2$ is invertible and let $(E,q,L)$ be a line-bundle-valued quadratic space of even rank
 $n$ over $Y$; when $L=\calO_Y$ and $Y=\Spec S$, these data merely amount to specifying a quadratic space of rank $n$
 over the ring $S$. According to \cite[Dfn.~1.9]{auel_parimala_suresh_2015}, $q$ is said to have \emph{simple
 degeneration} if for every $y\in Y$, the specialization of $q$ to $k(y)$ is a quadratic form whose radical has
 dimension at most $1$. In this case, it shown in \cite[Prp.~1.11]{auel_parimala_suresh_2015} that the even Clifford
 algebra $C_0(q)$, which is a sheaf of $\calO_Y$-algebras, is Azumaya over its centre $Z(q)$. Furthermore, the sheaf
 $Z(q)$ corresponds to a flat double covering $\pi:X\to Y$, which ramifies at the points $y\in Y$ where $q_{k(y)}$
 is degenerate. As such, $\pi$ is a good quotient relative to the involution $\lambda:X\to X$ induced by the involution
 of $Z(q)$ given by $x\mapsto \mathrm{Tr}_{Z(q)/\calO_Y}(x)-x$ on sections. Abusing the notation, we realize $C_0(q)$
 as an Azumaya algebra over $X$.
	
 Suppose $q$ has simple degeneration. We moreover assume that $Y$ is integral, regular and noetherian with generic
 point $\xi$ and that $q_{k(\xi)}$ is nondegenerate, although it is likely that these assumptions are unnecessary. The
 algebra $C_0(q)$ has a canonical involution $\tau_0$, see \cite[\S1.8]{auel_mpim_2011}, 
 and by applying
 \cite[Prp.~8.4]{knus_book_1998-1} to $C_0(q_{k(\xi)})$, we see that 
 $\tau_0$ is of the first kind when
 $n \equiv 0 \pmod 4 $ and a $\lambda$-involution when $n \equiv 2 \pmod 4 $. Furthermore, in the case
 $n \equiv 0 \pmod 4 $, we have $\transf_\lambda([C_0(q)])=0$ because $\transf_\lambda([C_0(q_{k(\xi)})])=0$ by
 \cite[Thm.~9.12]{knus_book_1998-1}, and $\Br(Y)\to \Br(k(\xi))$ is injective by
 \cite[Cor.~1.8]{grothendieck_groupe_1968} (or \cite[Thm.~7.2]{auslander_brauer_1960} in the affine case). It therefore
 follows from Theorem~\ref{TH:Saltman-ramified} that there exists $A'\in [C_0(q)]$ with $\deg A'=2\deg C_0(q)=2^{n}$
 that admits a $\lambda$-involution. We expect that the choice of $A'$ and its involution can be done canonically in
 $q$, and with no restrictions on $Y$.
	
 With the observations just made, it is possible that our work could facilitate the study of Clifford invariants of
 non-regular quadratic forms, e.g.\ in \cite{voight_2011}, \cite{auel_fibrations_2014},
 \cite{auel_parimala_suresh_2015}.
\end{example}

\section{Topology and Classifying Spaces}
\label{sec:topology}

The remainder of this paper is concerned with constructing a quadratic \'etale map of complex varieties $X\to Y$ and an
Azumaya algebra $A$ over over $X$ such that $A$ is Brauer equivalent to an algebra $A'$ with a $\lambda$-involution,
$\lambda$ being the non-trivial $Y$-automorphism of $X$, but such that the smallest degree of any such $A'$ is $2\deg
A$.

We recall that in this particular case, a Brauer equivalent algebra of degree $2\deg A$ admitting a $\lambda$-involution is guaranteed to exist
 by a theorem of Knus, Parimala and Srinivas \cite[Th.~4.2]{knus_azumaya_1990}; this has\benw{changed} been
generalized in Theorem~\ref{TH:Saltman-ramified}. An analogous example in which $\lambda:X\to X$ is the trivial
involution was exhibited in \cite{asher_auel_azumaya_2017}.

The example, which is constructed in Section~\ref{sec:the-example}, will be obtained by means of topological obstruction
theory, similarly to the methods of \cite{antieau_unramified_2014}, \cite{asher_auel_azumaya_2017} and related
works.
That is, the desired properties of $A$ above will be verified by
establishing them for the topological Azumaya algebra $A(\C)$ over the complexification $X(\C)$,
%see Example~\ref{ex:complexRealization}, 
whereas the latter will be done by means of
certain homotopy invariants.

This section is foundational, describing in part an approach to topological Azumaya algebras with involution via
equivariant homotopy theory. The main points are that Azumaya algebras with involution correspond to principal
$\PGL_n(\CC)$-bundles with involution---a fact that is true even outside the topological context, but that we have
not emphasized until now---, that there are equivariant classifying spaces for such bundles, and that their theory is
tractable if one restricts to considering spaces $X$ on which the $C_2$-action is trivial or free.

\subsection{Preliminaries}

In this section and the next, all topological spaces will be tacitly assumed to have a number of desirable
properties. All spaces appearing will be assumed to be compactly generated, Hausdorff, paracompact and locally contractible.

Throughout, we work in the category of $C_2$-topological spaces and $C_2$-equivariant maps. There are two notions of homotopy one can consider for
maps in this setting, the \textit{fine}, in which homotopies are themselves required to be $C_2$-equivariant, and the
\textit{coarse}, where non-equivariant homotopies are allowed. These two notions each have model structures appropriate
to them, the fine and the coarse. In the fine model structure, the weak equivalences are the equivariant maps
$f: X \to Y$ inducing weak equivalences on fixed point sets $f: X^G \to Y^G$ where $G$ is either the group $C_2$ or the
subgroup $\{1\}$.
In the coarse structure,
it is required only that $f: X \to Y$ be a weak equivalence when the $C_2$-action is disregarded, that is, only the
subgroup $\{1\}$ is considered. The identity functor is a left Quillen functor from the coarse to the fine. This is a
synthesis of the theory of \cite{dwyer_singular_1984} with \cite{elmendorf_systems_1983}.

\begin{notation}
 The notation $[X,Y]$ is used to denote the set of maps between two (possibly unpointed) objects $X$ and $Y$ in a homotopy
 category. The notation $[X,Y]_{C_2}$ will be used to denote the set of maps between $X$ and $Y$ in the fine
 $C_2$-equivariant homotopy category, whereas $[X,Y]_{C_2\text{-coarse}}$ will be used for the coarse structure.
\end{notation}

\begin{remark} In the case of the coarse model structure, the cofibrant objects include the $C_2$-CW-complexes with free
$C_2$-action, and if $X$ is a $C_2$-CW-complex, then the construction $X \times EC_2 \to X$ furnishes a cofibrant
replacement of $X$.% The cofibrant objects in the pointed coarse structure correspondingly include the spaces $X_+=X
% \coprod \ast$, where $X$ is a free $C_2$-CW-complex.
\end{remark}

All spaces are fibrant in both the coarse and the fine model structures, which implies the following standard result.

\begin{proposition} \label{prop:equalityForFreeAction}
 If $X$ is a free $C_2$-CW-complex and $Y$ is a $C_2$-space, then there is a natural bijection
\[ [X, Y]_{C_2} \longleftrightarrow [ X, Y]_{C_2\text{-coarse}}. \]
\end{proposition}

It is well known that $C_2$-equivariant homotopy theory in the coarse sense is equivalent to homotopy theory
carried out over the base space $B C_2$. We refer to \cite[Sec.~8]{shulman_parametrized_2008} for a sophisticated
general account of this equivalence. Specifically, the Borel construction $X \mapsto X \times^{C_2} E C_2$ and the
relative mapping space $Y \mapsto \Map_{B C_2}( E C_2, Y)$ form a Quillen equivalence between $C_2$-equivariant spaces
with the coarse structure, and spaces over $B C_2$, endowed with what \cite{shulman_parametrized_2008} calls the
``mixed'' structure on spaces over $B C_2$.\benw{I owe this reference to Bert Guillou.}
\begin{proposition} \label{PR:bundles-over-BCtwo}
Suppose $X$ and $Y$ are $C_2$-spaces with $X$ being a $C_2$-CW-complex. Then the Borel construction
$(\cdot )\times^{C_2} EC_2$ gives rise to a natural bijection
\[
\left\{
\begin{array}{c}
\text{coarse $C_2$-homotopy classes} \\
\text{of maps $X \to Y$}
\end{array}
\right\}
~\cong ~
\left\{
\begin{array}{c}
\text{homotopy classes of maps} \\
\text{$X\times^{C_2} EC_2
 \to Y \times^{C_2} EC_2$ over $B C_2$}
\end{array}
\right\} \ .
\]
\end{proposition}

\subsection{Equivariant Bundles and Classifying Spaces}
\label{subsec:equivClass}

There is a general theory of equivariant bundles and classifying spaces, more general indeed than what is required in
this paper. All examples we consider are of the following form:

\begin{definition} \label{def:twistedGbundle}
 Suppose we are given a topological $C_2$-group $G$, or equivalently,
 a topological group $G$ equipped with an involutary automorphism $\tau:G\to G$. 
 A \textit{principal
 $G$-bundle with
 a $\tau$-involution}, or just \textit{principal $G$-bundle
 with involution}, on a $C_2$-space $X$ is a map $\pi: E \to X$ in $C_2$-spaces such that:
\begin{enumerate}
\item $\pi: E \to X$ is a principal $G$-bundle,
\item the actions of $C_2$ and of $G$ on $E$ are compatible, in the sense that if $c \in C_2$, $g \in G$ and $e \in E$,
then
\[ 
c \cdot (e \cdot g) = (c \cdot e) \cdot ( c \cdot g ). 
\]
\end{enumerate}
\end{definition}

\begin{remark}\label{RM:c-two-torsor-equiv-dfn}
 This concept admits an equivalent definition. Any $G$-bundle $E$, equivariant or not, may be pulled back along the
 involution $\lambda$ of $X$, in order to form $\lambda^* E \to X$. One may then twist the the $G$-action on
 $\lambda^* E$ by changing the structure group along $\tau: G \to G$, forming $E^*:=G \times_\tau \lambda^* E$. This
 may be identified with $\lambda^* E\cong E$ as a topological space over $X$, but with a different $G$-action. The definition
 of principal $G$-bundle with involution given above is equivalent to asking that $\pi: E \to X$ be a principal $G$-bundle 
 together with a $G$-bundles morphism $f$ of order $2$ from $\pi: E \to X$ to $\pi: E^* \to X$. On the underlying spaces, $f$
 must be an isomorphism of order $2$ of $E$ over $X$, which is equivalent to a $C_2$-action on $E$
 making $\pi: E \to X$ equivariant. The fact that $f: E \to E^*$ is 
 an isomorphism of principal $G$-bundles is exactly the
 relation 
 $c \cdot ( e \cdot g) = (c\cdot e) \cdot (c \cdot g)$
 above.
\end{remark}

Because the automorphism $\tau:G\to G$ is not assumed to be trivial, this notion is more general than the most basic notion of
`equivariant principal $G$-bundle', but at the same time, because the sequence
\[ 1 \to G \to G \rtimes C_2 \to C_2 \to 1 \] is split, it is less general than the most general case considered in
\cite{may_remarks_1990}.

One may construct a $C_2$-equivariant classifying space for $C_2$-equivariant principal $G$-bundles, as in
\cite{may_remarks_1990}*{Thm.~5}. We will take the time to explain the procedure, since some of the details will be
important later \footnote{We remark that in our case, the group called $\Gamma$ in \cite{may_remarks_1990} is a
semidirect product, so $EG \times EC_2$, with an appropriate $\Gamma$-action, is a model for $E\Gamma$. This allows us
to replace the space of sections of $EC_2 \to E\Gamma$ by the space of maps $EC_2 \to EG$, an argument that appears in
\cite{Guillou2017}*{Sec.~5, p.~21}.}

\begin{notation}
 The notation $EG \to B G$ will be used for a construction of the classifying space of a topological group $G$,
 functorial in $G$.
\end{notation}
 By functoriality, if $G$ admits a $C_2$-action, then $EG \to B G$ admits a $C_2$-action. While the ordinary homotopy
 type of $EG \to B G$ is well defined, irrespective of the model we choose, the $C_2$-equivariant type is not. The
 construction $E_{C_2} G \to B_{C_2}G$ outlined below is a specific choice of such a type.

 Start with a $EG \to B G$. Now consider the space of continuous functions $\mathcal{C}(EC_2 , EG)$. It is endowed with
 both a $G$-action, induced directly by the $G$-action on $EG$, and by a $C_2$-action given by conjugation of the
 map. The two actions together induce an action of ${G \rtimes C_2}$ on $\mathcal{C}(EC_2 , EG)$, which is contractible,
 and consequently a $C_2$-action on $\mathcal{C}(EC_2 , EG)/G$, which is a model for $B G$. The resulting map
\[ \mathcal{C}(EC_2, EG) \to \mathcal{C}(EC_2, EG)/G \] is a map of $C_2$-spaces, and will be denoted
\[ E_{C_2}G \to B_{C_2}G. \]
We remark that in \cite{may_remarks_1990} and other sources, May and coauthors denote these spaces $E(G; G \rtimes C_2)$
and $B (G; G \rtimes C_2)$.

Furthermore, the map $EC_2 \to \ast$ induces a map $EG = \mathcal{C}(\pt, EG) \to \mathcal{C}(EC_2, EG) =
E_{C_2}G$. This map is $G \rtimes C_2$-equivariant, and induces a $C_2$-equivariant commutative square
\begin{equation}
 \label{eq:9} \xymatrix{ EG \ar^{\sim}[r] \ar[d] & E_{C_2} G \ar[d] \\ B G \ar^\sim[r] & B_{C_2} G, }
\end{equation} in which the horizontal maps are coarse, but not necessarily fine, $C_2$-weak equivalences. The map
$E_{C_2}G \to B_{C_2}G$ is a classifying space for principal $G$-bundles with involution.

\begin{proposition} \label{prop:existUnivCGbundle}
 If $X$ is a $C_2$-CW-complex, then there is a natural bijection between $ [X, B_{C_2}G]_{C_2}$ and the set of isomorphism classes of principal $G$-bundles with involution on $X$.
\end{proposition} 

We refer to \cite[Thm.~5]{may_remarks_1990} for the proof.

\begin{proposition} \label{prop:freeCase}
 If $X$ is a free $C_2$-CW-complex, then the following are naturally isomorphic
 \begin{enumerate}[label=(\alph*)]
 \item $[X, B_{C_2}G]_{C_2}$,
 \item $[X, B_{C_2}G]_{C_2\text{-coarse}}$,
 \item $[X, BG]_{C_2}$,
 \item $[X, BG]_{C_2\text{-coarse}}$,
 \item The set of isomorphism classes of principal $G$ bundles with involution on $X$.
 \end{enumerate}
\end{proposition}
\begin{proof}
 The equivalences all follow from Propositions \ref{prop:equalityForFreeAction}, \ref{prop:existUnivCGbundle} and Diagram
 \eqref{eq:9}.
\end{proof}

\begin{remark}
 Proposition~\ref{prop:freeCase} means that if one is willing to restrict one's attention to spaces with free $C_2$-action, then the
 construction of $E_{C_2}G \to B_{C_2}G$ from $EG \to BG$ is not necessary. The $C_2$-action given by the functoriality
 of the construction of $BG$ is sufficient.
\end{remark}

\begin{remark}\label{RM:two-copies-construction}
 Let $G$ be a topological group. One may give $G \times G$ the $C_2$-action which interchanges the two factors. Then the
 resulting classifying space $BG \times BG$ also admits this action. In this instance, the space $B_{C_2}(G \times G)$
 is $C_2$-equivalent 
 to $BG \times BG$ with the interchange action, which may be verified by testing on $C_2$-fixed
 points, for instance.
 
 The construction of taking a space $Y$ and producing $Y \times Y$ with the $C_2$-action interchanging the factors is
 right adjoint to the forgetful functor. Suppose $X$ is a $C_2$-space, then
 \begin{equation}
 \label{eq:4} [ X , BG ] \iso [X , BG \times BG]_{C_2}
 \end{equation} where the set on the left is the set of maps in the nonequivariant homotopy category.
\end{remark}

\subsection{The Case of $\PGL_n$-bundles}
\label{subsec:PGL-bundles}

For the rest of this section, we write $\GL_n$, $\PGL_n$ etc.~for the Lie group of complex points, $\GL_n(\C)$,
$\PGL_n(\C)$ and so on.

We now specify $C_2$-actions on groups that will appear in the sequel. There is a $C_2$-action on $\GL_n$ 
in which the non-trivial element acts via $A
\mapsto A^{-\tr}$, the transpose-inverse. This passes to certain subquotients of $\GL_n$, and we will use it as the
$C_2$-action on the groups $\mu_n$, $\C^\times$, $\SL_n$ and $\PGL_n$, all viewed either as subgroups or as quotients of
$\GL_n$. Specifically, we write $-\tr: \PGL_n \to \PGL_n$ for the outer automorphism $A \mapsto A^{-\tr}$.

There is also a $C_2$-action on $\GL_n \times \GL_n$ given by interchanging the factors and then applying the
transpose-inverse, so that the induced involution is
\[ (A,B) \mapsto (B^{-\tr}, A^{-\tr})\ . \]
This will be used for certain subquotients of this group, including $\mu_n \times \mu_n$, $\C^\times \times
\C^\times$, $\SL_n \times \SL_n$ and $\PGL_n \times \PGL_n$.

There is a diagonal inclusion $\GL_n \to \GL_n \times \GL_n$, given by $A \mapsto (A , A)$. It is $C_2$-equivariant,
and induces similar maps for the aforementioned subquotients of $\GL_n$.

One may form $C_2$-equivariant classifying spaces for the groups named above, as outlined in Subsection
\ref{subsec:equivClass}. Among the possibilities, two are particularly useful to us: $B_{C_2}\PGL_n$ and $B_{C_2}
(\PGL_n \times \PGL_n)$.

\begin{proposition} \label{pr:equivPGLn}
 Let $X$ be a $C_2$-CW-complex with corresponding involution $\lambda$, and let $n$ be a natural number. Then the following sets are in
 natural bijective correspondence:
 \begin{enumerate}[label=(\alph*)]
 \item \label{prep:1} Isomorphism classes of degree-$n$ topological Azumaya algebras with $\lambda$-involution on $X$,
 \item \label{prep:2} Isomorphism classes of principal $\PGL_n$-bundles with involution on $X$,
 \item \label{prep:3} $[X, B_{C_2} \PGL_n ]_{C_2}$.
 \end{enumerate}
\end{proposition}
\begin{proof}
 There is a well-known bijection between Azumaya algebras of degree $n$ on $X$ and principal $\PGL_n$-bundles, since
 $\PGL_n(\CC)$ is the automorphism group of $M:=\Mat_{n \times n}(\C)$ as a $\C$-algebra, see Subsection~\ref{subsec:Az-algs}.
 Let $A$ be an Azumaya algebra
 of degree $n$ on $X$ and $P$ the associated principal $\PGL_n$-bundle.

The functor of taking opposite algebras on Azumaya algebras corresponds to the functor of change of group along
 $-\tr: \PGL_n \to \PGL_n$ of principal $\PGL_n$-bundles; this can be seen at the level of clutching functions. Indeed,
 note that $m\mapsto m^{\tr}:\Mat_{n\times n}(\CC)\to \Mat_{n\times n}(\CC)^\op$
 is a $\CC$-algebra isomorphism.
 If one chooses coordinates for $A$ on two open sets of $X$ on which it trivializes, then the clutching function
 $f: \Mat_{n \times n}(\CC) \times (U_1 \cap U_2)\to \Mat_{n \times n}(\CC) \times (U_1\cap U_2)$ given by
 $m \mapsto xmx^{-1}$, for some $x : (U_1 \cap U_2) \to \PGL_n(\CC)$. For the same choice of coordinates over both
 $U_1$ and $U_2$, the clutching function $f^{\op}$ of the opposite algebra is given by
 $m^\tr \mapsto (xmx^{-1})^\tr = x^{-\tr} m^\tr x^\tr$.

 Therefore, the data of an isomorphism of $A \to A^\op$ of order $2$ over the involution $\lambda: X \to X$ is
 equivalent to an order-$2$ self-map of the associated principal $\PGL_n$-bundle, $P \to P^*$ over 
 $X$, where $P^*$
 denotes the principal $\PGL_n$-bundle 
 \[ P^* := \PGL_n \times_{-\tr} \lambda^* P.\]
 As explained in Remark~\ref{RM:c-two-torsor-equiv-dfn}, 
 this is equivalent to the definition of principal $\PGL_n$-bundle with involution in Definition
 \ref{def:twistedGbundle}; thus establishing the equivalence of \ref{prep:1} and \ref{prep:2}.

 The equivalence between \ref{prep:2} and \ref{prep:3} is an application of Proposition \ref{prop:existUnivCGbundle}.
\end{proof}

The space $B_{C_2}(\PGL_n \times \PGL_n)$, by similar methods, is seen to classify ordered pairs of $\PGL_n$-bundles on
a $C_2$-space $X$, such that the one is obtained
from the other by twisting relative to the involutions of $X$ and $\PGL_n$. But the category of such ordered pairs is identical
to the category of ordinary $\PGL_n$-bundles on the space $X$, forgetting the $C_2$-action.

This last fact also manifests itself algebraically via Remark~\ref{RM:two-copies-construction} in the following way: Suppose $G $ is a subgroup of $\GL_n$ closed under taking transposes,
or a quotient of $\GL_n$
by such a subgroup, let $(G\times G, \alpha)$ 
denote the product group with the involution $(A, B) \mapsto
(B^{-\tr}, A^{-\tr})$, and let $(G \times G, i)$ denote the product group with the involution exchanging $A$ and
$B$. Then $(A, B) \mapsto (A, B^{-\tr})$ is a $C_2$-equivariant isomorphism between these two groups with involution.

We will apply the classifying space theory developed above in the two extreme cases where the $C_2$-action on $X$ is
trivial and when it is free.

\subsection{Trivial Action}

Suppose $X$ is equipped with a trivial $C_2$-action. Then principal $\PGL_n$-bundles 
with
involution on $X$ are classified by $[X, B_{C_2}\PGL_n]_{C_2} = [X, (B_{C_2}\PGL_n)^{C_2}]$.

\begin{proposition} \label{pr:trivialActionTopology}
Let $n$ be a positive integer. Then the fixed locus $(B_{C_2}\PGL_n)^{C_2}$ is homeomorphic to
\begin{enumerate}[label=(\roman*)]
\item $B\PO_n\, \sqcup \, B\PSp_n$ if $n$ is even;
\item $B\PO_n$ if $n$ is odd.
\end{enumerate}
\end{proposition}
\begin{proof}
 We may calculate the fixed-point-sets of $B_{C_2}(\PGL_n)$ by means of \cite{may_remarks_1990}*{Thm.~7}. We explain
 the application of this theorem in the current case.

If $A \in \PGL_n$ is a matrix such that $A A^{-\tr} = I_n$, then $(A, -\tr) \in \Gamma := \PGL_n \rtimes C_2$ generates a
subgroup that maps isomorphically onto $C_2$ and intersects $\PGL_n$ trivially. Denote by $\PGL_n^{(A, -\tr)}$ the
commutant of $(A,-\tr)$ in $\PGL_n$, i.e.,
the
subgroup of $\PGL_n$ consisting of elements $X$ such that $X^{-\tr}=A^{-1}X A$. 
We write $A\sim A'$ if $(A,-\tr)$ and $(A',-\tr)$ are conjugate under $\PGL_n$, or equivalently,
if there exists $X\in\PGL_n$ such that $X AX^{\tr}=A'$.
Then the theorem asserts that
\[ (B_{C_2}\PGL_n)^{C_2}= \bigsqcup_{A} B( \PGL^{(A, -\tr)}) \] as 
$A$ runs over equivalence classes of elements $A\in \PGL_n$ satisfying $A^{-\tr}A=I_n$.

When $n$ is even, say $n=2m$, there are two such equivalence classes, namely the class of $I_n$ and the class of $h_{2m}(-1)$, in the
notation of Example \ref{ex:split_model_example}, as can be calculated directly. The fixed points under the action are
those matrices for which $B^{-\tr} = B$ in the first case and $B^{-\tr} =h_{2m}(-1) B h_{2m}(-1)^{-1}$ in the second,
which is to say, the subgroups of orthogonal and of symplectic matrices respectively. 
We therefore 
deduce
\[ (B_{C_2}\PGL_n)^{C_2}= B \PO_n \, \sqcup \, B \PSp_n . \]

When $n$ is odd, the argument is much the same, but only $B \PO_n $ occurs.
\end{proof}

\begin{remark}
 By Theorem \ref{TH:types-for-schemes}, we know that there are two types of involutions on Azumaya algebras over connected
 topological spaces with trivial action, the symplectic and orthogonal. By means of Examples \ref{ex:symplectic} and
 \ref{ex:orthogonal}, we know that the orthogonal and symplectic Azumaya algebras with involution are equivalent to principal
 bundles for the groups $\PO_n$ and $\PSp_n$, the latter when $n$ is even. Proposition \ref{pr:trivialActionTopology}
 has recovered these observations via equivariant homotopy theory.
\end{remark}

\subsection{Free Action}
\label{subsec:free-action}

Now we address the case where the action of $C_2$ on $X$ is free. 
In this case, the quotient map $X\to Y:=X/C_2$ is a two-sheeted covering space map.

\begin{proposition} \label{pr:freeClass}
 Let $X$ be a free
 $C_2$-CW-complex, with $C_2$ acting by the involution $\lambda$, and let $Y=X/C_2$. Consider $Y$ as a space over $B
 C_2$, or alternatively, as a space equipped with a distinguished class $\alpha \in \Hoh^1(Y, C_2)$. There are natural
 bijections between the following:
 \begin{enumerate}[label=(\alph*)]
 \item \label{item:pr:freeClass:a} Isomorphism classes of degree-$n$ 
 topological Azumaya algebras over $X$ equipped with 
 a $\lambda$-involution. 
 \item \label{item:pr:freeClass:b} 
 $[X, B\PGL_n]_{C_2}$.
 \item \label{item:pr:freeClass:c} 
 $[Y, B (\PGL_n \rtimes C_2) ]_{BC_2}$. 
 \item \label{item:pr:freeClass:d} Elements of the preimage of $\alpha$ under $\Hoh^1(Y, \PGL_n \rtimes C_2) \to \Hoh^1(Y, C_2)$.
 \end{enumerate}
\end{proposition}

\begin{proof}
	Propositions~\ref{pr:equivPGLn} and~\ref{prop:freeCase} give the bijection
	between \ref{item:pr:freeClass:a} and \ref{item:pr:freeClass:b},
	and Proposition~\ref{PR:bundles-over-BCtwo} gives a bijection between
	\ref{item:pr:freeClass:b} and \ref{item:pr:freeClass:c}.
	The one-to-one correspondence between \ref{item:pr:freeClass:c}
	and \ref{item:pr:freeClass:d} is standard.
\end{proof}

We continue to assume that $X$ is a free $C_2$-CW-complex and let $Y=X/C_2$.
We would like to have a classifying-space-level understanding of the cohomological transfer map
$\transf_{X/Y}:\Hoh^2(X,\Gm)\to \Hoh^2(Y,\Gm)$ considered 
in Subsection~\ref{subsec:coh-transfer}.

To that end, let $\mu$ denote the discrete group $\mu_n$ or the topological group $\units{\C}$. We endow $\mu$ with the
involution $a \mapsto a^{-1}$, give $\mu \times \mu$ the involution $(a,b) \mapsto (b^{-1}, a^{-1})$, and let
$\mu^{\text{triv}}$ denote $\mu$ with the trivial action.

The map $\mu \times \mu \to \mu^{\text{triv}}$ defined by $(a,b) \mapsto ab^{-1}$ is $C_2$-equivariant, and its kernel
consists of pairs of the form $(a,a)$, which is the image of the diagonal map $\mu \to \mu \times \mu$. That is, there
is a $C_2$-equivariant short exact sequence of $C_2$-groups
\[ 1 \to \mu \to \mu \times \mu \to \mu^{\text{triv}} \to 1 \] and therefore, a sequence of $C_2$-equivariant maps in
which any three consecutive terms form a homotopy fibre sequence:
\[ \mu \to \mu \times \mu \to \mu^{\text{triv}} \to B\mu \to B(\mu \times \mu) \to B \mu^{\text{triv}} \to B^2 \mu \to
\cdots .\]
Any such homotopy fibre sequence is a homotopy fibre sequence in the $C_2$-equivariant coarse structure. These
constructions are plainly natural with respect to inclusion of subgroups of $\C^\times$.

Now, if $X$ is a free $C_2$-CW-complex, then thanks to
Proposition~\ref{prop:freeCase}, one arrives at a long exact sequence of abelian groups
\begin{equation}
 \label{eq:6} \dots \to [X, B^i \mu ]_{C_2} \to [X, B^i \mu \times B^i \mu]_{C_2} \to [X , B^i \mu^{\text{triv}} ]_{C_2}
\to \cdots.
\end{equation} Since $\mu \times \mu$ with this action is isomorphic to $\mu \times \mu$ with the interchange action, it
follows that $[X, B^i \mu \times B^i \mu]_{C_2} \cong [X, B^i \mu] = \Hoh^i(X, \mu)$. Moreover, $ [X , B^i
\mu^{\text{triv}} ]_{C_2}$ is simply $[Y, B^i \mu] = \Hoh^i(Y, \mu)$.

Therefore, the sequence of \eqref{eq:6} reduces in this case to
\begin{equation}
 \label{eq:6a} \cdots \to [X, B^i \mu ]_{C_2} \to \Hoh^i(X, \mu) \xrightarrow{\transf}\Hoh^i(Y, \mu) \to \cdots.
\end{equation} 
When $\mu=\units{\C}$ and $i=2$, 
the map denoted $\transf$ agrees with the transfer map defined in
 Subsection~\ref{subsec:coh-transfer}. Indeed, we know that the transfer map in Section \ref{sec:Saltman} agrees with
 the ordinary transfer map for a $2$-sheeted covering in the case at hand, Example
 \ref{EX:topological-free-transfer}. It suffices therefore to show that the map $\transf$ in \eqref{eq:6a} is the usual
 transfer map for a $2$-sheeted cover. The trivial case $X = Y \times C_2$ is elementary. The general case where $\pi: X \to Y$ is
 merely locally trivial can be deduced from the trivial case by viewing $\Hoh^i(Y, \mu)$ and $\Hoh^i(X,\mu)$ as \v Cech
 cohomology groups and calculating each using covers $\sh U$ of $Y$ and $\pi^{-1} \sh U$ of $X$ where $\sh U$
 trivializes the double cover $\pi$.

\begin{proposition} \label{pr:whenCores0} Let $\mu$ be $\mu_n$ or $\C^\times$, given the involution
 $z \mapsto z^{-1}$. Let $X$ be a space with free $C_2$-action, let $Y=X/C_2$, and let $\xi: X \to B^i \mu$ be an equivariant map,
 representing a cohomology class $\xi \in \Hoh^i(X, \mu)$. Then $\transf_{X/Y} \xi = 0$.
\end{proposition}
\begin{proof}
 Since $\xi : X \to B^n \mu$ is equivariant, and the action on $X$ is free, $\xi$ lies in the image of $[X,
 B^n\mu]_{C_2}$ in $[X, B^n\mu]= \Hoh^n(X,\mu)$. Thus, the result follows from the exact sequence 
 \eqref{eq:6a}.
\end{proof}

\section{An Azumaya Algebra with no Involution of the Second Kind}
\label{sec:the-example}

We finally construct the example promised
at the beginning of Section~\ref{sec:topology}.

\medskip 

Throughout, the notation $x\Z$ means a free cyclic group, written additively, with a named
generator $x$.
Recall that for a topological space $X$,
the sheaf cohomology group $ \Hoh^2(X,\units{\C}):=\Hoh^2(X,\cont(X,\units{\C}))$
is isomorphic to the singular cohomology group
$\Hoh^3(X,\Z)$, see \cite[\S2.1]{asher_auel_azumaya_2017}, for instance. We shall use the latter group for the most part.

\subsection{A Cohomological Obstruction}
\label{subsec:coh-counterex}

In all cases, the groups appearing in this subsection are the complex points of linear algebraic groups. In the interest of
brevity, the relevant linear algebraic group, e.g.~$\SL_n$, will be written in place of the group itself, e.g.~$\SL_n(\C)$.

Unless otherwise specified, groups appearing will be endowed with a $C_2$-action. For the groups $\SL_n$, the action is
that sending $A $ to $ A^{-\tr}$, which restricts to the action $r \mapsto r^{-1}$ on the central subgroup $\mu_n$. For
the groups $\SL_n \times \SL_n$, the action is that given by $(A, B) \mapsto (B^{-\tr},A^{-\tr})$, and similarly for
$\mu_n \times \mu_n$. The maps $\SL_n \to \SL_n \times \SL_n$ and $\mu_n \to \mu_n \times \mu_n$ are given by diagonal
inclusions.

Embed $\mu_n \hookrightarrow \SL_n \times \SL_n$ via $r \mapsto (r I_n, r I_n)$, and let $Q_n$ denote the group obtained as
the quotient of $\SL_n \times \SL_n$ by the image of $\mu_n$.

In the following diagram, the horizontal arrows of the first two rows are $C_2$-equi\-variant.
This induces $C_2$-actions on the groups in the third row so that
all arrows become $C_2$-equivariant. 
\begin{figure}[h]
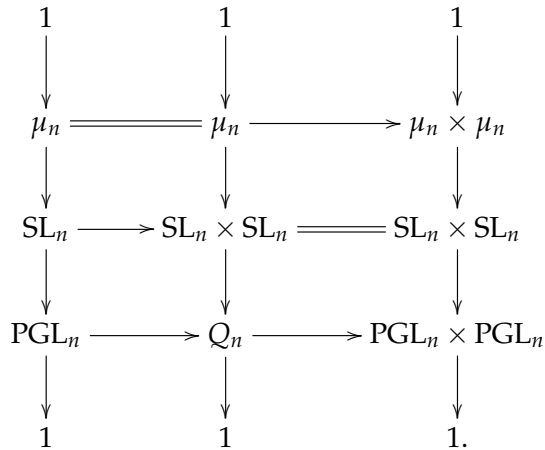

 \centering
 
 \caption{A Diagram of $C_2$ Groups}
 \label{fig:999}
\end{figure}

\begin{equation*} \xymatrix{ 1 \ar[d] & 1 \ar[d] & 1 \ar[d] \\ \mu_n \ar@{=}[r] \ar[d] & \mu_n \ar[r] \ar[d] & \mu_n
 \times \mu_n \ar[d] \\ \SL_n \ar[d] \ar[r] & \SL_n \times \SL_n \ar@{=}[r] \ar[d] & \SL_n\times \SL_n \ar[d] \\ \PGL_n
 \ar[r] \ar[d] & Q_n \ar[r] \ar[d] & \PGL_n \times \PGL_n \ar[d] \\ 1 & 1 & 1. }
 \end{equation*}

Each of the groups appearing above is equipped with a $C_2$-action, and consequently each may be extended to a
semidirect product with $C_2$, and equivariant classifying spaces of the form $B_{C_2}G$ may be constructed
as in Subsection~\ref{subsec:equivClass}. Since we
will consider equivariant maps with free $C_2$-action on the source, by Proposition \ref{prop:freeCase}, we may use any functorial model of $B G$ with its 
functorially-induced $C_2$-action instead.

\begin{proposition} \label{pr:C2cohoPGLn}
 The $C_2$-action on $\PGL_n$ induces an action on $\Hoh^*(B\PGL_n, \ZZ)$. In low degrees, this action is summarized
 by Table~\ref{tab:actions}.
 \begin{table}[h] \centering
 \begin{tabular}{c|c|c} $i$ & $\Hoh^i(B\PGL_n, \ZZ)$ & Action \\ \hline 0 & $\ZZ$ & trivial \\ 1 & 0 & - \\ 2 & 0 & -
 \\ 3 & $\alpha \ZZ/n$ & $\alpha \mapsto - \alpha$ \\ 4 & $ \tilde c_2 \ZZ $ & trivial
 \end{tabular}
 \caption{The $C_2$-action on $\Hoh^*(B\PGL_n, \ZZ)$.}
 \label{tab:actions}
 \end{table}
\end{proposition}
\begin{proof}
 The compatible $C_2$-actions on the terms of the exact sequence $1 \to \mu_n \to \SL_n \to \PGL_n \to 1$ induce an
 action on the fibre sequence $B \SL_n \to B \PGL_n \to B^2 \mu_n$, and therefore an action of $C_2$ on the associated
 Serre spectral sequence, which is illustrated in Figure~\ref{fig:SSS0}.
\begin{figure}[h]
 \centering
 \begin{equation*} \xymatrix@R=10pt{ \ZZ c_2 \ar^{d_5}[dddrrrrr] \\ 0 \\ 0 \\ \ZZ & 0 & 0 & \alpha \ZZ/n & 0 &
 \kappa \ZZ/(\epsilon n) }
 \end{equation*}
 \caption{A portion of the Serre spectral sequence in cohomology associated to $B\SL_n \to B\PGL_n \to B^2 \mu_n$.}
 \label{fig:SSS0}
 \end{figure}
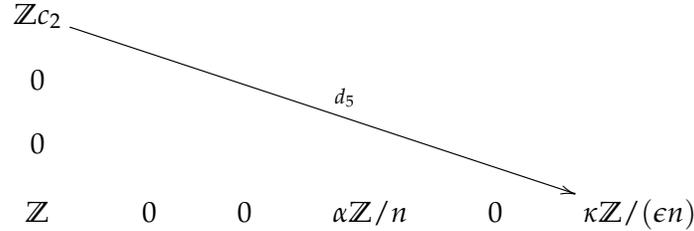 Here, $\epsilon$ is $2$ if $n$ is even and is $1$ otherwise.

 The action on $\alpha \ZZ/n$ is the same as the action of $C_2$ on $\mu_n = \ZZ/n$ itself, which is the sign
 action. The action on $\Hoh^4(\SL_n, \ZZ) = \ZZ c_2$ is calculated by identifying $\Hoh^*(\SL_n, \ZZ)$ as a subquotient
 of $\Hoh^*(B T_n, \ZZ)$, where $T_n$ is the maximal torus of diagonal matrices in $\GL_n$. Specifically, $\Hoh^*(B T_n,
 \ZZ) = \ZZ[\theta_1, \dots, \theta_n]$, where the $C_2$-action on $\theta_i$ is $\theta_i \mapsto - \theta_i$. Then the
 class $c_2$ in question may be identified with the image of the second elementary symmetric function in the $\theta_i$
 in $\Hoh^*(BT_n, \ZZ)/ (\sum_{i=1}^n \theta_i)$. It follows the action of $C_2$ on $c_2$ is trivial.

 We know from \cite[Proposition 4.4]{antieau_topological_2014} that the illustrated $d_5$ differential is
 surjective. Writing $\tilde c_2$ for $\epsilon n c_2$, it follows easily that the cohomology of $B\PGL_n$ takes the
 stated form, and carries the stated $C_2$-action.
\end{proof}

\begin{proposition}\label{pr:C2cohoQ}
Fix a natural number $n$, and let $\epsilon=\gcd(n,2)$. 
Let $S$ be the subgroup of $ c_2'\ZZ \oplus c_2''\ZZ$ consisting of terms $a c_2' + b c''_2$ where
$a + b \equiv 0 \pmod{ \epsilon n}$. The low-degree cohomology of $BQ_n$, along with its $C_2$-action, is summarized by Table~\ref{tab:actions2}.
\begin{table}[h]
 \centering
 \begin{tabular}{c|c|c} $i$ & $\Hoh^i(BQ_n, \ZZ)$ & Action \\ \hline 0 & $\ZZ$ & trivial \\ 1 & 0 & - \\ 2 & 0 & - \\ 3
 & $\alpha \ZZ/n$ & $\alpha \mapsto - \alpha$ \\ 4 & S & $ac'_2+bc''_2\mapsto bc'_2+ac''_2$
 \end{tabular}
 \caption{The $C_2$-action on $\Hoh^*(BQ_n, \ZZ)$.}
 \label{tab:actions2}
 \end{table} Moreover, the comparison map from $\Hoh^i(BQ_n, \ZZ)$ to $\Hoh^i(B\SL_n, \ZZ)$ is the evident
 identification map when $i \le 3$. When $i=4$, it is given by 
 $ac'_2+bc''_2 \mapsto \frac{a+b}{\epsilon n}\tilde{c}_2$.
\end{proposition}
\begin{proof}
 There is a fibre sequence $B(\SL_n \times \SL_n) \to BQ_n \to B^2 \mu_n$.
 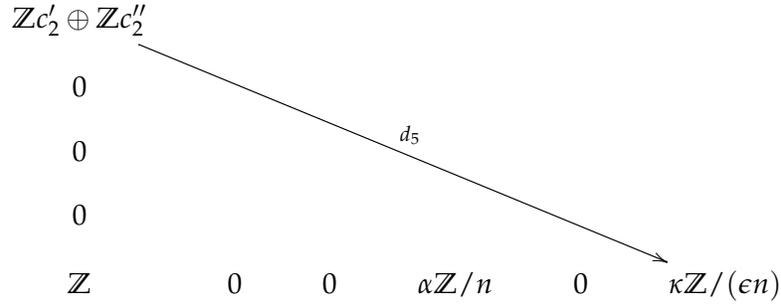
\begin{figure}[h] \centering
 \begin{equation*} \xymatrix@R=10pt{ \ZZ c_2' \oplus \ZZ c_2'' \ar^{d_5}[ddddrrrrr] \\ 0 \\ 0 \\ 0 \\ \ZZ & 0 & 0 & \alpha \ZZ/n
 & 0 & \kappa \ZZ/(\epsilon n) }
 \end{equation*}
 \caption{A portion of the Serre spectral sequence in cohomology associated to $B(\SL_n \times \SL_n) \to BQ_n \to
 B^2 \mu_n$.}
 \label{fig:SSS1}
 \end{figure} A portion of the associated Serre spectral sequence is shown in Figure~\ref{fig:SSS1}. 
 There is a comparison map of spectral sequences from this one to that of
 Figure~\ref{fig:SSS0}. The map identifies the bottom row of the two $\Eoh_2$-pages, and sends $c_2', c_2''$ both to $c_2$. It
 is compatible with the $C_2$-actions. The claimed results except
 the $C_2$-action on $\Hoh^4(B Q_n,\Z)$ all follow from the comparison map and the values in Table
 \ref{tab:actions}.
 As for the action on $\Hoh^4(B Q_n,\Z)$, as in the proof
 of Proposition~\ref{pr:C2cohoPGLn}, this can be deduced
 from the action on $\Hoh^*(B (T_n\times T_n),\Z)=\Z[\theta'_1,\dots,\theta'_n,\theta''_1,\dots,
 \theta''_n]$, which is given by $\theta'_i\mapsto -\theta''_i$ and $\theta''_i\mapsto -\theta'_i$.
\end{proof}

\begin{remark}\label{RM:generators-for-H-two}
 From Figures \ref{fig:SSS0} and \ref{fig:SSS1}, we deduce that the maps $B \PGL_n \to B^2 \mu_n$ and $B Q_n \to B^2
 \mu_n$ both represent generators of the groups $\Hoh^2( B\PGL_n, \ZZ/n) \iso \ZZ/n$ and $\Hoh^2(B Q_n, \ZZ/n)\iso \ZZ/n$, respectively. 
 Moreover, the image of the former class under the 
 Bockstein map is a generator of $\Hoh^3( B\PGL_n, \ZZ)$,
 which is nothing but the tautological Brauer class $\alpha$ of $ B\PGL_n$. That is, if $r:X \to B\PGL_n$ is the classifying map for a $\PGL_n$-bundle, or equivalently, a degree-$n$ topological Azumaya algebra, then the Brauer class of that algebra is $r^*(\alpha)$.
\end{remark}

Our purpose in introducing the group $Q_n$ is to construct a group 
which is as close to $\PGL_n \times \PGL_n$ (with the
interchange action) as possible, but for which the transfer of all classes in $\Hoh^2(BQ_n, \ZZ/n)$ vanish.

\begin{proposition} \label{pr:coresQ0} Let $Q_n$ be as constructed above. Give the space $B Q_n \times EC_2$ the diagonal
 $C_2$-action. Then the transfer map, $\Hoh^2(B Q_n \times EC_2, \ZZ/n)
 \to \Hoh^2(B Q_n \times^{C_2} EC_2, \ZZ/n)$, considered at the end of Subsection \ref{subsec:free-action}, vanishes.
\end{proposition}
The space $BQ_n \times EC_2$ is weakly equivalent to $BQ_n$, but carries a free $C_2$-action.

\begin{proof}
 The action of $C_2$ on $B Q_n\times EC_2$ is free. 
	As noted in Remark~\ref{RM:generators-for-H-two},
 one generator $\alpha$ of $\Hoh^2(Q_n \times EC_2; \ZZ/n) \iso
 \Hoh^2(Q_n; \ZZ/n)$ is given by the map $BQ_n \to B^2 \mu_n$
 arising from the short exact sequence defining $Q_n$.
 This map is $C_2$-equivariant when $\mu_n$, and therefore $B^2
 \mu_n$, is given the standard involution, and therefore the result follows from Proposition \ref{pr:whenCores0}.
\end{proof}

\begin{proposition} \label{pr:Obstruction}
 Let $n$ be an even integer, and let $a$ be an odd integer. Suppose $f: X \to BQ_n$ is a $C_2$-equivariant map and a $6$-equivalence. Then there is no
 $C_2$-equivariant map $g:X \to B\PGL_{an}$ inducing a surjection on $\Hoh^3(\,\cdot\,, \ZZ)$.
\end{proposition}

\begin{proof}
	For the sake of contradiction, suppose that $g$ exists.
	
	By Remark~\ref{RM:generators-for-H-two},
	the comoposition 
	\[ X\to BQ_n\to B^2\mu_n\to B^2\units{\C}=B^3\Z,\]
	induced by $f$ and the inclusion $\mu_n \to \units{\C}$,
	represents a generator $\xi$ of $\Hoh^3(X,\Z)=\Hoh^3(B Q_n,\Z)=\Z/n$.
	As a result, there is $t\in\Z$ such that the composition
	\[X	\xrightarrow{g} B\PGL_{an}\to B^2\mu_{an} \to B^2\units{\C}=B^3\Z\]
	represents $t\cdot \xi$. Consequently,	
	the map $X\to B Q_n\to B^2\mu_n$ fits into 
 a homotopy-commutative square 
 \[ \xymatrix{ X \ar[r] \ar[d]^g & B^2 \mu_n \ar[d]^s \\ B\PGL_{an} \ar[r]
 & B^2 \mu_{an} } \]
 in which $s$ is the composition of $B^2\mu_n\to B^2\mu_{an}$ and
 the map $B^2\mu_{an}\to B^2\mu_{an}$ induced by $x\mapsto x^t:\mu_{an}\to \mu_{an}$.
 We extend this square into a homotopy commutative diagram 
\[ \xymatrix{ F \ar[r] \ar@{-->}[d] & X \ar[r] \ar[d]^g & B^2 \mu_n \ar[d]^s \\ B\SL_{an} \ar[r] & B\PGL_{an} \ar[r]
 & B^2 \mu_{an} } \] where both rows are homotopy fibre sequences, so $F$ is the homotopy fibre of $X \to B^2
\mu_n$. Strictly speaking, we carry this out in the (fine) $C_2$-equivariant model structure on topological spaces, using the dual
of \cite[Prop.~6.3.5]{hovey_model_1999} to deduce the existence of the dashed arrow in that category, so that it may be assumed to be
$C_2$-equivariant. Moreover, the space $F$ appearing in this argument has the appropriate non-equivariant homotopy type, since the
functor forgetting the $C_2$-action is a right Quillen functor, and therefore preserves fibre sequences.

Each of the two fibre sequences is associated to a Serre spectral sequence in cohomology. In the case of the lower row,
the $\Eoh_2$-page is represented in Figure \ref{fig:SSS0}, whereas in the case of the upper row, since $X$ is
$6$-equivalent to $B Q_n$, it is isomorphic on the $\Eoh_2$-page to the spectral sequence represented in Figure
\ref{fig:SSS1}. There is an induced map between these spectral sequences, and this map restricts to the following on the
$\Eoh_2^{\ast, 0}$-line:
\[ \xymatrix{ \ZZ \ar@{=}[d] & 0 & 0 & \ZZ/{an} \ar@{->>}[d] & 0 & \kappa\ZZ/(2an) \ar@{-->>}[d] & \dots \\ \ZZ & 0 & 0 & \ZZ/n
 & 0 & \kappa\ZZ/(2n) & \dots }. \]
We know the map on $\Hoh^5(\,\cdot\,, \ZZ)$ is surjective, because in each case, the group is generated by a class $\kappa$
for which $2 \kappa$ is $\beta( \iota^2)$, obtained by taking the canonical class $\iota$ in $\Hoh^2(B^2 \mu_n, \ZZ/n)$, resp.\
$\Hoh^2(B^2 \mu_{an}, \ZZ/an)$, squaring it, and applying the Bockstein map with image $\Hoh^5(\cdot,
\ZZ)$.
This may be deduced from \cite{cartan_determination_1954}, or from the Serre spectral sequence associated to the
path-loop fibration $B\mu_n \to \ast \to B^2 \mu_n$.

Since the map $g^*$ of spectral sequences is compatible with the $C_2$-action, 
it induces the following commutative square
\[
\xymatrix{
c_2 \Z \ar[r]^{d_5} \ar[d]^{g^*} &
\kappa \Z/(2an) \ar[d]^{g^*} \\
c'_2\Z\oplus c''_2\Z \ar[r]^{d_5} &
\kappa \Z/(2n).}
\] 
in which all arrows are $C_2$-equivariant.
Furthermore, the proofs of Propositions~\ref{pr:C2cohoPGLn} and~\ref{pr:C2cohoQ}
imply that both horizontal maps are surjective, that $C_2$
acts trivially on $c_2\Z$, $\kappa \Z/(2an)$ and $\kappa \Z/(2n)$,
and that
the non-trivial
element of $C_2$ interchanges $c'_2$ and $c''_2$. 
Now, $g^*(c_2)$ lies in
the $C_2$-fixed subgroup of $ c_2'\ZZ \oplus c_2''\ZZ$, which is to say $g^*(c_2) = m c_2'+m c_2''$ for some integer 
$m$. Then $d_5(g^*(c_2))$ is $2m$ times a generator of $\kappa\ZZ/(2n)$, and
hence not a generator of $\kappa \ZZ/(2n)$. On the other hand, $g^*(d_5(c_2))$ is a generator of
$\ZZ/(2n)$ by the previous paragraph, a contradiction. 
\end{proof}

\subsection{An Algebraic Counterexample}

\label{sec:algcoex}

In this section, we consider complex algebraic varieties.
In particular, all algebraic groups are complex algebraic groups.
Cohomology is understood to be \'etale cohomology in the context of varieties
and singular cohomology in the context of topological spaces.

A $C_2$-action on a variety $X$ will be called \emph{free} if there exists
a $C_2$-torsor $X\to Y$. In this case, $Y$ coincides with categorical quotient $X/C_2$
in the category
of varieties. Furthermore, if $C_2$ acts freely on $X$, then it also acts freely on $X(\C)$.
The converse holds when $X$ is affine or projective, see
Example~\ref{EX:important-exact-quo-scheme} and Proposition~\ref{PR:ramification-in-schemes}, 
but not in general.

\medskip

Fix an even positive integer $n$. We define the complex algebraic group $Q_n$ by means of the short exact sequence
\[ 1 \to \mu_n \xrightarrow{x\mapsto (x,x)} \SL_n \times \SL_n \to Q_n \to 1 \] 
so that $Q_n(\C)$ is the group $Q_n$ considered in the previous subsection.
There is a natural map $\Hoh^1 (\,\cdot\,, Q_n) \to
\Hoh^2(\,\cdot\,,\mu_n)$. Composing with the map $\Hoh^2(\,\cdot\,,\mu_n)\to \Hoh^2(\,\cdot\,,\Gm)$
induced by the inclusion $\mu_n\to\Gm$ allows us
to associate with every $Q_n$-torsor $P\to X$ an $n$-torsion class in
$\Hoh^2(X,\Gm)$. This association is natural, and is, in particular,
compatible with complex realization.

The first projection $\pi_1:\SL_n\times \SL_n\to \SL_n$ induces
a group homomorphism $\pi_1:Q_n\to \PGL_n$ (it is not $C_2$-equivariant). Using
this map,
we associate to every $Q_n$-torsor $P$ a $\PGL_n$-torsor, namely, $P\times^{Q_n} \PGL_n$.

\begin{lemma}
	With the previous notation, let $P\to X$ be a $Q_n$-torsor,
	let $\alpha$ be its associated class in $\Hoh^2(X,\Gm)$, and let
	$S\to X$ be its associated $\PGL_n$-torsor.
	Then $\alpha$ is the image of $S$ under the canonical map $\Hoh^1(X,\PGL_n)\to \Hoh^2(X,\Gm)$.
	In particular, $\alpha\in \Br(X)$.
\end{lemma}

\begin{proof}
	This follows by considering the following 
	morphism of short exact sequences
	and the induced morphism between the associated cohomology
	exact sequences.
	\[
	\xymatrix{
	\mu_n \ar[r]\ar@{^(->}[d] &
	\SL_n\times \SL_n \ar[r]\ar[d]^{\pi_1} &
	Q_n \ar[d]^{\pi_1} \\
	\Gm \ar[r] &
	\GL_n \ar[r] &
	\PGL_n
	} 
	\]
	Note that the vertical maps are not necessarily $C_2$-equivariant.
\end{proof}

\begin{proposition}\label{PR:Q-n-torsor}
 Maintaining the previous notation, there exists a smooth affine complex variety $X$ with free $C_2$-action,
	a $Q_n$-torsor $P\to X$ and a map $f:X(\CC)\to B Q_n(\C)$
	such that the following hold:
	\begin{enumerate}[label=(\roman*)]
		\item \label{item:PR:P-f-equivariant} The map $f:X(\CC)\to B Q_n(\C)$ is $C_2$-equivariant and a $6$-equivalence.
		\item \label{item:PR:P-f-corresponds} The homotopy class of $f$ corresponds to the principal $Q_n(\C)$-bundle
		$P(\CC)\to X(\CC)$.
		\item \label{item:PR:P-f-transfer} The Brauer class $\alpha\in \Br(X)\subseteq \Hoh^2(X,\Gm)$ associated with
		$P\to X$ has trivial image under 
		$\transf:\Hoh^2(X,\Gm)\to \Hoh^2(X/C_2,\Gm)$.
	\end{enumerate}	
\end{proposition}

	For later reference, and in keeping with the previous parts of this paper, we denote $X/C_2$ by $Y$.

\begin{proof}
As in Subsection~\ref{subsec:coh-counterex}, the group $Q_n$ is an affine algebraic group equipped with an algebraic $C_2$-action. Consequently, the split exact
extension $\Gamma$ in
\[ 1 \to Q_n \to \Gamma \to C_2 \to 1 \] is also an affine algebraic group, \cite{molnar_semi-direct_1977}*{Ex.~2.15
(c)}.

Therefore it is possible to follow \cite{totaro_chow_1999} and construct affine spaces $V$ on
which $\Gamma$ acts and such that 
$V$ becomes a $\Gamma$-torsor
after removing a locus, $Z$, of arbitrarily large codimension. 
In particular, $V-Z$ is a $Q_n$-torsor. 
Choose $V$ so that $Z$ has (complex) codimension at
least $4$.

Let $P=(V-Z)$ and $X=(V-Z)/Q_n$, and note
that both $P$ and $X$ carry $C_2$-actions.
The $C_2$-action on $X$ is free since
$V-Z$ is a $\Gamma$-torsor.
Moreover, one checks directly that $P(\C)\to X(\C)$
is a principal $Q_n(\C)$-bundle with involution, see Definition~\ref{def:twistedGbundle}.
Since $C_2$ acts freely on $X(\C)$, 
Proposition~\ref{prop:freeCase} implies that
 this principal bundle is represented by a 
map $f:X(\C)\to B Q_n(\C)$, which satisfies
conditions \ref{item:PR:P-f-equivariant} and \ref{item:PR:P-f-corresponds}.

By means of the equivariant Jouanolou device, \cite{hoyois_six_2017}*{Prop.~2.20}, we may assume that $X$ is a smooth affine
variety with these properties.\uriyaf{Can we say something about the dimension using Lefschez's theorem
despite the $C_2$-action?}

Let $\alpha\in \Br(X)$ denote the Brauer class associated with $P\to X$.
It remains to show that $\transf_{X/Y}(\alpha)=0$ in $\Br(Y)$, where $Y=X/C_2$.

To that end, let $\xi$ be the image of $P$ under $\Hoh^1(X,
Q) \to \Hoh^2(X, \mu_n)$, and similarly define $\xi(\C)$ as the image of $P(\C)$ under the analogous map in singular
cohomology. It is enough to check that $\transf(\xi)\in \Hoh^2(Y,\mu_n)$ vanishes.
There is a commutative diagram
\begin{equation*}
 \xymatrix@R10pt@C15pt{ & \Hoh_\et^2(X, \mu_n) \ar^{\transf}[rr] \ar'[d][dd] \ar_-{\iso}[dl] && \Hoh_\et(Y, \mu_n) \ar[dd]
\ar_-{\iso}[dl] \\ \Hoh^2(X(\CC), \ZZ/n) \ar[rr] \ar[dd] && \Hoh^2(Y(\C), \ZZ/n) \ar[dd] & \\ & \Hoh_\et^2(X, \Gm)
\ar'[r][rr] \ar[ld] && \Hoh_\et^2(Y,\Gm) \ar[ld]\\ \Hoh^3(X(\CC), \ZZ) \ar[rr] && \Hoh^3(Y(\CC),\ZZ) & }
\end{equation*} where each map from left to right is a transfer map, each map from back to front is a
complex-realization map, and the maps from top to bottom are induced by the inclusion $\mu_n \to \Gm$. The two indicated
maps are isomorphism by Artin's theorem. We also remark that $\Hoh^2(\,\cdot\,, \CC^\times) \cong \Hoh^3(\,\cdot\,, \ZZ)$, where
the first group is understood as sheaf cohomology with coefficients in the sheaf of nonvanishing continuous
complex-valued functions.
Now, the transfer of $\xi(\C) \in \Hoh^2(X(\C), \ZZ/n)$ is easily seen to be $0$ by comparison with $\Hoh^2(B Q_n \times
EC_2, \ZZ/n)$, where it is known to vanish by Proposition \ref{pr:coresQ0}.
This completes the proof.
\end{proof}

\begin{theorem} \label{th:mainCounterexample} 
For any even integer $n$, there exists a quadratic \'etale map $X \to Y $ of smooth affine complex varieties and an Azumaya algebra $A$ of degree $n$ over $X$ such that:
\begin{enumerate}[label=(\roman*)]
	\item The period and index of $\alpha=[A]$ are both $n$.
	\item $ \transf_{X/Y}(\alpha)=0$ in $\Br(Y)$.
	\item The degree of any Azumaya algebra Brauer equivalent to $A$ and admitting a $\lambda$-involution is
 divisible by $2n$---here $\lambda$ denotes the non-trivial involution of $X$ over $Y$.
\end{enumerate} 
\end{theorem}

In particular, we see that the minimal degree of an Azumaya algebra Brauer equivalent to $A$ and supporting a
$\lambda$-involution is at least $2n$. This bound is sharp by Theorem~\ref{TH:Saltman-ramified}.

\begin{proof}
 Construct $P\to X$ as in Proposition~\ref{PR:Q-n-torsor}, and let $A$ be the Azumaya algebra corresponding to the
 $\PGL_n$-torsor associated to $P$ by means of $\pi_1:Q_n \to \PGL_n$. The Azumaya algebra $A$ has degree $n$, and 
 the complex reaization of its Brauer class is a generator of
 $\Hoh^3(X(\C), \ZZ) \iso \ZZ/n$, since the map $BQ_n(\C) \to B \PGL_n(\C)$ induces an isomorphism on $\Hoh^3(\cdot, \ZZ)$,
 see Remark~\ref{RM:generators-for-H-two} and the proof of Proposition~\ref{pr:C2cohoQ}.
 In particular, the period and index of $\alpha$ must both be $n$.

 Let $a$ be an odd integer and suppose $A$ were equivalent to an Azumaya algebra of degree $an$ carrying a
 $\lambda$-involution. Then, by Proposition~\ref{pr:equivPGLn}, the complex realization of this algebra would
 correspond to a topological $\PGL_{an}(\C)$-bundle with involution. By Proposition \ref{prop:freeCase}, there would
 be a $C_2$-equivariant map $g: X(\CC) \to B\PGL_{an}(\CC)$ such that $\alpha \in \Br(X(\CC))$ was the image of the
 canonical Brauer class in $\Hoh^3( B\PGL_{an}(\CC),\Z)$, and therefore $g$ would induce a surjection in $\Hoh^3(\cdot,
 \ZZ)$, since the image of $g^*$ would contain $\alpha$.
 This is forbidden by Proposition~\ref{pr:Obstruction}.
\end{proof}

\begin{question}
	Does Theorem~\ref{th:mainCounterexample} hold when $n$ is odd?
\end{question}

\appendix
\section{The Stalks of The Ring of Continuous Complex Functions}

\label{AX:C_is_s_henselian}
Let $X$ be a topological space; we work throughout on the small site of $X$. Let $\sh O$ denote the sheaf of continuous
$\C$-valued functions on $X$.

Let $p$ be a point of $X$ and consider $p^* \sh O$. It is a local ring with maximal ideal denoted $\m$. An element $f \in p^*\sh
O$ is
the germ of a continuous $\C$-valued function at $p$, and the class $\bar f \in p^*\sh O/ \m \iso \C$ is the complex number
$f(p)$.

\begin{proposition}
 The local ring $p^* \sh O$ is strictly henselian.
\end{proposition}
\begin{proof}
 It suffices to prove the ring is a henselian ring as the residue field is $\C$.

Consider $R_n:=\C^n $ as an ordered set of roots of a degree-$n$ monic polynomial. There is a permutation action of the
symmetric group $\Sigma_n$ on $R_n$, and there is a homeomorphism $R_n / \Sigma_n \to \C^n$, where the map takes
$(\alpha_1, \dots , \alpha_n)$ to the coefficients of the polynomial $\prod_{i=1}^n (t- \alpha_i)$,
\cite{bhatia_space_1983}.

We embed $\Sigma_{n-1} \subset \Sigma_n$ as the permutations fixing the first element, and form $R_n/ \Sigma_{n-1}$.

The space $R_n/\Sigma_{n-1} = \C \times R_{n-1}/\Sigma_{n-1}$ represents monic polynomials of degree $n$ and a
distinguished `first' root, $(t-\alpha_1) q(t)$. There is a closed subset $\Gamma \subset R_n/\Sigma_{n-1}$, the locus where $q(\alpha_1) =
0$. Then $R_n/\Sigma_{n-1} - \Gamma \subset R_n/ \Sigma_{n-1}$ is an open subset representing the set of monic,
degree-$n$ polynomials having a distinguished `first' root which is not repeated. Since quotient maps of spaces given by
finite group actions are open maps, in the following diagram, every map appearing is an open map:
\begin{equation*}
 \xymatrix{ \C \times R_{n-1}/\Sigma_{n-1} - \Gamma \ar[d] \ar@/_5em/_{\pi}[dd] \\ \C \times R_{n-1}/\Sigma_{n-1} \ar[d]
 & R_n \ar[l] \ar[dl] \\ R_n/\Sigma_n. }
\end{equation*} We denote the composite map $\C \times R_{n-1}/\Sigma_{n-1} - \Gamma \to R_n/ \Sigma_n$ by $\pi$. It
sends a pair $(\alpha_1, q(t))$, for which $q(\alpha_1) \neq 0$, to $(t-\alpha_1)q(t)$.

Let $(\alpha_1, q(t)) \in \C \times R_{n-1}/\Sigma_{n-1} -\Gamma $ be such a pair. We claim that there exists an open
neighbourhood $V$ of $(\alpha_1, q(t))$ such that $\pi|_V : V \to R_n/\Sigma_n$ is a homeomorphism onto its image. Choose
an open neighbourhood of $(\alpha_1, q(t)) \in \C \times R_{n-1}/\Sigma_{n-1}$ of the form $V = B(\alpha_1; \epsilon)
\times B(q(t);\epsilon)$, being the product of an $\epsilon$-ball around $\alpha_1$ and around $q(t)$, where $\epsilon$ is
sufficiently small that none of the polynomials in $B(q(t);\epsilon)$ has any of the complex numbers in $B(\alpha_1;
\epsilon)$ as roots. It is immediate that $\pi$ is injective when restricted to this open set in $ \C \times
R_{n-1}/\Sigma_{n-1} -\Gamma$. Since $\pi$ is an open map, $\pi|_V$ is a homeomorphism onto its image, establishing the claim.

\smallskip
Suppose we are given a polynomial $h(t) \in p^*\sh O[t]$. Suppose
further that a non-repeated root, $\bar \alpha_1$, of $\bar h(t)$ is given, where $\bar h(t)$ is the reduction of $h(t)$
to $(p^*\calO/\frakm)[t]=\CC[t]$. We can write $\bar q(t) =
\bar h(t) / (t-\bar \alpha_1)$ in $\CC[t]$. Note that we do not yet assert that $\bar q(t)$ and $\bar \alpha_1$ are the reductions
of any specific elements in $p^*\sh O [t]$ or $p^*\sh O$. To prove that the ring is Henselian, we must find an element
$\alpha_1$ lifting $\bar \alpha_1$ and satisfying $h(\alpha_1)=0$.

The germ $h(t)$ has an extension to an open neighbourhood $U \ni p$. 

We have the data of a diagram
\begin{equation*}
 \xymatrix{ p \ar[d] \ar^(0.2){(\bar \alpha_1 \bar ,q)}[r] & \C \times R_{n-1}/\Sigma_{n-1} - \Gamma \ar^\pi[d] \\ U \ar^h[r]
 & R/\Sigma_n. }
\end{equation*} Around the image of $(\bar \alpha_1,\bar q)$ in $ \C \times R_{n-1}/\Sigma_{n-1}- \Gamma$ we can find an
open set $V$ such that $\pi|_V: V \to R/\Sigma_n$ is a homeomorphism onto the image, $W$. Then $W$ is an open set in
$R/\Sigma_n$ containing $\bar \alpha_1 \bar q = \bar h$. Since $h(p) = \bar h$, the preimage $h^{-1}(W)$ is an open subset of $U$ containing $p$. Since
$\pi_V : V \to W$ is a homeomorphism, we may lift the map
\begin{equation*}
 \xymatrix{ p \ar[d] \ar^(0.23){(\alpha_1,q)}[r] & \C \times R_{n-1}/\Sigma_{n-1}- \Gamma \ar^\pi[d] \\ h^{-1}(W)
 \ar@{-->}[ur] \ar^h[r] & R/\Sigma_n. }
\end{equation*} as indicated.

That is to say, there is a neighbourhood, $h^{-1}(W)$, of $p$ such that the factorization $\bar h(t) =( t- \bar \alpha_1)\bar q(t)$
can be extended on $h^{-1}(W)$ to a factorization $h(t) = (t-\alpha_1)q(t)$. In particular, the class of $\alpha_1$ in
$p^* \sh O$ is a root of the polynomial $h(t)$ extending $\bar \alpha_1$. This proves Hensel's lemma for $p^* \sh O$.
\end{proof}

\bibliography{Involutions,Standard_BibTex}

 \end{document}

%%% Local Variables:
%%% mode: latex
%%% TeX-master: t
%%% End:

%% file: header.tex
\usepackage{ifdraft}

\ifdraft{\usepackage[notref, notcite]{showkeys}}{}

\ifdraft{\usepackage[letterpaper, lmargin=1.5in]{geometry}}{\usepackage[letterpaper]{geometry}}
\usepackage{xargs}
\usepackage{amsmath, amsthm, mathrsfs, amssymb}
\usepackage{xypic}
\usepackage{graphicx}
\usepackage[usenames, dvipsnames, svgnames, table]{xcolor}
\usepackage{pdfsync}
\usepackage{bm}
\usepackage{enumitem}   
\usepackage[colorinlistoftodos,prependcaption,textsize=tiny]{todonotes} % This package can be troublesome
\usepackage[final, pagebackref=true, bookmarks=true, colorlinks=true, citecolor=green!50!black, linkcolor=red!50!black, ]{hyperref}
\usepackage[alphabetic]{amsrefs}
\ifdraft{
}

\theoremstyle{plain}
\newtheorem{thm}{Theorem}[section] % can be changed to 'subsection' for 3-number numbering.
\newtheorem{theorem}[thm]{Theorem}

\newtheorem{lem}[thm]{Lemma}
\newtheorem{lemma}[thm]{Lemma}
\newtheorem{cor}[thm]{Corollary}
\newtheorem{corollary}[thm]{Corollary}

\newtheorem{prp}[thm]{Proposition}
\newtheorem{proposition}[thm]{Proposition}

\newtheorem*{claim*}{Claim} %Claim
\newtheorem*{thm*}{Theorem}

\theoremstyle{definition}

\newtheorem{definition}[thm]{Definition}
\newtheorem{construction}[thm]{Construction}
\newtheorem{notation}[thm]{Notation}

\theoremstyle{remark}

\newtheorem{remark}[thm]{Remark}

\newtheorem{example}[thm]{Example}

\newtheorem{question}[thm]{Question}
\newtheorem{problem}[thm]{Problem}

\newcommand{\rank}{\operatorname{rank}}

\newcommand{\SL}{\operatorname{SL}}

\newcommand{\Gl}{\operatorname{GL}}
\newcommand{\Mat}{\operatorname{Mat}}

\newcommand{\U}{\operatorname{U}}

\newcommand{\PGL}{\operatorname{PGL}}
\newcommand{\PO}{\operatorname{PO}}
\newcommand{\PSp}{\operatorname{PSp}}

\newcommand{\Hom}{\operatorname{Hom}}
\newcommand{\colim}{\operatornamewithlimits{colim}}
\newcommand{\cat}[1]{\text{\bf #1}}

\newcommand{\tensor}{\otimes}

\newcommand{\spec}{\operatorname{Spec}}

\newcommand{\Spec}{\spec}

\newcommand{\m}{\mathfrak{m}}

\newcommand{\Gm}{\mathbb{G}_m}
 
\newcommand{\pt}{\ast}

\newcommand{\sHom}{\mathcal{H}\!om}

\newcommand{\sAut}{\mathcal{A}ut}

\newcommand{\A}{\mathbb{A}}

\newcommand{\sh}[1]{\mathcal{#1}}

\newcommand{\tr}{\text{\rm tr}}

\providecommand{\deg}{\operatorname{deg}}

\newcommand{\op}{\text{op}}

\newcommand{\GL}{\Gl}

\newcommand{\PU}{\operatorname{PU}}

\newcommand{\Br}{\operatorname{Br}}

\newcommand{\Eoh}{\mathrm{E}}

\newcommand{\Hoh}{\mathrm{H}}

\renewcommand{\top}{\mathrm{top}}

\newcommand{\Map}{\operatorname{Map}}

\newcommand{\id}{\mathrm{id}}
\newcommand{\isom}{\cong}
\newcommand{\iso}{\isom}

\newcommand{\C}{\mathbb{C}}
\newcommand{\N}{\mathbb{N}}
\newcommand{\Q}{\mathbb{Q}}

\newcommand{\Z}{\mathbb{Z}}

\newcommand{\ZZ}{\Z}
\newcommand{\CC}{\C}

\newcommand{\im}{\operatorname{im}}

\newcommand{\Aut}{\operatorname{Aut}}

\newcommand{\et}{\text{\'et}}

\newcommand{\Sp}{{\mathrm{Sp}}}

\newcommand{\Zar}{\operatorname{Zar}} % Zariski

\newcommand{\Fun}{\mathrm{Fun}}
\renewcommand{\deg}{\operatorname{deg}}

\newcommand{\Tors}{\cat{Tors}}

\newcommand{\Sh}{\cat{Sh}}

\newcommand{\fppf}{\mathrm{fppf}}

\newcommand{\Jac}{\mathrm{Jac}}

\newcommand{\ct}{\mathrm{ct}} % coarse type

\newcommand{\cores}{\operatorname{cores}}

\newcommand{\transf}{\operatorname{transf}}

\newcommand{\Az}{{\cat{Az}}}

\newcommand{\cont}{{\mathcal{C}}}

\newcommand{\Types}[2][]{{\mathrm{Typ}_{#1}({#2})}}
\newcommand{\CTypes}[2][]{{\mathrm{cTyp}_{#1}({#2})}}
\newcommand{\tp}{\mathrm{t}}

\newcommand{\cosk}{{\mathrm{cosk}}}

\newcommand{\syp}{{\mathrm{sp}}} % the symplectic involution

%% file: UriyaHeader.tex
\newcommand{\units}[1]{{#1^\times}}
\newcommand{\veps}{\varepsilon}

\newcommand{\nMat}[2]{\mathrm{M}_{#2}(#1)} % Matrix algebra

\newcommand{\calC}{{\mathcal{C}}}

\newcommand{\calO}{{\mathcal{O}}}

\newcommand{\bfW}{{\mathbf{W}}}
\newcommand{\bfX}{{\mathbf{X}}}
\newcommand{\bfY}{{\mathbf{Y}}}

\newcommand{\fraka}{{\mathfrak{a}}}
\newcommand{\frakb}{{\mathfrak{b}}}

\newcommand{\frakm}{{\mathfrak{m}}}

\newcommand{\frakp}{{\mathfrak{p}}}
\newcommand{\frakq}{{\mathfrak{q}}}

\newcommand{\frakM}{{\mathfrak{M}}}

\newcommand{\suchthat}{\,:\,}
\newcommand{\where}{\,|\,}
\newcommand{\quo}[1]{\overline{#1}}

\newcommand{\Simp}{{\boldsymbol{\Delta}}}

%% file: Involutions_of_Azumaya_Algebras_v2.bbl
\begin{bibdiv}
\begin{biblist}

\bib{auel_fibrations_2014}{article}{
      author={Auel, Asher},
      author={Bernardara, Marcello},
      author={Bolognesi, Michele},
       title={Fibrations in complete intersections of quadrics, {C}lifford
  algebras, derived categories, and rationality problems},
        date={2014},
        ISSN={0021-7824},
     journal={J. Math. Pures Appl. (9)},
      volume={102},
      number={1},
       pages={249\ndash 291},
  url={http://dx.doi.org.ezproxy.library.ubc.ca/10.1016/j.matpur.2013.11.009},
      review={\MR{3212256}},
}

\bib{asher_auel_azumaya_2017}{article}{
      author={Auel, Asher},
      author={First, Uriya~A.},
      author={Williams, Ben},
       title={Azumaya algebras without involution},
        date={2019},
        ISSN={1435-9855},
     journal={J. Eur. Math. Soc. (JEMS)},
      volume={21},
      number={3},
       pages={897\ndash 921},
         url={https://doi.org/10.4171/JEMS/855},
      review={\MR{3908769}},
}

\bib{auslander_brauer_1960}{article}{
      author={Auslander, Maurice},
      author={Goldman, Oscar},
       title={The {{Brauer}} group of a commutative ring},
        date={1960},
        ISSN={0002-9947},
     journal={Transactions of the American Mathematical Society},
      volume={97},
       pages={367\ndash 409},
}

\bib{artin_theorie_1972-1}{book}{
      editor={Artin, M.},
      editor={Grothendieck, A.},
      editor={Verdier, J.~L.},
       title={Th{\'e}orie des topos et cohomologie {\'e}tale des sch{\'e}mas.
  {{Tome}} 1: {{Th{\'e}orie}} des topos},
      series={Lecture Notes in Mathematics, Vol. 269},
   publisher={{Springer-Verlag}},
     address={Berlin},
        date={1972},
        note={S{\'e}minaire de G{\'e}om{\'e}trie Alg{\'e}brique du Bois-Marie
  1963\textendash{}1964 (SGA 4), Dirig{\'e} par M. Artin, A. Grothendieck, et
  J. L. Verdier. Avec la collaboration de N. Bourbaki, P. Deligne et B.
  Saint-Donat},
}

\bib{artin_theorie_1972}{book}{
      editor={Artin, M.},
      editor={Grothendieck, A.},
      editor={Verdier, J.~L.},
       title={Th{\'e}orie des topos et cohomologie {\'e}tale des sch{\'e}mas.
  {{Tome}} 2},
      series={Lecture Notes in Mathematics, Vol. 270},
   publisher={{Springer-Verlag}},
     address={Berlin},
        date={1972},
        note={S{\'e}minaire de G{\'e}om{\'e}trie Alg{\'e}brique du Bois-Marie
  1963\textendash{}1964 (SGA 4), Dirig{\'e} par M. Artin, A. Grothendieck et J.
  L. Verdier. Avec la collaboration de N. Bourbaki, P. Deligne et B.
  Saint-Donat},
}

\bib{albert_symmetric_1938}{article}{
      author={Albert, A.~Adrian},
       title={Symmetric and alternate matrices in an arbitrary field. {{I}}},
        date={1938},
        ISSN={0002-9947},
     journal={Trans. Amer. Math. Soc.},
      volume={43},
      number={3},
       pages={386\ndash 436},
      review={\MR{1501952}},
}

\bib{aguilar_transfer_2010}{article}{
      author={Aguilar, Marcelo~A.},
      author={Prieto, Carlos},
       title={The transfer for ramified covering {{$G$}}-maps},
        date={2010-03},
        ISSN={0933-7741},
     journal={Forum Mathematicum},
       pages={100319043555021\ndash },
}

\bib{auel_parimala_suresh_2015}{article}{
      author={Auel, Asher},
      author={Parimala, R.},
      author={Suresh, V.},
       title={Quadric surface bundles over surfaces},
        date={2015},
        ISSN={1431-0635},
     journal={Doc. Math.},
      number={Extra vol.: Alexander S. Merkurjev's sixtieth birthday},
       pages={31\ndash 70},
      review={\MR{3404375}},
}

\bib{artin_joins_1971}{article}{
      author={Artin, M.},
       title={On the joins of {{Hensel}} rings},
        date={1971},
        ISSN={0001-8708},
     journal={Advances in Mathematics},
      volume={7},
       pages={282\ndash 296 (1971)},
}

\bib{auel_mpim_2011}{article}{
      author={Auel, Asher},
       title={Clifford invariants of line bundle-valued quadratic forms},
        date={2011},
}

\bib{antieau_period-index_2014}{article}{
      author={Antieau, Benjamin},
      author={Williams, Ben},
       title={The period-index problem for twisted topological {{{\emph{K}}}}
  \textendash{}theory},
    language={en},
        date={2014-04},
        ISSN={1364-0380, 1465-3060},
     journal={Geometry \& Topology},
      volume={18},
      number={2},
       pages={1115\ndash 1148},
}

\bib{antieau_topological_2014}{article}{
      author={Antieau, Benjamin},
      author={Williams, Ben},
       title={The topological period-index problem over 6-complexes},
    language={en},
        date={2014-09},
        ISSN={1753-8416, 1753-8424},
     journal={Journal of Topology},
      volume={7},
      number={3},
       pages={617\ndash 640},
}

\bib{antieau_unramified_2014}{article}{
      author={Antieau, Benjamin},
      author={Williams, Ben},
       title={Unramified division algebras do not always contain
  \{\vphantom\}{A}\vphantom\{\}zumaya maximal orders},
    language={en},
        date={2014-07},
        ISSN={0020-9910, 1432-1297},
     journal={Invent. math.},
      volume={197},
      number={1},
       pages={47\ndash 56},
}

\bib{antieau_topology_2015}{article}{
      author={Antieau, Benjamin},
      author={Williams, Ben},
       title={Topology and {{Purity}} for {{Torsors}}},
        date={2015},
     journal={Documenta Mathematica},
      volume={20},
       pages={333\ndash 355},
}

\bib{bayer_17}{article}{
      author={Bayer-Fluckiger, Eva},
      author={First, Uriya~A.},
       title={Rationally isomorphic hermitian forms and torsors of some
  non-reductive groups},
        date={2017},
        ISSN={0001-8708},
     journal={Advances in Mathematics},
      volume={312},
       pages={150\ndash 184},
}

\bib{bosch_1990_Neron_models}{book}{
      author={Bosch, Siegfried},
      author={L\"{u}tkebohmert, Werner},
      author={Raynaud, Michel},
       title={N\'{e}ron models},
      series={Ergebnisse der Mathematik und ihrer Grenzgebiete (3) [Results in
  Mathematics and Related Areas (3)]},
   publisher={Springer-Verlag, Berlin},
        date={1990},
      volume={21},
        ISBN={3-540-50587-3},
         url={https://doi.org/10.1007/978-3-642-51438-8},
      review={\MR{1045822}},
}

\bib{bhatia_space_1983}{article}{
      author={Bhatia, Rajendra},
      author={Mukherjea, Kalyan~K.},
       title={The space of unordered tuples of complex numbers},
        date={1983-07},
        ISSN={0024-3795},
     journal={Linear Algebra and its Applications},
      volume={52},
       pages={765\ndash 768},
}

\bib{bruhat_I_72}{article}{
      author={Bruhat, Fran\c{c}ois},
      author={Tits, Jacques},
       title={Groupes r\'eductifs sur un corps local},
        date={1972},
        ISSN={0073-8301},
     journal={Inst. Hautes \'Etudes Sci. Publ. Math.},
      number={41},
       pages={5\ndash 251},
      review={\MR{0327923 (48 \#6265)}},
}

\bib{bruhat_II_84}{article}{
      author={Bruhat, Fran\c{c}ois},
      author={Tits, Jacques},
       title={Groupes r\'eductifs sur un corps local. {II}. {S}ch\'emas en
  groupes. {E}xistence d'une donn\'ee radicielle valu\'ee},
        date={1984},
        ISSN={0073-8301},
     journal={Inst. Hautes \'Etudes Sci. Publ. Math.},
      number={60},
       pages={197\ndash 376},
         url={http://www.numdam.org/item?id=PMIHES_1984__60__5_0},
      review={\MR{756316 (86c:20042)}},
}

\bib{bruhat_III_87}{article}{
      author={Bruhat, Fran\c{c}ois},
      author={Tits, Jacques},
       title={Groupes alg\'ebriques sur un corps local. {C}hapitre {III}.
  {C}ompl\'ements et applications \`a la cohomologie galoisienne},
        date={1987},
        ISSN={0040-8980},
     journal={J. Fac. Sci. Univ. Tokyo Sect. IA Math.},
      volume={34},
      number={3},
       pages={671\ndash 698},
      review={\MR{927605}},
}

\bib{cartan_determination_1954}{article}{
      author={Cartan, H.},
       title={D{\'e}termination des alg{\`e}bres
  \${{H}}\_*({{K}}($\backslash$pi,n);$\backslash$mathbf\{\vphantom\}{{Z}}\vphantom\{\})\$},
        date={1954/1955},
     journal={S{\'e}minaire H. Cartan},
      volume={7},
      number={1},
       pages={11\ndash 01\ndash 11\ndash 24},
}

\bib{de_jong_stacks_2017}{misc}{
      author={{de Jong}, A.~J.},
       title={Stacks {{Project}}},
         how={\url{http://stacks.math.columbia.edu/}},
        date={2017},
}

\bib{dugger_hypercovers_2004}{article}{
      author={Dugger, Daniel},
      author={Hollander, Sharon},
      author={Isaksen, Daniel~C.},
       title={Hypercovers and simplicial presheaves},
        date={2004},
        ISSN={0305-0041},
     journal={Mathematical Proceedings of the Cambridge Philosophical Society},
      volume={136},
      number={1},
       pages={9\ndash 51},
}

\bib{demeyer_separable_1971}{book}{
      author={DeMeyer, Frank},
      author={Ingraham, Edward},
       title={Separable algebras over commutative rings},
      series={Lecture Notes in Mathematics, Vol. 181},
   publisher={{Springer-Verlag}},
     address={Berlin},
        date={1971},
}

\bib{dwyer_singular_1984}{article}{
      author={Dwyer, W.~G.},
      author={Kan, D.~M.},
       title={Singular functors and realization functors},
        date={1984},
        ISSN={0019-3577},
     journal={Indagationes Mathematicae},
      volume={46},
      number={2},
       pages={147\ndash 153},
}

\bib{eisenbud_commutative_1995}{book}{
      author={Eisenbud, David},
       title={Commutative algebra},
      series={Graduate Texts in Mathematics},
   publisher={{Springer-Verlag}},
     address={New York},
        date={1995},
      volume={150},
        ISBN={0-387-94268-8 0-387-94269-6},
        note={With a view toward algebraic geometry},
}

\bib{elmendorf_systems_1983}{article}{
      author={Elmendorf, A.~D.},
       title={Systems of {{Fixed Point Sets}}},
        date={1983-05},
        ISSN={0002-9947},
     journal={Transactions of the American Mathematical Society},
      volume={277},
      number={1},
       pages={275\ndash 284},
}

\bib{ferrand_1998_norm_functors}{article}{
      author={Ferrand, Daniel},
       title={Un foncteur norme},
        date={1998},
        ISSN={0037-9484},
     journal={Bull. Soc. Math. France},
      volume={126},
      number={1},
       pages={1\ndash 49},
         url={http://www.numdam.org/item?id=BSMF_1998__126_1_1_0},
      review={\MR{1651380}},
}

\bib{first_15}{article}{
      author={First, Uriya~A.},
       title={Rings that are {M}orita equivalent to their opposites},
        date={2015},
        ISSN={0021-8693},
     journal={J. Algebra},
      volume={430},
       pages={26\ndash 61},
  url={http://dx.doi.org.ezproxy.library.ubc.ca/10.1016/j.jalgebra.2015.01.026},
      review={\MR{3323975}},
}

\bib{gille_gersten_2009}{article}{
      author={Gille, Stefan},
       title={A {G}ersten-{W}itt complex for {H}ermitian {W}itt groups of
  coherent algebras over schemes. {II}. {I}nvolution of the second kind},
        date={2009},
        ISSN={1865-2433},
     journal={J. K-Theory},
      volume={4},
      number={2},
       pages={347\ndash 377},
  url={http://dx.doi.org.ezproxy.library.ubc.ca/10.1017/is008008022jkt065},
      review={\MR{2551913}},
}

\bib{giraud_cohomologie_1971}{book}{
      author={Giraud, J.},
       title={Cohomologie non ab{\'e}lienne},
      series={Die Grundlehren der mathematischen Wissenschaften in
  Einzeldarstellungen mit besonderer Ber{\"u}cksichtigung der
  Anwendungsgebiete, Bd. 179},
   publisher={{Springer-Verlag}},
        date={1971},
      volume={Bd. 179.},
}

\bib{Guillou2017}{article}{
      author={Guillou, Bertrand},
      author={May, Peter},
      author={Merling, Mona},
       title={Categorical models for equivariant classifying spaces},
        date={2017-09},
        ISSN={1472-2747},
     journal={Algebraic \& Geometric Topology},
      volume={17},
      number={5},
       pages={2565\ndash 2602},
}

\bib{godement_topologie_1973}{book}{
      author={Godement, Roger},
       title={Topologie alg{\'e}brique et th{\'e}orie des faisceaux},
      series={Actualit{\'e}s Scientifiques et Industrielles},
   publisher={{Hermann}},
        date={1973},
      number={1252},
}

\bib{grothendieck_revetements_1971}{book}{
      author={Grothendieck, A.},
      author={Raynaud, M.},
       title={Rev{\^e}tements {\'e}tales et groupe fondamental},
      series={S{\'e}minaire de g{\'e}ometrie alg{\'e}brique du Bois Marie},
   publisher={{Springer-Verlag}},
     address={Berlin ; New York},
        date={1971},
      number={SGA 1},
        ISBN={978-3-540-05614-0 978-0-387-05614-2},
}

\bib{grothendieck_techniques_1960}{article}{
      author={Grothendieck, Alexander},
       title={Techniques de construction en g{\'e}om{\'e}trie analytique.
  {{II}}. {{G{\'e}n{\'e}ralit{\'e}s}} sur les espaces annel{\'e}s et les
  espaces analytiques},
    language={en},
        date={1960},
     journal={S{\'e}minaire Henri Cartan},
      volume={13},
      number={1},
       pages={1\ndash 14},
}

\bib{grothendieck_elements_1964}{article}{
      author={Grothendieck, A.},
       title={{\'E}l{\'e}ments de g{\'e}om{\'e}trie alg{\'e}brique
  (r{\'e}dig{\'e}s avec la collaboration de {{Jean Dieudonn{\'e}}}): {{IV}}.
  {{{\'E}tude}} locale des sch{\'e}mas et des morphismes de sch{\'e}mas.
  {{I}}},
        date={1964},
        ISSN={0073-8301},
     journal={Inst. Hautes {\'E}tudes Sci. Publ. Math.},
      number={20},
       pages={259},
      review={\MR{0173675}},
}

\bib{grothendieck_elements_1967}{article}{
      author={Grothendieck, A.},
       title={{\'E}l{\'e}ments de g{\'e}om{\'e}trie alg{\'e}brique
  (r{\'e}dig{\'e}s avec la collaboration de {{Jean Dieudonn{\'e}}}): {{IV}}.
  {{{\'E}tude}} locale des sch{\'e}mas et des morphismes de sch{\'e}mas
  {{IV}}},
        date={1967},
        ISSN={0073-8301},
     journal={Inst. Hautes {\'E}tudes Sci. Publ. Math.},
      number={32},
       pages={361},
      review={\MR{0238860}},
}

\bib{grothendieck_groupe_1968}{incollection}{
      author={Grothendieck, A.},
       title={Le groupe de {{Brauer}}. {{II}}. {{Th{\'e}orie}} cohomologique},
        date={1968},
   booktitle={Dix {{Expos{\'e}s}} sur la {{Cohomologie}} des {{Sch{\'e}mas}}},
   publisher={{North-Holland}},
     address={Amsterdam},
       pages={67\ndash 87},
}

\bib{grothendieck_groupe_1968-1}{incollection}{
      author={Grothendieck, Alexander},
       title={Le groupe de {{Brauer}}. {{I}}. {{Alg{\`e}bres}} d'{{Azumaya}} et
  interpr{\'e}tations diverses},
        date={1968},
   booktitle={Dix {{Expos{\'e}s}} sur la {{Cohomologie}} des {{Sch{\'e}mas}}},
   publisher={{North-Holland}},
     address={Amsterdam},
       pages={46\ndash 66},
}

\bib{grothendieck_1968_groupe_de_Brauer_III}{incollection}{
      author={Grothendieck, Alexander},
       title={Le groupe de {B}rauer. {III}. {E}xemples et compl\'{e}ments},
        date={1968},
   booktitle={Dix expos\'{e}s sur la cohomologie des sch\'{e}mas},
      series={Adv. Stud. Pure Math.},
      volume={3},
   publisher={North-Holland, Amsterdam},
       pages={88\ndash 188},
      review={\MR{244271}},
}

\bib{hovey_model_1999}{book}{
      author={Hovey, Mark},
       title={Model {{Categories}}},
      series={Mathematical Surveys and Monographs},
   publisher={{American Mathematical Society}},
     address={Providence, RI},
        date={1999},
      volume={63},
        ISBN={0-8218-1359-5},
}

\bib{hoyois_six_2017}{article}{
      author={Hoyois, Marc},
       title={The six operations in equivariant motivic homotopy theory},
        date={2017-01},
        ISSN={0001-8708},
     journal={Advances in Mathematics},
      volume={305},
       pages={197\ndash 279},
}

\bib{knus_book_1998-1}{book}{
      author={Knus, Max-Albert},
      author={Merkurjev, Alexander},
      author={Rost, Markus},
      author={Tignol, Jean-Pierre},
       title={The book of involutions},
      series={American Mathematical Society Colloquium Publications},
   publisher={{American Mathematical Society, Providence, RI}},
        date={1998},
      volume={44},
        ISBN={978-0-8218-0904-4},
      review={\MR{1632779}},
}

\bib{knus_quadratic_1991}{book}{
      author={Knus, Max-Albert},
       title={Quadratic and hermitian forms over rings},
      series={Grundlehren der mathematischen Wissenschaften},
   publisher={{Springer-Verlag}},
     address={Berlin ; New York},
        date={1991},
      number={294},
        ISBN={978-0-387-52117-6 978-3-540-52117-4},
}

\bib{knus_azumaya_1990}{article}{
      author={Knus, M.-A.},
      author={Parimala, R.},
      author={Srinivas, V.},
       title={Azumaya algebras with involutions},
        date={1990},
        ISSN={0021-8693},
     journal={J. Algebra},
      volume={130},
      number={1},
       pages={65\ndash 82},
      review={\MR{1045736}},
}

\bib{may_remarks_1990}{article}{
      author={May, J.~P.},
       title={Some remarks on equivariant bundles and classifying spaces},
        date={1990},
        ISSN={0303-1179},
     journal={Ast{\'e}risque},
      number={191},
       pages={7, 239\ndash 253},
        note={International Conference on Homotopy Theory (Marseille-Luminy,
  1988)},
      review={\MR{1098973}},
}

\bib{may_simplicial_1992}{book}{
      author={May, J.~P.},
       title={Simplicial objects in algebraic topology},
      series={Chicago Lectures in Mathematics},
   publisher={{University of Chicago Press}},
     address={Chicago, IL},
        date={1992},
        ISBN={0-226-51181-2},
        note={Reprint of the 1967 original},
}

\bib{may_equivariant_1996}{book}{
      author={May, J.~P.},
       title={Equivariant homotopy and cohomology theory},
      series={CBMS Regional Conference Series in Mathematics},
   publisher={{Published for the Conference Board of the Mathematical Sciences,
  Washington, DC}},
        date={1996},
      volume={91},
        ISBN={0-8218-0319-0},
        note={With contributions by M. Cole, G. Comeza{\~n}a, S. Costenoble, A.
  D. Elmendorf, J. P. C. Greenlees, L. G. Lewis, Jr., R. J. Piacenza, G.
  Triantafillou, and S. Waner},
}

\bib{merling_equivariant_2017}{article}{
      author={Merling, Mona},
       title={Equivariant algebraic {{$K$}}-theory of {{$G$}}-rings},
        date={2017-01},
        ISSN={0025-5874},
     journal={Mathematische Zeitschrift},
      volume={285},
      number={3},
       pages={1205\ndash 1248},
}

\bib{milne_etale_1980}{book}{
      author={Milne, James~S},
       title={{\'E}tale cohomology},
      series={Princeton Mathematical Series},
   publisher={{Princeton University Press}},
     address={Princeton, N.J.},
        date={1980},
      volume={33},
        ISBN={0-691-08238-3},
}

\bib{mac_lane_sheaves_1992}{book}{
      author={Mac~Lane, Saunders},
      author={Moerdijk, Ieke},
       title={Sheaves in geometry and logic},
      series={Universitext},
   publisher={{Springer-Verlag}},
        date={1992},
        ISBN={978-0-387-97710-2},
}

\bib{molnar_semi-direct_1977}{article}{
      author={Molnar, Richard~K},
       title={Semi-direct products of {{Hopf}} algebras},
        date={1977-07},
        ISSN={0021-8693},
     journal={Journal of Algebra},
      volume={47},
      number={1},
       pages={29\ndash 51},
}

\bib{ojanguren_71}{article}{
      author={Ojanguren, M.},
      author={Sridharan, R.},
       title={Cancellation of {A}zumaya algebras},
        date={1971},
        ISSN={0021-8693},
     journal={J. Algebra},
      volume={18},
       pages={501\ndash 505},
  url={http://dx.doi.org.ezproxy.library.ubc.ca/10.1016/0021-8693(71)90133-5},
      review={\MR{0276271}},
}

\bib{piacenza_transfer_1984}{article}{
      author={Piacenza, Robert},
       title={Transfer in {{Generalized Sheaf Cohomology}}},
        date={1984-04},
        ISSN={0002-9939},
     journal={Proceedings of the American Mathematical Society},
      volume={90},
      number={4},
       pages={653\ndash 656},
}

\bib{parimala_92}{article}{
      author={Parimala, R.},
      author={Srinivas, V.},
       title={Analogues of the {B}rauer group for algebras with involution},
        date={1992},
        ISSN={0012-7094},
     journal={Duke Math. J.},
      volume={66},
      number={2},
       pages={207\ndash 237},
         url={http://dx.doi.org/10.1215/S0012-7094-92-06606-3},
      review={\MR{1162189}},
}

\bib{reiner_maximal_1975-1}{book}{
      author={Reiner, Irving},
       title={Maximal orders},
      series={L.M.S. monographs},
   publisher={{Academic Press}},
     address={London ; New York},
        date={1975},
      number={5},
        ISBN={978-0-12-586650-7},
}

\bib{riehm_orthogonal_rep_2001}{article}{
      author={Riehm, Carl~R.},
       title={Orthogonal, symplectic and unitary representations of finite
  groups},
        date={2001},
        ISSN={0002-9947},
     journal={Trans. Amer. Math. Soc.},
      volume={353},
      number={12},
       pages={4687\ndash 4727},
  url={http://dx.doi.org.ezproxy.library.ubc.ca/10.1090/S0002-9947-01-02807-0},
      review={\MR{1852079}},
}

\bib{saltman_azumaya_1978}{article}{
      author={Saltman, David~J.},
       title={Azumaya algebras with involution},
        date={1978-06},
        ISSN={0021-8693},
     journal={Journal of Algebra},
      volume={52},
      number={2},
       pages={526\ndash 539},
}

\bib{scharlau_quadratic_1985}{book}{
      author={Scharlau, Winfried},
       title={Quadratic and {{Hermitian Forms}}},
      series={Grundlehren der mathematischen Wissenschaften},
   publisher={{Springer Berlin Heidelberg}},
     address={Berlin, Heidelberg},
        date={1985},
      volume={270},
        ISBN={978-3-642-69973-3 978-3-642-69971-9},
}

\bib{shulman_parametrized_2008}{article}{
      author={Shulman, Michael~A.},
       title={Parametrized spaces model locally constant homotopy sheaves},
        date={2008-01},
        ISSN={0166-8641},
     journal={Topology and its Applications},
      volume={155},
      number={5},
       pages={412\ndash 432},
}

\bib{totaro_chow_1999}{incollection}{
      author={Totaro, Burt},
       title={The {{Chow}} ring of a classifying space},
        date={1999},
   booktitle={Algebraic $k$-theory ({{Seattle}}, {{WA}}, 1997)},
      series={Proc. Sympos. Pure Math.},
      volume={67},
   publisher={{Amer. Math. Soc.}},
     address={Providence, RI},
       pages={249\ndash 281},
}

\bib{vistoli_1989_intersection_theory}{article}{
      author={Vistoli, Angelo},
       title={Intersection theory on algebraic stacks and on their moduli
  spaces},
        date={1989},
        ISSN={0020-9910},
     journal={Invent. Math.},
      volume={97},
      number={3},
       pages={613\ndash 670},
         url={https://doi.org/10.1007/BF01388892},
      review={\MR{1005008}},
}

\bib{voight_2011}{article}{
      author={Voight, John},
       title={Characterizing quaternion rings over an arbitrary base},
        date={2011},
        ISSN={0075-4102},
     journal={J. Reine Angew. Math.},
      volume={657},
       pages={113\ndash 134},
  url={http://dx.doi.org.ezproxy.library.ubc.ca/10.1515/CRELLE.2011.054},
      review={\MR{2824785}},
}

\bib{voight_2011_standard_involution}{article}{
      author={Voight, John},
       title={Rings of low rank with a standard involution},
        date={2011},
        ISSN={0019-2082},
     journal={Illinois J. Math.},
      volume={55},
      number={3},
       pages={1135\ndash 1154 (2013)},
         url={http://projecteuclid.org/euclid.ijm/1369841800},
      review={\MR{3069299}},
}

\bib{verschoren_98}{incollection}{
      author={Verschoren, A.},
      author={Vidal, C.},
       title={An alternative approach to involutive {B}rauer groups},
        date={1998},
   booktitle={Rings, {H}opf algebras, and {B}rauer groups
  ({A}ntwerp/{B}russels, 1996)},
      series={Lecture Notes in Pure and Appl. Math.},
      volume={197},
   publisher={Dekker, New York},
       pages={311\ndash 324},
      review={\MR{1615912}},
}

\end{biblist}
\end{bibdiv}
